\documentclass[hidelinks]{article}
\usepackage
[
a4paper,
left=2.5cm,
right=2.5cm,
top=2.5cm,
bottom=2.5cm,
]
{geometry}

\usepackage{amsmath}
\usepackage{amsfonts}
\usepackage{amsthm}
\usepackage{amssymb}
\usepackage{latexsym}
\usepackage{enumitem} 
\usepackage{mathrsfs}
\usepackage{mathtools}
\usepackage[]{algorithm2e}
\usepackage{hyperref}
\usepackage{xcolor}
\usepackage{setspace}
\usepackage[date=year,doi=false,isbn=false,url=false,eprint=false,citestyle=numeric-comp]{biblatex}
\usepackage{bbm}
\renewbibmacro{in:}{}

\def\R{\mathbb{R}}

\DeclareMathOperator{\Var}{Var}
\DeclareMathOperator{\Cov}{Cov}

\DeclareMathOperator{\supp}{supp}
\DeclareMathOperator{\spn}{span}

\DeclareMathOperator{\sgn}{sign}

\newtheorem{theorem}{Theorem}
\newtheorem{proposition}{Proposition}
\newtheorem{lemma}{Lemma}
\newtheorem{corollary}{Corollary}
\newtheorem{definition}{Definition}

\theoremstyle{definition}
\newtheorem{remark}{Remark}

\title{Minimax rates for sparse signal detection under correlation}
\author{Subhodh Kotekal and Chao Gao\thanks{The research of C. Gao is partially supported by NSF CAREER award DMS-1847590 and NSF grant CCF-1934931.} \\ \\ \textit{University of Chicago}}
\date{}

\begin{document}
    \maketitle
    \begin{abstract}
        We fully characterize the nonasymptotic minimax separation rate for sparse signal detection in the Gaussian sequence model with \(p\) equicorrelated observations, generalizing a result of Collier, Comminges, and Tsybakov \cite{collierMinimaxEstimationLinear2017}. As a consequence of the rate characterization, we find that strong correlation is a blessing, moderate correlation is a curse, and weak correlation is irrelevant. Moreover, the threshold correlation level yielding a blessing exhibits phase transitions at the \(\sqrt{p}\) and \(p-\sqrt{p}\) sparsity levels. We also establish the emergence of new phase transitions in the minimax separation rate with a subtle dependence on the correlation level. Additionally, we study group structured correlations and derive the minimax separation rate in a model including multiple random effects. The group structure turns out to fundamentally change the detection problem from the equicorrelated case and different phenomena appear in the separation rate.
    \end{abstract}

    \section{Introduction}

    A broad research program involving the characterization of fundamental limits of testing global null hypotheses against structured alternatives has witnessed vigorous and fruitful activity \cite{liuMinimaxRatesSparse2021, ingsterNonparametricGoodnessoffitTesting2003,hallInnovatedHigherCriticism2010a,donohoHigherCriticismDetecting2004a, caiTwosampleTestHigh2014,collierMinimaxEstimationLinear2017,ingsterDetectionBoundarySparse2010a,mukherjeeGlobalTestingSparse2018,mukherjeeDetectionThresholdsVmodel2018, xuSignalDetectionDegree2021,butuceaDetectionSparseSubmatrix2013a,baraudNonasymptoticMinimaxRates2002b,ingsterMinimaxSignalDetection2012,dumbgenMultiscaleTestingQualitative2001}. The program can be traced back to the seminal work of Ingster (see \cite{ingsterNonparametricGoodnessoffitTesting2003} for an overview) which established the minimax hypothesis testing framework and settled a litany of fundamental questions in the setting of Gaussian models. Attention has largely been focused on testing in signal plus Gaussian noise models such as sequence, regression, and matrix models; independence of the noise is a crucial ingredient to the derivation of minimax separation rates in existing work. Non-Gaussian models exhibiting dependence have been studied recently \cite{mukherjeeDetectionThresholdsVmodel2018, mukherjeeGlobalTestingSparse2018,fromontAdaptiveTestsHomogeneity2011,xuSignalDetectionDegree2021} and striking results have been obtained. Consideration of minimax testing problems in Gaussian models with dependence is quite limited \cite{enikeevaBumpDetectionPresence2020,liuMinimaxRatesSparse2021,hallInnovatedHigherCriticism2010a, caiTwosampleTestHigh2014}. In an effort to address this dearth in the literature, we study the fundamental limits of sparse signal detection in Gaussian sequence mixed models.

    Linear mixed models have remained a mainstay throughout the history of applied statistical practice and are often the first tool the statistician reaches for when faced with correlated data in structured settings. Due to their ubiquity in applications, linear mixed models are arguably the most natural setting one ought to first consider when investigating the effect of dependence on detection limits. An oft-repeated aphorism among practitioners is that mixed models enable a ``borrowing of information" across observations and are thus advantageous when engaging with statistical tasks. This purported advantage often manifests in practical data analysis; parameter estimates and inference conclusions appear to be more reasonable compared to results from alternative fixed effect models. This piece of statistical folklore is frequently offered as justification for the use of a mixed model or a particular experimental design. In a sense, the present paper investigates the folklore in the context of sparse signal detection. 
    
    To study the essence of the phenomenon, we initially consider a simple signal plus noise model with equicorrelated observations,
    \begin{equation}\label{model:additive_model}
        X_i = \theta_i + \sqrt{\gamma} W + \sqrt{1-\gamma} Z_i
    \end{equation}
    for \(1 \leq i \leq p\) where \(\theta \in \R^p\) denotes the fixed effects, \(\gamma \in [0, 1]\) denotes the correlation level, and \(W, Z_1,...,Z_p\) are independent and identically distributed standard Gaussian random variables. Conventionally, \(W\) is referred to as the random effect; its presence immediately implies the observations are equicorrelated. The model implies that the vector of observations is marginally distributed as   
    \begin{equation}\label{model:single_random_effect}
        X \sim N\left(\theta, (1-\gamma)I_p + \gamma \mathbf{1}_p\mathbf{1}_p^\intercal\right)  
    \end{equation}
    where \(I_p \in \R^{p \times p}\) denotes the identity matrix and \(\mathbf{1}_p \in \R^p\) denotes the vector with all entries equal to one. To describe the sparse signal detection problem, for \(1 \leq s \leq p\) and \(\varepsilon > 0\) define the parameter space 
    \begin{equation}\label{parameter:alternative}
        \Theta(p, s, \varepsilon) := \left\{ \theta \in \R^p : ||\theta|| \geq \varepsilon \text{ and } ||\theta||_0 \leq s \right\}.
    \end{equation}
    Having observed \(X\) from the model (\ref{model:single_random_effect}), the sparse signal detection problem is the problem of testing the hypotheses
    \begin{align}
        H_0 &: \theta = 0, \label{problem:test_equicorrelated_1}\\
        H_1 &: \theta \in \Theta(p, s, \varepsilon) \label{problem:test_equicorrelated_2}.
    \end{align}
    The sparsity level is denoted by \(s\) and \(\varepsilon\) controls the separation between the alternative and null hypotheses. Intuitively, the testing problem becomes more difficult as \(\varepsilon\) decreases. The purpose of this paper is to characterize a fundamental quantity known as the minimax separation rate. In preparation for its definition, let us first define the minimax testing risk 
    \begin{equation}\label{def:testing_risk}
        \mathcal{R}(\varepsilon) := \inf_{\varphi} \left\{ P_{0, \gamma}\{\varphi = 1\} + \sup_{\theta \in \Theta(p, s, \varepsilon)} P_{\theta, \gamma} \{\varphi = 0\} \right\},
    \end{equation}
    where the infimum runs over all tests (measurable functions) \(\varphi : \R^p \to \{0, 1\}\). Here, \(P_{\theta, \gamma}\) denotes the probability measure associated with the distribution \(N\left(\theta, (1-\gamma)I_p + \gamma \mathbf{1}_p\mathbf{1}_p^\intercal\right)\). 
    \begin{definition}\label{def:separation_rate}
        We say \(\varepsilon^* = \varepsilon^*(p, s, \gamma)\) is the \text{minimax separation rate} for the hypothesis testing problem (\ref{problem:test_equicorrelated_1})-(\ref{problem:test_equicorrelated_2}) with parameters \((p, s, \gamma)\) if
        \begin{enumerate}[label=(\roman*)]
            \item for every \(\eta \in (0, 1)\) there exists a constant \(C_\eta > 0\) depending only on \(\eta\) such that \(C > C_\eta\) implies \(\mathcal{R}(C\varepsilon^*) \leq \eta\),
            \item for every \(\eta \in (0, 1)\) there exists a constant \(c_\eta > 0\) depending only on \(\eta\) such that \(0 < c < c_\eta\) implies \(\mathcal{R}(c\varepsilon^*) \geq 1-\eta\).
        \end{enumerate}
    \end{definition}
    The minimax separation rate \(\varepsilon^*\) is a nonasymptotic quantity defined for every configuration of \((p, s, \gamma)\) such that \(1 \leq s \leq p\) are natural numbers and \(\gamma \in [0, 1]\). Note it is unique only up to multiplication by universal constants. The minimax separation rate characterizes the fundamental difficulty (up to absolute factors) of the testing problem. Concretely, it denotes the order of the signal magnitude \(||\theta||\) which is necessary and sufficient to distinguish between the null and alternative hypotheses with arbitrarily small testing risk. 

    \subsection{Related work}\label{section:related_work}
    In the simplest setting with \(s = p\) and \(\gamma = 0\), it is well known (see for example \cite{baraudNonasymptoticMinimaxRates2002b}) that
    \begin{equation*}
        \varepsilon^*(p, p, 0)^2 \asymp \sqrt{p}.    
    \end{equation*}
    The test which rejects the null hypothesis when \(||X||^2 > p + C\sqrt{p}\) for a suitably chosen positive constant \(C\) is rate-optimal. Considering the general sparsity case in the independent setting, Collier et al. \cite{collierMinimaxEstimationLinear2017} generalized the results of \cite{baraudNonasymptoticMinimaxRates2002b} and showed
    \begin{equation}\label{rate:independent}
        \varepsilon^*(p, s, 0)^2 \asymp 
        \begin{cases}
            s \log\left(1+\frac{p}{s^2}\right) &\text{if } s < \sqrt{p}, \\
            \sqrt{p} &\text{if } s \geq \sqrt{p}.
        \end{cases}
    \end{equation}
    The minimax separation rate exhibits a phase transition at \(s = \sqrt{p}\). When \(s < \sqrt{p}\), detection is possible for essentially smaller signal magnitudes compared to the dense case \((s = p)\). However, sparsity offers no help in detection once \(s \geq \sqrt{p}\).
    
    In the case \(s = p\) with \(\gamma \in [0, 1)\), a simple diagonalization argument shows that
    \begin{equation*}
        \varepsilon^*(p, p, \gamma)^2 \asymp ||(1-\gamma)I_p + \gamma \mathbf{1}_p\mathbf{1}_p^\intercal||_F. 
    \end{equation*}
    Noting \(||(1-\gamma)I_p + \gamma \mathbf{1}_p\mathbf{1}_p^\intercal||_F \asymp (1-\gamma)\sqrt{p} + (1-\gamma + \gamma p)\), we see that \(\varepsilon^*(p, p, \gamma)=\omega(\varepsilon^*(p, p, 0))\) when \(\gamma = \omega(\frac{1}{\sqrt{p}})\). It is also readily seen that \(\varepsilon^*(p, p, \gamma) \asymp \varepsilon^*(p, p, 0)\) when \(\gamma \lesssim \frac{1}{\sqrt{p}}\) and \(1-\gamma \asymp 1\). In other words, nontrivial correlation is a curse when \(s = p\).
    
    To elaborate on the diagonalization argument, the spectral theorem asserts one can write \((1-\gamma)I_p + \gamma \mathbf{1}_p\mathbf{1}_p^\intercal = Q\Lambda Q^\intercal\) where \(Q \in \R^{p \times p}\) is an orthogonal matrix and \(\Lambda \in \R^{p \times p}\) is a diagonal matrix with diagonal entries given by the eigenvalues of \((1-\gamma)I_p + \gamma \mathbf{1}_p\mathbf{1}_p^\intercal\) in descending order. Thus \(QX \sim N(Q\theta, \Lambda)\). Since \(||Q\theta||^2 = ||\theta||^2\), it immediately follows in the case \(s = p\) that the testing problem (\ref{problem:test_equicorrelated_1})-(\ref{problem:test_equicorrelated_2}) is equivalent to testing \(H_0 : Q\theta = 0\) against \(H_1 : ||Q\theta||^2 \geq \varepsilon^2\). A standard lower bound construction and an analysis of the test which rejects the null hypothesis when \(||QX||^2 > p + C||(1-\gamma)I_p + \gamma \mathbf{1}_{p}\mathbf{1}_p^\intercal||_F\) for a suitable positive constant \(C\) establishes the separation rate. It is readily seen that the same argument furnishes the minimax separation rate for a general covariance matrix \(\Sigma\) when \(s = p\), namely \((\varepsilon^*)^2 \asymp ||\Sigma||_F\) (see for example \cite{liuMinimaxRatesSparse2021}). 

    Notably, the diagonalization argument employed in the \(s = p\) case fails to carry over to the \(s < p\) case. Even though the null hypothesis \(\theta = 0\) is equivalent to \(Q\theta = 0\), it need not be the case when \(s < p\) that \(\theta \in \Theta(p, s, \varepsilon)\) is equivalent to \(Q\theta \in \Theta(p, s, \varepsilon)\). Specifically the sparsity is not preserved, in that \(||\theta||_0 \leq s\) need not be equivalent to \(||Q\theta||_0 \leq s\). Careful study of the interaction between the covariance matrix and the sparsity is needed.

    With the desire to understand the \(s < p\) case, some existing work in the literature has investigated how covariance affects sparse signal detection in other settings. Hall and Jin \cite{hallInnovatedHigherCriticism2010a} consider the observational model \(Y \sim N(\theta, \Omega_p^{-1})\) in which the precision matrix \(\Omega_p\) exhibits polynomial off-diagonal decay. In a Bayesian setup of the testing problem in which the support of \(\theta\) is drawn uniformly at random and coordinates in the support are set equal to an elevated value \(\mu\), Hall and Jin derive a detection boundary in an asymptotic setting with sparsity calibration \(s = p^{1-\beta}\) with \(\beta \in (0, 1)\) denoting a sparsity parameter. The derived detection boundary sharply describes the configurations of \((\mu, \beta)\) for which consistent detection is possible, and the boundary depends on limiting functionals of \(\Omega_p\) as \(p\) grows to infinity. Cai et al. \cite{caiTwosampleTestHigh2014} consider a two-sample version of the testing problem with common precision matrix, and prove asymptotic optimality of a proposed test under a similar Bayesian setup and a sparsity assumption on the precision matrix. The setting of \cite{caiTwosampleTestHigh2014} is also asymptotic and uses a different separation metric defining the alternative. 
    
    The present paper differs from \cite{hallInnovatedHigherCriticism2010a,caiTwosampleTestHigh2014} in significant ways. First, our results are nonasymptotic. Second, our results are minimax in nature. Third, the composite alternative is separated from the null hypothesis in Euclidean norm. Finally, the assumptions of polynomial off-diagonal decay or sparsity of the precision matrix needed in \cite{hallInnovatedHigherCriticism2010a,caiTwosampleTestHigh2014} are not satisfied in the problems we study. 

    Let us return to the setting of testing (\ref{problem:test_equicorrelated_1})-(\ref{problem:test_equicorrelated_2}) with observation (\ref{model:single_random_effect}). Recently, Liu et al. \cite{liuMinimaxRatesSparse2021} showed that if \(s \leq p^{1/5}\) and \(p\) is larger than some absolute constant, then
    \begin{equation*}
        \varepsilon^*(p, s, \gamma)^2 \asymp (1-\gamma) s \log\left(1 + \frac{p}{s^2}\right). 
    \end{equation*}
    In particular, the minimax separation rate satisfies \(\varepsilon^*(p, s, \gamma) = o(\varepsilon^*(p, s, 0))\) when \(1-\gamma = o(1)\) and \(\varepsilon^*(p, s, \gamma) \asymp \varepsilon^*(p, s, 0)\) when \(1-\gamma \asymp 1\). In other words, Liu and coauthors show that correlation is a blessing when \(s \leq p^{1/5}\). In stark contrast, correlation is a curse when \(s = p\) as discussed above. The aforementioned rates suggest the existence of a striking relationship between ambient dimension, sparsity, and correlation. A full characterization of \(\varepsilon^*(p, s, \gamma)\) for all configurations of \(p, s,\) and \(\gamma\) has not yet been established. The purpose of the present paper is to fill this gap in the literature.

    \subsection{Main contributions}\label{section:main_contributions}
    Our first main contribution is Theorem \ref{thm:minimax_rate} in which we obtain a complete characterization of the minimax separation rate, that is,
    \begin{equation*}
        \varepsilon^*(p, s, \gamma)^2 \asymp 
        \begin{cases}
            (1-\gamma) s \log\left(1 + \frac{p}{s^2}\right) &\text{if } s < \sqrt{p}, \\
            (1-\gamma)\sqrt{p} + \frac{(1-\gamma)p^{3/2}}{p-s} \wedge \left(1-\gamma+\gamma p\right) &\text{if } \sqrt{p} \leq s \leq p-\sqrt{p}, \\
            (1-\gamma)\sqrt{p} + (1-\gamma)p\log\left(1 + \frac{p}{(p-s)^2}\right) \wedge \left(1-\gamma+\gamma p\right) &\text{if } p-\sqrt{p} < s \leq p.
        \end{cases}
    \end{equation*}
    This result generalizes the previously known results discussed in Section \ref{section:related_work} and reveals a number of novel phenomena. A direct comparison of \(\varepsilon^*(p, s, \gamma)^2\) and \(\varepsilon^*(p, s, 0)^2\) will reveal the blessing of strong correlation, the curse of moderate correlation, and the irrelevance of weak correlation. Furthermore, it will be seen that the correlation level \(\gamma\) has a subtle effect on the emergence of new phase transitions. We discuss the various correlation regimes in turn.
    
    \begin{itemize}
        \item \textbf{Weak correlation:} 
        Suppose \(\gamma \lesssim \frac{1}{\sqrt{p}}\). Then
        \begin{equation*}
            \varepsilon^*(p, s, \gamma)^2 \asymp 
            \begin{cases}
                s \log\left(1 + \frac{p}{s^2}\right) &\text{if } s < \sqrt{p}, \\
                \sqrt{p} &\text{if } s \geq \sqrt{p}. 
            \end{cases}
        \end{equation*}
        Notice \(\varepsilon^*(p, s, \gamma)^2 \asymp \varepsilon^*(p, s, 0)^2\) and so there is no difference between the settings of no correlation and weak correlation in terms of the minimax separation rate. In other words, weak correlation is irrelevant. 

        \item \textbf{Moderate correlation:} 
        Suppose \(\frac{c}{\sqrt{p}} \leq \gamma \leq \frac{1}{2} \vee \frac{c}{\sqrt{p}}\) for an absolute constant \(c > 1\). Then  
        \begin{equation*}
            \varepsilon^*(p, s, \gamma)^2 \asymp 
            \begin{cases}
                s \log\left(1 + \frac{p}{s^2}\right) &\text{if } s < \sqrt{p}, \\
                \frac{p^{3/2}}{p-s} &\text{if } \sqrt{p} \leq s \leq p - \frac{\sqrt{p}}{\gamma}, \\
                \gamma p &\text{if } p-\frac{\sqrt{p}}{\gamma} < s \leq p.
            \end{cases}
        \end{equation*}
        The separation rate exhibits two phase transitions when the correlation is moderate. The phase transition at \(s = \sqrt{p}\) carries over from the independent setting, but the phase transition at \(s = p - \frac{\sqrt{p}}{\gamma}\) is novel. Note that \(\varepsilon^*(p, s, \gamma)^2 \asymp \varepsilon^*(p, s, 0)^2\) when \(p-s \asymp p\). However, it is readily seen that \(\varepsilon^*(p, s, \gamma)^2 = \omega(\varepsilon^*(p, s, 0)^2)\) for \(p-s = o(p)\) when \(\gamma = \omega(p^{-1/2})\). In other words, moderate correlation is a curse. 
    
        \item \textbf{Strong correlation:} 
        Suppose \(\frac{c}{\sqrt{p}} \vee \frac{1}{2} \leq \gamma \leq 1\) for an absolute constant \(c > 1\). Then
        \begin{equation*}
            \varepsilon^*(p, s, \gamma)^2 \asymp 
            \begin{cases}
                (1-\gamma) s \log\left(1 + \frac{p}{s^2}\right) &\text{if } s < \sqrt{p}, \\
                \frac{(1-\gamma)p^{3/2}}{p-s} &\text{if } \sqrt{p} \leq s \leq p - \sqrt{p}, \\
                (1-\gamma)p \log\left(1 + \frac{p}{(p-s)^2}\right) &\text{if } p - \sqrt{p} < s < p - e^{-(1-\gamma)^{-1}}\sqrt{p}, \\
                p &\text{if } p - e^{-(1-\gamma)^{-1}}\sqrt{p} \leq s \leq p. 
            \end{cases}
        \end{equation*}
        The separation rate now exhibits three phase transitions, namely at the sparsity levels \(\sqrt{p}, p - \sqrt{p}\), and \(e^{-(1-\gamma)^{-1}}\sqrt{p}\). The location of this last phase transition exhibits an exponential rate in \((1-\gamma)^{-1}\). In contrast, polynomial dependence on \(\gamma\) was exhibited in the moderate correlation regime discussed in the previous enumerated point. A direct comparison of \(\varepsilon^*(p, s, \gamma)^2\) to \(\varepsilon^*(p, s, 0)^2\) reveals that, for a given sparsity level, suitably strong correlation is a blessing. A quantitative description is given in Remark \ref{remark:blessing}. Conversely, insufficiently strong correlation turns out to be a curse for signal detection (see Remark \ref{remark:curse}).
    \end{itemize}
    The last enumerated point above hides a curiosity when the correlation is very strong. Notably, a discontinuity at \(s = p\) appears in the minimax separation rate when \(1-\gamma = o\left(\frac{1}{\log(ep)}\right)\). See Remark \ref{remark:discontinuity} for additional discussion.
   
    In Section \ref{section:multiple_random_effects}, we study sparse signal detection in the setting of group structured correlations. Specifically, we consider a mixed model with \(R\) random effects. Each of \(R\) groups contains \(\frac{p}{R}\) observations which are equicorrelated with correlation \(\gamma\). Observations in different groups are independent. Such mixed models are employed when observations exhibit correlations in clustering structures. Leveraging our insight from the \(R = 1\) case, we obtain the minimax separation rate for general \(R\). Salient features of the separation rate are discussed in Section \ref{section:multiple_random_effects}.

    The testing procedure we construct en route to proving Theorem \ref{thm:minimax_rate} requires knowledge of the sparsity level. In Section \ref{section:sparsity_adaptive}, it is shown that an adaptive procedure which achieves the minimax rate can be obtained by scanning over all sparsity levels. In Section \ref{section:rank_one_correlation}, we consider the sparse signal detection problem in mixed models exhibiting different correlation patterns. Specifically, we consider covariance matrices of the form \((1-\gamma)I_p + \gamma vv^\intercal\) for \(v \in \R^p\) with \(||v||=\sqrt{p}\). We obtain a partial characterization of the minimax separation rate.

    \subsection{Notation}
    This section defines frequently used notation. For a natural number \(n\), denote \([n] := \{1,..., n\}\). For \(a, b \in \R\) the notation \(a \lesssim b\) denotes the existence of a universal constant \(c > 0\) such that \(a \leq cb\). The notation \(a \gtrsim b\) is used to denote \(b \lesssim a\). Additionally \(a \asymp b\) denotes \(a \lesssim b\) and \(a \gtrsim b\). The symbol \(:=\) is frequently used when defining a quantity or object. Furthermore, we frequently use \(a \vee b := \max(a, b)\) and \(a \wedge b := \min(a, b)\). We generically use the notation \(\mathbbm{1}_A\) to denote the indicator function for an event \(A\). For a vector \(v \in \R^p\) and a subset \(S \subset [p]\), we sometimes use the notation \(v_S \in \R^{p}\) to denote the vector with coordinate \(i\) equal to \(v_i\) if \(i \in S\) and zero otherwise. In other cases, the notation \(v_S \in \R^{|S|}\) denotes the subvector of dimension \(|S|\) corresponding to the coordinates in \(S\). The context will clarify between the two different notational uses of \(v_S\). In particular, we will frequently make use of the notation \(\mathbf{1}_S := (\mathbf{1}_p)_S\) in this way. Additionally, \(||v||_0 := \sum_{i=1}^{p} \mathbbm{1}_{\{v_i \neq 0\}}\), \(||v||_1 := \sum_{i=1}^{p} |v_i|\), and \(||v||^2 := \sum_{i=1}^{p} v_i^2\). We also frequently make use of the notation \(\bar{v} = p^{-1}\sum_{i=1}^{p} v_i\), though in some cases the notation is used to denote something specified in advance. For a vector \(v \in \R^p\), the notation \(\supp(v)\) refers to the support of \(v\), namely the set \(\{i \in [p] : v_i \neq 0\}\). For two probability measures \(P\) and \(Q\) on a measurable space \((\mathcal{X}, \mathcal{A})\), the total variation distance is defined as \(d_{TV}(P, Q) := \sup_{A \in \mathcal{A}} |P(A) - Q(A)|\). If \(P\) is absolutely continuous with respect to \(Q\), then the \(\chi^2\)-divergence is defined as \(\chi^2(P||Q) := \int_{\mathcal{X}} \left(\frac{dP}{dQ} - 1\right)^2 \, dQ\). For sequences \(\{a_k\}_{k=1}^{\infty}\) and \(\{b_k\}_{k=1}^{\infty}\), the notation \(a_k = o(b_k)\) denotes \(\lim_{k \to \infty} \frac{a_k}{b_k} = 0\) and the notation \(a_k = \omega(b_k)\) is used to denote \(b_k = o(a_k)\). For a matrix \(A \in \R^{m \times n}\), the Frobenius norm of \(A\) is denoted as \(||A||_F = \sqrt{\sum_{i=1}^{m}\sum_{j=1}^{n} a_{ij}^2}\). 
    
    \section{Equicorrelated observations}
    In this section, the testing problem (\ref{problem:test_equicorrelated_1})-(\ref{problem:test_equicorrelated_2}) with observation (\ref{model:single_random_effect}) is studied, culminating in a complete characterization of the minimax separation rate.

    \subsection{Perfect correlation}\label{section:perfect_correlation}
    It is convenient to consider the extreme case \(\gamma = 1\) separately from the case \(\gamma \in [0, 1)\). Indeed, the covariance matrix \((1-\gamma)I_p + \gamma \mathbf{1}_p\mathbf{1}_p^\intercal\) is invertible for \(\gamma \in [0, 1)\), but noninvertible for \(\gamma = 1\). When \(\gamma = 1\), the minimax separation rate is easily established. 
    \begin{proposition}\label{prop:perfect_correlation}
        If \(1 \leq s \leq p\), then 
        \begin{equation*}
            \varepsilon^*(p, s, 1)^2 \asymp 
            \begin{cases}
                0 &\text{if } s < p, \\
                p &\text{if } s = p.
            \end{cases}
        \end{equation*}
    \end{proposition}
    
    After some thought, both the degeneracy for \(s < p\) and the harsh discontinuity at \(s = p\) in the minimax separation rate is unsurprising. To see this, consider from (\ref{model:additive_model}) that the observation can be written as \(X = \theta + W\mathbf{1}_p\) when \(\gamma = 1\). Projecting to orthogonal subspaces yields \(X - \bar{X}\mathbf{1}_p = \theta - \bar{\theta}\mathbf{1}_p\) and \(\sqrt{p}\bar{X} = \sqrt{p}\bar{\theta} + \sqrt{p}W\). Notice that \(X - \bar{X}\mathbf{1}_p\) is almost surely equal to \(\theta - \bar{\theta}\mathbf{1}_p\), that is to say all of the noise has been removed! Since \(s < p\) and \(||\theta||_0 \leq s\) implies that \(\theta \not \in \spn\{\mathbf{1}_p\}\), it immediately follows that the test which rejects the null hypothesis when \(X - \bar{X}\mathbf{1}_p\) is nonzero achieves zero testing risk. Hence the separation rate is degenerate for \(s < p\). When \(s = p\), the alternative hypothesis now includes some \(\theta\) that lie in \(\spn\{\mathbf{1}_p\}\) and so the statistic \(X - \bar{X}\mathbf{1}_p\) cannot be relied on exclusively. Rather, the potential information in the remaining piece \(\sqrt{p}\bar{X}\) must be incorporated. The testing procedure which rejects when \(||X||^2\) is suitably large can be used, yielding \(\varepsilon^*(p,p,1)^2 \lesssim p\). The lower bound \(\varepsilon^*(p, p, 1)^2 \gtrsim p\) is easily shown by a reduction to a two-point, simple vs simple hypothesis testing problem (see Section \ref{section:proofs} for details).

    Having dealt with the case \(\gamma = 1\), the majority of the remainder of the paper focuses on pinning down the minimax separation rate when \(\gamma \in [0, 1)\). Formal statements of results will specify that \(\gamma \in [0, 1)\), but remarks and exposition frequently do not and so require the reader to rely on context.

    \subsection{Decomposition into a pair of testing problems}\label{section:problem_decomp}
    Our strategy involves decomposing the testing problem (\ref{problem:test_equicorrelated_1})-(\ref{problem:test_equicorrelated_2}) into a pair of testing problems (denoted as Problem I and Problem II) to be studied separately. Once each of the pair has been studied, those results can be assembled to furnish a characterization for \(\varepsilon^*(p, s, \gamma)\). Letting \(1 \leq s \leq p\) and \(\varepsilon > 0\), define the parameter spaces 
    \begin{align}
        \Theta_{\mathcal{I}}(p, s, \varepsilon) &:= \left\{ \theta \in \R^p : ||\theta - \bar{\theta}\mathbf{1}_p|| \geq \varepsilon \text{ and } ||\theta||_0 \leq s \right\}, \label{space:Theta1}\\
        \Theta_{\mathcal{II}}(p, s, \varepsilon) &:= \left\{ \theta \in \R^p : ||\bar{\theta}\mathbf{1}_p|| \geq \varepsilon \text{ and } ||\theta||_0 \leq s\right\}. \label{space:Theta2} 
    \end{align}

    We have defined \(\Theta_{\mathcal{I}}(p, s, \varepsilon)\) and \(\Theta_{\mathcal{II}}(p, s, \varepsilon)\) in this way as we will employ the same strategy of projecting to orthogonal subspaces employed in Section \ref{section:perfect_correlation}. We now define Problem I and Problem II. 
    
    \begin{definition}[Problem I]
        For \(1 \leq s \leq p\) and with observation (\ref{model:single_random_effect}), Problem I is defined to be the testing problem 
        \begin{align}
            H_0 &: \theta = 0, \label{problem:ProblemI_1}\\
            H_1 &: \theta \in \Theta_{\mathcal{I}}(p, s, \varepsilon_1)\label{problem:ProblemI_2}
        \end{align}
        where \(\varepsilon_1 > 0\). 
    \end{definition}
        
    The minimax testing risk for Problem I is defined as 
    \begin{equation*}
        \mathcal{R}_{\mathcal{I}}(\varepsilon_1) := \inf_{\varphi}\left\{P_{0, \gamma}\{\varphi = 1\} + \sup_{\theta \in \Theta_{\mathcal{I}}(p, s, \varepsilon_1)}P_{\theta, \gamma}\{\varphi = 0\} \right\}   
    \end{equation*}
    where the infimum runs over all tests (measurable functions) \(\varphi : \R^p \to \{0, 1\}\). 
    
    \begin{definition}[Problem II]
        For \(1 \leq s \leq p\) and with observation (\ref{model:single_random_effect}), Problem II is defined to be the testing problem
        \begin{align}
            H_0 &: \theta = 0, \label{problem:ProblemII_1} \\
            H_1 &: \theta \in \Theta_{\mathcal{II}}(p, s, \varepsilon_2) \label{problem:ProblemII_2}
        \end{align}
        where \(\varepsilon_2 > 0\). 
    \end{definition}
    
    The minimax testing risk for Problem II is defined as 
    \begin{equation*}
        \mathcal{R}_{\mathcal{II}}(\varepsilon_2) := \inf_{\varphi} \left\{ P_{0, \gamma} \{\varphi = 1\} + \sup_{\theta \in \Theta_{\mathcal{II}}(p, s, \varepsilon_2)} P_{\theta, \gamma}\{\varphi = 0\} \right\}
    \end{equation*}
    where the infimum runs over all tests (measurable functions) \(\varphi : \R^p \to \{0, 1\}\). The minimax separation rates \(\varepsilon_1^* = \varepsilon_1^*(p, s, \gamma)\) and \(\varepsilon_2^* = \varepsilon_2^*(p, s, \gamma)\) for Problem I and Problem II are defined analogously to the definition of \(\varepsilon^*\). Problem I and Problem II are related to our original testing problem (\ref{problem:test_equicorrelated_1})-(\ref{problem:test_equicorrelated_2}) in a simple yet important way. 

    \begin{lemma}\label{lemma:problem_decomposition}
        If \(1 \leq s \leq p\), then \((\varepsilon^*)^2 \asymp (\varepsilon_1^*)^2 + (\varepsilon_2^*)^2\).
    \end{lemma}

    Lemma \ref{lemma:problem_decomposition} validates the strategy of decomposing the original problem into Problem I and Problem II. In particular, it suffices to characterize the minimax separation rates \(\varepsilon_1^*\) and \(\varepsilon_2^*\) separately. In fact, the problem can be further reduced. 

    \begin{lemma} \label{lemma:problemI_equiv}
        If \(1 \leq s \leq \frac{p}{2}\), then \((\varepsilon^*)^2 \asymp (\varepsilon_1^*)^2\).
    \end{lemma}

    Lemma \ref{lemma:problemI_equiv} asserts that the original testing problem is equivalent to Problem I (with respect to minimax separation rates) for \(1 \leq s \leq \frac{p}{2}\). In conjunction with Lemma \ref{lemma:problem_decomposition}, Lemma \ref{lemma:problemI_equiv} implies that one need only characterize \(\varepsilon_2^*\) for the limited range of sparsity levels \(\frac{p}{2} \leq s \leq p\) for the purpose of characterizing \(\varepsilon^*\). However, a full investigation of Problem I for all sparsity levels is still needed.

    \subsection{Problem I}\label{section:ProblemI}

    The main result of this section is the characterization of the minimax separation rate for Problem I. For ease of notation throughout this section, we set 
    \begin{equation} \label{rate:problemI}
        \psi_1^2 := \begin{cases}
            (1-\gamma) s \log\left(1 + \frac{p}{s^2}\right) &\text{if } s < \sqrt{p}, \\
            (1-\gamma)\sqrt{p} &\text{if } s \geq \sqrt{p}. 
        \end{cases}
    \end{equation}
    \begin{theorem}\label{thm:ProblemI}
        Assume \(p \geq 2\). If \(1 \leq s \leq p\) and \(\gamma \in [0,1)\), then \(\varepsilon_1^*(p, s, \gamma)^2 \asymp \psi_1^2\).
    \end{theorem}
    Note that the case \(p = 1\) is vacuous as \(\Theta_\mathcal{I}(1, 1, \varepsilon)\) is empty for all \(\varepsilon > 0\), and so trivially \(\varepsilon_1^*(1, 1, \gamma) = 0\). To derive the minimax separation rate, we construct matching (up to absolute constants) upper and lower bounds. To obtain the upper bound, we formulate a testing procedure with appropriate testing risk. At the heart of our procedure is a simple transformation of the data which decorrelates the observations. Recall that \(X \sim N(\theta, (1-\gamma)I_p + \gamma \mathbf{1}_p\mathbf{1}_p^\intercal)\). Drawing \(\xi \sim N(0, 1)\) independently from \(X\), consider the transformation 
    \begin{equation*}
        \widetilde{X} := \frac{1}{\sqrt{1-\gamma}}\left(I_p - \frac{1}{p} \mathbf{1}_p\mathbf{1}_p^\intercal\right) X + \frac{\xi}{\sqrt{p}} \mathbf{1}_p. 
    \end{equation*}
    It is readily seen that \(\widetilde{X} \sim N\left( \frac{\theta - \bar{\theta}\mathbf{1}_p}{\sqrt{1-\gamma}}, I_p \right)\). Our testing procedure will be applied to \(\widetilde{X}\) and depends on whether the ``sparse'' regime (\(s < \sqrt{p}\)) or the ``dense'' regime (\(s \geq \sqrt{p}\)) is in force. For the dense regime, we use a \(\chi^2\)-type test applied to the transformed data. Namely, the test
    \begin{equation}\label{test:problemI_chisquare}
        \varphi^{\chi^2}_{r} := \mathbbm{1}_{\left\{||\widetilde{X}||^2 > p + r\sqrt{p}\right\}}
    \end{equation}
    is used, where \(r\) is set by the statistician to attain a desired level of the testing risk. For the sparse regime, a statistic formulated by Collier et al. \cite{collierMinimaxEstimationLinear2017} can be applied to the transformed data. In detail, for \(t \geq 0\) define 
    \begin{equation*}
        Y_t := \sum_{i = 1}^{p} (\widetilde{X}_i^2 - \alpha_{t})\mathbbm{1}_{\{|\widetilde{X}_i| \geq t\}}
    \end{equation*}
    where 
    \begin{equation}\label{eqn:alpha_t}
        \alpha_t := \frac{E(g^2 \mathbbm{1}_{\{|g| \geq t\}})}{P\{|g| \geq t\}}
    \end{equation}
    with \(g \sim N(0, 1)\). With this testing statistic in hand, define the test 
    \begin{equation}\label{test:problemI_tsybakov}
        \varphi_{t, r} := \mathbbm{1}_{\left\{Y_t > r \right\}}.
    \end{equation}
    where \(r > 0\) is set by the statistician to attain a desired level of testing risk. With these tests, the upper bound \(\varepsilon_1^*(p, s, \gamma)^2 \lesssim \psi_1^2\) can be obtained.

    \begin{proposition}[Upper bound]\label{prop:problemI_upperbound}
        Assume \(p \geq 4\). Let \(1 \leq s \leq p\), \(\gamma \in [0, 1)\). If \(\eta \in (0, 1)\), then there exists a constant \(C_\eta > 0\) depending only on \(\eta\) such that for all \(C > C_\eta\), the testing procedure 
        \begin{equation}\label{test:problemI}
            \varphi_{I}^* = 
            \begin{cases}
                \varphi_{t^*, r^*} &\text{if } s < \sqrt{p}, \\
                \varphi^{\chi^2}_{C^2/2} &\text{if } s \geq \sqrt{p}
            \end{cases}
        \end{equation}
        with \(t^* = \sqrt{2 \log\left(1 + \frac{p}{s^2}\right)}\) and \(r^* = \frac{C^2}{32} s \log\left(1 + \frac{p}{s^2}\right)\) satisfies 
        \begin{equation*}
            P_{0, \gamma}\{\varphi^*_{I} = 1\} + \sup_{\theta \in \Theta_{\mathcal{I}}(p, s, C\psi_1)} P_{\theta, \gamma}\left\{ \varphi^*_{I} = 0\right\} \leq \eta
        \end{equation*}
        Here, \(\varphi^{\chi^2}_{C^2/2}\) and \(\varphi_{t^*, r^*}\) are given by (\ref{test:problemI_chisquare}) and (\ref{test:problemI_tsybakov}) respectively.
    \end{proposition}

    Before proceeding to the statement of a matching lower bound, a discussion is in order regarding the use of Collier's (and his coauthors') \cite{collierMinimaxEstimationLinear2017} statistic in the setting of Problem I. Firstly, this statistic was initially formulated in the sparse mean setting with independent noise, that is, in the setting with observation \(\zeta \sim N\left(\mu, I_p\right)\). Collier et al. \cite{collierMinimaxEstimationLinear2017} essentially prove (see also \cite{liuMinimaxRatesSparse2021}) that the statistic \(\sum_{i=1}^{p} (\zeta_i^2 - \alpha_{t^*})\mathbbm{1}_{\{|\zeta_i| \geq t^*\}}\) with \(t^* = \sqrt{2 \log\left(1 + \frac{p}{s^2}\right)}\) furnishes a rate-optimal test for testing \(H_0 : \mu = 0\) against \(H_1 : \mu \in \Theta(p, s, \varepsilon)\) for \(1\leq s \leq p\) and \(\varepsilon > 0\). 
    
    In contrast, we have applied Collier's (and coauthors') statistic to \(\widetilde{X}\). Notice \(\widetilde{X}\) has mean \(\frac{\theta - \bar{\theta}\mathbf{1}_p}{\sqrt{1-\gamma}}\), which need not be \(s\)-sparse. In fact, when \(s < p\) it follows that \(||\theta - \bar{\theta}\mathbf{1}_p||_0 \leq s\) if and only if \(||\theta||_0 \leq s\) and \(\theta \in \spn\{\mathbf{1}_p\}^\perp\). Importantly, we do \emph{not} have that \(\theta-\bar{\theta}\mathbf{1}_p\) is \(s\)-sparse for all \(\theta \in \Theta_{\mathcal{I}}(p, s, \varepsilon_1)\). Thus, at first glance, it is surprising that the statistic of \cite{collierMinimaxEstimationLinear2017} can be used to furnish an optimal test for Problem I in the sparse regime \(s < \sqrt{p}\).

    We extend the applicability of Collier's and coauthors' statistic to the setting where \(\mu\) exhibits a sparse subvector which has norm of the same order (see Proposition \ref{prop:supp_tsybakov}). This extension enables us to use the statistic with the data \(\widetilde{X}\) because Corollary \ref{corollary:supp_approximation} in Section \ref{section:geometric_observations} implies that if \(||\theta||_0 \leq s\) then 
    \begin{equation*}
        ||\theta-\bar{\theta}\mathbf{1}_{\supp(\theta)}||^2 \geq ||\theta||^2 \cdot \frac{p-2s}{p}.
    \end{equation*}
    As we are interested in the case \(s < \sqrt{p}\), it is immediately seen that there exists a sparse sub-vector of \(\theta - \bar{\theta}\mathbf{1}_{p}\) with the same squared Euclidean norm order.  More explicitly, we have shown \(||\theta - \bar{\theta}\mathbf{1}_{\supp(\theta)}||^2 \gtrsim \varepsilon_1^2\) for all \(\theta \in \Theta_{\mathcal{I}}(p, s, \varepsilon_1)\). This result gives some insight into an upper bound for \(\varepsilon_1^*\) when \(s < \sqrt{p}\). Given that \(\theta - \bar{\theta}\mathbf{1}_{\supp(\theta)}\) is \(s\)-sparse, an informed guess based on \cite{collierMinimaxEstimationLinear2017} is \((\varepsilon_1^*)^2 \lesssim (1-\gamma) s \log\left(1 + \frac{p}{s^2}\right)\) for \(s < \sqrt{p}\). Proposition \ref{prop:problemI_upperbound} confirms the correctness of this guess.
     
    The next proposition establishes a matching lower bound (up to absolute constants) for \(\varepsilon^*_1\), thus implying that the testing procedure \(\varphi^*\) is rate optimal. 

    \begin{proposition}[Lower bound]\label{prop:problemI_lowerbound}
        Assume \(p \geq 2\). Let \(1 \leq s \leq p\) and \(\gamma \in [0, 1)\). If \(\eta \in (0, 1)\), then there exists a constant \(c_\eta > 0\) depending only on \(\eta\) such that \(\mathcal{R}_{\mathcal{I}}(c\psi_1) \geq 1-\eta\) for all \(c < c_\eta\).
    \end{proposition}

    Combining Propositions \ref{prop:problemI_upperbound} and \ref{prop:problemI_lowerbound} yields Theorem \ref{thm:ProblemI} for \(p \geq 4\). The cases \(p = 2,3\) are handled very easily (see  Section \ref{section:proofs}). 
    
    \subsection{Problem II} \label{section:ProblemII}

    The main result of this section is a characterization of the minimax separation rate \(\varepsilon_2^*\) for Problem II. However, since Problem II is interesting only to the extent to which it is informative about \(\varepsilon^*\), Lemma \ref{lemma:problemI_equiv} implies that separation rate for Problem II need only be pinned down for \(s \geq \frac{p}{2}\). Narrowing focus to this regime, we obtain the following theorem. For ease of notation throughout this section, set
    \begin{equation} \label{rate:problemII}
        \psi_2^2 := 
        \begin{cases}
            \frac{(1-\gamma)p^{3/2}}{p-s} \wedge (1-\gamma+\gamma p) &\text{if } \frac{p}{2} \leq s \leq p-\sqrt{p}, \\
            (1-\gamma) p \log\left(1 + \frac{p}{(p-s)^2}\right) \wedge (1-\gamma+\gamma p) &\text{if } p-\sqrt{p} < s \leq p. 
        \end{cases}
    \end{equation}

    \begin{theorem}[Problem II]\label{thm:ProblemII}
        If \(\frac{p}{2} \leq s \leq p\) and \(\gamma \in [0, 1)\), then \(\varepsilon_2^*(p, s, \gamma)^2 \asymp \psi_2^2\).
    \end{theorem}

    As in the analysis of Problem I, matching upper and lower bounds are constructed to derive the minimax separation rate. A combination of three tests is used. First, the \(\chi^2\)-type test \(\varphi^{\chi^2}_{r}\) defined in (\ref{test:problemI_chisquare}) is used. Second, define the test 
    \begin{equation}\label{test:problemII_linear}
        \varphi^{\mathbf{1}_p}_{r} := \mathbbm{1}_{\{\langle p^{-1/2} \mathbf{1}_p, X \rangle^2 > (1-\gamma+\gamma p)(1 + r)\}}. 
    \end{equation}
    Finally, the test \(\varphi_{t, r}\) defined in (\ref{test:problemI_tsybakov}) based on the statistic formulated by Collier et al. \cite{collierMinimaxEstimationLinear2017} is again used. 

    \begin{proposition}[Upper bound]\label{prop:problemII_upperbound}
        Let \(\frac{p}{2} \leq s \leq p\) and \(\gamma \in [0, 1)\). If \(\eta \in (0, 1)\), then there exists a constant \(C_\eta > 0\) depending only on \(\eta\) such that for all \(C > C_{\eta}\) the testing procedure 
        \begin{equation}\label{test:problemII}
            \varphi_{II}^* = 
            \begin{cases}
                \varphi^{\chi^2}_{C^2/2} \vee \varphi^{\mathbf{1}_p}_{C^2/2} &\text{if } \frac{p}{2} \leq s \leq p-\sqrt{p}, \\
                \varphi_{\widetilde{t}, \widetilde{r}} \vee \varphi^{\mathbf{1}_p}_{C^2/2} &\text{if } p-\sqrt{p} < s < p, \\
                \varphi^{\mathbf{1}_p}_{C^2/2} &\text{if } s = p
            \end{cases}
        \end{equation}
        with \(\widetilde{t} = \sqrt{2 \log\left(1 + \frac{p}{(p-s)^2}\right)}\) and \(\widetilde{r} = \frac{C^2}{8} (p-s)\log\left(1 + \frac{p}{(p-s)^2}\right)\) satisfies 
        \begin{equation*}
            P_{0, \gamma}\{\varphi^*_{II} = 1\} + \sup_{\theta \in \Theta_{\mathcal{II}}(p, s, C\psi_2)}P_{\theta, \gamma}\{\varphi^*_{II} = 0\} \leq \eta.
        \end{equation*}
        Here, the constituent tests in (\ref{test:problemII}) are given by (\ref{test:problemI_chisquare}), (\ref{test:problemI_tsybakov}), and (\ref{test:problemII_linear}). 
    \end{proposition}

    The three tests are sensitive to different regions of the parameter space. Firstly, the test \(\varphi_{r}^{\mathbf{1}_p}\) is sensitive to those \(\theta\) with \(||\bar{\theta}\mathbf{1}_p||^2 \gtrsim 1-\gamma+\gamma p\). Secondly, the test \(\varphi_{r}^{\chi^2}\) is sensitive to signals with \(||\bar{\theta}\mathbf{1}_p||^2 \gtrsim \frac{(1-\gamma)p^{3/2}}{p-s}\). Finally, the test \(\varphi_{\widetilde{t}, r}\) is sensitive to signals with \(||\bar{\theta}\mathbf{1}_p||^2 \gtrsim (1-\gamma) p \log\left(1 + \frac{p}{(p-s)^2}\right)\) when \(1 \leq p - s < \sqrt{p}\). 

    Counterintuitively, these last two tests make use of the transformed data \(\widetilde{X}\) despite the fact that its mean is orthogonal to \(\bar{\theta}\mathbf{1}_p\). Recall that the alternative hypothesis parameter space is \(\Theta_{\mathcal{II}}(p, s, \varepsilon_2) = \{\theta \in \R^p : ||\bar{\theta}\mathbf{1}_p|| \geq \varepsilon_2, ||\theta||_0 \leq s\}\) and so separation from the null hypothesis is given along the one-dimensional subspace \(\spn\{\mathbf{1}_p\}\). At first glance, a strategy using \(\widetilde{X}\) appears to be completely hopeless since there may possibly be choices of \(\theta\) in the alternative hypothesis such that \(||\bar{\theta}\mathbf{1}_p||^2\) can be made large while \(||\theta - \bar{\theta}\mathbf{1}_p||^2\) can be made arbitrarily small due to the orthogonality of \(\bar{\theta}\mathbf{1}_p\) and \(\theta - \bar{\theta}\mathbf{1}_p\). In other words, it may possibly be the case that the distribution of \(\widetilde{X}\) can be made to be arbitrarily close to a pure noise distribution despite the fact that the alternative hypothesis is well separated from the null hypothesis. 
    
    However, this line of reasoning neglects the sparsity of \(\theta\). Corollary \ref{corollary:orthog_approximation} in Section \ref{section:geometric_observations}, along with the trivial inequality \(\varepsilon_2^2 \leq ||\bar{\theta}\mathbf{1}_{p}||^2 \leq ||\theta||^2\) for \(\theta \in \Theta_{\mathcal{II}}(p, s, \varepsilon_2)\), implies that 
    \begin{equation}\label{bound:orthog_problemII}
        \left|\left|\frac{\theta - \bar{\theta}\mathbf{1}_p}{\sqrt{1-\gamma}}\right|\right|^2 \geq \frac{\varepsilon_2^2}{1-\gamma} \cdot \frac{p-s}{p}
    \end{equation}
    for all \(\theta \in \Theta_{\mathcal{II}}(p, s, \varepsilon_2)\). With this bound in hand, observe that \(\varphi_{r}^{\chi^2}\) is sensitive to signals with \(||\bar{\theta}\mathbf{1}_p||^2 \gtrsim \frac{(1-\gamma)p^{3/2}}{p-s}\) because (\ref{bound:orthog_problemII}) implies \(\left|\left|\frac{\theta - \bar{\theta}\mathbf{1}_p}{\sqrt{1-\gamma}} \right|\right|^2 \gtrsim \sqrt{p}\). 
    
    Turning our attention to \(\varphi_{t, r}\), it can be intuitively seen that the test is sensitive to signals with \(||\bar{\theta}\mathbf{1}_p||^2 \gtrsim (1-\gamma) p \log\left(1 + \frac{p}{(p-s)^2}\right)\) when \(1 \leq p-s < \sqrt{p}\). Let \(T \subset \supp(\theta)^c\) with \(|T| = p - s\) and note existence is guaranteed since \(|\supp(\theta)^c| \geq p-s\). Consider \(\left|\left| \frac{(\theta - \bar{\theta}\mathbf{1}_p)_{T}}{\sqrt{1-\gamma}} \right|\right|^2 = \frac{||\bar{\theta}\mathbf{1}_p||^2}{1-\gamma} \cdot \frac{p-s}{p} \gtrsim (p-s)\log\left(1 + \frac{p}{(p-s)^2}\right)\). Since \(1 \leq p-s < \sqrt{p}\), it ought to follow from the choice \(t = \widetilde{t} = \sqrt{2\log\left(1 + \frac{p}{(p-s)^2}\right)}\) and from an analogue of our Problem I analysis that Collier and coauthor's statistic applied to \(\widetilde{X}\) furnishes a sensitive test. The proof of Proposition \ref{prop:problemII_upperbound} confirms this intuition. 
    
    The next proposition establishes a matching lower bound (up to absolute constants) for \(\varepsilon_2^*\) when \(\frac{p}{2} \leq s \leq p\). 
    \begin{proposition}[Lower bound]\label{prop:problemII_lowerbound}
        Let \(\frac{p}{2} \leq s \leq p\) and \(\gamma \in [0, 1)\). If \(\eta \in (0, 1)\), then there exists a constant \(c_\eta > 0\) depending only on \(\eta\) such that \(\mathcal{R}_{\mathcal{II}}(c\psi_2) \geq 1-\eta\) for all \(c < c_\eta\). 
    \end{proposition}
    The proof of the lower bound sheds light on features of the minimax separation rate for Problem II. For ease of notation, let \(\Sigma := (1-\gamma)I_p + \gamma \mathbf{1}_p\mathbf{1}_p^\intercal\). When \(\frac{p}{2} \leq s \leq p\), it turns out it is useful to consider a simple mean shift transformation. The problem of testing \(N(0, \Sigma)\) against \(N(\theta, \Sigma)\) for \(\theta \in \Theta_{\mathcal{II}}(p, s, \varepsilon)\) is clearly equivalent to testing \(N(\nu^*, \Sigma)\) against \(N(\theta - \nu^*, \Sigma)\) where \(\nu^* = \frac{c\psi_2\sqrt{p}}{s} \mathbf{1}_p\) with \(c\) a positive constant to be chosen later. Note that the leading eigenvalue of \(\Sigma\) is \(1-\gamma+\gamma p\) and is associated to the unit eigenvector \(p^{-1/2}\mathbf{1}_p\). Since \(s \asymp p\), we have \(\frac{\sqrt{p}}{s} \asymp p^{-1/2}\) and so \(||\nu^*||^2 \asymp \psi_2^2\). When \(\psi_2^2 \leq 1-\gamma+\gamma p\), it can be shown that \(d_{TV}(N(\nu^*, \Sigma), N(0, \Sigma))\) can be made arbitrarily small by choosing \(c\) suitably small but independently of \(p, s\), and \(\gamma\). Therefore, when \(\psi_2^2 \leq 1-\gamma+\gamma p\) the problem of testing \(N(\nu^*, \Sigma)\) against \(N(\theta - \nu^*, \Sigma)\) is statistically indistinguishable from the problem of testing \(N(0, \Sigma)\) against \(N(\theta - \nu^*, \Sigma)\). This latter problem will show why the terms involving \(p-s\) appear in the separation rate. Before we elaborate on this point, we pause to note that this mean shift transformation explains why \(\varepsilon_2^*\) involves taking a minimum with \(1-\gamma+\gamma p\). Furthermore the utility of \(\varphi_r^{\mathbf{1}_p}\) is clarified as it is sensitive precisely to those signals with \(||\bar{\theta}\mathbf{1}_p||^2 \gtrsim 1-\gamma+\gamma p\).

    Examining the problem of testing \(N(0, \Sigma)\) against \(N(\theta, \Sigma)\) for \(\theta \in \Theta_{\mathcal{II}}(p, s, c\psi_2)\), a lower bound for \(\varepsilon_2^*\) can be constructed by considering the prior distribution \(\pi\) supported on \(\Theta_{\mathcal{II}}(p, s, c\psi_2)\) in which a draw \(\theta \sim \pi\) is obtained by setting \(\theta = \frac{c\psi_2\sqrt{p}}{s} \mathbf{1}_S\) for a uniformly drawn size \(s\) subset \(S \subset [p]\). When \(\psi_2^2 \leq 1 - \gamma + \gamma p\), we can consider the statistically indistinguishable problem of testing \(N(0, \Sigma)\) against \(N(\theta - \nu^*, \Sigma)\) for \(\theta \sim \pi\). It is immediately clear that \(\theta - \nu^* = -\frac{c\psi_2 \sqrt{p}}{s}\mathbf{1}_T\) where \(T\) is a uniformly drawn size \(p-s\) subset \(T \subset [p]\). The original testing problem has transformed into a problem of detecting a \(p-s\) sparse signal, and so the \(p-s\) terms appear in the rate. Thus the juncture in the rate at \(s = p-\sqrt{p}\) is explained as it exactly corresponds to \(p-s = \sqrt{p}\). 

    \subsection{Synthesis}\label{section:synthesis}
    With characterizations of the minimax separation rates of Problem I and Problem II in hand, Lemma \ref{lemma:problem_decomposition} enables a quick deduction about \(\varepsilon^*\). Specifically, combining Proposition \ref{prop:perfect_correlation}, Theorem \ref{thm:ProblemI}, Theorem \ref{thm:ProblemII}, Lemma \ref{lemma:problemI_equiv}, and Lemma \ref{lemma:problem_decomposition} immediately yields a full characterization of \(\varepsilon^*\). Note that while Theorem \ref{thm:ProblemI} does not explicitly cover the case \(p = 1\), it is clear from (\ref{space:Theta1}) that \(\varepsilon_1^*(1, 1, \gamma) = 0\) and so Lemma \ref{lemma:problem_decomposition} implies \(\varepsilon^*(1, 1, 0) \asymp \varepsilon_2^*(1, 1, 0)\). 
    
    \begin{theorem}\label{thm:minimax_rate}
        If \(1 \leq s \leq p\) and \(\gamma \in [0, 1]\), then 
        \begin{equation*}
            \varepsilon^*(p, s, \gamma)^2 \asymp 
            \begin{cases}
                (1-\gamma)s\log\left(1 + \frac{p}{s^2}\right) &\text{if } s < \sqrt{p}, \\
                (1-\gamma)\sqrt{p} + \frac{(1-\gamma)p^{3/2}}{p-s} \wedge (1-\gamma+\gamma p) &\text{if } \sqrt{p} \leq s \leq p-\sqrt{p}, \\
                (1-\gamma)\sqrt{p} + (1-\gamma)p \log\left(1 + \frac{p}{(p-s)^2}\right) \wedge (1-\gamma+\gamma p) &\text{if } p-\sqrt{p} < s \leq p. 
            \end{cases}
        \end{equation*}
    \end{theorem}

    Theorem \ref{thm:minimax_rate} generalizes results in the literature and gives a full description of how ambient dimension, sparsity, and correlation level interact and impact the minimax separation rate for the signal detection problem (\ref{problem:test_equicorrelated_1})-(\ref{problem:test_equicorrelated_2}). The following remarks contextualize and note some novel phenomenon revealed by Theorem \ref{thm:minimax_rate}. 
    
    \begin{remark}[Recovering known results]
        The special cases previously known in the literature are subsumed by Theorem \ref{thm:minimax_rate}. For example, substituting \(\gamma = 0\) to obtain \(\varepsilon^*(p, s, 0)\) immediately recovers the minimax separation rate established by Collier et al. \cite{collierMinimaxEstimationLinear2017} for all \(1 \leq s \leq p\). We also recover \(\varepsilon^*(p, p, \gamma)^2 \asymp ||(1-\gamma)I_p + \gamma \mathbf{1}_p\mathbf{1}_p^\intercal||_F\) for \(\gamma \in [0, 1)\). Furthermore, \(\varepsilon^*(p,s,\gamma)\) matches the rate established by Liu et al. \cite{liuMinimaxRatesSparse2021} for \(s \leq p^{1/5}\).
    \end{remark}
    
    \begin{remark}[Discontinuity]\label{remark:discontinuity}
        Suppose \(1-\gamma \lesssim \frac{1}{\log(ep)}\) and note 
        \begin{equation}\label{regime:moderate_correlation}
            \varepsilon^*(p, s, \gamma)^2 \asymp 
            \begin{cases}
                (1-\gamma)s\log\left(1 + \frac{p}{s^2}\right) &\text{if } s < \sqrt{p}, \\
                \frac{(1-\gamma)p^{3/2}}{p-s} &\text{if } \sqrt{p} \leq s \leq p-\sqrt{p}, \\
                (1-\gamma) p \log\left(1 + \frac{p}{(p-s)^2}\right) &\text{if } p-\sqrt{p} < s < p, \\
                p &\text{if } s = p.
            \end{cases}
        \end{equation}
        Note \(\varepsilon^*(p, p-1, \gamma)^2 \asymp (1-\gamma) p \log\left(1 + p \right) \lesssim p \asymp \varepsilon^*(p, p, \gamma)^2\) when \(1-\gamma \lesssim \frac{1}{\log(ep)}\). A discontinuity in the minimax separation rate emerges at \(s = p\) when \(1 - \gamma = o\left(\frac{1}{\log(ep)}\right)\) as then \(\varepsilon^*(p, p-1, \gamma)^2 = o(\varepsilon^*(p, p, \gamma)^2)\). Notice that the discontinuity disappears once \(1-\gamma \gtrsim \frac{1}{\log(ep)}\). This discontinuity in the correlation regime \(1-\gamma = o\left(\frac{1}{\log(ep)}\right)\) appears to be an extension of the discontinuity phenomenon witnessed in Proposition \ref{prop:perfect_correlation} (in which the extreme case \(\gamma = 1\) was in force). Theorem \ref{thm:minimax_rate} thus reveals not only that detection of completely dense signals (\(s = p\)) can be fundamentally different from detection of non-completely dense signals (\(s < p\)) by virtue of this discontinuity, but that this fundamental difference exists if and only if the correlation is suitably strong.
    \end{remark}

    \begin{remark}[The blessing of strong correlation]\label{remark:blessing}
        Theorem \ref{thm:minimax_rate} enables a comparison between the correlated and independent settings. We determine how strong the correlation must be such that it is a blessing for signal detection. A direct and straight-forward comparison yields the following result. For \(1 \leq s < p\), set 
        \begin{equation*}
            1-\gamma^* := 
            \begin{cases}
                1 &\text{if } s < \sqrt{p}, \\
                \frac{p-s}{p} &\text{if } \sqrt{p} \leq s \leq p - \sqrt{p}, \\
                \frac{1}{\sqrt{p}\log\left(1 + \frac{p}{(p-s)^2}\right)} &\text{if } p-\sqrt{p} < s < p.
            \end{cases}
        \end{equation*}
        If \(\gamma \in [0, 1]\) and \(1 \leq s < p\), then \(1-\gamma = o(1-\gamma^*)\) if and only if \(\varepsilon^*(p, s, \gamma) = o(\varepsilon^*(p, s, 0))\). Notably, the threshold correlation level exhibits phase transitions at the \(\sqrt{p}\) and \(p-\sqrt{p}\) sparsity levels. 
    \end{remark}

    \begin{remark}[The curse of moderate correlation]\label{remark:curse}
        We can also use Theorem 3 to determine the correlation levels which hinder signal detection. Concretely, we characterize for which correlation levels we have \(\varepsilon^*(p, s, \gamma) = \omega(\varepsilon^*(p, s, 0))\). For \(\sqrt{p} \leq s \leq p\), set
        \begin{equation*}
            1-\gamma_* := 
            \begin{cases}
                \frac{p-s}{p} &\text{if } \sqrt{p} \leq s \leq p-\sqrt{p}, \\
                \frac{1}{\sqrt{p}\log\left(1 + \frac{p}{(p-s)^2}\right)} &\text{if } p-\sqrt{p} < s < p, \\
                0 &\text{if } s = p.
            \end{cases}
        \end{equation*}
        Suppose \(\sqrt{p} \leq s \leq p\) and \(\gamma = \omega\left(\frac{1}{\sqrt{p}}\right)\). Then \(1-\gamma = \omega(1-\gamma_*)\) if and only if \(\varepsilon^*(p, s, \gamma) = \omega(\varepsilon^*(p, s, 0))\).  Note that correlation is never a curse in the sparsity regime \(s < \sqrt{p}\), rather it can only be a blessing or irrelevant. 
    \end{remark}

    \begin{remark}[The irrelevance of weak correlation]\label{remark:irrelevancy}
        Remark \ref{remark:curse} requires \(\gamma =\omega\left(\frac{1}{\sqrt{p}}\right)\) in the course of showing that moderate correlation is a curse. When the correlation is weak, meaning \(\gamma \lesssim \frac{1}{\sqrt{p}}\) and \(1-\gamma \asymp 1\), it is immediately seen that \(\varepsilon^*(p, s, \gamma) \asymp \varepsilon^*(p, s, 0)\) for all \(1 \leq s \leq p\).
    \end{remark}

    \begin{remark}[Implications for study design]\label{remark:study_design_comparison}
        In many scientific fields, researchers regularly design and conduct experimental studies. Often, the choice is between a within-subjects (i.e. repeated measures) design or a between-subjects (i.e. between groups) design. If the purpose of the study is to test a global null hypothesis against sparse alternatives, Theorem \ref{thm:minimax_rate} offers some guidance on how to choose.

        To set the stage, the within-subjects study design is described first. Suppose there are \(p\) treatments and researchers are interested in testing the global null hypothesis (all \(p\) treatment effects are zero) against a sparse alternative. In a within-subjects study design, researchers recruit \(n\) individuals and apply all \(p\) treatments to each individual. Since each individual receives each treatment, it is expected that responses to treatments are correlated within each individual. Concretely, letting \(Y_{ij}\) denote the response of individual \(i\) to treatment \(j\), a random effect model is commonly used to model the response,
        \begin{equation*}
            Y_{ij} = \tau_j + \sqrt{\gamma} \omega_i + \sqrt{1-\gamma} \zeta_{ij}
        \end{equation*}
        where \(\tau_j\) denotes the treatment effect of treatment \(j\) and \(\gamma \in [0, 1]\) denotes the common correlation of responses across treatments within individuals. For simplicity, assume \(\gamma\) is known. Further, \(\omega_i\) denotes the individual level random effect, and \(\zeta_{ij}\) denotes the additive noise for the response of individual \(i\) to treatment \(j\). It is assumed \(\{\omega_i\}_{i}\) and \(\{\zeta_{ij}\}_{ij}\) are all independent standard Gaussian random variables. Note that responses for different individuals are assumed to be independent. Letting \(\tau = (\tau_1,...,\tau_p) \in \R^p\), the testing problem of interest is testing \(H_0 : \tau = 0\) against \(H_1: \tau \in \Theta(p, s, \varepsilon)\) where \(1 \leq s \leq p\) and \(\varepsilon > 0\). Appealing to sufficiency, one can average across individuals to obtain (with an abuse of notation) \(Y_j = \tau_j + \frac{\sqrt{\gamma}}{n} \bar{\omega} + \frac{\sqrt{1-\gamma}}{n} \zeta_j\) yielding, in vector form \(Y := (Y_1,...,Y_p)\), the observational model
        \begin{equation}\label{model:within_subjects}
            \sqrt{n} Y \sim N\left(\sqrt{n} \tau, (1-\gamma)I_p + \gamma \mathbf{1}_p\mathbf{1}_p^\intercal \right).
        \end{equation}

        In a between-subjects study design, researchers recruit \(n\) individuals for each of \(p\) treatments, i.e. a total of \(np\) individuals are enrolled. Each individual is assigned to only one of the \(p\) treatments. A natural and common assumption is independence between individuals, and so the observation model is given by 
        \begin{equation*}
            Y_{ij} = \tau_j + \zeta_{ij}. 
        \end{equation*}
        The testing problem of interest is as before. An appeal to sufficiency followed by averaging over individuals in each treatment group yields, in vector form with \(Y := (Y_1,...,Y_p)\), the observational model 
        \begin{equation}\label{model:between_subjects}
            \sqrt{n} Y \sim N\left(\sqrt{n} \tau, I_p\right). 
        \end{equation}

        In essence, researchers are choosing between observation models (\ref{model:within_subjects}) and (\ref{model:between_subjects}) when deciding between a within-subjects study design and a between-subjects study design. If it is known that \(\tau\) is \(s\)-sparse with \(s < \sqrt{p}\), then Remark \ref{remark:blessing} indicates researchers ought to choose the within-subject study design. Explicitly, this is because \(\varepsilon^*(p,s,\gamma)^2 \lesssim \varepsilon^*(p, s, 0)^2\) holds for all \(s < \sqrt{p}\) and \(\gamma \in [0, 1]\). On the other hand, if \(s > \sqrt{p}\), then the correlation must be sufficiently strong in order to recommend the within-subject study design on the basis of minimax separation rates (see Remark \ref{remark:blessing}). In some settings, researchers may be able to select the level of correlation \(\gamma\); Remark \ref{remark:blessing} clearly prescribes how strong \(\gamma\) ought to be and so researchers should always choose to implement a within-subject design. Of course, Theorem \ref{thm:minimax_rate} is relevant only when the goal is testing the global null hypothesis. For multiple testing related goals, Theorem \ref{thm:minimax_rate} has nothing to say. Investigating how \(\gamma\) affects the fundamental limits of multiple testing tasks is a rich direction for further work \cite{fromontFamilyWiseSeparationRates2016}.
    \end{remark}
    
    As seen in the proof of Lemma \ref{lemma:problem_decomposition}, a rate optimal testing procedure is given by taking maximum of the two rate optimal tests in Propositions \ref{prop:problemI_upperbound} and \ref{prop:problemII_upperbound}. At the risk of redundancy but in the interest of clarity, we explicitly state this procedure's optimality in the following proposition. We only explicitly state the result for \(\gamma \in [0, 1)\) as the case \(\gamma = 1\) does not require a synthesis of Problem I and Problem II (see Section \ref{section:perfect_correlation}). In preparation for the statement of the result, set \(\psi^2\) to be the right hand side of Theorem \ref{thm:minimax_rate}.

    \begin{proposition}[Rate optimal test]\label{prop:full_test}
        Let \(1 \leq s \leq p\) and \(\gamma \in [0, 1)\). If \(\eta \in (0, 1)\), then there exists a constant \(C_\eta > 0\) depending only on \(\eta\) such that for every \(C > C_\eta\) the testing procedure 
        \begin{equation}
            \varphi^* := \varphi_{I}^* \vee \varphi_{II}^* \label{test:full_test}
        \end{equation}
        satisfies 
        \begin{equation*}
            P_{0, \gamma}\{\varphi^* = 1\} + \sup_{\theta \in \Theta(p, s, C\psi)} P_{\theta, \gamma}\left\{ \varphi^* = 0 \right\} \leq \eta
        \end{equation*}
        for all \(C > C_\eta\). Here, \(\varphi_{I}^*\) and \(\varphi_{II}^*\) are the tests given by (\ref{test:problemI}) and (\ref{test:problemII}) respectively.
    \end{proposition}

    \section{Multiple random effects}\label{section:multiple_random_effects}
    In this section, we study sparse signal detection under group structured correlation by considering a model with multiple random effects. Suppose we have \(p\) observations with each observation belonging exclusively to one of \(R\) equally-sized groups. For ease, we assume \(R\) divides \(p\) and so each group is of size \(\frac{p}{R}\). Let \(B_1,...,B_R \subset [p]\) denote the groups and note that \(B_1,...,B_R\) are mutually disjoint. Let \(B : [p] \to [R]\) denote the group membership function, that is, \(B(i)\) denotes the group label of individual \(i\). It is assumed the group membership function is known, as is frequently the case when employing mixed models in applications. Letting \(\gamma \in [0, 1]\) denote the correlation level, consider the following Gaussian sequence mixed model 
    \begin{equation}
        X_i = \theta_i + \sqrt{\gamma} W_{B(i)} + \sqrt{1-\gamma}Z_i \label{model:multiple_random_effects_additive}
    \end{equation}
    where \(\theta = (\theta_1,...,\theta_p) \in \R^p\) denotes the fixed effects and \(W_1,...,W_R,Z_1,...,Z_p\) are independent and identically distributed standard Gaussian random variables. Note that \(W_1,...,W_R\) are the random effects, and their presence induces correlation only between those individuals in the same group. Explicitly, 
    \begin{equation*}
        \Cov(X_i, X_j) = 
        \begin{cases}
            1 & \text{if } i = j, \\
            \gamma &\text{if } i \neq j \text{ and } B(i) = B(j), \\
            0 &\text{otherwise.}
        \end{cases}
    \end{equation*}
    In vector form, the marginal distribution of \(X := (X_1,...,X_p)\) is given by 
    \begin{equation}\label{model:multiple_random_effects}
        X \sim N\left(\theta, (1-\gamma)I_p + \gamma \sum_{k=1}^{R} \mathbf{1}_{B_k}\mathbf{1}_{B_k}^\intercal\right).
    \end{equation}
    Here, \(\mathbf{1}_{B_k} \in \R^{p}\) denotes the vector with those coordinates in \(B_k\) equal to one and the remaining coordinates equal to zero. With this model in hand, a sparse signal detection problem can be formulated with the parameter space (\ref{parameter:alternative}). Concretely, given the observation \(X\) following the model (\ref{model:multiple_random_effects}), the sparse signal detection problem of interest is 
    \begin{align}
        H_0 &: \theta = 0, \label{problem:multiple_random_effects_1}\\
        H_1 &: \theta \in \Theta(p, s, \varepsilon) \label{problem:multiple_random_effects_2}
    \end{align}
    where \(1 \leq s \leq p\) and \(\varepsilon > 0\). A minimax testing risk for this testing problem can be defined analogously to (\ref{def:testing_risk}). At the risk of notational abuse but in the interest in brevity, in this section we overload notation and refer to the minimax testing risk of problem (\ref{problem:multiple_random_effects_1})-(\ref{problem:multiple_random_effects_2}) also as \(\mathcal{R}(\varepsilon)\). We can define a minimax separation rate for the testing problem (\ref{problem:multiple_random_effects_1})-(\ref{problem:multiple_random_effects_2}) similar to Definition \ref{def:separation_rate}. However, now the separation rate will exhibit a dependence on the number of groups \(R\). 

    \begin{definition}
        We say \(\varepsilon^* = \varepsilon^*(p, s, \gamma, R)\) is the minimax separation rate for the hypothesis testing problem (\ref{problem:multiple_random_effects_1})-(\ref{problem:multiple_random_effects_2}) with parameters \((p, s, \gamma, R)\) if 
        \begin{enumerate}[label=(\roman*)]
            \item for every \(\eta \in (0, 1)\) there exists a constant \(C_\eta > 0\) depending only on \(\eta\) such that \(C > C_\eta\) implies \(\mathcal{R}(C\varepsilon^*) \leq \eta\),

            \item for every \(\eta \in (0, 1)\) there exists a constant \(c_\eta > 0\) depending only on \(\eta\) such that \(0 < c < c_\eta\) implies \(\mathcal{R}(c\varepsilon^*) \geq 1-\eta\).
        \end{enumerate}
    \end{definition}
    At first, it may seem that the minimax separation rate should depend on the whole collection of groups \(B_1,...,B_R\) rather than just on \(R\). However, this is not the case as one can find a suitable permutation matrix \(\Pi\) such that the covariance matrix of \(\Pi X\) is block diagonal with blocks ordered corresponding to \(B_1,...,B_R\). Note \(\Pi \theta\) satisfies \(||\Pi\theta|| = ||\theta||\) and \(||\Pi\theta||_0 = ||\theta||_0\) since \(\Pi\) simply permutes coordinates. Consequently, it is equivalent to test \(H_0 : \Pi \theta = 0\) against \(H_1 : \Pi \theta \in \Theta(p, s, \varepsilon)\). Thus, the choice of groups is immaterial. 

    The main result of this section is a characterization of the minimax separation rate. For clarity, we give the rate for each sparsity regime in a separate statement. The following results focus on the correlation regime \(\gamma \in [0, 1)\). The perfect correlation case \((\gamma = 1)\) is addressed in Section \ref{section:R_perfect_correlation}.

    \begin{theorem}\label{thm:R_sparse}
        If \(R \in [p]\) divides \(p\), \(1 \leq s \leq \frac{p}{4R}\), and \(\gamma \in [0, 1)\), then 
        \begin{equation*}
            \varepsilon^*(p, s, \gamma, R)^2 \asymp \psi_1^2 
        \end{equation*}
        where \(\psi_1^2\) is given by (\ref{rate:problemI}).
    \end{theorem}
    \begin{theorem}\label{thm:pR_dense}
        Suppose \(R \in [p]\) divides \(p\), \(\frac{p}{4R} < s < \frac{p}{R}\), and \(\gamma \in [0, 1)\). Set 
        \begin{equation}\label{rate:pR_dense_upsilon}
            \upsilon^2 := 
            \begin{cases}
                \frac{(1-\gamma)p}{p-Rs}\left(\sqrt{\frac{p}{R} \log(eR)} + \log(R) \right) \wedge \left(1-\gamma+\gamma \frac{p}{R}\right) \log(eR) &\text{if } \frac{p}{4R} < s \leq \frac{p}{R} - \sqrt{\frac{p}{R}\log(eR)},\\
                \frac{(1-\gamma) p}{p-Rs}\left(\left( \frac{p}{R} - s \right)\log\left(1 + \frac{Rp \log(eR)}{(p-Rs)^2}\right) + \log(R)\right) \wedge \left(1 - \gamma + \gamma \frac{p}{R}\right)\log(eR), &\text{if } \frac{p}{R} - \sqrt{\frac{p}{R}\log(eR)} < s < \frac{p}{R}.
            \end{cases}
        \end{equation}
        Then \(\varepsilon^*(p, s, \gamma, R)^2 \asymp \psi_1^2 + \upsilon^2\) where \(\psi_1^2\) is given by (\ref{rate:problemI}).
    \end{theorem}

    \begin{theorem}\label{thm:R_avg_problem}
        If \(R \in [p]\) divides \(p\), \(\frac{p}{R} \leq s \leq p,\) and \(\gamma \in [0, 1)\), then 
        \begin{equation*}
            \varepsilon^*(p, s, \gamma, R)^2 \asymp  
            \begin{cases}
                \psi_1^2 + \left(1 - \gamma + \gamma \frac{p}{R}\right) \frac{Rs}{p} \log\left(1 + \frac{p^2}{Rs^2}\right) &\text{if } \frac{p}{R} \leq s < \frac{p}{\sqrt{R}}, \\
                (1-\gamma)\sqrt{p} + \left(1-\gamma+\gamma \frac{p}{R}\right)\sqrt{R} &\text{if } s \geq \frac{p}{\sqrt{R}}
            \end{cases}
        \end{equation*}
        where \(\psi_1^2\) is given by (\ref{rate:problemI}).
    \end{theorem}
    
    Theorems \ref{thm:R_sparse} through \ref{thm:R_avg_problem} indicate that the sparse signal detection problem with group structured correlation is fundamentally different than the detection problem under equicorrelation. In particular, the minimax separation rate is not a simple extension of the rate stated in Theorem \ref{thm:minimax_rate}. Before describing the testing procedures employed, a few remarks discussing the rates presented in Theorems \ref{thm:R_sparse} through \ref{thm:R_avg_problem} are in order. 
    \begin{remark}[Single random effect]
        The case \(R = 1\) is the setting of the observation model (\ref{model:single_random_effect}). Substituting \(R = 1\) and collecting the results of Theorems \ref{thm:R_sparse} through \ref{thm:R_avg_problem} yields precisely the separation rate of Theorem \ref{thm:minimax_rate}.
    \end{remark}

    \begin{remark}[Bounded number of random effects]
        When \(R \asymp 1\) and \(p\) is larger than some absolute constant, the separation rate has order 
        \begin{equation*}
            \varepsilon^*(p, s, \gamma, R)^2 \asymp 
            \begin{cases}
                (1-\gamma) s \log\left(1 + \frac{p}{s^2}\right) &\text{if } s < \sqrt{p}, \\
                (1-\gamma)\sqrt{p} + \frac{(1-\gamma)p^{3/2}}{p-Rs} \wedge \left(1 - \gamma + \gamma p\right) &\text{if } \sqrt{p} < s \leq \frac{p}{R} - \sqrt{\frac{p}{R}}, \\
                (1-\gamma)\sqrt{p} + (1-\gamma)p \log\left(1 + \frac{p}{(p-Rs)^2}\right) \wedge \left(1 - \gamma + \gamma p \right) &\text{if } \frac{p}{R} - \sqrt{\frac{p}{R}} < s < \frac{p}{R}, \\
                (1-\gamma)\sqrt{p} + \left(1 - \gamma + \gamma p\right) &\text{if } \frac{p}{R} \leq s \leq p. 
            \end{cases}
        \end{equation*}
        The phase transition at \(s = \sqrt{p}\) is a carryover from the independent setting. Interestingly, the presence of \(R > 1\) random effects results in a rate different from \(\varepsilon^*(p, s, \gamma, 1)\) even when \(R \asymp 1\). 
    \end{remark}

    \begin{remark}[Observation level random effects]
        At the other extreme, the case \(R = p\) is the setting in which there is a unique random effect per observation and so the observations are independent. Indeed, the marginal distribution of the data is \(X \sim N(\theta, I_p)\). The separation rate is given by 
        \begin{equation*}
            \varepsilon^*(p, s, \gamma, p)^2 \asymp 
            \begin{cases}
                s \log\left(1 + \frac{p}{s^2}\right) &\text{if } s < \sqrt{p}, \\
                \sqrt{p} &\text{if } s \geq \sqrt{p}.
            \end{cases}
        \end{equation*}
        We have recovered exactly the result of \cite{collierMinimaxEstimationLinear2017}. Note that the same result is recovered when \(\gamma = 0\) for any \(R \in [p]\). When \(\gamma = 0\), the covariance matrix of the data \(X\) is simply \(I_p\) and the grouping structure is irrelevant. 
    \end{remark}

    \begin{remark}[Rate behavior and sparsity regimes]
        Theorems \ref{thm:R_sparse} through \ref{thm:R_avg_problem} reveal three different sparsity regimes in which the minimax separation rate exhibits distinctive features. In the regime \(1 \leq s \leq \frac{p}{4R}\), the minimax separation rate matches the rate (\ref{rate:problemI}). Provided the sparsity is small enough, the grouping structure does not affect the form of the rate. 
        
        In the sparsity regime \(\frac{p}{4R} < s < \frac{p}{R}\), the situation is more complicated. The juncture at \(s = \frac{p}{R} - \sqrt{\frac{p}{R} \log(eR)}\) is a generalization of the juncture at \(s = p-\sqrt{p}\) from Theorem \ref{thm:minimax_rate}. Intuitively, the separation rate is driven by those signals \(\theta\) which are supported on some group \(B_k\) (see Section \ref{section:pR_dense} for more discussion). The group on which \(\theta\) is supported on is unknown; this ignorance leads to a difference in the rate compared to the equicorrelated setting of Theorem \ref{thm:minimax_rate}. One must pay the price of an additive \(\log(R)\) term and a multiplicative \(\log(eR)\) term. The form of the juncture \(s = \frac{p}{R} - \sqrt{\frac{p}{R} \log(eR)}\) also reflects the ignorance about which group supports the signal. Fixing a group \(B_k\), detecting a signal on \(B_k\) is similar to detecting a \(s\)-sparse signal in ambient dimension \(\frac{p}{R}\) with equicorrelated observations. The intuition from Theorem \ref{thm:minimax_rate} explains the multiplicative factor of \(\frac{(1-\gamma)\frac{p}{R}}{\frac{p}{R} - s}\). 
        
        In the regime \(s \geq \frac{p}{R}\), the \(p\)-dimensional signal detection problem reduces down to an \(R\)-dimensional signal detection problem with independent observations and sparsity \(\frac{Rs}{p}\). The juncture at \(s = \frac{p}{R}\) is discussed in Remark \ref{remark:pR_discontinuity}. The phase transition at \(s = \frac{p}{\sqrt{R}}\) turns out to be driven by the same phenomenon occurring at \(s = \sqrt{p}\) in the independent setting proved in \cite{collierMinimaxEstimationLinear2017}. Thus a phase transition appears at sparsity level \(s\) such that \(\frac{Rs}{p} = \sqrt{R}\), that is to say, at \(s = \frac{p}{\sqrt{R}}\). Section \ref{section:R_avg} discusses the situation in more detail.
    \end{remark}

    \begin{remark}[Discontinuity at \(\frac{p}{R}\)]\label{remark:pR_discontinuity}
        A discontinuity appears in the separation rate at \(s = \frac{p}{R}\). After a quick calculation, Theorems \ref{thm:pR_dense} and \ref{thm:R_avg_problem} show that if \(R > 1\) is large enough such that \(\log(eR) \gtrsim \log\left(1 + \frac{p}{R}\log(eR)\right)\), then \(1-\gamma = o(1)\) implies \(\varepsilon^*\left(p, \frac{p}{R}-1, \gamma\right) = o\left(\varepsilon^*\left(p, \frac{p}{R}, \gamma\right)\right)\). On the other hand in which \(\log(eR) \lesssim \log\left(1 + \frac{p}{R}\log(eR)\right)\), another calculation shows that \(1-\gamma = o\left( \frac{\log(eR)}{\log\left(1 + \frac{p}{R}\log(eR)\right)} \right)\) implies \(\varepsilon^*\left(p, \frac{p}{R} - 1, \gamma\right) = o\left(\varepsilon^*\left(p, \frac{p}{R}, \gamma\right)\right)\). The discontinuity at \(s = \frac{p}{R}\) is a generalization of that discussed in Remark \ref{remark:discontinuity}.
    \end{remark}

    \begin{remark}[Implications for study design]
        Theorems \ref{thm:R_sparse} through \ref{thm:R_avg_problem} offer additional guidance to choosing between a within-subjects design or a between-subjects design. Remark \ref{remark:study_design_comparison} advocated for a within-subject design. Concretely, it was prescribed that \(n\) individuals be recruited and each exposed to the \(p\) treatments. However, it is impractical when \(p\) is large as it may be infeasible to expose every individual to every treatment. A more practical design is to split up the subjects into \(R\) equally-sized groups in which each group is exposed only to \(\frac{p}{R}\) treatments. Adopting the notation of Remark \ref{remark:study_design_comparison}, the observation model of such a design is
        \begin{equation*}
            \sqrt{n}Y\sim N\left(\sqrt{n}\tau, (1-\gamma)I_p + \gamma \sum_{k=1}^{R} \mathbf{1}_{B_k}\mathbf{1}_{B_k}^\intercal \right).
        \end{equation*}
        For simplicity, let us suppose it is known that the sparsity of the potential signal satisfies \(s \leq \frac{p}{2}\). The setting of Remark \ref{remark:study_design_comparison} exhibits separation rate of order \(\varepsilon^*(p, s, \gamma, 1)^2 \asymp \psi_1^2\). How does one choose \(R\)? Larger \(R\) is more practical in that an individual need only be exposed to \(\frac{p}{R}\) treatments, but the order of the separation rate may increase. Theorems \ref{thm:R_sparse} through \ref{thm:R_avg_problem} indicate that if \(R \geq \frac{p}{4s}\), then indeed the separation rate exhibits \(\varepsilon^*(p, s, \gamma, R)^2 \gtrsim \psi_1^2\). However, if \(R \leq \frac{p}{4s}\), then \(\varepsilon^*(p, s, \gamma, R)^2 \asymp \psi_1^2\). Amazingly, we can have the best of both worlds by selecting \(R = \left\lfloor \frac{p}{4s}\right\rfloor\). In other words, we reap the benefit of needing only to expose individuals to at most \(\frac{p}{\frac{p}{4s}} = 4s\) treatments and we sustain no increase in the separation rate! The benefit is substantial when \(s\) is small. For \(s > \frac{p}{2}\), one can directly compare \(\varepsilon^*(p, s, \gamma, 1)\) and \(\varepsilon^*(p, s, \gamma, R)\) and choose \(R\) appropriately. 
    \end{remark}

    \subsection{Preliminaries}\label{section:R_prelim}
    This subsection sets up the notation which will be in force for the proofs presented in Section \ref{section:proofs}. A general strategy for constructing testing procedures is also discussed. Afterwards, each sparsity regime is examined in detail and the testing procedures attaining the upper bounds established in Theorems \ref{thm:R_sparse} through \ref{thm:R_avg_problem} are discussed.
    
    Throughout this section and the proofs presented in Section \ref{section:proofs}, a collection of \(R\) equally-sized disjoint groups \(\{B_k\}_{k=1}^{R}\) will be fixed, where \(R \in [p]\) is an integer dividing \(p\). The matrix \(I_{B_k}\) will refer to the diagonal \(p \times p\) matrix in which the \(i\)th diagonal entry is equal to one if \(i \in B_k\) and equal to zero otherwise. For a vector \(\theta \in \R^p\), we denote \(\theta_{B_k} \in \R^p\) to be the vector with \(i\)th entry equal to \(\theta_i\) if \(i \in B_k\) and equal to zero otherwise. In some cases, we also use \(\theta_{B_k}\) to denote the vector in \(\R^{p/R}\) obtained by taking only those coordinates \(i \in B_k\). Distinguishing notation is not used as context will make it clear. We denote \(\bar{\theta}_{B_k} := \frac{R}{p}\sum_{i \in B_k} \theta_i\). For \(\theta \in \R^p\) and \(\gamma \in [0, 1]\), the notation \(P_{\theta, \gamma, R}\) is used to denote the distribution given by (\ref{model:multiple_random_effects}).
    
    A general strategy for constructing testing procedures is rooted in the following proposition. 
    
    \begin{proposition}\label{prop:theta_upsilon}
        Suppose \(1 \leq s \leq p\) and \(\varepsilon > 0\). Then \(\Theta(p, s, \varepsilon) \subset \Upsilon_{\mathcal{I}}(p, s, \varepsilon) \cup \Upsilon_{\mathcal{II}}(p, s, \varepsilon)\)
        where 
        \begin{equation}\label{space:Upsilon1}
            \Upsilon_{\mathcal{I}}(p, s, \varepsilon) := \left\{ \theta \in \R^p : ||\theta||_0 \leq s \text{ and } \sum_{k=1}^{R} ||\theta_{B_k} - \bar{\theta}_{B_k}\mathbf{1}_{B_k \cap \supp(\theta)}||^2 \geq \frac{\varepsilon^2}{8}\right\}
        \end{equation}
        and 
        \begin{equation}\label{space:Upsilon2}
            \Upsilon_{\mathcal{II}}(p, s, \varepsilon) := \left\{ \theta \in \R^p : ||\theta||_0 \leq s \text{ and } \sum_{\substack{k \in [R] : \\ |B_k \cap \supp(\theta)| > \frac{p}{4R}}} ||\bar{\theta}_{B_k}\mathbf{1}_{B_k}||^2 \geq \frac{\varepsilon^2}{8} \right\}.
        \end{equation}
    \end{proposition}

    Proposition \ref{prop:theta_upsilon} asserts that it suffices to construct constituent tests which are separately sensitive to \(\Upsilon_{\mathcal{I}}(p, s, \varepsilon)\) are \(\Upsilon_{\mathcal{II}}(p, s, \varepsilon)\), and then combine them to furnish a testing procedure sensitive to \(\Theta(p, s, \varepsilon)\). The space \(\Upsilon_{\mathcal{I}}(p, s, \varepsilon)\) is attractive in that for any vector we are guaranteed that a \(s\)-sparse subvector has not too small norm. The space \(\Upsilon_{\mathcal{II}}(p, s, \varepsilon)\) is attractive in that it guarantees averaging does not diminish the norm too much; in particular, the \(R\)-dimensional vector \((\bar{\theta}_{B_1},...,\bar{\theta}_{B_R})\) will end up having sufficiently large norm for detection. Intuitively, \(\Upsilon_{\mathcal{II}}(p, s, \varepsilon)\) reduces a \(p\)-dimensional problem to an \(R\)-dimensional problem. Of course, the constituent tests we employ will depend heavily on which sparsity regime is in force. 

    \subsection{\texorpdfstring{Regime \(1 \leq s \leq \frac{p}{4R}\)}{Regime 1 <= s <= p/4R}}
    In the regime \(1 \leq s \leq \frac{p}{4R}\), Proposition \ref{prop:theta_upsilon} implies \(\Theta(p, s, \varepsilon) \subset \Upsilon_{\mathcal{I}}(p, s, \varepsilon)\). It is also immediate to see \(\Upsilon_{\mathcal{I}}(p, s, \sqrt{8}\varepsilon) \subset \Theta(p, s, \varepsilon)\). Consequently, the \(\varepsilon^*(p, s, \gamma, R)^2\) matches (up to absolute factors) the minimax separation rate when testing against \(\Upsilon_{\mathcal{I}}(p, s, \varepsilon)\). Armed with the knowledge that attention can be restricted to \(\Upsilon_{\mathcal{I}}(p, s, \varepsilon)\), insight from Section \ref{section:ProblemI} can be leveraged. 
    
    Specifically, the following transformation is considered. For \(k \in [R]\) define \(\widetilde{X}_{B_k} \in \R^{p/R}\) with
    \begin{equation*}
        \widetilde{X}_{B_k} := \frac{1}{\sqrt{1-\gamma}}\left(I_{B_k} - \frac{R}{p} \mathbf{1}_{B_k}\mathbf{1}_{B_k}^\intercal \right) X_{B_k} + \frac{\xi_k}{\sqrt{p/R}} \mathbf{1}_{B_k}
    \end{equation*}
    where \(\xi_1,...,\xi_{\frac{p}{R}} \overset{iid}{\sim} N(0, 1)\) are drawn independently of \(X\). Note that \(\widetilde{X}_{B_k} \sim N\left(\frac{\theta_{B_k} - \bar{\theta}_{B_k}\mathbf{1}_{B_k}}{\sqrt{1-\gamma}}, I_{\frac{p}{R}}\right)\) and \(\{\widetilde{X}_{B_k}\}_{k=1}^{R}\) are independent. For \(t > 0\), define the testing statistic 
    \begin{equation*}
        Y_t := \sum_{k=1}^{R} \sum_{j=1}^{p/R} \left((\widetilde{X}_{B_k})_j^2 - \alpha_{t}\right)\mathbbm{1}_{\{|(\widetilde{X}_{B_k})_j| \geq t\}}
    \end{equation*}
    where \(\alpha_t\) is given by (\ref{eqn:alpha_t}). The corresponding test is 
    \begin{equation}\label{test:orthogonal_tsybakov}
        \varphi_{t, r} := \mathbbm{1}_{\{Y_t > r\}}
    \end{equation}
    where \(r \in \R\) is used to set the testing risk. It turns out \(\varphi_{t, r}\) can only be used when \(s < \sqrt{p}\). If \(s \geq \sqrt{p}\), the testing statistic \(\sum_{k=1}^{R} ||\widetilde{X}_{B_k}||^2\) will be used. The corresponding test is 
    \begin{equation}\label{test:block_orthogonal_chisquare}
        \varphi_{r}^{\chi^2} := \mathbbm{1}_{\left\{\sum_{k=1}^{R} ||\widetilde{X}_{B_k}||^2 > p + r \sqrt{p}\right\}}
    \end{equation}
    where \(r > 0\) is used to set the testing risk. 
    
    \begin{proposition}[Upper bound]\label{prop:ubound_R_sparse}
        Suppose \(1 \leq s \leq \frac{p}{4R}\) and \(\gamma \in [0, 1)\). If \(\eta \in (0, 1)\), then there exists a constant \(C_\eta > 0\) depending only on \(\eta\) such that for all \(C > C_\eta\) the testing procedure 
        \begin{equation}\label{test:R_problem1}
            \varphi_1 := 
            \begin{cases}
                \varphi_{t^*, r^*} &\text{if } s < \sqrt{p}, \\
                \varphi_{C^2/16}^{\chi^2} &\text{if } s \geq \sqrt{p}
            \end{cases}
        \end{equation}
        with \(t^* = \sqrt{2 \log\left(1 + \frac{p}{s^2}\right)}\) and \(r^* = \frac{C^2}{64} s\log\left(1 + \frac{p}{s^2}\right)\) satisfies
        \begin{equation*}
            P_{0, \gamma, R}\left\{ \varphi_1 = 1\right\} + \sup_{\theta \in \Theta(p, s, C\psi_1)} P_{\theta, \gamma, R} \left\{ \varphi_1 = 0 \right\} \leq \eta.
        \end{equation*}
        Here, \(\psi_1^2\) is given by (\ref{rate:problemI}).
    \end{proposition}

    The optimality of the test is established by furnishing a matching lower bound. 
    \begin{proposition}[Lower bound]\label{prop:lbound_R_sparse}
        Suppose \(1 \leq s \leq p\) and \(\gamma \in [0, 1)\). If \(\eta \in (0, 1)\), then there exists a constant \(c_\eta > 0\) depending only on \(\eta\) such that for all \(0 < c < c_\eta\) we have \(\mathcal{R}(c\psi_1) \geq 1-\eta\). Here, \(\psi_1^2\) is given by (\ref{rate:problemI}).
    \end{proposition}
    Note that the Proposition \ref{prop:lbound_R_sparse} holds for all \(1 \leq s \leq p\). Propositions \ref{prop:ubound_R_sparse} and \ref{prop:lbound_R_sparse} together give Theorem \ref{thm:R_sparse}.

    \subsection{\texorpdfstring{Regime \(\frac{p}{4R} < s \leq \frac{p}{R} - \sqrt{\frac{p}{R}\log(eR)}\)}{Regime p/4R < s <= p/R - sqrt(p log(eR)/R) }}\label{section:pR_dense}
    In the regime \(\frac{p}{4R} < s \leq \frac{p}{R} - \sqrt{\frac{p}{R}\log(eR)}\), constituent tests sensitive to \(\Upsilon_{\mathcal{II}}(p, s, \varepsilon)\) are now needed. Since \(s < \frac{p}{R}\), we have \(\frac{s}{p/(4R)} \leq 4\). Consequently, if \(\theta \in \Upsilon_{\mathcal{II}}(p, s, \varepsilon)\) then the number of groups \(B_k\) such that \(|B_k \cap \supp(\theta)| > \frac{p}{4R}\) is at most \(4\). Therefore, the triangle inequality implies that there exists a \(B_{k^*}\) such that \(||\bar{\theta}_{B_{k^*}}\mathbf{1}_{B_{k^*}}||^2 \geq \frac{\varepsilon^2}{32}\). In other words, there is a group in which the norm of the signal is large. A natural idea is to scan over all groups. Moreover, since \(s \leq \frac{p}{R} - \sqrt{\frac{p}{R}\log(eR)}\), the signal inside that group is sparse. 

    Two constituent tests sensitive to \(\Upsilon_{\mathcal{II}}(p, s, \varepsilon)\) in different rate regimes will be used.  For \(k \in [R]\), define 

    \begin{equation}\label{test:pR_dense_chisquare}
        \varphi_r^{\chi^2-\text{scan}} := \max_{k \in [R]} \mathbbm{1}_{\left\{ ||\widetilde{X}_{B_k}||^2 > \frac{p}{R} + 2\sqrt{\frac{p}{R} r} + 2r \right\}}
    \end{equation}
    where \(r \in \R\) is used to set the desired testing risk. Additionally, define 
    \begin{equation}\label{test:linear_scan}
        \varphi_r^{\mathbf{1}-\text{scan}} := \max_{k \in [R]} \mathbbm{1}_{\left\{ \left\langle \sqrt{\frac{R}{p}}\mathbf{1}_{B_k}, X \right\rangle^2 > \left(1-\gamma+\gamma \frac{p}{R}\right)\left(1 + r\right) \right\}}.
    \end{equation}
    The test \(\varphi_1\) given by (\ref{test:R_problem1}) will also be used for sensitivity to \(\Upsilon_{\mathcal{I}}(p, s, \varepsilon)\). 

    \begin{proposition}[Upper bound]\label{prop:ubound_pR_dense}
        Suppose \(\frac{p}{4R} < s \leq \frac{p}{R} - \sqrt{\frac{p}{R}\log(eR)}\) and \(\gamma \in [0, 1)\). If \(\eta \in (0, 1)\), then there exists a constant \(C_\eta > 0\) depending only on \(\eta\) such that for all \(C > C_\eta\) the testing procedure \(\varphi^* := \varphi_1 \vee \varphi_2\) where
        \begin{equation*}
            \varphi_2 := 
            \begin{cases}
                \varphi_{\frac{C^2}{128}\log(eR)}^{\chi^2-\text{scan}} &\text{if } \upsilon^2 = \frac{(1-\gamma)p}{p-Rs} \left(\sqrt{\frac{p}{R}\log(eR)} + \log(R)\right), \\
                \varphi_{\frac{C^2}{64}\log(eR)}^{\mathbf{1}-\text{scan}} &\text{if } \upsilon^2 = \left(1 - \gamma + \gamma \frac{p}{R}\right) \log(eR)
            \end{cases}
        \end{equation*}
        satisfies 
        \begin{equation*}
            P_{0, \gamma, R}\{\varphi^* = 1\} + \sup_{\theta \in \Theta(p, s, C(\psi_1 \vee \upsilon))} P_{\theta, \gamma, R}\{\varphi^* = 0\} \leq \eta.
        \end{equation*}
        Here, \(\psi_1^2\) is given by (\ref{rate:problemI}), \(\varphi_1\) is given by (\ref{test:R_problem1}), and \(\upsilon^2\) is given by (\ref{rate:pR_dense_upsilon}).
    \end{proposition}
    
    We obtain a lower bound which matches the upper bound. 
    \begin{proposition}[Lower bound]\label{prop:lbound_pR_dense}
        Suppose \(\frac{p}{4R} < s \leq \frac{p}{R} - \sqrt{\frac{p}{R}\log(eR)}\) and \(\gamma \in [0, 1)\). If \(\eta \in (0, 1)\), then there exists a constant \(c_\eta > 0\) such that for all \(0 < c < c_\eta\) we have \(\mathcal{R}(c(\psi_1 \vee \upsilon)) \geq 1-\eta\) where \(\psi_1^2\) is given by (\ref{rate:problemI}) and \(\upsilon^2\) is given by (\ref{rate:pR_dense_upsilon}).
    \end{proposition}

    \subsection{\texorpdfstring{Regime \(\frac{p}{R} - \sqrt{\frac{p}{R}\log(eR)} < s < \frac{p}{R}\)}{Regime p/R - sqrt(p log(eR)/R) < s < p/R}}
    The regime \(\frac{p}{R} - \sqrt{\frac{p}{R}\log(eR)} < s < \frac{p}{R}\) is similar to the previous sparsity regime. The strategy of scanning over all groups is employed here as well. For \(k \in [R]\) and \(t > 0\) define the testing statistic 
    \begin{equation*}
        Y_t^{(k)} := \sum_{i \in B_k} \left((\widetilde{X}_{B_k})_i^2 - \alpha_t\right) \mathbbm{1}_{\left\{|(\widetilde{X}_{B_k})_i| \geq t\right\}}
    \end{equation*}
    where \(\alpha_t\) is given by (\ref{eqn:alpha_t}). With this testing statistic, define the test 
    \begin{equation}\label{test:pR_very_dense_tsybakov}
        \varphi_{t, r}^{\text{scan}} := \max_{k \in [R]} \mathbbm{1}_{\left\{Y_t^{(k)} > r\right\}}
    \end{equation}
    where \(r \in \R\) is used to set the desired testing risk. The test \(\varphi_{r}^{\mathbf{1}-\text{scan}}\) will be employed as well for sensitivity to \(\Upsilon_{\mathcal{II}}(p, s, \varepsilon)\). The test \(\varphi_1\) given by (\ref{test:R_problem1}) will be used for sensitivity to \(\Upsilon_{\mathcal{I}}(p, s, \varepsilon)\). 

    \begin{proposition}[Upper bound]\label{prop:ubound_pR_very_dense}
        Suppose \(s > \frac{p}{4R}\). Further suppose \(\frac{p}{R} - \sqrt{\frac{p}{R}\log(eR)} < s < \frac{p}{R}\) and \(\gamma \in [0, 1)\). If \(\eta \in (0, 1)\), then there exists a constant \(C_\eta > 0\) depending only on \(\eta\) such that for all \(C > C_\eta\) the testing procedure \(\varphi^* := \varphi_1 \vee \varphi_2\) where 
        \begin{equation*}
            \varphi_2 := 
            \begin{cases}
                \varphi_{\widetilde{t}, \widetilde{r}}^{\text{scan}} &\text{if } \upsilon^2 = \frac{(1-\gamma)p}{p-Rs}\left( \left(\frac{p}{R} - s\right)\log\left(1 + \frac{Rp \log(eR)}{(p-Rs)^2}\right)+ \log(R) \right), \\
                \varphi_{\frac{C^2}{64}\log(eR)}^{\mathbf{1}-\text{scan}} &\text{if } \upsilon^2 = \left(1 -\gamma + \gamma\frac{p}{R}\right)\log(eR)
            \end{cases}
        \end{equation*}
        with \(\widetilde{t} := \sqrt{2 \log\left(1 + \frac{Rp \log(eR)}{(p-Rs)^2}\right)}\) and \(\widetilde{r} := \frac{C^2}{64} \left(\left( \frac{p}{R} - s\right) \log\left(1 + \frac{Rp\log(eR)}{(p-Rs)^2}\right) + \log(R)\right)\) satisfies 
        \begin{equation*}
            P_{0, \gamma, R}\{\varphi^* = 1\} + \sup_{\theta \in \Theta(p, s, C(\psi_1 \vee \upsilon))} P_{\theta, \gamma, R}\{\varphi^* = 0\} \leq \eta.
        \end{equation*}
        Here, \(\psi_1^2\) is given by (\ref{rate:problemI}), \(\varphi_1\) is given by (\ref{test:R_problem1}), and \(\upsilon^2\) is given by (\ref{rate:pR_dense_upsilon}).
    \end{proposition}

    We obtain a matching lower bound.

    \begin{proposition}[Lower bound]\label{prop:lbound_pR_very_dense}
        Suppose \(s > \frac{p}{4R}\). Further suppose \(\frac{p}{R} - \sqrt{\frac{p}{R}\log(eR)} < s < \frac{p}{R}\) and \(\gamma \in [0, 1)\). If \(\eta \in (0, 1)\), then there exists a constant \(c_\eta > 0\) depending only on \(\eta\) such that for all \(0 < c < c_\eta\) we have \(\mathcal{R}(c(\psi_1 \vee \upsilon)) \geq 1-\eta\) where \(\psi_1^2\) is given by (\ref{rate:problemI}) and \(\upsilon^2\) is given by (\ref{rate:pR_dense_upsilon}).
    \end{proposition}

    Propositions \ref{prop:ubound_pR_dense}, \ref{prop:lbound_pR_dense}, \ref{prop:ubound_pR_very_dense}, and \ref{prop:lbound_pR_very_dense} are combined to give Theorem \ref{thm:pR_dense}. 

    \subsection{\texorpdfstring{Regime \(\frac{p}{R} \leq s \leq p\)}{Regime p/R <= s <= p\)}}\label{section:R_avg}
    In the regime \(\frac{p}{R} \leq s \leq p\), the space \(\Upsilon_{\mathcal{II}}(p, s, \varepsilon)\) plays a more prominent role. Specifically, the space can be thought of as consisting of those \(\theta\) such that \((\bar{\theta}_{B_1},...,\bar{\theta}_{B_R})\) exhibits a \(\frac{s}{p/R}\)-sparse subvector with not too small squared norm. A natural approach is thus to examine the \(R\)-dimensional sparse signal detection problem with the observation \((\bar{X}_{B_1},...,\bar{X}_{B_R})\).

    For sensitivity to the space \(\Upsilon_{\mathcal{I}}(p, s, \varepsilon)\), the test \(\varphi_1\) given by (\ref{test:R_problem1}) is used. For sensitivity to \(\Upsilon_{\mathcal{II}}(p, s, \varepsilon)\), consider the following construction. Note that \(\{\bar{X}_{B_k}\}_{k=1}^{R}\) are independent and that 
    \begin{equation*}
        \sqrt{\frac{p}{R}}\bar{X}_{B_k} \sim N\left(\sqrt{\frac{p}{R}}\bar{\theta}_{B_k} , 1 - \gamma + \gamma \frac{p}{R} \right).
    \end{equation*}
    For \(t > 0\) define the testing statistic 
    \begin{equation*}
        \bar{Y}_t := \sum_{k=1}^{R} \left(\frac{\frac{p}{R} \bar{X}_{B_k}^2}{1-\gamma+\gamma \frac{p}{R}} - \alpha_{t}\right)\mathbbm{1}_{\left\{ \left|\frac{\sqrt{\frac{p}{R}} \bar{X}_{B_k}}{\sqrt{1-\gamma+\gamma \frac{p}{R}}} \right| \geq t\right\}}
    \end{equation*}
    where \(\alpha_t\) is given by (\ref{eqn:alpha_t}), and the corresponding test
    \begin{equation}\label{test:average_tsybakov}
        \bar{\varphi}_{t, r} := \mathbbm{1}_{\{\bar{Y}_t > r\}}.
    \end{equation}
    The test \(\bar{\varphi}_{t, r}\) is useful when \(s < \frac{p}{\sqrt{R}}\). When \(s \geq \frac{p}{\sqrt{R}}\), the testing statistic \(\sum_{k=1}^{R} ||\bar{X}_{B_k}\mathbf{1}_{B_k}||^2\) is used. Specifically, the test 
    \begin{equation}\label{test:average_chisquare}
        \bar{\varphi}_{r}^{\chi^2} := \mathbbm{1}_{\left\{\sum_{k=1}^{R} ||\bar{X}_{B_k} \mathbf{1}_{B_k}||^2 > \left(1 - \gamma + \gamma \frac{p}{R}\right)(R + r\sqrt{R})\right\}}
    \end{equation}
    will be used, where \(r \in \R\) is used to set the testing risk. The final testing procedure combines \(\varphi_1\), \(\bar{\varphi}_{t, r}\), and \(\bar{\varphi}_{r}^{\chi^2}\). The following proposition establishes its performance. 
    
    \begin{proposition}[Upper bound]\label{prop:R_avg_ubound}
        Suppose \(\frac{p}{R} \leq s \leq p\) and \(\gamma \in [0, 1)\). Set
        \begin{equation*}
            \bar{\rho}^2 := 
            \begin{cases}
                \left(1-\gamma+\gamma \frac{p}{R} \right) \frac{4Rs}{p} \log\left(1 + \frac{p^2}{16Rs^2}\right) &\text{if } s < \frac{p}{\sqrt{R}}, \\
                \left(1 - \gamma + \gamma \frac{p}{R}\right) \sqrt{R} &\text{if } s \geq \frac{p}{\sqrt{R}}. 
            \end{cases}
        \end{equation*}
        If \(\eta \in (0, 1)\), then there exists a constant \(C_\eta\) depending only on \(\eta\) such that for all \(C > C_\eta\), the testing procedure \(\varphi^* := \varphi_1 \vee \varphi_2\) satisfies 
        \begin{equation*}
            P_{0, \gamma, R}\left\{ \varphi^* = 1 \right\} + \sup_{\theta \in \Theta(p, s, C(\psi_1 \vee \bar{\rho}))} P_{\theta, \gamma, R}\{\varphi^* = 0\} \leq \eta. 
        \end{equation*}
        Here, \(\psi_1^2\) is given by (\ref{rate:problemI}), \(\varphi_1\) is given by (\ref{test:R_problem1}), and 
        \begin{equation}\label{test:R_avg}
            \varphi_2 := 
            \begin{cases}
                \bar{\varphi}_{\bar{t}, \bar{r}} &\text{if } s < \frac{p}{\sqrt{R}}, \\
                \bar{\varphi}_{C^2/16}^{\chi^2} &\text{if } s \geq \frac{p}{\sqrt{R}}
            \end{cases}
        \end{equation}
        with \(\bar{t} = \sqrt{2 \log\left(1 + \frac{p^2}{Rs^2}\right)}\) and \(\bar{r} = \frac{C^2}{64} \frac{4Rs}{p}\log\left(1 + \frac{p^2}{16Rs^2}\right)\).
    \end{proposition}

    The optimality of our testing procedure is implied by a matching lower bound.

    \begin{proposition}[Lower bound]\label{prop:R_avg_lbound}
        Suppose \(\frac{p}{R} \leq s \leq p\) and \(\gamma \in [0, 1)\). Set 
        \begin{equation*}
            \underline{\rho}^2 := 
            \begin{cases}
                \left(1 - \gamma + \gamma \frac{p}{R}\right) \frac{Rs}{p} \log\left(1 + \frac{p^2}{Rs^2}\right) &\text{if } s < \frac{p}{\sqrt{R}}, \\
                \left(1 - \gamma + \gamma \frac{p}{R}\right) \sqrt{R} &\text{if } s \geq \frac{p}{\sqrt{R}}.
            \end{cases}
        \end{equation*}
        If \(\eta \in (0, 1)\), then there exists a constant \(c_\eta > 0\) depending only on \(\eta\) such that for all \(0 < c < c_\eta\) we have \(\mathcal{R}(c(\psi_1 \vee \underline{\rho})) \geq 1-\eta\) where \(\psi_1^2\) is given by (\ref{rate:problemI}). 
    \end{proposition}
    Propositions \ref{prop:R_avg_ubound} and \ref{prop:R_avg_lbound} together imply Theorem \ref{thm:R_avg_problem}.

    \subsection{Perfect correlation}\label{section:R_perfect_correlation}
    In this section, the case \(\gamma = 1\) is addressed.
    \begin{theorem}\label{thm:R_perfect_correlation}
        If \(s \in [p]\) and \(R \in [p]\) divides \(p\), then 
        \begin{equation*}
            \varepsilon^*(p, s, 1, R)^2 \asymp 
            \begin{cases}
                0 &\text{if } s < \frac{p}{R}, \\
                s \log\left(1 + \frac{p^2}{Rs^2}\right) &\text{if } \frac{p}{R} \leq s < \frac{p}{\sqrt{R}}, \\
                \frac{p}{\sqrt{R}} &\text{if } \frac{p}{\sqrt{R}} \leq s \leq p. 
            \end{cases}
        \end{equation*}
    \end{theorem}

    When \(s < \frac{p}{R}\), the situation is similar to the one discussed in Section \ref{section:perfect_correlation}. Specifically, for \(k \in [R]\) one can write \(X_{B_k} = \theta_{B_k} + W_k \mathbf{1}_{B_k}\) where \(W_1,...,W_R \overset{iid}{\sim} N(0, 1)\). Consequently, \(X_{B_k} - \bar{X}_{B_k}\mathbf{1}_{B_k} = \theta_{B_k} - \bar{\theta}_{B_k}\) for all \(k\) almost surely. Now for any \(\varepsilon > 0\), there must be a \(k^*\) such that \(\theta_{B_{k^*}} \not \in \spn\{\mathbf{1}_{B_{k^*}}\}\) for \(\theta \in \Theta(p, s, \varepsilon)\). Consequently, the test which rejects the null hypothesis when \(\sum_{k=1}^{R} ||X_{B_k} - \bar{X}_{B_k}\mathbf{1}_{B_k}||^2\) is nonzero attains zero testing risk. Since \(\varepsilon > 0\) was arbitrary, the minimax separation rate is degenerate when \(s < \frac{p}{R}\). 

    When \(s \geq \frac{p}{R}\), the minimax separation rate becomes nontrivial due to those signals \(\theta \in \Theta(p, s, \varepsilon)\) which are constant on the groups, i.e. \(\theta_{B_k} \in \spn\left\{ \mathbf{1}_{B_k}\right\}\) for all \(k\). Consequently, the effective dimension of the problem reduces to \(R\) and the effective sparsity is \(\left\lfloor \frac{s}{p/R}\right\rfloor\). Essentially, the problem reduces to a sparse signal detection problem given independent observations \(\sqrt{\frac{p}{R}} \bar{X}_{B_k} \sim N\left( \sqrt{\frac{p}{R}}\bar{\theta}_{B_k}, \frac{p}{R}\right)\) for \(k \in [R]\). The result of Collier et al. \cite{collierMinimaxEstimationLinear2017} thus predicts that the squared separation rate should be of order \(\frac{p}{R} \cdot \frac{s}{p/R} \log\left(1 + \frac{p^2}{Rs^2}\right)\) for \(\frac{s}{p/R} < \sqrt{R}\) and of order \(\frac{p}{R} \cdot \sqrt{R}\) for \(\frac{s}{p/R} \geq \sqrt{R}\). Indeed, Theorem \ref{thm:R_perfect_correlation} confirms this prediction.

    \section{Discussion}

    \subsection{Unknown sparsity}\label{section:sparsity_adaptive}
    The rate-optimal testing procedure (\ref{test:full_test}) requires knowledge of the sparsity level. In practice, this information is typically unknown to the statistician and so it is of practical interest to construct testing procedures which adapt to the sparsity level. Fortunately, it turns out that the simple idea of scanning over all sparsity levels and using (\ref{test:full_test}) at each level furnishes an adaptive testing procedure which achieves the minimax separation rate. 
    
    To describe the construction, we recall the constituent tests making up (\ref{test:full_test}) and make the dependence on the sparsity explicit. First, we will make use of the test (\ref{test:problemI_chisquare}). Notice it does not require knowledge of the sparsity. Second, we will make use of the test (\ref{test:problemI_tsybakov}) with the choice \(t = t(s)\) and \(r = r(s)\) depending on the sparsity, that is, we use \(\varphi_{t(s), r(s)} = \mathbbm{1}_{\{Y_{t(s)} > r(s)\}}\). Lastly, we make use of the test (\ref{test:problemII_linear}) which does not depend on the sparsity. The idea of scanning over all sparsity levels amounts to taking a maximum over these constituent tests, meaning we define
    \begin{equation}\label{test:adaptive}
        \varphi_{\text{adaptive}} := \left(\max_{1 \leq s < \sqrt{p}} \varphi_{t(s), r(s)} \right) \vee \varphi_{r_{\chi^2}}^{\chi^2} \vee \left(\max_{p-\sqrt{p} < s < p} \varphi_{\widetilde{t}(s), \widetilde{r}(s)}\right) \vee \varphi^{\mathbf{1}_p}_{r_{\mathbf{1}_p}}. 
    \end{equation}
    where \(t(s)\) and \(\widetilde{t}(s)\) are thresholds used in the test (\ref{test:problemI_tsybakov}) with the dependence on the sparsity \(s\) made explicit. The cutoffs \(r(s)\) and \(\widetilde{r}(s)\) also depend on the sparsity. The cutoffs \(r_{\chi^2}\) and \(r_{\mathbf{1}_p}\) for the tests (\ref{test:problemI_chisquare}) and (\ref{test:problemII_linear}) do not depend on \(s\).
    \begin{theorem}\label{thm:sparsity_adaptive}
        Suppose \(1 \leq s^* \leq p\) and \(\gamma \in [0, 1)\). If \(\eta \in (0, 1)\), then there exists a constant \(C_\eta > 0\) depending only on \(\eta\) such that for all \(C > C_\eta\) the test \(\varphi_{\text{adaptive}}\) given in (\ref{test:adaptive}) with the choices \(t(s) = \sqrt{2 \log\left(1 + \frac{p}{s^2}\right)}, \widetilde{t}(s) = \sqrt{2 \log\left(1 + \frac{p}{(p-s)^2}\right)}, r(s) = \frac{C^2}{32} s \log\left(1 + \frac{p}{s^2}\right), \widetilde{r}(s) = \frac{C^2}{8} (p-s)\log\left(1 + \frac{p}{(p-s)^2}\right), r_{\chi^2} = \frac{C^2}{2}, r_{\mathbf{1}_p} = \frac{C^2}{2}\) satisfies
        \begin{equation*}
            P_{0, \gamma}\{\varphi_{\text{adaptive}} = 1\} + \sup_{\theta \in \Theta(p, s^*, C\psi)} P_{\theta, \gamma}\{\varphi_{\text{adaptive}} = 0\} \leq \eta
        \end{equation*} 
        where \(\psi^2\) is given by the right hand side of Theorem \ref{thm:minimax_rate}.
    \end{theorem}

    \subsection{Rank-one correlation patterns}\label{section:rank_one_correlation}
    In this section, we extend the model (\ref{model:single_random_effect}) by examining ``rank-one'' correlation patterns. Specifically, consider the random effects model
    \begin{equation}\label{model:rank_one_correlation_additive}
        X_i = \theta_i + \sqrt{\gamma} Wv_i + \sqrt{1-\gamma} Z_i
    \end{equation}
    for \(1 \leq i \leq p\) where \(W,Z_1,...,Z_p\) are independent and identically distributed standard Gaussian random variables. Here, \(\gamma \in [0, 1]\) denotes the correlation level and \(v := (v_1,...,v_p) \in \R^p\). The marginal distribution of the observation \(X := (X_1,...,X_p)\) is given by
    \begin{equation}\label{model:rank_one_correlation}
        X \sim N\left( \theta, (1-\gamma) I_p + \gamma vv^\intercal \right)
    \end{equation}
    where \(\theta \in \R^p\), \(I_p \in \R^{p \times p}\) denotes the identity matrix. We assume the scaling \(||v|| = \sqrt{p}\) in order to keep the scale of \(\gamma\) the same as in (\ref{model:single_random_effect}). We say that (\ref{model:rank_one_correlation}) exhibits a ``rank-one'' correlation pattern since the covariance matrix \((1-\gamma)I_p + \gamma vv^\intercal\) is a convex combination between \(I_p\) and a rank-one positive-semidefinite matrix. The model (\ref{model:rank_one_correlation_additive}) is known as a ``random slope'' model in the linear mixed models literature \cite{agrestiFoundationsLinearGeneralized2015,gelmanDataAnalysisUsing2006a}. Specifically, each \(v_i\) can be interpreted as a one-dimensional covariate associated to the individual \(i\). The random effect \(\sqrt{\gamma}W\) then has the interpretation as a random slope since \(W\) is random. While the results of this section may have implications for practical modelling and study design concerns, we do not dwell on viewing (\ref{model:rank_one_correlation_additive}) as a random slope model. Rather, our discussion exclusively refers to \(v\) as a correlation pattern.
    
    We focus on the setting where \(v\) is known and present a partial result. Given an observation \(X\) from the model (\ref{model:rank_one_correlation}) in which \(v\) is known, we are interested in the hypothesis testing problem (\ref{problem:multiple_random_effects_1})-(\ref{problem:multiple_random_effects_2}). The minimax testing risk is defined to be 
    \begin{equation}\label{risk:known_v}
        \mathcal{R}(\varepsilon, v) = \inf_{\varphi}\left\{ P_{0, \gamma, v}\left\{\varphi = 1\right\} + \sup_{\theta \in \Theta(p, s, \varepsilon)} P_{\theta, \gamma, v}\left\{\varphi = 0 \right\} \right\}
    \end{equation}
    where the infimum runs over all measurable functions \(\varphi : \R^p \to \{0, 1\}\). The minimax separation rate \(\varepsilon^* = \varepsilon^*(p, s, \gamma, v)\) can be defined analogously to Definition \ref{def:separation_rate}. Note, however, that the rate depends on \(v\). Indeed, understanding how \(v\) affects the rate is the principal motivation in this section.

    A slight break in the narrative flow is necessary to first deal with the case \(\gamma = 1\). The situation is very similar to that described in Section \ref{section:perfect_correlation} with some slight modifications to deal with the direction of the correlation pattern. 
    \begin{proposition}\label{prop:known_v_perfect_correlation}
        If \(1 \leq s \leq p\), then 
        \begin{equation*}
            \varepsilon^*(p, s, 1, v)^2 \asymp 
            \begin{cases}
                0 &\text{if } s < ||v||_0, \\
                p &\text{if } s \geq ||v||_0.
            \end{cases}
        \end{equation*}
    \end{proposition}
    
    For the discussion in the remainder of this section, the reader should keep the setting \(\gamma \in [0, 1)\) in mind. In preparation for the statement of our main result define 
    \begin{equation}\label{eqn:omega_v}
        \omega(v) := \max\left\{ 0 \leq s \leq p : \max_{\substack{S \subset [p], \\ |S| \leq s}} ||v_S||^2 \leq \frac{p}{4}\right\}.
    \end{equation}
    Note \(\omega(v) < ||v||_0\) since \(||v||^2 = p\). The main result of this section is the following. 
    \begin{theorem}\label{thm:known_v_rate}
        Suppose \(1 \leq s \leq \omega(v)\) and \(\gamma \in [0, 1]\). If \(\omega(v) \geq \sqrt{p}\), then 
        \begin{equation*}
            \varepsilon^*(p, s, \gamma, v)^2 \asymp 
            \begin{cases}
                (1-\gamma)s \log\left(1 + \frac{p}{s^2}\right) &\text{if } s \leq \sqrt{p}, \\
                (1-\gamma) \sqrt{p} &\text{if } \sqrt{p} < s \leq \omega(v).
            \end{cases}
        \end{equation*}
        If \(\omega(v) < \sqrt{p}\), then 
        \begin{equation*}
            \varepsilon^*(p, s, \gamma, v)^2 \asymp (1-\gamma) s \log\left(1 + \frac{p}{s^2}\right).
        \end{equation*}
    \end{theorem}

    Interestingly, for any \(v, v' \in \R^p\) with \(||v|| = ||v'|| = \sqrt{p}\), we have the rate equivalence \(\varepsilon^*(p, s, \gamma, v)^2 \asymp \varepsilon^*(p, s, \gamma, v')^2\) for all \(1 \leq s \leq \omega(v) \wedge \omega(v')\). As an example, Theorem \ref{thm:known_v_rate} implies that if \(v, v' \in \{-1, 1\}^p\), then \(\varepsilon^*(p , s, \gamma, v)^2 \asymp \varepsilon^*(p, s, \gamma, v')^2\). Surprisingly, this holds even when, say, \(p\) is even and \(v, v' \in \{-1, 1\}^p\) are chosen to be orthogonal. To consider an example exhibiting heterogeneity in the magnitude order of the coordinates, consider the choice \(v = \mathbf{1}_p\) and \(v'\) being the vector with the first \(\sqrt{p}\) coordinates equal to \(p^{1/4}\) and the remaining coordinates equal to zero. It follows from Theorem \ref{thm:known_v_rate} that \(\varepsilon^*(p, s, \gamma, v)^2 \asymp \varepsilon^*(p, s, \gamma, v')^2 \asymp (1-\gamma) s \log\left(1 + \frac{p}{s^2}\right)\) for \(s \leq \frac{\sqrt{p}}{4}\). It is quite intriguing to see that the minimax separation rates are equivalent for \(s \leq \frac{\sqrt{p}}{4}\) even though \(\langle \frac{v}{||v||}, \frac{v'}{||v'||} \rangle = p^{-1/4}\), meaning that \(v\) and \(v'\) become orthogonal in the limit as \(p \to \infty\). 

    Theorem \ref{thm:known_v_rate} only characterizes the minimax separation rate for a limited range of sparsity levels. As hinted by the preceding discussion, the rate appears to exhibit a subtle dependence on the direction \(v\). Unfortunately, a complete and rigorous characterization of the rate is outside our grasp and so the problem remains open. From the preceding discussion, one might be tempted to conjecture that the minimax separation rate for any direction \(v''\) and any \(\gamma \in [0, 1)\) should be of the order \((1-\gamma) s \log\left(1 + \frac{p}{s^2}\right)\) for \(s \lesssim \sqrt{p}\). 
    
    Yet the choice \(v'' = \sqrt{p}e_1\) where \(e_1\) is the first standard basis vector in \(\R^p\) falsifies the conjecture in the regime \(\gamma = \omega(p^{-1/2})\). In particular, a reduction to a simple versus simple testing problem establishes \(\varepsilon^*(p, s, \gamma, v'')^2 \gtrsim 1-\gamma+\gamma p\) for \(s \geq 1\). Observe that \(1-\gamma+\gamma p = \omega((1-\gamma)\sqrt{p})\) when \(\gamma = \omega(p^{-1/2})\). Since \((1-\gamma)\sqrt{p} \gtrsim (1-\gamma)s \log\left(1 + \frac{p}{s^2}\right)\) for \(s \lesssim \sqrt{p}\), the aforementioned conjecture is false when \(\gamma = \omega(p^{-1/2})\). Note that in the regime \(\gamma \lesssim p^{-1/2}\) and \(1-\gamma \asymp 1\) we have \(1-\gamma+\gamma p \lesssim \sqrt{p} \asymp (1-\gamma)\sqrt{p}\). So the lower bound \(\varepsilon^*(p, s, \gamma, v'')^2 \gtrsim 1-\gamma+\gamma p\) does not falsify the conjecture in this correlation regime when \(s \asymp \sqrt{p}\). 

    \subsection{Future directions}
    The existing literature studying fundamental detection limits in dependent Gaussian models is quite limited. We hope the present work in the stylized setting of Gaussian sequence mixed models is a small piece at the beginning of a budding research program engaging the statistical community. There are a number of directions for future work building upon our results. An extension of Theorem \ref{thm:known_v_rate} to a full characterization of the minimax separation rate for any \(||v|| = \sqrt{p}\) would be of interest. A very ambitious goal is a full characterization of the minimax separation rate for a general covariance matrix \(\Sigma\), immensely extending Theorem 7 of \cite{liuMinimaxRatesSparse2021}. Arguably, this goal is too ambitious as such a result must encapsulate the two substantially different separation rates from Theorem \ref{thm:minimax_rate} and Theorems \ref{thm:R_sparse} through \ref{thm:R_avg_problem}. Instead, a more tractable program involves studying the sparse signal detection problem with specific and explicit covariance structures of interest, as done in the present paper. For example, one might consider auto-regressive structures, various kernel matrices, convex combination between an \(r\)-rank matrix and the identity matrix, etc. 
    
    Another direction for further work is to consider different separation metrics for the alternative hypothesis rather than the Euclidean norm. Our arguments relied heavily on the Pythagorean identity for squared Euclidean norm, and it would be interesting to see how the rates differ under other separation metrics. Additionally, other structured signals rather than just sparse ones can be considered. Structured signal detection has been studied in Ising models \cite{debDetectingStructuredSignals2020} as well as in Gaussian models with independent observations \cite{arias-castroDetectionAnomalousCluster2011a,arias-castroSearchingTrailEvidence2008a}, and extensions to Gaussian models exhibiting dependence are of interest.

    \section{Proofs}\label{section:proofs}
    In this section, all proofs for the results presented in the main body of the paper are provided. 

    \subsection{Some simple geometric consequences of sparsity}\label{section:geometric_observations}
    In this section, we state and prove some simple geometric results about sparse vectors. While elementary, they are invaluable for deriving upper bounds for the minimax separation rates \(\varepsilon_1^*\) and \(\varepsilon_2^*\) and are used repeatedly.

    \begin{lemma}\label{lemma:v_orthog_approximation}
        Suppose \(v \in \R^p\) with \(||v|| = \sqrt{p}\). If \(\theta \in \R^p\) and \(||\theta||_0 \leq s\), then 
        \begin{equation*}
            \left|\left|\theta - \frac{1}{p}\langle v, \theta \rangle v\right|\right|^2 \geq ||\theta||^2 \cdot \frac{p - \max_{\substack{S \subset [p], \\ |S| \leq s}} ||v_S||^2}{p}.
        \end{equation*}
    \end{lemma}
    \begin{proof}
        The Pythagorean identity gives \(||\theta - p^{-1}\langle v, \theta \rangle v||^2 = ||\theta||^2 - p^{-1} \langle v, \theta\rangle^2\). Since \(||\theta||_0 \leq s\), it follows that 
        \begin{align*}
            ||\theta - p^{-1}\langle v, \theta \rangle v||^2 &\geq ||\theta||^2 - p^{-1} \cdot \max_{\substack{||\mu||_0 \leq s, \\ ||\mu|| = ||\theta||}} \langle v, \mu\rangle^2.
        \end{align*}
        Since \(||\mu||_0 \leq s\), it follows that \(\langle v, \mu\rangle = \langle v_S, \mu\rangle\) where \(S = \supp(\mu)\). Therefore, we have by Cauchy-Schwarz 
        \begin{align*}
            \max_{\substack{||\mu||_0 \leq s, \\ ||\mu|| = ||\theta||}} \langle v, \mu\rangle^2 = \max_{\substack{S \subset [p], \\ |S| \leq s}} \max_{\substack{\supp(\mu) = S, \\ ||\mu|| = ||\theta||}} \langle v_S, \mu\rangle^2 \leq ||\theta||^2 \cdot \max_{\substack{S \subset [p], \\ |S| \leq s}} ||v_S||^2. 
        \end{align*}
        Hence, 
        \begin{equation*}
            ||\theta - p^{-1}\langle v, \theta \rangle v||^2 \geq ||\theta||^2 - p^{-1} ||\theta||^2 \cdot \max_{\substack{S \subset [p], \\ |S| \leq s}} ||v_S||^2 = ||\theta||^2 \cdot \frac{p - \max_{\substack{S \subset [p], \\ |S| \leq s}} ||v_S||^2}{p}
        \end{equation*}
        as desired.
    \end{proof}
    \begin{lemma}\label{lemma:v_supp_approximation}
        Suppose \(v \in \R^p\) with \(||v|| = \sqrt{p}\). If \(\theta \in \R^p\) and \(||\theta||_0 \leq s\), then 
        \begin{equation*}
            \left|\left| \theta - \frac{1}{p}\langle v, \theta \rangle v_{\supp(\theta)} \right|\right|^2 \geq ||\theta||^2 \cdot \frac{p - 2\max_{\substack{S \subset [p], \\ |S| \leq s}} ||v_S||^2}{p}.
        \end{equation*}
    \end{lemma}
    \begin{proof}
        Consider that 
        \begin{align*}
            \left|\left| \theta - \frac{1}{p}\langle v, \theta \rangle v_{\supp(\theta)} \right|\right|^2 &= \left|\left| \theta - \frac{1}{p}\langle v, \theta \rangle v\right|\right|^2 - \left|\left|\frac{1}{p}\langle v, \theta \rangle v_{\supp(\theta)^c}\right|\right|^2 \\
            &\geq \left|\left|\theta - \frac{1}{p}\langle v, \theta \rangle v\right|\right|^2 - \left|\left|\frac{1}{p} \langle v, \theta \rangle v \right|\right|^2 \\
            &= 2\left|\left|\theta - \frac{1}{p}\langle v, \theta \rangle v\right|\right|^2 - ||\theta||^2.
        \end{align*}
        Applying Lemma \ref{lemma:v_orthog_approximation}, we have 
        \begin{align*}
            2\left|\left| \theta - \frac{1}{p}\langle v, \theta \rangle v\right|\right|^2 - ||\theta||^2 &\geq 2 ||\theta||^2 \cdot \frac{p - \max_{\substack{S \subset [p], \\ |S| \leq s}} ||v_S||^2}{p} - ||\theta||^2 \\
            &= ||\theta||^2 \cdot \frac{p - 2\max_{\substack{S \subset [p], \\ |S| \leq s}} ||v_S||^2}{p}
        \end{align*}
        as desired.
    \end{proof}

    The special case of \(v = \mathbf{1}_p\) is obviously of great interest to us, so we explicitly state the corresponding results as corollaries. 
    \begin{corollary}\label{corollary:orthog_approximation}
        If \(1 \leq s \leq p\), \(\theta \in \R^p\), and \(||\theta||_0 \leq s\), then \(||\theta - \bar{\theta}\mathbf{1}_p||^2 \geq ||\theta||^2 \cdot \frac{p-s}{p}\).
    \end{corollary}

    \begin{corollary}\label{corollary:supp_approximation}
        If \(1 \leq s \leq p\), \(\theta \in \R^p\), and \(||\theta||_0 \leq s\), then \(||\theta - \bar{\theta}\mathbf{1}_{\supp(\theta)}||^2 \geq ||\theta||^2 \cdot \frac{p-2s}{p}\).
    \end{corollary}
    
    \subsection{Key results for upper bounds}

    The following results will be repeatedly used in our arguments. Recall that the statistic of \cite{collierMinimaxEstimationLinear2017} was designed for the sparse signal detection problem in which the mean vector of the Gaussian observation with spherical covariance is sparse. Both \cite{collierMinimaxEstimationLinear2017} and \cite{liuMinimaxRatesSparse2021} essentially prove the statistic's utility. Proposition \ref{prop:supp_tsybakov} crucially extends the statistic's applicability to the setting in which the mean vector exhibits a sparse subvector with large norm. Notably, the mean vector itself need not be sparse. To prove Proposition \ref{prop:supp_tsybakov}, we follow the strategy of \cite{liuMinimaxRatesSparse2021} but with important modifications to expand the statistic's applicability. Recall the notation \(\alpha_t\) given by (\ref{eqn:alpha_t}). 

    \begin{proposition}\label{prop:supp_tsybakov}
        Suppose \(1 \leq s < 4\sqrt{p}\) and \(Y \sim N(\mu, \sigma^2I_p)\). For \(\varepsilon > 0\), define the parameter space
        \begin{equation}\label{space:supp_space}
            \mathscr{M}(p, s, \varepsilon) := \left\{ \mu \in \R^p : ||\mu_S||^2 \geq \varepsilon^2 \text{ for some } S \subset [p] \text{ with } |S|\leq s \right\}.    
        \end{equation}
        If \(\eta \in (0, 1)\), then there exists a constant \(C_\eta > 0\) depending only on \(\eta\) such that for all \(C > C_\eta\)
        \begin{equation*}
            P_{0, \sigma^2}\left\{ \sum_{i=1}^{p} \left(\frac{Y_i^2}{\sigma^2} - \alpha_{t^*}\right)\mathbbm{1}_{\{|Y_i/\sigma| \geq t^*\}} > r^* \right\} + \sup_{\mu \in \mathscr{M}\left(p, s, C\sigma \sqrt{s\log\left(1 + \frac{p}{s^2}\right)}\right)} P_{\mu, \sigma^2}\left\{ \sum_{i=1}^{p} \left(\frac{Y_i^2}{\sigma^2} - \alpha_{t^*}\right)\mathbbm{1}_{\{|Y_i/\sigma| \geq t^*\}} \leq r^* \right\} \leq \eta
        \end{equation*}
        where \(t^* = \sqrt{2 \log\left(1 + \frac{p}{s^2}\right)}\) and \(r^* = \frac{C^2}{8} s \log\left(1 + \frac{p}{s^2}\right)\). In fact, we may take 
        \begin{equation}\label{eqn:general_tsybakov_Ceta}
            C_\eta = \sqrt{\frac{9 \cdot 8 \cdot 2}{\log\left(\frac{17}{16}\right)}} \vee \sqrt{\log\left(\frac{2}{\eta}\right) \cdot \frac{9 \cdot 8 \cdot 2}{\log\left(\frac{17}{16}\right)}} \vee \sqrt{256} \vee \left(\frac{1024 C_1}{\eta/2}\right)^{1/4} \vee \sqrt{\frac{128C_1}{\eta/2}}
        \end{equation}
        where \(C_1\) is the constant from Lemma \ref{lemma:tsybakov_variance}.
    \end{proposition}
    \begin{proof}
        Fix \(\eta \in (0, 1)\). Let \(C > C_\eta\). For ease of notation, let \(\psi^2 = \sigma^2 s\log\left(1 + \frac{p}{s^2}\right)\). We first bound the type I error. First, consider that \(1 \leq s < 4\sqrt{p}\) implies
        \begin{align*}
            C^2\frac{\log\left(\frac{17}{16}\right)}{2} (1 + s) \leq C^2 s \log\left(\frac{17}{16}\right) \leq C^2 s \log\left(1 + \frac{p}{s^2}\right).
        \end{align*}
        Since \(C^2\frac{\log\left(\frac{17}{16}\right)}{2 \cdot 9 \cdot 8} > C_\eta^2\frac{\log\left(\frac{17}{16}\right)}{2 \cdot 9 \cdot 8} \geq 1\), it follows that \(s\sqrt{C^2\frac{\log\left(\frac{17}{16}\right)}{2 \cdot 9 \cdot 8}} + C^2\frac{\log\left(\frac{17}{16}\right)}{2 \cdot 9 \cdot 8} \leq \frac{C^2}{9 \cdot 8} s \log\left(1 + \frac{p}{s^2}\right)\). Therefore, 
        \begin{equation*}
            9\sqrt{pe^{-(t^*)^2/2} C^2 \frac{\log\left(\frac{17}{16}\right)}{2 \cdot 9 \cdot 8}} + C^2 \frac{\log\left(\frac{17}{16}\right)}{2 \cdot 8} \leq r^*.
        \end{equation*}
        Thus the type I error is bounded as 
        \begin{align*}
            &P_{0, \sigma^2}\left\{ \sum_{i=1}^{p} \left(\frac{Y_i^2}{\sigma^2} - \alpha_{t^*}\right)\mathbbm{1}_{\{|Y_i/\sigma| \geq t^*\}} > r^* \right\} \\
            &\leq P_{0, \sigma^2}\left\{ \sum_{i=1}^{p} \left(\frac{Y_i^2}{\sigma^2} - \alpha_{t^*}\right)\mathbbm{1}_{\{|Y_i/\sigma| \geq t^*\}} > 9\sqrt{pe^{-(t^*)^2/2} C^2 \frac{\log\left(\frac{17}{16}\right)}{2 \cdot 9 \cdot 8}} + C^2 \frac{\log\left(\frac{17}{16}\right)}{2 \cdot 8} \right\} \\
            &\leq \exp\left(-C^2\frac{\log\left(\frac{17}{16}\right)}{2 \cdot 9 \cdot 8} \right) \\
            &\leq \exp\left(-C_\eta^2\frac{\log\left(\frac{17}{16}\right)}{2 \cdot 9 \cdot 8} \right) \\
            &\leq \frac{\eta}{2}
        \end{align*}
        where we have used Lemma \ref{lemma:tsybakov_typeI_error} and \(C_\eta^2 \geq \log\left(\frac{2}{\eta}\right) \cdot \frac{2 \cdot 8 \cdot 9}{\log\left(\frac{17}{16}\right)}\).
        
        We now shift our attention to bounding the type II error. For \(\mu \in \mathscr{M}(p, s, C\psi)\) let \(S(\mu) \subset [p]\) denote any subset such that \(|S(\mu)| \leq s\) and \(||\mu_{S(\mu)}||^2 \geq C^2 \psi^2\). Note such a subset must exist by definition of \(\mathscr{M}(p, s, C\psi)\). Let \(Z_1,...,Z_p \overset{iid}{\sim}N(0, 1)\) be independent of \(Y\). Define the random vector \(Y' \in \R^p\) with 
        \begin{equation*}
            Y'_i = 
            \begin{cases}
                \frac{Y_i}{\sigma} &\text{if } i \in S(\mu), \\
                Z_i &\text{if } i \in S(\mu)^c.
            \end{cases}
        \end{equation*}
        It is clear that if \(g \sim N(m, 1)\) and \(t > 0\), then \((g^2 - \alpha_t)\mathbbm{1}_{\{|g| \geq t\}}\) is stochastically increasing in \(|m|\). Consequently, \(\left( \frac{Y_i^2}{\sigma^2} - \alpha_{t^*}\right)\mathbbm{1}_{\{|Y_i/\sigma| \geq t^*\}}\) is stochastically larger than \((Z_i^2 - \alpha_{t^*})\mathbbm{1}_{\{|Z_i| \geq t^*\}}\) for all \(1 \leq i \leq p\). Since the collection \(\{Y_i\}_{i=1}^{p}\) are mutually independent, it follows that \(\sum_{i=1}^{p} \left(\frac{Y_i^2}{\sigma^2} - \alpha_{t^*}\right)\mathbbm{1}_{\{|Y_i / \sigma| \geq t^*\}}\) is stochastically larger than 
        \begin{equation*}
            \sum_{i \in S(\mu)} \left( \frac{Y_i^2}{\sigma^2} - \alpha_{t^*} \right)\mathbbm{1}_{\{|Y_i/\sigma| \geq t^*\}} + \sum_{i \in S(\mu)^c} (Z_i^2 - \alpha_{t^*})\mathbbm{1}_{\{|Z_i| \geq t^*\}}.
        \end{equation*}
        The above display is exactly \(\sum_{i=1}^{p} ( Y_i'^2 - \alpha_{t^*})\mathbbm{1}_{\{|Y_i'| \geq t^*\}}\). By the stochastic ordering, we have 
        \begin{equation}
            \sup_{\mu \in \mathscr{M}(p, s, C\psi)} P_{\mu, \sigma^2}\left\{ \sum_{i=1}^{p} \left(\frac{Y_i^2}{\sigma^2} - \alpha_{t^*} \right)\mathbbm{1}_{\{|Y_i/\sigma| \geq t^*\}} \leq r^* \right\}
            \leq \sup_{\mu \in \mathscr{M}(p, s, C\psi)} P_{\mu, \sigma^2}\left\{ \sum_{i=1}^{p} \left(Y_i'^2 - \alpha_{t^*} \right)\mathbbm{1}_{\{|Y_i'| \geq t^*\}} \leq r^* \right\}. \label{eqn:supp_tsybakov_bound}
        \end{equation}
        Consider that \(Y' \sim N(\mu', I_p)\) with 
        \begin{equation*}
            \mu'_i =
            \begin{cases}
                \frac{\mu_i}{\sigma} &\text{if } i \in S(\mu), \\
                0 &\text{if } i \in S(\mu)^c.
            \end{cases}
        \end{equation*}
        Consequently, \(||\mu'||_0 = |S(\mu)| \leq s\) and \(||\mu'||^2 = \frac{||\mu_{S(\mu)}||^2}{\sigma^2} \geq \frac{C^2\psi^2}{\sigma^2} \geq C^2 s \log\left(1 + \frac{p}{s^2}\right)\). Therefore
        \begin{align*}
            &\sup_{\mu \in \mathscr{M}(p, s, C\psi)} P_{\mu, \sigma^2}\left\{ \sum_{i=1}^{p} \left(Y_i'^2 - \alpha_{t^*} \right)\mathbbm{1}_{\{|Y_i'| \geq t^*\}} \leq r^*\right\} \\
            &\leq \sup_{\substack{||\mu'||_0 \leq s, \\ ||\mu'||^2 \geq C^2 s \log\left(1 + \frac{p}{s^2}\right)}} P_{Y' \sim N(\mu', I_p)}\left\{  \sum_{i=1}^{p} \left(Y_i'^2 - \alpha_{t^*} \right)\mathbbm{1}_{\{|Y_i'| \geq t^*\}}  \leq r^* \right\}.
        \end{align*}
        We first calculate bounds on the variance and expectation of the statistic on the right hand side before working to obtain the desired bound. For ease, we suppress the subscript \(Y' \sim N(\mu', I_p)\) on the associated probability and expectation operators. For \(Y' \sim N(\mu', I_p)\) with \(||\mu'||_0 \leq s\) and \(||\mu'|| \geq C^2 s \log\left(1 + \frac{p}{s^2}\right)\), we have by Lemma \ref{lemma:tsybakov_expectation}
        \begin{align*}
            &E\left(\sum_{i=1}^{p} \left(Y_i'^2 - \alpha_{t^*}\right)\mathbbm{1}_{\{|Y_i'| \geq t^*\}} \right) \\
            &= \sum_{i \in \supp(\mu')} E\left(\left(Y_i'^2 - \alpha_{t^*}\right)\mathbbm{1}_{\{|Y_i'| \geq t^*\}}\right) + \sum_{i \in \supp(\mu')^c} E\left(\left(Y_i'^2 - \alpha_{t^*}\right)\mathbbm{1}_{\{|Y_i'| \geq t^*\}}\right) \\
            &= \sum_{i \in \supp(\mu')} E\left(\left(Y_i'^2 - \alpha_{t^*}\right)\mathbbm{1}_{\{|Y_i'| \geq t^*\}}\right) \\
            &\geq \sum_{i \in \supp(\mu'): |\mu_i'| \geq 8t^*} E\left(\left(Y_i'^2 - \alpha_{t^*}\right)\mathbbm{1}_{\{|Y_i'| \geq t^*\}}\right) \\
            &\geq \sum_{i \in \supp(\mu'): |\mu_i'| \geq 8t^*} \frac{\mu_i'^2}{2} \\
            &= \frac{||\mu'||^2}{2} - \sum_{i \in \supp(\mu) : |\mu_i'| < 8t^*} \frac{\mu_i'^2}{2} \\
            &\geq \frac{||\mu'||^2}{2} - 32s(t^*)^2 \\
            &= \frac{||\mu'||^2}{2} - 64 s \log\left(1 + \frac{p}{s^2}\right) \\
            &= \frac{||\mu'||^2}{2} \left(1 - \frac{128}{C^2}\right).
        \end{align*}
        Now using that \(C^2 > C_\eta^2 \geq 256\) which implies \(\frac{1}{2}\left(1 - \frac{128}{C^2}\right) \geq \frac{1}{4}\), we have 
        \begin{equation}\label{eqn:expectation_lowerbound}
            E\left(\sum_{i=1}^{p} \left(Y_i'^2 - \alpha_{t^*}\right)\mathbbm{1}_{\{|Y_i'| \geq t^*\}} \right) \geq \frac{||\mu'||^2}{4}.
        \end{equation} 

        Turning our attention to the variance, consider that by independence and Lemma \ref{lemma:tsybakov_variance}
        \begin{align*}
            &\Var\left(\sum_{i=1}^{p} \left(Y_i'- \alpha_{t^*}\right)\mathbbm{1}_{\{|Y_i'| \geq t^*\}} \right) \\
            &= \sum_{i=1}^{p} \Var\left( \left(Y_i'^2 - \alpha_{t^*}\right)\mathbbm{1}_{\{|Y_i'| \geq t^*\}} \right) \\
            &= \sum_{i \in \supp(\mu')^c} \Var\left( \left(Y_i'^2 - \alpha_{t^*}\right)\mathbbm{1}_{\{|Y_i'| \geq t^*\}} \right) + \sum_{i \in \supp(\mu')} \Var\left( \left(Y_i'^2 - \alpha_{t^*}\right)\mathbbm{1}_{\{|Y_i'| \geq t^*\}} \right) \\
            &\leq C_1(p-||\mu'||_0)(t^*)^3e^{-(t^*)^2/2} + \sum_{i \in \supp(\mu')} \Var\left( \left(Y_i'^2 - \alpha_{t^*}\right)\mathbbm{1}_{\{|Y_i'| \geq t^*\}} \right) \\
            &\leq C_1p(t^*)^4e^{-(t^*)^2/2} + \sum_{i \in \supp(\mu') : |\mu_i'| \geq 2t^*} C_1 \mu_i'^2 + \sum_{i \in \supp(\mu') : |\mu_i'| < 2t^*} C_1(t^*)^4 \\
            &\leq 4C_1p \log^2\left(1 + \frac{p}{s^2}\right) \left(1 + \frac{p}{s^2}\right)^{-1} + C_1||\mu'||^2 + 4C_1 s \log^2\left(1 + \frac{p}{s^2}\right) \\
            &\leq 4C_1 s^2 \log^2\left(1 + \frac{p}{s^2}\right) + C_1||\mu'||^2 + 4C_1 s \log^2\left(1 + \frac{p}{s^2}\right) \\
            &\leq 8C_1s^2 \log^2\left(1 + \frac{p}{s^2}\right) + C_1||\mu'||^2
        \end{align*}
        where \(C_1\) is the constant from Lemma \ref{lemma:tsybakov_variance}. We can now proceed to prove the desired result in the statement of the lemma. Since \(r^* = \frac{C^2}{8}s \log\left(1 + \frac{p}{s^2}\right) \leq \frac{||\mu'||^2}{8}\), an application of Chebyshev's inequality with (\ref{eqn:expectation_lowerbound}) and the above variance upper bound yields
        \begin{align*}
            &\sup_{\substack{||\mu'||_0 \leq s, \\ ||\mu'||^2 \geq C^2 s\log\left(1 + \frac{p}{s^2}\right)}} P\left\{ \sum_{i=1}^{p} \left(Y_i'^2 - \alpha_{t^*}\right)\mathbbm{1}_{\{|Y_i'| \geq t^*\}} \leq r^* \right\} \\
            &\leq \sup_{\substack{||\mu'||_0 \leq s, \\ ||\mu'||^2 \geq C^2 s\log\left(1 + \frac{p}{s^2}\right)}} \frac{\Var\left(\sum_{i=1}^{p} \left(Y_i'^2 - \alpha_{t^*}\right)\mathbbm{1}_{\{|Y_i'| \geq t^*\}} \right) }{\left(E\left(\sum_{i=1}^{p} \left(Y_i'^2 - \alpha_{t^*}\right)\mathbbm{1}_{\{|Y_i'| \geq t^*\}}  \right)  - r^*\right)^2} \\
            &\leq \sup_{\substack{||\mu'||_0 \leq s, \\ ||\mu'||^2 \geq C^2 s\log\left(1 + \frac{p}{s^2}\right)}} \frac{8C_1s^2\log^2\left(1 + \frac{p}{s^2}\right) + C_1||\mu'||^2}{\frac{||\mu'||^4}{64}} \\
            &\leq \frac{512C_1 s^2\log^2\left(1 + \frac{p}{s^2}\right)}{C^4s^2\log^2\left(1 + \frac{p}{s^2}\right)} + \sup_{\substack{||\mu'||_0 \leq s, \\ ||\mu'||^2 \geq C^2 s\log\left(1 + \frac{p}{s^2}\right)}} \frac{64C_1}{||\mu'||^2} \\
            &\leq \frac{512C_1}{C^4} + \frac{64C_1}{C^2 s\log\left(1 + \frac{p}{s^2}\right)} \\
            &\leq \frac{512C_1}{C_\eta^4} + \frac{64C_1}{C_\eta^2} \\
            &\leq \frac{\eta}{2}.
        \end{align*}
        Plugging this bound into (\ref{eqn:supp_tsybakov_bound}) shows that the type II error is bounded by \(\frac{\eta}{2}\). Thus the sum of type I and type II errors is bounded by \(\eta\). Since \(C > C_\eta\) was arbitrary and \(\eta \in (0, 1)\) was arbitrary, the proof is complete.
    \end{proof}

    \subsection{Proofs of Proposition \ref{prop:perfect_correlation}, Lemma \ref{lemma:problem_decomposition}, and Lemma \ref{lemma:problemI_equiv}}
    We begin by proving Proposition \ref{prop:perfect_correlation}. 
    \begin{proof}[Proof of Proposition \ref{prop:perfect_correlation}]
        The argument proceeds by separately considering the two cases \(s < p\) and \(s = p\). 

        \textbf{Case 1:} Suppose \(1 \leq s < p\). It trivially holds \(\varepsilon^*(p, s, 1)^2 \geq 0\) and so the lower bound is proved. Define the test \(\varphi^* = \mathbf{1}_{\{X-\bar{X}\mathbf{1}_p \neq 0\}}\). Observe that under the data-generating process \(P_{\theta, 1}\), we have \(X - \bar{X}\mathbf{1}_p = \theta - \bar{\theta}\mathbf{1}_p\) almost surely. Since \(s < p\) and \(||\theta||_0 \leq s\) implies that \(\theta \not \in \spn\{\mathbf{1}_p\}\setminus \{0\}\), it immediately follows that if \(||\theta||_0 \leq s\), then \(\theta \neq 0\) if and only if \(\theta - \bar{\theta}\mathbf{1}_p \neq 0\). We now explicitly bound the type I and type II errors. Examining the type I error first, consider \(P_{0, 1}\{\varphi^* = 1\} = P_{0, 1}\{X-\bar{X}\mathbf{1}_p \neq 0\} = 0\) since \(X-\bar{X}\mathbf{1}_p = 0\) almost surely under \(P_{0, 1}\). Examining the type II error, consider that for any \(\varepsilon > 0\), we have
        \begin{align*}
            \sup_{\theta \in \Theta(p, s, \varepsilon)} P_{\theta, 1} \left\{ \varphi^* = 0 \right\} = \sup_{\theta \in \Theta(p, s, \varepsilon)} P_{\theta, 1} \left\{ X-\bar{X}\mathbf{1}_p = 0\right\} = 0
        \end{align*}
        since \(X - \bar{X}\mathbf{1}_p = \theta-\bar{\theta}\mathbf{1}_p\) almost surely under \(P_{\theta, 1}\) and \(\theta - \bar{\theta}\mathbf{1}_p \neq 0\) for all \(\theta \in \Theta(p, s, \varepsilon)\). Therefore, the sum of the type I and II errors of \(\varphi^*\) is equal to \(0\). Hence, we have \(\mathcal{R}(\varepsilon) = 0\). Since \(\varepsilon > 0\) was arbitrary, we have proved the upper bound. \newline

        \textbf{Case 2:} Suppose \(s = p\). We first prove the upper bound. Let \(\eta \in (0, 1)\) and set \(C_{\eta}\) to be any value satisfying \(\frac{16}{C_{\eta}^4} + \frac{16}{C_{\eta}^2} \leq \eta\). Note that such a \(C_{\eta}\) clearly exists by taking \(C_\eta\) sufficiently large depending only on \(\eta\). Let \(C > C_\eta\). Define the test \(\varphi^* = \mathbbm{1}_{\left\{||X||^2 > p + \frac{C^2}{2} p \right\}}\). Under the data generating process \(P_{\theta, 1}\), we have \(||X||^2 = ||X - \bar{X}\mathbf{1}_p||^2 + ||\bar{X}\mathbf{1}_p||^2 \sim ||\theta - \bar{\theta}\mathbf{1}_p||^2 + p\chi^2_1(\bar{\theta}^2)\). Consequently, \(E_{\theta,1}(||X||^2) = ||\theta||^2 + p\) and \(\Var_{\theta,1}(||X||^2) = 2p^2 + 4p||\bar{\theta}\mathbf{1}_p||^2\). Examining the type I error, consider 
        \begin{align*}
            P_{0,1}\left\{\varphi^* = 1 \right\} = P_{0, 1}\left\{||X||^2 > p + \frac{C^2}{2} p  \right\} \leq \frac{\Var_{0, 1}(||X||^2)}{\frac{C^4}{4} p^2} = \frac{8p^2}{C^4p^2} = \frac{8}{C^4} \leq \frac{8}{C_{\eta}^4}.
        \end{align*}
        Now examining the type II error, consider by Chebyshev's inequality
        \begin{align*}
            \sup_{\theta \in \Theta(p, p, C\sqrt{p})} P_{\theta, 1}\left\{ \varphi^* = 0 \right\} &= \sup_{\theta \in \Theta(p, p, C\sqrt{p})} P_{\theta, 1}\left\{ ||X||^2 \leq p + \frac{C^2}{2} p \right\} \\
            &\leq \sup_{\theta \in \Theta(p, p, C\sqrt{p})} P_{\theta, 1}\left\{||\theta||^2 - \frac{C^2}{2} p \leq p + ||\theta||^2 - ||X||^2 \right\} \\
            &\leq \sup_{\theta \in \Theta(p, p, C\sqrt{p})} \frac{\Var_{\theta,1}(||X||^2)}{(||\theta||^2 - \frac{C^2}{2} p)^2} \\
            &= \sup_{\theta \in \Theta(p, p, C\sqrt{p})} \frac{2p^2 + 4p||\bar{\theta}\mathbf{1}_p||^2}{(||\theta||^2 - \frac{C^2}{2} p)^2} \\
            &\leq \frac{2p^2}{\frac{C^4}{4}p^2} + \sup_{\theta \in \Theta(p, p, C\sqrt{p})} \frac{4p||\bar{\theta}\mathbf{1}_p||^2}{\frac{1}{4} ||\theta||^4} \\
            &\leq \frac{8}{C^4} + \sup_{\theta \in \Theta(p, p, C\sqrt{p})} \frac{16p}{||\theta||^2} \\
            &\leq \frac{8}{C_{\eta}^4} + \frac{16}{C_{\eta}^2}. 
        \end{align*}
        Therefore, it follows that \(P_{0,1}\left\{\varphi^* = 1 \right\} + \sup_{\theta \in \Theta(p, p, C\sqrt{p})} P_{\theta, 1}\left\{ \varphi^* = 0 \right\} \leq \eta\). Since \(C > C_\eta\) was arbitrary and \(\eta \in (0, 1)\) was arbitrary, it follows that \(\varepsilon^*(p, p, 1)^2 \lesssim p\).

        We now prove the lower bound. Let \(\eta \in (0, 1)\) and set \(c_{\eta} := \sqrt{\log\left(1 + 4\eta^2\right)}\). Let \(0 < c < c_\eta\). To prove the lower bound, we make use of Lemma \ref{lemma:general_lower_bound}. Let \(\pi\) be the prior on \(\Theta(p, p, c\sqrt{p})\) which is a point mass at \(c\mathbf{1}_p\). Note that \(|| c\mathbf{1}_p|| = c \sqrt{p}\) and so \(\pi\) is indeed supported on \(\Theta(p, p, c\sqrt{p})\). A direct calculation shows 
        \begin{align*}
            \chi^2(P_{c\mathbf{1}_p, 1} || P_{0, 1}) = \chi^2(N(cp, p^2) || N(0, p^2)) = \exp\left(c^2p^2 \cdot \frac{1}{p^2}\right) - 1 = e^{c^2} - 1 \leq e^{c_{\eta}^2} - 1.
        \end{align*}
        Therefore, \(1 - \frac{1}{2}\sqrt{\chi^2(P_{c\mathbf{1}_p, 1} || P_{0, 1})} \geq 1 - \frac{1}{2}\sqrt{e^{c_\eta^2} - 1} = 1-\eta\). Lemma \ref{lemma:general_lower_bound} thus implies \(\mathcal{R}(c\sqrt{p}) \geq 1-\eta\). Since \(0 < c < c_\eta\) was arbitrary and \(\eta \in (0, 1)\) was arbitrary, it follows that \(\varepsilon^*(p, p, 1)^2 \gtrsim p\).
    \end{proof}

    We now prove Lemmas \ref{lemma:problem_decomposition} and \ref{lemma:problemI_equiv} which validate the strategy of separately considering Problem I and Problem II.
    \begin{proof}[Proof of Lemma \ref{lemma:problem_decomposition}]
        For any \(\varepsilon, \varepsilon_1,\varepsilon_2 > 0\) such that \(\varepsilon^2 = \varepsilon_1^2 + \varepsilon_2^2\), we have the inclusion \(\Theta(p, s, \varepsilon) \subset \Theta_{\mathcal{I}}(p, s, \varepsilon_1) \cup \Theta_{\mathcal{II}}(p, s, \varepsilon_2)\). Indeed, if \(||\theta||^2 \geq \varepsilon^2\), then either \(||\theta - \bar{\theta}\mathbf{1}_p||^2 \geq \varepsilon_1^2\) or \(||\bar{\theta}\mathbf{1}_p||^2 \geq \varepsilon_2^2\) since \(||\theta||^2 = ||\theta-\bar{\theta}\mathbf{1}_p||^2 + ||\bar{\theta}\mathbf{1}_p||^2\). Since \(\varepsilon_1^*\) and \(\varepsilon_2^*\) are the minimax separation rates for Problem I and Problem II respectively consider the following. There must exist tests \(\varphi_1^*\) and \(\varphi_2^*\) such that for every \(\eta \in (0, 1)\) there exists a constants \(C_{\eta, 1}, C_{\eta, 2} > 0\) such that for all \(C_1 > C_{\eta, 1}\) and \(C_2 > C_{\eta, 2}\)
        \begin{align*}
            &P_{0, \gamma}\{\varphi_1 = 1\} + \sup_{\theta \in \Theta_{\mathcal{I}}(p, s, C_1\varepsilon_1^*)}P_{\theta, \gamma}\{\varphi_1 = 0\} \leq \frac{\eta}{2}, \\
            &P_{0, \gamma}\{\varphi_2 = 1\} + \sup_{\theta \in \Theta_{\mathcal{II}}(p, s, C_2\varepsilon_2^*)}P_{\theta, \gamma}\{\varphi_2 = 0\} \leq \frac{\eta}{2}. 
        \end{align*}
        Let \(\varphi^* := \varphi_1^* \vee \varphi_2^*\) and for \(\eta \in (0, 1)\) let \(C_\eta := C_{\eta, 1} \vee C_{\eta, 2}\). Then for all \(C > C_\eta\) we have 
        \begin{align*}
            &P_{0, \gamma}\{\varphi^* = 1\} + \sup_{\theta \in \Theta\left(p, s, C\sqrt{(\varepsilon_1^*)^2 + (\varepsilon_2^*)^2}\right)}P_{\theta, \gamma}\{\varphi^* = 0\} \\
            &\leq P_{0, \gamma}\{\varphi^*_1 = 1\} + P_{0, \gamma}\{\varphi_2^* = 1\} + \sup_{\theta \in \Theta_{\mathcal{I}}(p, s, C\varepsilon_1^*)\cup \Theta_{\mathcal{II}}(p, s, C\varepsilon_2^*)}P_{\theta, \gamma}\{\varphi^* = 0\} \\
            &\leq \left\{P_{0, \gamma}\{\varphi^*_1 = 1\} + \sup_{\theta \in \Theta_{\mathcal{I}}(p, s, C\varepsilon_1^*)}P_{\theta, \gamma}\{\varphi^*_1 = 0\}\right\} + \left\{P_{0, \gamma}\{\varphi^*_2 = 1\} + \sup_{\theta \in \Theta_{\mathcal{II}}(p, s, C\varepsilon_2^*)}P_{\theta, \gamma}\{\varphi^*_2 = 0\}\right\} \\
            &\leq \eta.
        \end{align*}
        Since \(C > C_\eta\) was arbitrary and \(\eta \in (0, 1)\) was arbitrary, we have shown \(\varepsilon^* \lesssim \varepsilon_1^* + \varepsilon_2^*\). It remains to show \(\varepsilon^* \gtrsim \varepsilon_1^* + \varepsilon_2^*\). To show this, consider that we have the inclusions \(\Theta_{\mathcal{I}}(p, s, \varepsilon) \subset \Theta(p, s, \varepsilon)\) and \(\Theta_{\mathcal{II}}(p, s, \varepsilon) \subset \Theta(p, s, \varepsilon)\). Consequently, we have \(\mathcal{R}_{\mathcal{I}}(\varepsilon) \leq \mathcal{R}(\varepsilon)\) and \(\mathcal{R}_{\mathcal{II}}(\varepsilon) \leq \mathcal{R}(\varepsilon)\). Therefore, \(\varepsilon_1^* \lesssim \varepsilon^*\) and \(\varepsilon_2^* \lesssim \varepsilon^*\), which immediately yields \(\varepsilon_1^* + \varepsilon_2^* \lesssim \varepsilon^*\). Thus, the desired result is proved.
    \end{proof}

    \begin{proof}[Proof of Lemma \ref{lemma:problemI_equiv}]
        By hypothesis we have \(1 \leq s \leq \frac{p}{2}\). Note that \(s \leq \frac{p}{2}\) implies \(\frac{p-s}{p} \geq \frac{1}{2}\). Hence, Corollary \ref{corollary:orthog_approximation} implies that if \(||\theta||_0 \leq s\), then \(||\theta - \bar{\theta}\mathbf{1}_p||^2 \geq \frac{1}{2}||\theta||^2\). Therefore, for \(\varepsilon > 0\) the chain of inclusions \(\Theta_{\mathcal{I}}(p, s, \varepsilon) \subset \Theta(p, s, \varepsilon) \subset \Theta_{\mathcal{I}}(p, s, \varepsilon/\sqrt{2})\) holds. The first inclusion follows from the fact that \(||\theta-\bar{\theta}\mathbf{1}_p||^2 \geq \varepsilon^2\) implies \(||\theta||^2 \geq \varepsilon^2\). The inclusions immediately gives us the chain of inequalities \(\mathcal{R}_{\mathcal{I}}(\varepsilon) \leq \mathcal{R}(\varepsilon) \leq \mathcal{R}_{\mathcal{I}}(\varepsilon/\sqrt{2})\). It immediately follows that \(\varepsilon^*(p, s, \gamma) \asymp \varepsilon_1^*(p, s, \gamma)\) as desired.
    \end{proof}

    \subsection{Proofs of results in Section \ref{section:ProblemI}}
    \subsubsection{Upper bound}
        In this section, we prove Proposition \ref{prop:problemI_upperbound}. For clarity and ease of reading, we separately evaluate the performance of the constituent tests and combine the results later to prove Proposition \ref{prop:problemI_upperbound}.

        \begin{lemma}\label{lemma:problem1_chisquare}
            Suppose \(1 \leq s \leq p\) and \(\gamma \in [0, 1)\). If \(\eta \in (0, 1)\), then there exists a constant \(C_\eta > 0\) depending only on \(\eta\) such that for all \(C > C_\eta\), the test \(\varphi_{C^2/2}^{\chi^2}\) given by (\ref{test:problemI_chisquare}) satisfies 
            \begin{equation*}
                P_{0, \gamma}\left\{\varphi_{C^2/2}^{\chi^2} = 1\right\} + \sup_{\theta \in \Theta_{\mathcal{I}}\left(p, s, C\sqrt{(1-\gamma)\sqrt{p}}\right)} P_{\theta, \gamma}\left\{ \varphi_{C^2/2}^{\chi^2} = 0 \right\} \leq \eta.
            \end{equation*} 
        \end{lemma}
        \begin{proof}
            Fix \(\eta \in (0, 1)\) and let \(C_\eta > 0\) be any value satisfying the inequality \(\frac{16}{C_\eta^4} + \frac{16}{C_\eta^2} \leq \eta\). Note that such a \(C_{\eta}\) clearly exists by taking \(C_\eta\) sufficiently large. Let \(C > C_\eta\). Under the data-generating process \(P_{\theta, \gamma}\), recall that \(\widetilde{X} \sim N\left(\frac{\theta - \bar{\theta}\mathbf{1}_p}{\sqrt{1-\gamma}}, I_p\right)\) where \(\widetilde{X}\) is the transformed data used in the definition of the test \(\varphi_{C^2/2}^{\chi^2}\). Therefore, \(||\widetilde{X}||^2 \sim \chi^2_p\left(\frac{||\theta-\bar{\theta}\mathbf{1}_p||^2}{1-\gamma}\right)\) and so \(E_{\theta, \gamma}(||\widetilde{X}||^2) = p + \frac{||\theta - \bar{\theta}\mathbf{1}_p||^2}{1-\gamma}\) and \(\Var_{\theta, \gamma}(||\widetilde{X}||^2) = 2p + 4\frac{||\theta - \bar{\theta}\mathbf{1}_p||^2}{1-\gamma}\). We now explicitly bound the type I and type II errors of the test. 
            
            Examining first the type I error, consider by Chebyshev's inequality 
            \begin{align*}
                P_{0, \gamma}\left\{\varphi_{C^2/2}^{\chi^2} = 1\right\} = P_{0, \gamma}\left\{ ||\widetilde{X}||^2 > p + \frac{C^2}{2} \sqrt{p} \right\} \leq \frac{\Var_{0, \gamma}(||\widetilde{X}||^2)}{\frac{C^4}{4} p} = \frac{8}{C^4} \leq \frac{8}{C_\eta^4}.
            \end{align*}

            Turning our attention to the type II error, consider again by Chebyshev's inequality
            \begin{align*}
                \sup_{\theta \in \Theta_{\mathcal{I}}\left(p, s, C\sqrt{(1-\gamma)\sqrt{p}}\right)} P_{\theta, \gamma}\left\{ \varphi_{C^2/2}^{\chi^2} = 0 \right\} &= \sup_{\theta \in \Theta_{\mathcal{I}}\left(p, s, C\sqrt{(1-\gamma)\sqrt{p}}\right)} P_{\theta, \gamma}\left\{ ||\widetilde{X}||^2 \leq p + \frac{C^2}{2}\sqrt{p}\right\} \\
                &= \sup_{\theta \in \Theta_{\mathcal{I}}\left(p, s, C\sqrt{(1-\gamma)\sqrt{p}}\right)} P_{\theta, \gamma}\left\{ \frac{||\theta - \bar{\theta}\mathbf{1}_p||^2}{1-\gamma} - \frac{C^2}{2} \sqrt{p} \leq p + \frac{||\theta - \bar{\theta}\mathbf{1}_p||^2}{1-\gamma} - ||\widetilde{X}||^2\right\} \\
                &\leq \sup_{\theta \in \Theta_{\mathcal{I}}\left(p, s, C\sqrt{(1-\gamma)\sqrt{p}}\right)} \frac{\Var_{\theta, \gamma}(||\widetilde{X}||^2)}{\left(\frac{||\theta - \bar{\theta}\mathbf{1}_p||^2}{1-\gamma} - \frac{C^2}{2} \sqrt{p}\right)^2} \\
                &= \sup_{\theta \in \Theta_{\mathcal{I}}\left(p, s, C\sqrt{(1-\gamma)\sqrt{p}}\right)} \frac{2p + 4\frac{||\theta-\bar{\theta}\mathbf{1}_p||^2}{1-\gamma}}{\left(\frac{||\theta - \bar{\theta}\mathbf{1}_p||^2}{1-\gamma} - \frac{C^2}{2} \sqrt{p}\right)^2} \\
                &\leq \frac{2p}{\frac{C^4}{4}p} + \sup_{\theta \in \Theta_{\mathcal{I}}\left(p, s, C\sqrt{(1-\gamma)\sqrt{p}}\right)}  \frac{4\frac{||\theta-\bar{\theta}\mathbf{1}_p||^2}{1-\gamma}}{\frac{1}{4}\frac{||\theta -\bar{\theta}\mathbf{1}_p||^4}{(1-\gamma)^2}} \\
                &= \frac{8}{C^4} + \sup_{\theta \in \Theta_{\mathcal{I}}\left(p, s, C\sqrt{(1-\gamma)\sqrt{p}}\right)} \frac{16(1-\gamma)}{||\theta - \bar{\theta}\mathbf{1}_p||^2} \\
                &\leq \frac{8}{C^4} + \frac{16}{C^2 \sqrt{p}} \\
                &\leq \frac{8}{C_\eta^4} + \frac{16}{C_\eta^2}. 
            \end{align*}
            Therefore, the sum of type I and type II errors is bounded by 
            \begin{align*}
                P_{0, \gamma}\left\{\varphi_{C^2/2}^{\chi^2} = 1\right\} + \sup_{\theta \in \Theta_{\mathcal{I}}\left(p, s, C\sqrt{(1-\gamma)\sqrt{p}}\right)} P_{\theta, \gamma}\left\{ \varphi_{C^2/2}^{\chi^2} = 0 \right\} \leq \frac{16}{C_\eta^4} + \frac{16}{C_\eta^2} \leq \eta.
            \end{align*}
            Since \(C > C_\eta\) was arbitrary and \(\eta \in (0, 1)\) was arbitrary, we have proved the desired result.
        \end{proof}

        \begin{lemma}\label{lemma:problem1_tsybakov}
            Assume \(p \geq 4\). Suppose \(1 \leq s < \sqrt{p}\) and \(\gamma \in [0, 1)\). If \(\eta \in (0, 1)\), then there exists a constant \(C_\eta > 0\) depending only on \(\eta\) such that for all \(C > C_\eta\), the test \(\varphi_{t^*, r^*}\) given by (\ref{test:problemI_tsybakov}) with \(t^* = \sqrt{2\log\left(1 + \frac{p}{s^2}\right)}\) and \(r^* = \frac{C^2}{32} s \log\left(1 + \frac{p}{s^2}\right)\) satisfies 
            \begin{equation*}
                P_{0, \gamma}\left\{\varphi_{t^*, r^*} = 1 \right\} + \sup_{\theta \in \Theta_{\mathcal{I}}\left(p, s, C\sqrt{(1-\gamma) s \log\left(1 + \frac{p}{s^2}\right)}\right)} P_{\theta, \gamma}\left\{ \varphi_{t^*, r^*} = 0 \right\} \leq \eta.
            \end{equation*}
        \end{lemma}
        \begin{proof}
            Fix \(\eta \in (0, 1)\) and set \(C_\eta := 2C_\eta^*\) where \(C_\eta^*\) is given by Proposition \ref{prop:supp_tsybakov}. Let \(C > C_\eta\). For \(\theta \in \Theta_{\mathcal{I}}\left(p, s, C\sqrt{(1-\gamma)s\log\left(1 + \frac{p}{s^2}\right)}\right)\), an application of Corollary \ref{corollary:supp_approximation} along with \(p \geq 4\) and \(1 \leq s < \sqrt{p}\) being an integer yields 
            \begin{equation*}
                \left|\left|\frac{\theta - \bar{\theta}\mathbf{1}_{\supp(\theta)}}{\sqrt{1-\gamma}}\right|\right|^2 \geq \frac{||\theta - \bar{\theta}\mathbf{1}_{p}||^2}{1-\gamma} \cdot \frac{p-2s}{p} \geq \frac{||\theta - \bar{\theta}\mathbf{1}_{p}||^2}{1-\gamma} \cdot \frac{1}{4} \geq \frac{C^2}{4} s \log\left(1 + \frac{p}{s^2}\right).
            \end{equation*}
            Hence, we have shown that \(\theta \in \Theta_{\mathcal{I}}\left(p, s, C\sqrt{(1-\gamma)s\log\left(1 + \frac{p}{s^2}\right)}\right)\) implies \(\frac{\theta - \bar{\theta}\mathbf{1}_p}{\sqrt{1-\gamma}} \in \mathscr{M}\left(p, s, \frac{C}{2}\sqrt{s\log\left(1 + \frac{p}{s^2}\right)}\right)\) where the latter parameter space is given by (\ref{space:supp_space}). Recall under the data-generating process \(P_{\theta, \gamma}\) we have \(\widetilde{X} \sim N\left(\frac{\theta - \bar{\theta}\mathbf{1}_p}{\sqrt{1-\gamma}}, I_p \right)\) where \(\widetilde{X}\) is given in the definition \(\varphi_{t^*, r^*}\) (see (\ref{test:problemI_tsybakov})). Now since \(\frac{C}{2} > \frac{C_\eta}{2} = C_\eta^*\) and \(r^* = \frac{C^2/4}{8} s\log\left(1 + \frac{p}{s^2}\right)\), Proposition \ref{prop:supp_tsybakov} immediately implies 
            \begin{equation*}
                P_{0, \gamma}\left\{\varphi_{t^*, r^*} = 1 \right\} + \sup_{\theta \in \Theta_{\mathcal{I}}\left(p, s, C\sqrt{(1-\gamma) s \log\left(1 + \frac{p}{s^2}\right)}\right)} P_{\theta, \gamma}\left\{ \varphi_{t^*, r^*} = 0 \right\} \leq \eta.
            \end{equation*}
            Since \(C > C_\eta\) was arbitrary and \(\eta \in (0, 1)\) was arbitrary, the desired result has been proved.
        \end{proof}

        \begin{proof}[Proof of Proposition \ref{prop:problemI_upperbound}]
            Fix \(\eta \in (0, 1)\) and let \(C_{\eta} = C_{\eta, 1} \vee C_{\eta, 2}  \) where \(C_{\eta, 1}\) and \(C_{\eta, 2}\) are the constants at level \(\eta\) from Lemmas \ref{lemma:problem1_tsybakov} and \ref{lemma:problem1_chisquare} respectively. Let \(C > C_{\eta}\). Recall that \(p \geq 4\) is assumed. 

            \textbf{Case 1:} \(1 \leq s < \sqrt{p}\). Then \(\varphi_I^* = \varphi_{t^*, r^*}\) and \(\psi_1^2 = (1-\gamma)s \log\left(1 + \frac{p}{s^2}\right)\). Since \(C > C_{\eta, 1}\), it immediately follows from Lemma \ref{lemma:problem1_tsybakov} that \(P_{0, \gamma}\left\{ \varphi_I^* = 1 \right\} + \sup_{\theta \in \Theta_{\mathcal{I}}(p, s, C\psi_1)} P_{\theta, \gamma}\left\{ \varphi_{I}^* = 0 \right\} \leq \eta\).

            \textbf{Case 2:} Suppose \(s \geq \sqrt{p}\). Then \(\varphi_{I}^* = \varphi_{C^2/2}^{\chi^2}\) and \(\psi_1^2 = (1-\gamma)\sqrt{p}\). Since \(C > C_{\eta, 2}\), it immediately follows from Lemma \ref{lemma:problem1_chisquare} that \(P_{0, \gamma}\left\{\varphi_I^* = 1 \right\} + \sup_{\theta \in \Theta_{\mathcal{I}}(p, s, C\psi_1)} P_{\theta, \gamma}\left\{ \varphi_{I}^* = 0 \right\} \leq \eta\). Since \(C > C_\eta\) was arbitrary and \(\eta \in (0, 1)\) was arbitrary, we have proved the desired result.
        \end{proof}

    \subsubsection{Lower bound}
    In this section, we prove Proposition \ref{prop:problemI_lowerbound}. The strategy for proving the lower bound is the typical one found in the literature. In particular, as laid out by Lemma \ref{lemma:general_lower_bound}, we will construct a prior distribution on the alternative hypothesis and bound the \(\chi^2\)-divergence between the null distribution and the mixture distribution induced by the prior. Fortunately, in our Gaussian observation setting, the so called Ingster-Suslina method \cite{ingsterNonparametricGoodnessoffitTesting2003} (see Lemma \ref{lemma:Ingster_Suslina} below) can be used to calculate the \(\chi^2\)-divergence.  
    
    In preparation for the proof of Proposition \ref{prop:problemI_lowerbound}, the following lemma is needed.
    \begin{lemma}\label{lemma:v_inverse}
        If \(\gamma \in [0, 1)\) and \(v \in \R^p\) satisfies \(||v|| = \sqrt{p}\), then 
        \begin{equation*}
            \left((1-\gamma)I_p + \gamma vv^\intercal\right)^{-1} = \frac{1}{1-\gamma}\left( I_p - \frac{1}{p}vv^\intercal \right) + \frac{1}{1-\gamma+\gamma p} \cdot \frac{1}{p}vv^\intercal. 
        \end{equation*}
    \end{lemma}
    \begin{proof}
        The result follows immediately from an application of the Sherman-Morrison formula.
    \end{proof}

    The proof of Proposition \ref{prop:problemI_lowerbound} uses the same prior distribution used in the lower bound argument of Collier et al. \cite{collierMinimaxEstimationLinear2017}. In fact, the calculation in our proof of Proposition \ref{prop:problemI_lowerbound} is similar to the calculation in Collier et al. \cite{collierMinimaxEstimationLinear2017} despite the correlation in our setting; it turns out no significant extra work needs to be done to obtain the lower bound for Problem I. 

    \begin{proof}[Proof of Proposition \ref{prop:problemI_lowerbound}]
        Fix \(\eta \in (0, 1)\) and set 
        \begin{equation*}
            c_\eta := \sqrt{\frac{2-\sqrt{2}}{2}} \wedge \sqrt{\frac{2-\sqrt{2}}{2} \log\left(1 + 4\eta^2\right)} \wedge \sqrt{\frac{2-\sqrt{2}}{2} \log\left(1 + \log\left(1 + 4\eta^2\right)\right)}
        \end{equation*}
        Let \(0 < c < c_\eta\). We deal with the two sparsity regimes separately. 

        \textbf{Case 1:} Suppose \(1 \leq s < \sqrt{p}\). Note \(\psi_1^2 = (1-\gamma)s \log\left(1 + \frac{p}{s^2}\right)\). Let \(\pi\) be the prior on \(\Theta_{\mathcal{I}}(p, s, c\psi_1)\) in which a draw \(\mu \sim \pi\) is obtained by drawing a subset \(S \subset [p]\) of size \(s\) uniformly at random from the set of all subsets of \([p]\) of size \(s\) and setting 
        \begin{equation*}
            \mu_i := 
            \begin{cases}
                \sqrt{\frac{2}{2-\sqrt{2}}} \cdot \frac{c\psi_1}{\sqrt{s}} &\text{if } i \in S, \\
                0 &\text{if } i \in S^c.
            \end{cases}
        \end{equation*}
        Note that we can write \(\mu = \sqrt{\frac{2}{2-\sqrt{2}}} \frac{c\psi_1}{\sqrt{s}} \mathbf{1}_S\). Since \(|S| = s\), we clearly have \(||\mu||_0 \leq s\). By Corollary \ref{corollary:orthog_approximation}, we have \(||\mu - \bar{\mu}\mathbf{1}_p||^2 \geq ||\mu||^2 \cdot \frac{p-s}{p} \geq  \frac{2}{2-\sqrt{2}} c^2 \psi_1^2 \cdot \frac{p-s}{p} \geq c^2 \psi_1^2 \). The last inequality follows from \(s < \sqrt{p}\) and \(p \geq 2\). Hence, \(\pi\) is indeed supported on \(\Theta_{\mathcal{I}}(p, s, c\psi_1)\). 

        Let \(P_{\pi, \gamma} = \int P_{\theta, \gamma} \pi(d\theta)\) denote the Gaussian mixture induced by \(\pi\). By Lemma \ref{lemma:Ingster_Suslina} and Lemma \ref{lemma:v_inverse}, 
        \begin{align*}
            \chi^2(P_{\pi, \gamma} || P_{0, \gamma}) = E\left[ \exp \left( \left\langle \theta, \left[ \frac{1}{1-\gamma} I_p - \frac{\gamma}{(1-\gamma)^2 + (1-\gamma)\gamma p} \mathbf{1}_p \mathbf{1}_p^\intercal \right] \widetilde{\theta}  \right\rangle \right)\right] - 1
        \end{align*}
        where \(\theta, \widetilde{\theta} \overset{iid}{\sim} \pi\). Write \(\theta = \sqrt{\frac{2}{2-\sqrt{2}}} \frac{c\psi_1}{\sqrt{s}} \mathbf{1}_S\) and \(\widetilde{\theta} = \sqrt{\frac{2}{2-\sqrt{2}}} \frac{c \psi_1}{\sqrt{s}} \mathbf{1}_{\widetilde{S}}\) for \(S, \widetilde{S}\) iid uniformly drawn subsets of \([p]\) of size \(s\). Then, we have 
        \begin{align*}
            \left\langle \theta, \left[ \frac{1}{1-\gamma} I_p - \frac{\gamma}{(1-\gamma)^2 + (1-\gamma)\gamma p} \mathbf{1}_p \mathbf{1}_p^\intercal \right] \widetilde{\theta}  \right\rangle &= \frac{2}{2-\sqrt{2}} \frac{c^2 \psi_1^2}{s} \left[\frac{1}{1-\gamma} |S \cap \widetilde{S}| - \frac{\gamma s^2}{(1-\gamma)^2 + \gamma(1-\gamma)p}\right] \\
            &\leq \frac{2}{2-\sqrt{2}} \frac{c^2\psi_1^2}{s(1-\gamma)} |S \cap \widetilde{S}|
        \end{align*}
        where the final inequality is due to \(\gamma \geq 0\). Hence, 
        \begin{align}
            \chi^2(P_{\pi, \gamma} || P_{0, \gamma}) &\leq E\left[\exp\left( \frac{2}{2-\sqrt{2}} \frac{c^2\psi_1^2}{s(1-\gamma)} |S \cap \widetilde{S}| \right)\right] - 1 \\
            &\leq \left(1 - \frac{s}{p} + \frac{s}{p} \exp\left(\frac{2}{2-\sqrt{2}} \frac{c^2\psi_1^2}{s(1-\gamma)} \right)\right)^s - 1\label{eqn:chi_square_divergence}
        \end{align}
        where the final inequality follows from the fact that \(|S \cap \widetilde{S}|\) is distributed according to the hypergeometric distribution with probability mass function given in Lemma \ref{lemma:hypergeometric}. Recalling that \(\psi_1^2 = (1-\gamma) s \log\left(1 + \frac{p}{s^2}\right)\), we have 
        \begin{align*}
            \left(1 - \frac{s}{p} + \frac{s}{p} \exp\left(\frac{2}{2-\sqrt{2}} \frac{c^2\psi_1^2}{s(1-\gamma)} \right)\right)^s &= \left(1 - \frac{s}{p} + \frac{s}{p} \exp\left(\frac{2}{2-\sqrt{2}} c^2 \log\left(1 + \frac{p}{s^2}\right) \right)\right)^s \\
            &= \left(1 - \frac{s}{p} + \frac{s}{p}\left(1 + \frac{p}{s^2}\right)^{\frac{2}{2-\sqrt{2}} c^2}\right)^s \\
            &\leq \left(1 + \frac{1}{s} \frac{2}{2-\sqrt{2}} c^2 \right)^s \\
            &\leq \exp\left(\frac{2}{2-\sqrt{2}} c^2\right).
        \end{align*}
        where the penultimate inequality follows from the fact that \(\frac{2}{2-\sqrt{2}} c^2 < \frac{2}{2-\sqrt{2}} c_\eta^2 \leq 1\) along with the inequality \((1+x)^\delta - 1 \leq \delta x\) for all \(0 < \delta < 1\) and \(x > 0\). Therefore, it follows that \(\chi^2(P_{\pi, \gamma} || P_{0, \gamma}) \leq \exp\left(\frac{2}{2-\sqrt{2}}c^2\right) - 1 \leq \exp\left( \frac{2}{2-\sqrt{2}} c_\eta^2 \right) - 1 \leq 4\eta^2\). An application of Lemma \ref{lemma:general_lower_bound} yields \(\mathcal{R}_{\mathcal{I}}(c\psi_1) \geq 1 - \frac{1}{2}\sqrt{\chi^2(P_{\pi, \gamma} || P_{0, \gamma})} \geq 1-\eta\). Since \(0 < c < c_\eta\) was arbitrary and \(\eta \in (0, 1)\) was arbitrary, we have proved the desired result for the case \(1 \leq s < \sqrt{p}\).

        \textbf{Case 2:} Suppose \(s \geq \sqrt{p}\). Note \(\psi_1^2 = (1-\gamma)\sqrt{p}\). Without loss of generality, assume \(\sqrt{p}\) is an integer. Repeating exactly the argument presented in Case 1 except now replacing every instance of \(s\) with \(\sqrt{p}\) yields 
        \begin{align*}
            \chi^2(P_{\pi, \gamma} || P_{0, \gamma}) &\leq \left(1 - \frac{1}{\sqrt{p}} + \frac{1}{\sqrt{p}} \exp\left(\frac{2}{2-\sqrt{2}} \frac{c^2\psi_1^2}{\sqrt{p}(1-\gamma)}\right)\right)^{\sqrt{p}} - 1 \\
            &\leq \left(1 + \frac{1}{\sqrt{p}} \left[\exp\left(\frac{2}{2-\sqrt{2}} c^2 \right) - 1\right] \right)^{\sqrt{p}} - 1\\
            &\leq \exp\left(\exp\left(\frac{2}{2-\sqrt{2}} c^2\right) - 1\right) - 1 \\
            &\leq \exp\left(\exp\left(\frac{2}{2-\sqrt{2}} c_\eta^2\right) - 1\right) - 1 \\
            &\leq 4\eta^2. 
        \end{align*}
        An application of Lemma \ref{lemma:general_lower_bound} yields \(\mathcal{R}_{\mathcal{I}}(c\psi_1) \geq 1 - \frac{1}{2}\sqrt{\chi^2(P_{\pi, \gamma} || P_{0, \gamma})} \geq 1-\eta\). Since \(0 < c < c_\eta\) was arbitrary and \(\eta \in (0, 1)\) was arbitrary, we have proved the desired result for the case \(s \geq \sqrt{p}\). 
    \end{proof}

    \begin{proof}[Proof of Theorem \ref{thm:ProblemI}]
        If \(p \geq 4\), then Propositions \ref{prop:problemI_upperbound} and \ref{prop:problemII_lowerbound} can be combined to immediately yield the desired result. Suppose \(p = 2\) or \(p = 3\). Proposition \ref{prop:problemI_lowerbound} already gives the lower bound. To prove the upper bound, one can repeat the analysis of Lemma \ref{lemma:problem1_chisquare} to show that the \(\chi^2\)-test is sensitive for all \(1 \leq s \leq p\). The proof is complete. 
    \end{proof}

    \subsection{Proof of results in Section \ref{section:ProblemII}}

    \subsubsection{Upper bound}
    To prove the upper bound, the constituent tests are examined separately and later combined to prove Proposition \ref{prop:problemII_upperbound}.

    \begin{lemma}\label{lemma:problemII_linear}
        Suppose \(1 \leq s \leq p\) and \(\gamma \in [0, 1]\). If \(\eta \in (0, 1)\), then there exists a constant \(C_\eta\) depending only on \(\eta\) such that for all \(C > C_\eta\) the testing procedure 
        \begin{equation*}
            \varphi_{C^2/2}^{\mathbf{1}_p} := \mathbbm{1}_{\{\langle p^{-1/2} \mathbf{1}_p, X\rangle^2 > (1-\gamma+\gamma p) \left(1 + C^2/2\right)\}}
        \end{equation*}
        satisfies 
        \begin{equation*}
            P_{0, \gamma}\left\{\varphi_{C^2/2}^{\mathbf{1}_p} = 1 \right\} + \sup_{\theta \in \Theta_{\mathcal{II}}(p, s, C\sqrt{1-\gamma+\gamma p})} P_{\theta, \gamma}\left\{\varphi_{C^2/2}^{\mathbf{1}_p} = 0\right\} \leq \eta. 
        \end{equation*}
    \end{lemma}
    \begin{proof}
        Fix \(\eta \in (0, 1)\) and let \(C_{\eta} := \left(\frac{16}{\eta}\right)^{1/4} \vee \left(\frac{32}{\eta}\right)^{1/4} \vee \sqrt{\frac{64}{\eta}}\). Clearly such a value of \(C_\eta\) exists since \(C_\eta\) can be taken sufficiently large depending only \(\eta\). Let \(C > C_\eta\). Consider under the data-generating process \(P_{\theta, \gamma}\) we have \(\langle p^{-1/2}\mathbf{1}_p, X\rangle \sim N\left( \sqrt{p}\bar{\theta}, 1-\gamma+\gamma p\right)\). Therefore, \(\langle p^{-1/2}\mathbf{1}_p, X\rangle^2 \sim (1-\gamma+\gamma p) \chi^2_1\left(\frac{||\bar{\theta}\mathbf{1}_p||^2}{1-\gamma+\gamma p} \right)\). Moreover, \(E_{\theta, \gamma}\left( \langle p^{-1/2} \mathbf{1}_p, X\rangle^2 \right) = (1-\gamma+\gamma p) + ||\bar{\theta}\mathbf{1}_p||^2\) and \(\Var_{\theta, \gamma}\left( \langle p^{-1/2}\mathbf{1}_p, X\rangle^2 \right) = 2(1-\gamma+\gamma p)^2 + 4(1-\gamma+\gamma p)||\bar{\theta}\mathbf{1}_p||^2\). Examining the type I error, consider 
        \begin{align*}
            P_{0, \gamma}\left\{ \langle p^{-1/2}\mathbf{1}_p, X\rangle^2 > (1-\gamma+\gamma p)\left(1 + \frac{C^2}{2}\right) \right\} \leq \frac{\Var_{0, \gamma}\left( \langle p^{-1/2}\mathbf{1}_p, X\rangle \right)}{\frac{C^4}{4}(1-\gamma+\gamma p)^2} \leq \frac{8}{C_\eta^4} \leq \frac{\eta}{2}
        \end{align*}
        since \(C_\eta \geq \left(\frac{16}{\eta}\right)^{1/4}\). Examining the type II error, consider 
        \begin{align*}
            &\sup_{\theta \in \Theta_{\mathcal{II}}(p, s, C\sqrt{1-\gamma+\gamma p})} P_{\theta, \gamma}\left\{\varphi_{C^2/2}^{\mathbf{1}_p} = 0 \right\} \\
            &= \sup_{\theta \in \Theta_{\mathcal{II}}(p, s, C\sqrt{1-\gamma+\gamma p})} P_{\theta, \gamma}\left\{\langle p^{-1/2}\mathbf{1}_p, X\rangle^2 \leq (1-\gamma+\gamma p)\left(1 + \frac{C^2}{2}\right) \right\} \\
            &\leq \sup_{\theta \in \Theta_{\mathcal{II}}(p, s, C\sqrt{1-\gamma+\gamma p})} P_{\theta, \gamma}\left\{||\bar{\theta}\mathbf{1}_p||^2 - (1-\gamma+\gamma p)\frac{C^2}{2} \leq (1-\gamma+\gamma p) + ||\bar{\theta}\mathbf{1}_p||^2 - \langle p^{-1/2}\mathbf{1}_p, X\rangle^2 \right\} \\
            &\leq \sup_{\theta \in \Theta_{\mathcal{II}}(p, s, C\sqrt{1-\gamma+\gamma p})}\frac{\Var_{\theta, \gamma}\left(\langle p^{-1/2}\mathbf{1}_p, X\rangle^2\right)}{\left(||\bar{\theta}\mathbf{1}_p||^2 - \frac{C^2}{2}(1-\gamma+\gamma p)\right)^2} \\
            &\leq \frac{2(1-\gamma+\gamma p)^2}{\frac{C^4}{4}(1-\gamma+\gamma p)^2} + \sup_{\theta \in \Theta_{\mathcal{II}}(p, s, C\sqrt{1-\gamma+\gamma p})} \frac{4(1-\gamma+\gamma p) ||\bar{\theta}\mathbf{1}_p||^2}{\frac{1}{4} ||\bar{\theta}\mathbf{1}_p||^4} \\
            &\leq \frac{8}{C^4} + \frac{16}{C^2} \\
            &\leq \frac{\eta}{2}.
        \end{align*}
        Thus, \(P_{\theta, \gamma}\left\{ \varphi_{C^2/2}^{\mathbf{1}_p} = 1 \right\} + \sup_{\theta \in \Theta_{\mathcal{II}}(p, s, C\sqrt{1-\gamma+\gamma p})} P_{\theta, \gamma}\left\{\varphi_{C^2/2}^{\mathbf{1}_p} = 0 \right\} \leq \frac{\eta}{2} + \frac{\eta}{2} \leq \eta\). Since \(C > C_\eta\) was arbitrary and \(\eta \in (0, 1)\) was arbitrary, the desired result is thus proved. 
    \end{proof}

    \begin{lemma}\label{lemma:problemII_chisquare}
        Suppose \(1 \leq s < p\) and \(\gamma \in [0, 1)\). If \(\eta \in (0, 1)\), then there exists a constant \(C_\eta\) depending only on \(\eta\) such that for all \(C > C_\eta\) the testing procedure \(\varphi_{C^2/2}^{\chi^2}\) given by (\ref{test:problemI_chisquare}) satisfies 
        \begin{equation*}
            P_{0, \gamma}\left\{ \varphi_{C^2/2}^{\chi^2} = 1  \right\} + \sup_{\theta \in \Theta_{\mathcal{II}}\left(p, s, C\sqrt{\frac{(1-\gamma)p^{3/2}}{p-s}}\right)} P_{\theta, \gamma}\left\{ \varphi_{C^2/2}^{\chi^2} = 0 \right\} \leq \eta.
        \end{equation*}
    \end{lemma}
    \begin{proof}
        Fix \(\eta \in (0, 1)\) and let \(C_{\eta} := \left(\frac{16}{\eta}\right)^{1/4} \vee \left(\frac{32}{\eta}\right)^{1/4} \vee \sqrt{\frac{64}{\eta}}\). Let \(C > C_\eta\). For ease of notation, denote \(\kappa := \frac{(1-\gamma)p^{3/2}}{p-s}\). Note that the same test was analyzed in Lemma \ref{lemma:problem1_chisquare}. Recall from the proof of Lemma \ref{lemma:problem1_chisquare} that \(||\widetilde{X}||^2 \sim \chi^2_p\left(\frac{||\theta - \bar{\theta}\mathbf{1}_p||^2}{1-\gamma}\right)\) under the data-generating process \(P_{\theta, \gamma}\). Therefore, \(E_{\theta, \gamma}\left(||\widetilde{X}||^2\right) = p + \frac{||\theta - \bar{\theta}\mathbf{1}_p||^2}{1-\gamma}\) and \(\Var_{\theta, \gamma}\left(||\widetilde{X}||^2 \right) = 2p + 4\frac{||\theta - \bar{\theta}\mathbf{1}_p||^2}{1-\gamma}\). The type I error analysis of the proof of Lemma \ref{lemma:problem1_chisquare} can be exactly repeated here to yield \(P_{0, \gamma}\left\{ \varphi_{C^2/2}^{\chi^2} = 1 \right\} \leq \frac{\eta}{2}\). Before turning our attention to the type II error, we make a preliminary observation. For every \(\theta \in \Theta_{\mathcal{II}}(p, s, C\sqrt{\kappa})\), we have by Corollary \ref{corollary:orthog_approximation}
        \begin{align*}
            ||\theta - \bar{\theta}\mathbf{1}_p||^2 \geq ||\theta||^2\cdot \frac{p-s}{p} \geq ||\bar{\theta}\mathbf{1}_p||^2 \cdot \frac{p-s}{p} \geq C^2\kappa \cdot \frac{p-s}{p} = C^2 (1-\gamma) \sqrt{p}.
        \end{align*}
        In other words, we have the inclusion \(\Theta_{\mathcal{II}}(p, s, C\sqrt{\kappa}) \subset \Theta_{\mathcal{I}}(p, s, C\sqrt{(1-\gamma)\sqrt{p}})\). Consequently,
        \begin{align*}
            \sup_{\theta \in \Theta_{\mathcal{II}}(p, s, C\sqrt{\kappa})} P_{\theta, \gamma}\left\{\varphi_{C^2/2}^{\chi^2} = 0 \right\} \leq \sup_{\theta \in \Theta_{\mathcal{I}}(p, s, C\sqrt{(1-\gamma)\sqrt{p}})} P_{\theta, \gamma}\left\{ \varphi_{C^2/2}^{\chi^2} = 0 \right\} \leq \frac{8}{C_\eta^4} + \frac{16}{C_\eta^2} \leq \frac{\eta}{2}
        \end{align*}
        where the second inequality is obtained by repeating exactly the type II error analysis in the proof of Lemma \ref{lemma:problem1_chisquare}. Therefore, \(P_{0, \gamma}\left\{ \varphi_{C^2/2}^{\chi^2} = 1 \right\} + \sup_{\theta \in \Theta_{\mathcal{II}}(p, s, C\sqrt{\kappa})} P_{\theta, \gamma}\left\{\varphi_{C^2/2}^{\chi^2} = 0 \right\} \leq \frac{\eta}{2} + \frac{\eta}{2} \leq \eta\). Since \(C > C_\eta\) was arbitrary and \(\eta \in (0, 1)\) was arbitrary, the desired result has thus been proved.
    \end{proof}

    \begin{lemma}\label{lemma:problemII_tsybakov}
        Suppose \(p - \sqrt{p} < s < p\) and \(\gamma \in [0, 1)\). If \(\eta \in (0, 1)\), then there exists a constant \(C_\eta\) depending only on \(\eta\) such that for all \(C > C_\eta\) the testing procedure \(\varphi_{\widetilde{t}, \widetilde{r}}\) with \(\widetilde{t} = \sqrt{2 \log\left(1 + \frac{p}{(p-s)^2}\right)}\) and \(\widetilde{r} = \frac{C^2}{8} (p-s) \log\left(1 + \frac{p}{(p-s)^2}\right)\) given by (\ref{test:problemI_tsybakov}) satisfies 
        \begin{equation*}
            P_{0,\gamma}\left\{ \varphi_{\widetilde{t}, \widetilde{r}} = 1 \right\} + \sup_{\theta \in \Theta_{\mathcal{II}}\left(p, s, C \sqrt{(1-\gamma) p \log\left(1 + \frac{p}{(p-s)^2}\right)} \right)}P_{\theta, \gamma}\left\{ \varphi_{\widetilde{t}, \widetilde{r}} = 0 \right\} \leq \eta. 
        \end{equation*}
    \end{lemma}
    \begin{proof}
        Fix \(\eta \in (0, 1)\) and set \(C_\eta\) as in Proposition \ref{prop:supp_tsybakov}. Let \(C > C_\eta\). For ease of notation, set \(\kappa:= (1-\gamma)p \log\left(1 + \frac{p}{(p-s)^2}\right)\). Fix \(\theta \in \Theta_{\mathcal{II}}\left(p, s, C \sqrt{\kappa}\right)\). Since \(||\theta||_0 \leq s\), there exists a set \(T \subset \supp(\theta)^c\) such that \(|T| = p-s\). It immediately follows 
        \begin{equation*}
            \left|\left|\left(\frac{\theta-\bar{\theta}\mathbf{1}_p}{\sqrt{1-\gamma}} \right)_T\right|\right|^2 = \frac{\left|\left|-\bar{\theta}\mathbf{1}_{T} \right|\right|^2}{1-\gamma} = \frac{\bar{\theta}^2 (p-s)}{1-\gamma} \geq \frac{C^2\kappa (p-s)}{p(1-\gamma)} = C^2(p-s)\log\left(1 + \frac{p}{(p-s)^2}\right).
        \end{equation*}
        Thus, we have shown that \(\theta \in  \Theta_{\mathcal{II}}\left(p, s, C \sqrt{\kappa} \right)\) implies \(\frac{\theta-\bar{\theta}\mathbf{1}_p}{\sqrt{1-\gamma}} \in \mathscr{M}\left(p, p-s, C\sqrt{(p-s)\log\left(1 + \frac{p}{(p-s)^2}\right)}\right)\) where the latter parameter space is given by (\ref{space:supp_space}). Consider that \(1 \leq p-s < \sqrt{p}\) and \(\widetilde{X} \sim N\left(\frac{\theta - \bar{\theta}\mathbf{1}_p}{\sqrt{1-\gamma}}, I_p\right)\) under the data-generating process \(P_{\theta, \gamma}\) (where \(\widetilde{X}\) is given in the definition of \(\varphi_{\widetilde{t}, \widetilde{r}}\) as in (\ref{test:problemI_tsybakov})). Since \(C > C_\eta\) and \(\widetilde{r} = \frac{C^2}{8} (p-s) \log\left(1 + \frac{p}{(p-s)^2}\right)\), an application of Proposition \ref{prop:supp_tsybakov} yields 
        \begin{equation*}
            P_{0,\gamma}\left\{ \varphi_{\widetilde{t}, \widetilde{r}} = 1 \right\} + \sup_{\theta \in \Theta_{\mathcal{II}}\left(p, s, C \sqrt{\kappa} \right)}P_{\theta, \gamma}\left\{ \varphi_{\widetilde{t}, \widetilde{r}} = 0 \right\} \leq \eta. 
        \end{equation*}
        Since \(C > C_\eta\) was arbitrary and \(\eta \in (0, 1)\) was arbitrary, the desired result has been proved.
    \end{proof}

    Having examined the constituent tests, we are ready to prove Proposition \ref{prop:problemII_upperbound}.

    \begin{proof}[Proof of Proposition \ref{prop:problemII_upperbound}]
        Fix \(\eta \in (0, 1)\) and let \(C_{\eta} := C_{\eta/2, 1} \vee C_{\eta/2, 2} \vee C_{\eta/2, 3}\) where \(C_{\eta/2, 1}, C_{\eta/2, 2},\) and \(C_{\eta/2, 3}\) are the constants at level \(\frac{\eta}{2}\) from Lemmas \ref{lemma:problemII_linear}, \ref{lemma:problemII_chisquare}, and \ref{lemma:problemII_tsybakov} respectively. Let \(C > C_\eta\). We consider the separate cases. 

        \textbf{Case 1:} Suppose \(\frac{p}{2} \leq s \leq p-\sqrt{p}\) and \(\psi_2^2 = (1-\gamma+\gamma p)\). Then \(\varphi_{II}^* = \varphi_{C^2/2}^{\chi^2} \vee \varphi_{C^2/2}^{\mathbf{1}_p}\). Therefore, since \(C > C_{\eta/2, 1}\), we have by Lemma \ref{lemma:problemII_linear}
        \begin{align*}
            &P_{0, \gamma}\{\varphi_{II}^* = 1\} + \sup_{\theta \in \Theta_{\mathcal{II}}(p, s, C\psi_2)} P_{\theta, \gamma}\{ \varphi_{II}^* = 0 \} \\
            &\leq P_{0, \gamma}\{\varphi_{C^2/2}^{\chi^2} = 1\} + P_{0, \gamma}\{\varphi_{C^2/2}^{\mathbf{1}_p} = 1\} + \sup_{\theta \in \Theta_{\mathcal{II}}(p, s, C\psi_2)} P_{\theta, \gamma} \{ \varphi_{C^2/2}^{\mathbf{1}_p} = 0\} \\
            &\leq \frac{\eta}{2} + \frac{\eta}{2} \\
            &\leq \eta
        \end{align*}
        where we have also used Lemma \ref{lemma:problemII_chisquare} because \(C > C_{\eta/2,2}\).
        
        \textbf{Case 2:} Suppose \(\frac{p}{2} \leq s \leq p-\sqrt{p}\) and \(\psi^2_2 = \frac{(1-\gamma)p^{3/2}}{p-s}\). Then \(\varphi_{II}^* = \varphi_{C^2/2}^{\chi^2} \vee \varphi_{C^2/2}^{\mathbf{1}_p}\). Therefore, since \(C > C_{\eta/2, 2}\), we have by Lemma \ref{lemma:problemII_chisquare}
        \begin{align*}
            & P_{0, \gamma}\{\varphi_{II}^* = 1\} + \sup_{\theta \in \Theta_{\mathcal{II}}(p, s, C\psi_2)} P_{\theta, \gamma}\{ \varphi_{II}^* = 0 \} \\
            &\leq P_{0, \gamma}\{\varphi_{C^2/2}^{\mathbf{1}_p} = 1\} + P_{0, \gamma}\{\varphi_{C^2/2}^{\chi^2} = 1\} + \sup_{\theta \in \Theta_{\mathcal{II}}(p, s, C\psi_2)} P_{\theta, \gamma}\{\varphi_{C^2/2}^{\chi^2} = 0\} \\
            &\leq \frac{\eta}{2} + \frac{\eta}{2} \\
            &\leq \eta
        \end{align*}
        where we have also used Lemma \ref{lemma:problemII_linear} because \(C > C_{\eta/2, 1}\).

        \textbf{Case 3:} Suppose \(p-\sqrt{p} < s < p\) and \(\psi^2_2 = 1-\gamma+\gamma p\). Then \(\varphi_{II}^* = \varphi_{\widetilde{r}, \widetilde{t}} \vee \varphi_{C^2/2}^{\mathbf{1}_p}\). Therefore, since \(C > C_{\eta/2, 1}\), we have by Lemma \ref{lemma:problemII_linear}
        \begin{align*}
            &P_{0, \gamma}\{\varphi_{II}^* = 1\} + \sup_{\theta \in \Theta_{\mathcal{II}}(p, s, C\psi_2)} P_{\theta, \gamma}\{ \varphi_{II}^* = 0 \} \\
            &\leq P_{0, \gamma}\{\varphi_{\widetilde{t}, \widetilde{r}} = 1\} + P_{0, \gamma}\{\varphi_{C^2/2}^{\mathbf{1}_p} = 1\} + \sup_{\theta \in \Theta_{\mathcal{II}}(p, s, C\psi_2)} P_{\theta, \gamma} \{ \varphi_{C^2/2}^{\mathbf{1}_p} = 0\} \\
            &\leq \frac{\eta}{2} + \frac{\eta}{2} \\
            &\leq \eta
        \end{align*}
        where we have also used Lemma \ref{lemma:problemII_tsybakov} because \(C > C_{\eta/2, 3}\). 

        \textbf{Case 4:} Suppose \(p - \sqrt{p} < s < p\) and \(\psi^2_2 = (1-\gamma)p\log\left(1 + \frac{p}{(p-s)^2}\right)\). Then \(\varphi_{II}^* = \varphi_{\widetilde{r}, \widetilde{t}} \vee \varphi_{C^2/2}^{\mathbf{1}_p}\). Therefore, since \(C > C_{\eta/2, 3}\), we have by Lemma \ref{lemma:problemII_linear}
        \begin{align*}
            & P_{0, \gamma}\{\varphi_{II}^* = 1\} + \sup_{\theta \in \Theta_{\mathcal{II}}(p, s, C\psi_2)} P_{\theta, \gamma}\{ \varphi_{II}^* = 0 \} \\
            &\leq P_{0, \gamma}\{\varphi_{C^2/2}^{\mathbf{1}_p} = 1\} + P_{0, \gamma}\{\varphi_{\widetilde{t}, \widetilde{r}} = 1\} + \sup_{\theta \in \Theta_{\mathcal{II}}(p, s, C\psi_2)} P_{\theta, \gamma}\{\varphi_{\widetilde{t}, \widetilde{r}} = 0\} \\
            &\leq \frac{\eta}{2} + \frac{\eta}{2} \\
            &\leq \eta
        \end{align*}
        where we have also used Lemma \ref{lemma:problemII_linear} because \(C > C_{\eta/2, 1}\). 

        \textbf{Case 5:} Suppose \(s = p\). Then \(\psi_2^2 = 1-\gamma+\gamma p\) and \(\varphi_{II}^* = \varphi_{C^2/2}^{\mathbf{1}_p}\). Since \(C > C_{\eta/2, 1}\), Lemma \ref{lemma:problemII_linear} immediately implies 
        \begin{equation*}
            P_{0, \gamma}\{ \varphi_{C^2/2}^{\mathbf{1}_p} = 1\} + \sup_{\theta \in \Theta_{\mathcal{II}}(p, s, C\psi_2)} P_{\theta, \gamma} \{\varphi_{C^2/2}^{\mathbf{1}_p} = 0\} \leq \frac{\eta}{2} \leq \eta. 
        \end{equation*}
        
        All of the cases have now been dealt with. Since \(C > C_{\eta}\) was arbitrary and \(\eta \in (0, 1)\) was arbitrary, the desired result has been proved.
    \end{proof}

    \subsubsection{Lower bound}
    In this section, we prove Proposition \ref{prop:problemII_lowerbound}. A number of preliminary lemmas are needed.

    \begin{lemma}\label{lemma:dtv_null_bound}
        Suppose \(\gamma \in [0, 1)\). Let \(P_{\theta, \gamma}\) denote the distribution \(N(\theta, (1-\gamma)I_p + \gamma \mathbf{1}_p\mathbf{1}_p^\intercal)\) with \(\theta \in \R^p\). If \(\mu = m \mathbf{1}_p\) for some \(m \in \R\), then 
        \begin{equation*}
            d_{TV}(P_{\mu, \gamma}, P_{0, \gamma}) \leq \frac{1}{2}\sqrt{\exp\left(\frac{pm^2}{1-\gamma+\gamma p}\right) - 1}.
        \end{equation*}
    \end{lemma}
    \begin{proof}
        By Lemma \ref{lemma:chisquare_tv_bound}, Lemma \ref{lemma:Ingster_Suslina}, and Lemma \ref{lemma:v_inverse} we have 
        \begin{align*}
            d_{TV}(P_{\mu, \gamma}, P_{0, \gamma}) &\leq \frac{1}{2} \sqrt{\chi^2(P_{\mu, \gamma} || P_{0, \gamma})} \\
            &= \frac{1}{2}\sqrt{\exp\left(\left\langle \mu, \left[\frac{1}{1-\gamma}\left(I_p - \frac{1}{p}\mathbf{1}_p\mathbf{1}_p^\intercal\right) + \frac{1}{1-\gamma+\gamma p} \cdot \frac{1}{p}\mathbf{1}_p\mathbf{1}_p^\intercal \right] \mu \right\rangle \right) - 1} \\
            &= \frac{1}{2}\sqrt{\exp\left(\frac{m^2}{1-\gamma+\gamma p}\left\langle \mathbf{1}_p, \left[\frac{1}{p}\mathbf{1}_p\mathbf{1}_p^\intercal\right] \mathbf{1}_p\right\rangle \right) - 1} \\
            &=\frac{1}{2}\sqrt{\exp\left(  \frac{p m^2}{1-\gamma+\gamma p} \right) - 1}.
        \end{align*}
    \end{proof}

    \begin{lemma}\label{lemma:problemII_lowerbound}
        Suppose \(\frac{p}{2} \leq s < p\). Set 
        \begin{equation*}
            \kappa := 
            \begin{cases}
                \frac{(1-\gamma)p^{3/2}}{p-s} \wedge (1-\gamma + \gamma p) &\text{if } \frac{p}{2} \leq s \leq p-\sqrt{p}, \\
                (1-\gamma)p \log\left(1 + \frac{p}{(p-s)^2}\right) \wedge (1-\gamma+\gamma p) &\text{if } p-\sqrt{p} < s < p. 
            \end{cases}
        \end{equation*}
        If \(\eta \in (0, 1)\), then there exists a constant \(c_\eta > 0\) depending only on \(\eta\) such that for all \(0 < c < c_\eta\) we have \(\mathcal{R}_{\mathcal{II}}(c\sqrt{\kappa}) \geq 1-\eta\). 
    \end{lemma}
    \begin{proof}
        Fix \(\eta \in (0, 1)\) and set \(c_\eta := \frac{1}{4} \wedge \frac{\log(1+\eta^2)}{8e+1} \wedge \sqrt{\frac{1}{4}\log(1+\eta^2)}\). Let \(0 < c < c_\eta\). Let \(\pi\) be the prior on \(\Theta_{\mathcal{II}}(p, s, c\sqrt{\kappa})\) in which a drawn \(\mu \sim \pi\) is obtained by drawing a subset \(S \subset [p]\) of size \(s\) uniformly at random and setting 
        \begin{equation*}
            \mu_i :=
            \begin{cases}
                \frac{c\sqrt{\kappa p}}{s} &\text{if } i \in S, \\
                0 &\text{if } i \in S^c.
            \end{cases}
        \end{equation*}
        Note that if \(\mu \sim \pi\), then \(||\mu||_0 = s\) and \(p \bar{\mu}^2 = p \left(\frac{c\sqrt{\kappa p}}{s} \cdot \frac{s}{p} \right)^2 = c^2\kappa\). Thus, \(\pi\) is indeed supported on \(\Theta(p, s, c\sqrt{\kappa})\). Let \(P_{\pi, \gamma}\) denote the Gaussian mixture \(\int_{\theta} P_{\theta, \gamma} \, \pi(d\theta)\). Consider that 
        \begin{equation}\label{eqn:problem2_tv_lbound}
            \mathcal{R}_{\mathcal{II}}(c\sqrt{\kappa}) \geq \inf_{\varphi} \left\{ P_{0, \gamma} \{\varphi = 1\} + P_{\pi, \gamma} \{\varphi = 0\} \right\} = 1 - d_{TV}(P_{0, \gamma}, P_{\pi, \gamma})
        \end{equation}
        where the final equality is given by the Neyman-Pearson lemma. Now consider the invertible transformation 
        \begin{equation}\label{eqn:problemII_lbound_Y_defn}
            Y := X - \frac{c\sqrt{\kappa p}}{s}\mathbf{1}_p. 
        \end{equation}
        Under the data-generating process \(X \sim P_{0, \gamma}\), we clearly have \(Y \sim P_{\nu^*, \gamma}\) where \(\nu^* := -\frac{c\sqrt{\kappa p}}{s}\mathbf{1}_p\). Under the process \(X \sim P_{\pi, \gamma}\), we have \(Y \sim P_{\widetilde{\pi}, \gamma}\) where \(P_{\widetilde{\pi}, \gamma} = \int P_{\theta, \gamma} \widetilde{\pi}(d\theta)\) is the mixture induced by the prior \(\widetilde{\pi}\). Here, \(\widetilde{\pi}\) is the prior in which a draw \(\mu \sim \widetilde{\pi}\) is obtained by uniformly drawing a size \(p-s\) subset \(T \subset [p]\) and setting 
        \begin{equation*}
            \mu_i := 
            \begin{cases}
                -\frac{c\sqrt{\kappa p}}{s} &\text{if } i \in T, \\
                0 &\text{if } i \in T^c.
            \end{cases}
        \end{equation*}
        Since the transformation furnishing \(Y\) is an invertible transformation, we have that \(d_{TV}(P_{0, \gamma}, P_{\pi, \gamma}) = d_{TV}(P_{\nu^*, \gamma}, P_{\widetilde{\pi}, \gamma})\). Since \(d_{TV}\) is a metric, an application of triangle inequality and Lemma \ref{lemma:dtv_null_bound} to (\ref{eqn:problemII_lbound_Y_defn}) yields
        \begin{align}
            \mathcal{R}_{\mathcal{II}}(c\sqrt{\kappa}) &\geq 1 - d_{TV}(P_{\nu^*, \gamma}, P_{0, \gamma}) - d_{TV}(P_{0, \gamma}, P_{\widetilde{\pi}, \gamma}) \\
            &\geq 1 - \frac{1}{2}\sqrt{\exp\left(\frac{p^2\kappa}{s^2(1-\gamma+\gamma p)}\right) - 1} - d_{TV}(P_{0, \gamma}, P_{\widetilde{\pi}, \gamma}).\label{eqn:dtv_triangle_inequality} 
        \end{align}
        Hence, it suffices to produce a lower bound of the right hand side of (\ref{eqn:dtv_triangle_inequality}) in order to furnish a lower bound for \(\mathcal{R}_{\mathcal{II}}(c\sqrt{\kappa})\). We examine each term separately. First, the inequalities \(\kappa \leq 1-\gamma+\gamma p\) and \(\frac{p}{2} \leq s\) together yield
        \begin{equation}\label{eqn:dtv_nu_star_bound}
           \frac{1}{2}\sqrt{\exp\left(\frac{p^2}{s^2} \frac{c^2\kappa}{1-\gamma+\gamma p}\right) - 1} \leq \frac{1}{2}\sqrt{\exp\left(4c^2\right) - 1} \leq \frac{\eta}{2}
        \end{equation}
        where we have used \(c < c_\eta \leq \sqrt{\frac{1}{4}\log(1+\eta^2)}\) to give the second inequality.
        
        We now work to furnish a lower bound for the term \(1 - d_{TV}(P_{0, \gamma}, P_{\widetilde{\pi}, \gamma})\) in the right hand side of (\ref{eqn:dtv_triangle_inequality}). Note that \(1 - d_{TV}(P_{0, \gamma}, P_{\widetilde{\pi}, \gamma}) \geq 1-\frac{1}{2}\sqrt{\chi^2(P_{\widetilde{\pi}, \gamma} || P_{0, \gamma})}\) and so it suffices to bound the \(\chi^2\)-divergence from above. By Lemma \ref{lemma:Ingster_Suslina}, 
        \begin{equation*}
            \chi^2(P_{\widetilde{\pi}, \gamma} ||P_{0, \gamma}) = E\left[ \exp\left( \left\langle \theta, \left[\frac{1}{1-\gamma}\left(I_p - \frac{1}{p}\mathbf{1}_p\mathbf{1}_p^\intercal \right) + \frac{1}{1-\gamma+\gamma p}\cdot \frac{1}{p}\mathbf{1}_p\mathbf{1}_p^\intercal \right] \widetilde{\theta} \right\rangle \right) \right] - 1
        \end{equation*}
        where \(\theta, \widetilde{\theta} \overset{iid}{\sim} \widetilde{\pi}\). Note we can write \(\theta = -\frac{c\sqrt{\kappa p}}{s}\mathbf{1}_T\) and \(\widetilde{\theta} = - \frac{c\sqrt{\kappa p}}{s}\mathbf{1}_{\widetilde{T}}\) where \(T, \widetilde{T}\) are iid uniformly drawn subsets of \([p]\) of size \(p-s\). We now consider the two cases. \newline

        \textbf{Case 1:} Suppose \(\frac{p}{2} \leq s \leq p-\sqrt{p}\). Then \(\kappa = \frac{(1-\gamma)p^{3/2}}{p-s} \wedge (1-\gamma+\gamma p)\). We have 
        \begin{equation*}
            \left\langle \theta, \left[\frac{1}{1-\gamma}\left(I_p - \frac{1}{p}\mathbf{1}_p\mathbf{1}_p^\intercal \right) + \frac{1}{1-\gamma+\gamma p}\cdot \frac{1}{p}\mathbf{1}_p\mathbf{1}_p^\intercal \right] \widetilde{\theta} \right\rangle = \frac{c^2\kappa p}{s^2(1-\gamma)} \left( |T \cap \widetilde{T}| - \frac{(p-s)^2}{p}\right) + \frac{c^2 \kappa}{1-\gamma+\gamma p} \cdot \frac{(p-s)^2}{s^2}.
        \end{equation*}
        Note that \(|T \cap \widetilde{T}|\) is distributed according to the hypergeometric distribution 
        \begin{equation*}
            P\left\{|T \cap \widetilde{T}| = k\right\} = \frac{\binom{p-s}{k} \binom{s}{p-s-k}}{\binom{p}{p-s}}
        \end{equation*}
        for \(0 \leq k \leq p-s\). Note also that since \(\kappa \leq 1-\gamma+\gamma p\) and since \(p-s \leq \frac{p}{2} \leq s\), we have \(\frac{c^2\kappa}{1-\gamma+\gamma p} \cdot \frac{(p-s)^2}{s^2} \leq c^2\). Therefore, 
        \begin{equation*}
            \frac{c^2\kappa p}{s^2(1-\gamma)} \left( |T \cap \widetilde{T}| - \frac{(p-s)^2}{p}\right) + \frac{c^2 \kappa}{1-\gamma+\gamma p} \cdot \frac{(p-s)^2}{s^2} \leq \frac{c^2\kappa p}{(1-\gamma) s^2} |T \cap \widetilde{T}| - \frac{c^2 \kappa (p-s)^2}{(1-\gamma)s^2} + c^2.
        \end{equation*}
        By Lemma \ref{lemma:hypergeometric}, 
        \begin{align}
            \chi^2(P_{\widetilde{\pi}, \gamma}||P_{0, \gamma}) &\leq E\left[ \exp\left(\frac{c^2\kappa p}{(1-\gamma) s^2} |T \cap \widetilde{T}| - \frac{c^2 \kappa (p-s)^2}{(1-\gamma)s^2} + c^2 \right) \right] - 1 \\
            &= E\left[ \exp\left( \frac{c^2\kappa p}{(1-\gamma)s^2} |T \cap \widetilde{T}| \right)\right] \exp\left( - \frac{c^2\kappa (p-s)^2}{(1-\gamma)s^2} + c^2\right) - 1 \\
            &\leq \left(1 - \frac{p-s}{p} + \frac{p-s}{p} \exp\left(\frac{c^2\kappa p}{s^2(1-\gamma)}\right)\right)^{p-s} \exp\left( - \frac{c^2\kappa (p-s)^2}{(1-\gamma)s^2} + c^2\right) - 1. \label{eqn:problemII_chisquare_div_bound}
        \end{align}
        Consider 
        \begin{equation}\label{eqn:mgf_bound}
            \left(1 - \frac{p-s}{p} + \frac{p-s}{p}\exp\left(\frac{c^2\kappa p}{s^2(1-\gamma)}\right)\right)^{p-s} \leq \exp\left(\frac{(p-s)^2}{p} \left[ \exp\left(\frac{c^2\kappa p}{s^2(1-\gamma)}\right) - 1 \right]\right).
        \end{equation}
       Since \(\sqrt{p} \leq p-s \leq \frac{p}{2}\) and \(\kappa \leq \frac{(1-\gamma)p^{3/2}}{p-s}\), it follows that \(\frac{c^2\kappa p}{s^2 (1-\gamma)} \leq \frac{c^2p^{5/2}}{(p-s)s^2} \leq \frac{4c^2\sqrt{p}}{p-s} \leq 4c^2 \leq 1\). Here, we have used that \(c^2 < c_\eta^2 \leq \frac{1}{4}\). Therefore, we have 
       \begin{equation*}
        \exp\left(\frac{c^2\kappa p}{s^2(1-\gamma)}\right) - 1 \leq \frac{c^2 \kappa p}{s^2(1-\gamma)} + \frac{ec^4\kappa^2 p^2}{2s^4(1-\gamma)^2}
       \end{equation*}
       due to the inequality \(e^t - 1 \leq t + \frac{et^2}{2}\) for all \(t \in (0, 1)\). Hence, 
       \begin{align*}
        \exp\left(\frac{(p-s)^2}{p} \left[ \exp\left(\frac{c^2\kappa p}{s^2(1-\gamma)}\right) - 1 \right]\right) &\leq \exp\left( \frac{c^2 \kappa (p-s)^2}{s^2(1-\gamma)} + \frac{ec^4\kappa^2 p (p-s)^2}{2s^4(1-\gamma)^2} \right).
       \end{align*}
       Plugging the above bound into (\ref{eqn:mgf_bound}) followed by (\ref{eqn:problemII_chisquare_div_bound}) yields 
       \begin{align*}
        \chi^2(P_{\widetilde{\pi}, \gamma} || P_{0, \gamma}) &\leq \exp\left( \frac{c^2 \kappa (p-s)^2}{s^2(1-\gamma)} + \frac{ec^4\kappa^2 p (p-s)^2}{2s^4(1-\gamma)^2} \right) \exp\left(-\frac{c^2 \kappa(p-s)^2}{(1-\gamma)s^2} + c^2 \right) - 1 \\
        &= \exp\left(\frac{ec^4\kappa^2 p(p-s)^2}{2s^4(1-\gamma)^2} + c^2\right) - 1 \\
        &\leq \exp\left( \frac{ec^4p^4}{2s^4} + c^2 \right) - 1 \\
        &\leq \exp\left(8ec^4 + c^2\right) - 1
       \end{align*}
       where the penultimate inequality follows from \(\kappa^2 \leq \frac{(1-\gamma)^2 p^3}{(p-s)^2}\) and the final inequality follows from \(s^4 \geq \frac{p^4}{16}\). Since \(c < c_\eta \leq \frac{1}{4} \wedge \frac{\log\left(1 + \eta^2\right)}{8e+1}\), it follows that \(e^{8ec^4 + c^2} - 1 \leq e^{c^2(8e+1)} - 1 \leq \eta^2\). In other words, we have shown \(\chi^2(P_{\widetilde{\pi}, \gamma}||P_{0, \gamma}) \leq \eta^2\), and so 
       \begin{equation}\label{eqn:dtv_bound_pi_tilde}
        1 - d_{TV}(P_{0, \gamma}, P_{\widetilde{\pi}, \gamma}) \geq 1 - \frac{1}{2}\sqrt{\chi^2(P_{\widetilde{\pi}, \gamma} || P_{0, \gamma})} \geq 1 - \frac{\eta}{2}.
       \end{equation}
    
       Plugging in the bounds (\ref{eqn:dtv_nu_star_bound}) and (\ref{eqn:dtv_bound_pi_tilde}) into (\ref{eqn:dtv_triangle_inequality}) yields \(\mathcal{R}_{\mathcal{II}}(c\sqrt{\kappa}) \geq 1 - \eta\). Since \(0 < c < c_\eta\) was arbitrary and \(\eta \in (0, 1)\) was arbitrary, the desired result has thus been proved.

        \textbf{Case 2:} Suppose \(p-\sqrt{p} < s < p\). Then \(\kappa = (1-\gamma)p \log\left(1 + \frac{p}{(p-s)^2}\right) \wedge (1-\gamma+\gamma p)\). We have 
        \begin{align*}
            \left\langle \theta, \left[\frac{1}{1-\gamma} I_p - \frac{\gamma}{(1-\gamma)^2 + \gamma(1-\gamma)p} \mathbf{1}_p\mathbf{1}_p^\intercal\right] \widetilde{\theta}\right\rangle &= \frac{c^2 \kappa p}{s^2} \left(\frac{1}{1-\gamma} |T \cap \widetilde{T}| - \frac{\gamma (p-s)^2}{(1-\gamma)^2 + \gamma(1-\gamma) p} \right) \\
            &\leq \frac{c^2 \kappa p}{s^2 (1-\gamma)} |T \cap \widetilde{T}| \\
            &\leq 4c^2 \log\left(1 + \frac{p}{(p-s)^2}\right) |T \cap \widetilde{T}| 
        \end{align*}
        where we have used the previously established fact \(s \geq \frac{p}{2}\) as well as the inequality \(\kappa \leq (1-\gamma)p\log\left(1 + \frac{p}{(p-s)^2}\right)\). Applying Lemma \ref{lemma:hypergeometric} yields 
        \begin{align*}
            \chi^2(P_{\widetilde{\pi}, \gamma} || P_{0, \gamma}) &\leq E\left[ \exp\left(4c^2 \log\left(1 + \frac{p}{(p-s)^2}\right) |T \cap \widetilde{T}| \right) \right] - 1 \\
            &\leq \left(1 - \frac{p-s}{p} + \frac{p-s}{p} \exp\left(4c^2 \log\left(1 + \frac{p}{(p-s)^2}\right) \right)\right)^{p-s} - 1 \\
            &\leq \left(1 - \frac{p-s}{p} + \frac{p-s}{p}\left(1 + \frac{p}{(p-s)^2}\right)^{4c^2}\right)^{p-s} - 1 \\
            &\leq \left(1 + \frac{4c^2}{p-s}\right)^{p-s} - 1\\
            &\leq e^{4c^2} - 1
        \end{align*}
        where the penultimate inequality follows from the fact that \(4c^2 < 4c_\eta^2 \leq 1\) along with the inequality \((1+t)^{\delta} - 1 \leq \delta t\) for all \(0 < \delta < 1\) and \(t > 0\). Therefore, it follows that \(\chi^2(P_{\widetilde{\pi}, \gamma} || P_{0, \gamma}) \leq e^{4c^2} - 1 \leq e^{4c_\eta^2} - 1 \leq \eta^2\). Therefore, we have \(1 - d_{TV}(P_{\widetilde{\pi}, \gamma}, P_{0, \gamma}) \geq 1 - \frac{1}{2}\sqrt{\chi^2(P_{\widetilde{\pi}, \gamma} || P_{0, \gamma})} \geq 1 - \frac{\eta}{2}\). Substituting this bound and (\ref{eqn:dtv_nu_star_bound}) into (\ref{eqn:dtv_triangle_inequality}) yields \(\mathcal{R}_{\mathcal{II}}(c\sqrt{\kappa}) \geq 1 - \eta\). Since \(0 < c < c_\eta\) was arbitrary and \(\eta \in (0, 1)\) was arbitrary, the desired result has been proved.
    \end{proof}

    \begin{lemma}\label{lemma:problemII_p_lbound}
        Suppose \(s = p\) and \(\gamma \in [0, 1)\). If \(\eta \in (0, 1)\), then there exists a constant \(c_\eta > 0\) depending only on \(\eta\) such that for all \(0 < c < c_\eta\) we have \(\mathcal{R}_{\mathcal{II}}\left(c \sqrt{1-\gamma+\gamma p} \right) \geq 1-\eta\).
    \end{lemma}
    \begin{proof}
        Fix \(\eta \in (0, 1)\) and set \(c_{\eta} := \sqrt{\log(1 + \eta^2)}\). Let \(0 < c < c_\eta\). For ease of notation, set \(\kappa := 1-\gamma+\gamma p\). Let \(\pi\) be the prior on \(\Theta_{\mathcal{II}}(p, p, c\sqrt{\kappa})\) which is a point mass at \(\frac{c\sqrt{\kappa}}{\sqrt{p}} \mathbf{1}_p\). Consider that 
        \begin{align*}
            \mathcal{R}_{\mathcal{II}}(c\sqrt{\kappa}) \geq \inf_{\varphi} \left\{ P_{0, \gamma}\{\varphi = 1\} + P_{\frac{c\sqrt{\kappa}}{\sqrt{p}}\mathbf{1}_p, \gamma}\{\varphi = 0\} \right\} = 1 - d_{TV}\left(P_{0, \gamma}, P_{\frac{c\sqrt{\kappa}}{\sqrt{p}}\mathbf{1}_p, \gamma}\right)
        \end{align*}
        where the final equality is due to the Neyman-Pearson lemma. Lemma \ref{lemma:dtv_null_bound} implies 
        \begin{align*}
            1 - d_{TV}\left(P_{0, \gamma}, P_{\frac{c\sqrt{\kappa}}{\sqrt{p}}\mathbf{1}_p, \gamma}\right) &\geq 1 - \frac{1}{2}\sqrt{\exp\left( \frac{\frac{pc^2\kappa}{p}}{1-\gamma+\gamma p} \right)-1} \geq 1 - \frac{1}{2}\sqrt{\exp\left(c^2\right)- 1} \geq 1 - \eta.
        \end{align*}
        Since \(0 < c < c_\eta\) was arbitrary and \(\eta \in (0, 1)\) was arbitrary, we have proved the desired result.
    \end{proof}
    
    With all three sparsity regimes examined, we are now able to combine the results to prove Proposition \ref{prop:problemII_lowerbound}.

    \begin{proof}[Proof of Proposition \ref{prop:problemII_lowerbound}]
        Fix \(\eta \in (0, 1)\) and let \(c_{\eta} := c_{\eta,1} \wedge c_{\eta, 2} \) where \(c_{\eta, 1}, c_{\eta, 2}\) are the constants given in Lemmas \ref{lemma:problemII_lowerbound} and \ref{lemma:problemII_p_lbound} respectively. Let \(0 < c < c_\eta\). We now consider each case in turn. 

        \textbf{Case 1:} Suppose \(\frac{p}{2} \leq s \leq p-\sqrt{p}\). Then \(\psi_2^2 = \frac{(1-\gamma)p^{3/2}}{p-s} \wedge (1-\gamma+\gamma p)\). Since \(c < c_\eta \leq c_{\eta, 1}\), it follows by Lemma \ref{lemma:problemII_lowerbound} that \(\mathcal{R}_{\mathcal{II}}(c\psi_2) \geq 1-\eta\). 

        \textbf{Case 2:} Suppose \(p-\sqrt{p} < s < p\). Then \(\psi_2^2 = (1-\gamma)p\log\left(1 + \frac{p}{(p-s)^2}\right) \wedge (1-\gamma+\gamma p)\). Since \(c < c_{\eta} \leq c_{\eta, 1}\), it follows by Lemma \ref{lemma:problemII_lowerbound} that \(\mathcal{R}_{\mathcal{II}}(c\psi_2) \geq 1-\eta\). 

        \textbf{Case 3:} Suppose \(s = p\). Then \(\psi^2_2 = 1 - \gamma + \gamma p\). Since \(c < c_{\eta} \leq c_{\eta, 2}\), it follows by Lemma \ref{lemma:problemII_p_lbound} that \(\mathcal{R}_{\mathcal{II}}(c\psi_2) \geq 1-\eta\).

        All of the cases have been dealt with. Since \(0 < c < c_\eta\) was arbitrary and \(\eta \in (0, 1)\) was arbitrary, we have proved the desired result.
    \end{proof}

    \subsection{Proofs of results in Section \ref{section:synthesis}}
    \begin{proof}[Proof of Proposition \ref{prop:full_test}]
        Fix \(\eta \in (0, 1)\) and let \(C_{\eta} := \sqrt{2}C_{\eta/2, 1} \vee \sqrt{2}C_{\eta/2, 2}\) where \(C_{\eta/2, 1}, C_{\eta/2,2}\) are the constants corresponding to testing risk level \(\frac{\eta}{2}\) from the proofs of Propositions \ref{prop:problemI_upperbound} and \ref{prop:problemII_upperbound} respectively. Let \(C > C_\eta\). Note by the Pythagorean identity \(||\theta||^2 = ||\theta - \bar{\theta}\mathbf{1}_p||^2 + ||\bar{\theta}\mathbf{1}_p||^2\) we have the inclusion \(\Theta(p, s, C\psi) \subset \Theta_{\mathcal{I}}\left(p, s, \frac{C}{\sqrt{2}}\psi\right) \cup \Theta_{\mathcal{II}}\left(p, s, \frac{C}{\sqrt{2}}\psi\right)\). Note also that \(\psi^2 \geq \psi_1^2 \vee \psi_2^2\) where \(\psi_1^2\) and \(\psi_2^2\) are given by (\ref{rate:problemI}) and (\ref{rate:problemII}) respectively. Consequently, we have 
        \begin{align*}
            &P_{0, \gamma}\{\varphi^* = 1\} + \sup_{\theta \in \Theta(p, s, C\psi)} P_{\theta, \gamma}\{\varphi^* = 0\} \\
            &\leq P_{0, \gamma}\{\varphi_{I}^* = 1\} + P_{0, \gamma}\{\varphi_{II}^* = 1\} + \sup_{\theta \in \Theta_{\mathcal{I}}(p, s, \frac{C}{\sqrt{2}}\psi)} P_{\theta, \gamma} \{\varphi^* = 0\} + \sup_{\theta \in \Theta_{\mathcal{II}}(p, s, \frac{C}{\sqrt{2}}\psi)} P_{\theta, \gamma}\{\varphi^* = 0\} \\
            &\leq \left(P_{0, \gamma}\{\varphi_{I}^* = 1\} +  \sup_{\theta \in \Theta_{\mathcal{I}}(p, s, \frac{C}{\sqrt{2}}\psi)} P_{\theta, \gamma} \{\varphi_{I}^* = 0\} \right) + \left(P_{0, \gamma}\{\varphi_{II}^* = 1\} +  \sup_{\theta \in \Theta_{\mathcal{II}}(p, s, \frac{C}{\sqrt{2}}\psi)}P_{\theta, \gamma}\{\varphi_{II}^* = 0\}\right) \\
            &\leq \frac{\eta}{2} + \frac{\eta}{2} \\
            &= \eta
        \end{align*}
        where the penultimate inequality follows from \(\frac{C}{\sqrt{2}} > C_{\eta/2, 1} \vee C_{\eta/2, 2}\). Since \(C > C_\eta\) was arbitrary and \(\eta \in (0, 1)\) was arbitrary, the desired result has been proved.
    \end{proof}
    
    \subsection{Proofs of lower bounds in Section \ref{section:multiple_random_effects}}
    In this section, we present the proofs for the lower bounds presented in Section \ref{section:multiple_random_effects}. The strategy for proving the lower bound is the same as the one we employed in our previous arguments. Specifically, Lemma \ref{lemma:general_lower_bound} along with the Ingster-Suslina method (Lemma \ref{lemma:Ingster_Suslina}) will be used. Recall the notation which was set up in Section \ref{section:R_prelim}. Introducing an additional piece of notation, for a probability distribution \(\pi\) on \(\R^p\) we denote \(P_{\pi, \gamma, R} := \int P_{\theta, \gamma, R} \, \pi(d\theta)\) as the Gaussian mixture induced by \(\pi\). The following lemma, which we state without proof, is needed for our main arguments.
    \begin{lemma}\label{lemma:block_precision}
        If \(\gamma \in [0, 1)\), then 
        \begin{equation*}
            \left( (1-\gamma)I_p + \gamma \sum_{k=1}^{R} \mathbf{1}_{B_k}\mathbf{1}_{B_k}^\intercal \right)^{-1} = \sum_{k=1}^{R} \left(\frac{1}{1-\gamma}I_{B_k} - \frac{\gamma}{(1-\gamma)\left(1-\gamma+\gamma \frac{p}{R}\right)} \mathbf{1}_{B_k}\mathbf{1}_{B_k}^\intercal \right)
        \end{equation*}
        where \(I_{B_k}\) is the diagonal matrix with \(i\)th diagonal entry equal to one if \(i \in B_k\) and equal to zero otherwise.
    \end{lemma}

    \subsubsection{\texorpdfstring{Regime \(1 \leq s \leq \frac{p}{4R}\)}{Regime 1 <= s <= p/4R}}
    We are now ready to prove our first lower bound. Notice that it immediately gives the desired lower bound presented in Theorem \ref{thm:R_sparse} where the sparsity regime \(s < \frac{p}{4R}\) is in force. 

    \begin{proof}[Proof of Proposition \ref{prop:lbound_R_sparse}]
        Fix \(\eta \in (0, 1)\). Set \(c_\eta := 1 \wedge \sqrt{\log\left(1 + 4\eta^2\right)} \wedge \sqrt{\log(1 + \log(1 + 4\eta^2))}\). Let \(0 < c < c_\eta\). We deal with the two sparsity regimes separately. 

        \textbf{Case 1:} Suppose \(1 \leq s < \sqrt{p}\). Then \(\psi_1^2 = (1-\gamma)s \log\left(1 + \frac{p}{s^2}\right)\). Let \(\pi\) be the prior on \(\Theta(p, s, c\psi_1)\) in which a draw \(\mu \sim \pi\) is obtained by drawing a subset \(S \subset [p]\) of size \(s\) uniformly at random from the set of all subset of \([p]\) of size \(s\) and setting 
        \begin{equation*}
            \mu_i := 
            \begin{cases}
                \frac{c\psi_1}{\sqrt{s}} &\text{if } i \in S, \\
                0 &\text{if } i \in S^c. 
            \end{cases}
        \end{equation*}
        Note that \(||\mu||_0 = s\) and \(||\mu||^2 = c^2\psi_1^2\) almost surely, and so \(\pi\) is indeed supported on \(\Theta(p, s, c\psi_1)\). Let \(P_{\pi, \gamma, R} = \int P_{\theta, \gamma, R} \, \pi(d\theta)\) denote the Gaussian mixture induced by \(\pi\). By Lemma \ref{lemma:Ingster_Suslina} we have 
        \begin{equation*}
            \chi^2(P_{\pi, \gamma, R} || P_{0, \gamma, R}) = E\left[ \exp\left( \left\langle \theta, \left[(1-\gamma) I_p + \gamma \sum_{k=1}^{R} \mathbf{1}_{B_k}\mathbf{1}_{B_k}^\intercal \right]^{-1} \widetilde{\theta} \right\rangle \right) \right] - 1
        \end{equation*}
        where \(\theta, \widetilde{\theta} \overset{iid}{\sim} \pi\). Write \(\theta = \frac{c\psi_1}{\sqrt{s}}\mathbf{1}_S\) and \(\widetilde{\theta} = \frac{c\psi_1}{\sqrt{s}}\mathbf{1}_{\widetilde{S}}\) where \(S, \widetilde{S}\) are iid uniformly drawn subsets of \([p]\) of size \(s\). Then, we have by Lemma \ref{lemma:block_precision} 
        \begin{align*}
            \left\langle \theta, \left[(1-\gamma) I_p + \gamma \sum_{k=1}^{R} \mathbf{1}_{B_k}\mathbf{1}_{B_k}^\intercal \right]^{-1} \widetilde{\theta} \right\rangle &= \sum_{k=1}^{R} \left\langle \theta, \left(\frac{1}{1-\gamma} I_{B_k} - \frac{\gamma}{(1-\gamma)(1-\gamma + \gamma \frac{p}{R})} \mathbf{1}_{B_k}\mathbf{1}_{B_k}^\intercal \right) \widetilde{\theta} \right\rangle  \\
            &= \sum_{k=1}^{R} \left(\frac{c^2\psi_1^2}{s(1-\gamma)} |S \cap \widetilde{S} \cap B_k| - \frac{c^2 \psi_1^2 \gamma}{s(1-\gamma)\left(1 -\gamma + \gamma \frac{p}{R}\right)} |B_k \cap S| \cdot |B_k \cap \widetilde{S}|\right) \\
            &\leq \sum_{k=1}^{R} \frac{c^2\psi_1^2}{s(1-\gamma)} |S \cap \widetilde{S} \cap B_k| \\
            &\leq \frac{c^2\psi_1^2}{s(1-\gamma)} |S \cap \widetilde{S}|.
        \end{align*}
        Consider that \(|S \cap \widetilde{S}|\) is distributed according to the hypergeometric distribution with probability mass function given in Lemma \ref{lemma:hypergeometric}. A calculation similar to the Case 1 analysis (starting just before (\ref{eqn:chi_square_divergence})) in the proof of Proposition \ref{prop:problemI_lowerbound} shows that \(\chi^2(P_{\pi, \gamma, R} || P_{0, \gamma, R}) \leq e^{c^2} - 1 \leq e^{c_\eta^2} - 1 \leq 4\eta^2\). An application of Lemma \ref{lemma:general_lower_bound} gives \(\mathcal{R}(c\psi_1) \geq 1 - \frac{1}{2}\sqrt{\chi^2(P_{\pi, \gamma, R} || P_{0, \gamma, R})} \geq 1-\eta\). Since \(0 < c < c_\eta\) was arbitrary and \(\eta \in (0, 1)\) was arbitrary, the desired result has been proved in the case \(1 \leq s < \sqrt{p}\). 

        \textbf{Case 2:} Suppose \(s \geq \sqrt{p}\). Then \(\psi_1^2 = (1-\gamma)\sqrt{p}\). Without loss of generality, assume \(\sqrt{p}\) is an integer. Repeating exactly the argument presented in Case 1 but now with every instance of \(s\) replaced with \(\sqrt{p}\) yields \(\chi^2(P_{\pi, \gamma, R} || P_{0, \gamma, R}) \leq 4\eta^2\). This bound is obtained by following a calculation similar to the one in the Case 2 analysis of Proposition \ref{prop:problemI_lowerbound}. Lemma \ref{lemma:general_lower_bound} then implies \(\mathcal{R}(c\psi_1) \geq 1 - \frac{1}{2}\sqrt{\chi^2(P_{\pi, \gamma, R} || P_{0, \gamma, R})} \geq 1-\eta\). Since \(0 < c < c_\eta\) was arbitrary and \(\eta \in (0, 1)\) was arbitrary, we have proved the desired result for the case \(s \geq \sqrt{p}\). 
    \end{proof}

    \subsubsection{\texorpdfstring{Regime \(\frac{p}{4R} < s < \frac{p}{R}\)}{Regime p/4R < s < p/R}}
    A preliminary lemma is needed before we prove the lower bound of interest. 
    \begin{lemma}\label{lemma:NL_bound}
        If \(\eta \in (0, 1)\), then there exists a constant \(L_\eta > 0\) depending only on \(\eta\) such that for all \(0 < L < L_\eta\) we have 
        \begin{equation*}
            -\frac{1}{N} + \exp\left((L - 1)\log(N) + L\right) \leq 4\eta^2
        \end{equation*}
        for all positive integers \(N\). 
    \end{lemma}
    \begin{proof}
        Fix \(\eta \in (0, 1)\) and let \(L_\eta := \frac{\log(8\eta^2 + 1)}{2e} \wedge \log(1+4\eta^2) \wedge \frac{1}{2} \). Let \(0 < L < L_\eta\). In the case \(N = 1\), observe that \(-1 + e^L \leq -1 + e^{\log(1+4\eta^2)} = 4\eta^2\) as desired. Suppose \(N \geq 2\). Define the function \(f : [2, \infty) \to \R\) with \(f(x) = -\frac{1}{x} + \exp((L-1)\log(x) + L)\). Clearly \(f\) is differentiable on \((2, \infty)\) with \(f'(x) = x^{-2}\left(1 + (L-1)(ex)^L\right)\). Consider that \(f'(x) < 0\) if and only if \((L-1)(ex)^L < -1\). Since \(L < \frac{1}{2}\) and \(x \geq 2\) we have \((L-1)(ex)^L \leq (L-1)(2e)^L < -1\). Thus, \(f'(x) < 0\). Consequently, \(-\frac{1}{N} + \exp\left((L - 1)\log(N) + L\right) \leq -\frac{1}{2} + \exp\left((L - 1)\log(2) + L\right) = \frac{1}{2}\left((2e)^L - 1\right) \leq 4\eta^2\) since \(L < \frac{\log\left(8\eta^2+1\right)}{2e}\). The proof is complete.
    \end{proof}

    \begin{lemma}\label{lemma:lbound_pR_dense}
        Suppose \(\frac{p}{4R} < s < \frac{p}{R}\) and \(\gamma \in [0, 1)\). Set 
        \begin{equation*}
            \rho^2 := 
            \begin{cases}
                \frac{(1-\gamma)p}{p-Rs} \left(\sqrt{\frac{p}{R}\log(eR)} + \log(R) \right) \wedge \left(1-\gamma+\gamma \frac{p}{R}\right)\log(eR) &\text{if } \frac{p}{4R} < s \leq \frac{p}{R} - \sqrt{\frac{p}{R} \log(eR)}, \\
                \frac{(1-\gamma)p}{p-Rs} \left(\left(\frac{p}{R}-s\right)\log\left(1 + \frac{Rp\log(eR)}{\left(p - Rs\right)^2}\right)+ \log(R) \right) \wedge \left(1-\gamma+\gamma \frac{p}{R}\right)\log(eR) &\text{if } \frac{p}{R} - \sqrt{\frac{p}{R} \log(eR)} < s < \frac{p}{R}.
            \end{cases}
        \end{equation*}
        If \(\eta \in (0, 1)\), then there exists a constant \(c_\eta > 0\) depending only on \(\eta\) such that for all \(0 < c < c_\eta\) we have \(\mathcal{R}(c\rho) \geq 1-\eta\). 
    \end{lemma}
    \begin{proof}
        Fix \(\eta\in (0, 1)\) and set \(c_\eta = \frac{1}{\sqrt{32}} \wedge c_{\eta, 1} \wedge c_{\eta, 2} \wedge c_{\eta, 3}\) where \(c_{\eta, 1}, c_{\eta, 2}, c_{\eta, 3}\) will be set later. Let \(0 < c < c_\eta\). We will use Lemma \ref{lemma:general_lower_bound} to prove the desired result.
        
        \textbf{Case 1:} Suppose \(\frac{p}{4R} < s \leq \frac{p}{R} - \sqrt{\frac{p}{R}\log(eR)}\) and further suppose \(\sqrt{\frac{p}{R} \log(eR)} \geq \log(R)\). 
        Let \(\pi\) denote the prior supported on \(\Theta(p, s, c\rho)\) defined as follows. To sample \(\mu \sim \pi\), draw \(K \in [R]\) uniformly at random followed by drawing the subset \(S \subset B_K\) of size \(s\) uniformly at random from the collection of all size \(s\) subsets of \(B_K\). Then set 
        \begin{equation*}
            \mu_i = 
            \begin{cases}
                \frac{c\rho \sqrt{\frac{p}{R}}}{s} &\text{if }i \in S, \\
                0 &\text{if } i \in S^c.
            \end{cases}
        \end{equation*}
        Note that \(||\mu||_0 = s\) and \(||\mu||^2 = \frac{c^2\rho^2\left(\frac{p}{R}\right)}{s^2} \cdot s \geq 4c^2 \rho^2\). Thus, \(\pi\) is indeed supported on \(\Theta(p, s, c\rho)\). By Lemma \ref{lemma:Ingster_Suslina} and Lemma \ref{lemma:block_precision} we have 
        \begin{align*}
            \chi^2(P_{\pi, \gamma, R} || P_{0, \gamma, R}) + 1 = E\left[ \exp\left( \left\langle \theta, \left[\sum_{k=1}^{R} \Omega_k \right] \widetilde{\theta}  \right \rangle \right) \right]
        \end{align*}
        where \(\theta, \widetilde{\theta} \overset{iid}{\sim} \pi\) and \(\Omega_k := \frac{1}{1-\gamma}I_{B_k} - \frac{\gamma}{(1-\gamma)\left(1-\gamma+\gamma \frac{p}{R}\right)} \mathbf{1}_{B_k}\mathbf{1}_{B_k}^\intercal\) for \(1 \leq k \leq R\). Let \(K\) denote the group index and \(S\) the size \(s\) subset of \(B_K\) associated to \(\theta\). Likewise, let \(\widetilde{K}\) and \(\widetilde{S}\) denote the corresponding objects associated to \(\widetilde{\theta}\). A direct calculation shows 
        \begin{align*}
            \left\langle \theta, \left[\sum_{k=1}^{R} \Omega_k \right] \widetilde{\theta}  \right \rangle = \left\langle \theta , \Omega_K \widetilde{\theta}\right\rangle \cdot \mathbbm{1}_{\left\{K = \widetilde{K}\right\}}.
        \end{align*}
        Define \(\eta := \frac{c\rho\sqrt{\frac{p}{R}}}{s}\mathbf{1}_{B_K}\) and \(\widetilde{\eta} := \frac{c\rho\sqrt{\frac{p}{R}}}{s}\mathbf{1}_{B_{\widetilde{K}}}\). On the event \(\{K = \widetilde{K}\}\) we have
        \begin{align*}
            \left\langle \theta , \Omega_K \widetilde{\theta}\right\rangle &= \left\langle \theta - \eta, \Omega_K (\widetilde{\theta} -\widetilde{\eta})\right\rangle  + \langle \eta, \Omega_K \widetilde{\theta}\rangle + \langle \widetilde{\eta}, \Omega_K \theta \rangle - \langle \eta, \Omega_K \widetilde{\eta}\rangle\\
            &\leq \left\langle \theta - \eta, \Omega_K (\widetilde{\theta} -\widetilde{\eta})\right\rangle + \frac{2c^2\rho^2 \frac{p}{R}}{s\left(1 - \gamma + \gamma \frac{p}{R}\right)} - \frac{c^2\rho^2 \left(\frac{p}{R}\right)^2}{s^2 \left(1-\gamma+\gamma \frac{p}{R}\right)} \\
            &\leq \left\langle \theta - \eta, \Omega_K (\widetilde{\theta} -\widetilde{\eta})\right\rangle + \frac{7c^2\rho^2}{\left(1 - \gamma + \gamma \frac{p}{R}\right)} \\
            &\leq \left\langle \theta - \eta, \Omega_K (\widetilde{\theta} -\widetilde{\eta})\right\rangle + 7c^2 \log(eR)
        \end{align*}
        where we have used \(\frac{p}{4R} < s < \frac{p}{R}\) and the fact that \(\Omega_K = \frac{1}{1-\gamma}\left(I_{B_K} - \frac{R}{p}\mathbf{1}_{B_K}\mathbf{1}_{B_K}^\intercal\right) + \frac{1}{1-\gamma+\gamma \frac{p}{R}} \cdot \frac{R}{p} \mathbf{1}_{B_K}\mathbf{1}_{B_K}^\intercal\). We have also used \(\rho^2 \leq \left(1 -\gamma + \gamma \frac{p}{R}\right) \log(eR)\). To summarize our calculation so far, we have shown 
        \begin{align*}
            \chi^2(P_{\pi, \gamma, R} || P_{0, \gamma, R}) + 1 &\leq E\left[ \exp\left( \left[\left\langle \theta - \eta, \Omega_K (\widetilde{\theta} -\widetilde{\eta})\right\rangle + 7c^2\log(eR) \right] \cdot \mathbbm{1}_{\left\{ K = \widetilde{K} \right\}} \right) \right] \\
            &= 1 - \frac{1}{R} + \frac{1}{R} \cdot \left( \frac{1}{R}\sum_{k=1}^{R} E\left[ \exp\left( \frac{c^2\rho^2 \frac{p}{R}}{s}\left\langle \mathbf{1}_{T_k}, \Omega_k  \mathbf{1}_{\widetilde{T}_k}\right\rangle + 7c^2\log(eR)\right) \right]\right) \\
            &= 1 - \frac{1}{R} + \frac{e^{7c^2 \log(eR)}}{R^2}\sum_{k=1}^{R}  E\left[ \exp\left( \frac{c^2\rho^2 \frac{p}{R}}{s^2}\left\langle \mathbf{1}_{T_k}, \Omega_k  \mathbf{1}_{\widetilde{T}_k}\right\rangle \right) \right]
        \end{align*}
        where \(T_k, \widetilde{T}_k\) are iid uniformly drawn size \(\frac{p}{R} - s\) subsets of \(B_k\) for \(1 \leq k \leq R\). The expectation is the same no matter the value of \(k\) so we can further reduce 
        \begin{equation}\label{eqn:chi_square_div_bound}
            \chi^2(P_{\pi, \gamma, R} || P_{0, \gamma, R}) \leq - \frac{1}{R} + \frac{e^{7c^2\log(eR)}}{R} E\left[ \exp\left( \frac{c^2\rho^2 \frac{p}{R}}{s^2}\left\langle \mathbf{1}_{T_1}, \Omega_1  \mathbf{1}_{\widetilde{T}_1}\right\rangle \right) \right].
        \end{equation}
        Observe that 
        \begin{align*}
            \frac{c^2\rho^2 \frac{p}{R}}{s^2}\left\langle \mathbf{1}_{T_1}, \Omega_1  \mathbf{1}_{\widetilde{T}_1}\right\rangle &= \frac{c^2 \rho^2 \frac{p}{R}}{s^2 (1-\gamma)} \left(|T_1 \cap \widetilde{T}_1| - \frac{\left(\frac{p}{R} - s\right)^2}{\frac{p}{R}} \right) + \frac{c^2 \rho^2}{1-\gamma+\gamma \frac{p}{R}} \cdot \frac{\left(\frac{p}{R} - s\right)^2}{s^2}.
        \end{align*}
        Note that \(|T_1 \cap \widetilde{T}_1|\) is distributed according to the hypergeometric distribution 
        \begin{equation}\label{eqn:T1_hyp}
            P\left\{ |T_1 \cap \widetilde{T}_1| = k \right\} = \frac{ \binom{\frac{p}{R} - s}{k} \binom{s}{\frac{p}{R} - s - k} }{ \binom{\frac{p}{R}}{\frac{p}{R} - s} }
        \end{equation}
        for \(0 \leq k \leq \frac{p}{R} - s\). Note also that since \(\rho^2 \leq (1-\gamma+\gamma \frac{p}{R})\log(eR)\) and since \(\frac{p}{R} - s \leq 3s\), we have \(\frac{c^2 \rho^2}{1-\gamma+\gamma \frac{p}{R}}\cdot \frac{\left(\frac{p}{R}-s\right)^2}{s^2} \leq 9c^2\log(eR)\). Therefore, 
        \begin{equation*}
            \frac{c^2 \rho^2 \frac{p}{R}}{s^2 (1-\gamma)} \left(|T_1 \cap \widetilde{T}_1| - \frac{\left(\frac{p}{R} - s\right)^2}{\frac{p}{R}} \right) + \frac{c^2 \rho^2}{1-\gamma+\gamma \frac{p}{R}} \cdot \frac{\left(\frac{p}{R} - s\right)^2}{s^2} \leq \frac{c^2 \rho^2 \frac{p}{R}}{s^2 (1-\gamma)} \left(|T_1 \cap \widetilde{T}_1| - \frac{\left(\frac{p}{R} - s\right)^2}{\frac{p}{R}} \right) + 9c^2\log(eR).
        \end{equation*}
        By Lemma \ref{lemma:hypergeometric}, 
        \begin{align*}
            E\left[ \exp\left( \frac{c^2\rho^2 \frac{p}{R}}{s^2}\left\langle \mathbf{1}_{T_1}, \Omega_1  \mathbf{1}_{\widetilde{T}_1}\right\rangle \right) \right] &\leq \left(1 - \frac{\frac{p}{R} - s}{\frac{p}{R}} + \frac{\frac{p}{R} - s}{\frac{p}{R}} \exp\left(\frac{c^2\rho^2 \frac{p}{R}}{(1-\gamma)s^2} \right)\right)^{\frac{p}{R} - s} \exp\left( - \frac{c^2 \rho^2 \left(\frac{p}{R} - s\right)^2}{s^2(1-\gamma)} + 9c^2\log(eR)\right) \\
            &\leq \exp\left(\frac{\left(\frac{p}{R} - s\right)^2}{\frac{p}{R}} \left[\exp\left( \frac{c^2\rho^2 \frac{p}{R}}{s^2(1-\gamma)} \right) - 1\right] \right) \exp\left( - \frac{c^2 \rho^2 \left(\frac{p}{R} - s\right)^2}{s^2(1-\gamma)} + 9c^2\log(eR)\right).
        \end{align*}
        Since \(\sqrt{\frac{p}{R} \log(eR)} \leq \frac{p}{R} - s \leq \frac{3p}{4R}\) and \(\rho^2 \leq 2\frac{(1-\gamma)p}{p-Rs}\sqrt{\frac{p}{R} \log(eR)}\) because \(\sqrt{\frac{p}{R}\log(eR)} \geq \log(R)\), we have 
        \begin{equation*}
            \frac{c^2\rho^2 \frac{p}{R}}{s^2(1-\gamma)} \leq \frac{2c^2\left(\frac{p}{R}\right)^2\sqrt{\frac{p}{R}\log(eR)}}{s^2\left(\frac{p}{R}-s\right)} = 32c^2 < 1
        \end{equation*}
        where we have used \(c^2 < c_\eta^2 \leq \frac{1}{32}\). Applying the inequality \(e^t - 1 \leq t + \frac{et^2}{2}\) for \(t \in (0, 1)\) yields 
        \begin{equation*}
            \exp\left(\frac{\left(\frac{p}{R} - s\right)^2}{\frac{p}{R}} \left[\exp\left( \frac{c^2\rho^2 \frac{p}{R}}{s^2(1-\gamma)} \right) - 1\right] \right) \leq \exp\left(\frac{c^2\rho^2 \left(\frac{p}{R} - s\right)^2}{(1-\gamma)s^2} + \frac{ec^4 \rho^4 \frac{p}{R}\left(\frac{p}{R} - s\right)^2}{2s^4(1-\gamma)^2} \right).
        \end{equation*}
        Therefore, 
        \begin{align*}
            E\left[ \exp\left( \frac{c^2\rho^2 \frac{p}{R}}{s^2}\left\langle \mathbf{1}_{T_1}, \Omega_1  \mathbf{1}_{\widetilde{T}_1}\right\rangle \right) \right] &\leq \exp\left(\frac{c^2\rho^2 \left(\frac{p}{R} - s\right)^2}{(1-\gamma)s^2} + \frac{ec^4 \rho^4 \frac{p}{R}\left(\frac{p}{R} - s\right)^2}{2s^4(1-\gamma)^2} \right) \exp\left( - \frac{c^2 \rho^2 \left(\frac{p}{R} - s\right)^2}{s^2(1-\gamma)} + 9c^2\log(eR)\right) \\
            &= \exp\left(\frac{ec^4 \rho^4 \frac{p}{R}\left(\frac{p}{R} - s\right)^2}{2s^4(1-\gamma)^2} + 9c^2\log(eR)\right) \\
            &= \exp\left(\frac{4ec^4 (1-\gamma)^2 \left(\frac{p}{R}\right)^4 \log(eR) \cdot \left(\frac{p}{R}-s\right)^2}{2s^4 (1-\gamma)^2 \left(\frac{p}{R} - s\right)^2} + 9c^2\log(eR) \right) \\
            &\leq \exp\left(\left(512ec^4 + 9c^2\right) \log(eR)\right). 
        \end{align*}
        Plugging this bound into (\ref{eqn:chi_square_div_bound}) yields 
        \begin{align*}
            \chi^2(P_{\pi, \gamma, R} || P_{0, \gamma, R}) &\leq -\frac{1}{R} + \frac{1}{R}\exp\left(\left(512ec^4 + 16c^2\right) \log(eR)\right) \\
            &\leq -\frac{1}{R} + \exp\left(\left(512ec^4 + 16c^2 - 1\right) \log(R) + 512ec^4 + 16c^2\right). 
        \end{align*}
        Let \(L_{\eta}\) be the constant from Lemma \ref{lemma:NL_bound} at level \(\eta\) and let \(c_{\eta, 1}\) denote the largest constant such that \(0 < d < c_{\eta, 1}\) implies \(512ed^4 + 16d^2 \leq L_\eta\). Since \(c < c_\eta \leq c_{\eta, 1}\), we have that 
        \begin{equation*}
            \chi^2(P_{\pi, \gamma, R} || P_{0, \gamma, R}) \leq -\frac{1}{R} + \exp\left(\left(512ec^4 + 16c^2 - 1\right) \log(R) + 512ec^4 + 16c^2\right) \leq 4\eta^2. 
        \end{equation*}
        By Lemma \ref{lemma:general_lower_bound} it follows that 
        \begin{equation*}
            \mathcal{R}(c\rho) \geq 1 - \frac{1}{2}\sqrt{\chi^2(P_{\pi, \gamma, R} || P_{0, \gamma, R})} \geq 1-\eta.
        \end{equation*}
        Since \(0 < c < c_\eta\) was arbitrary and \(\eta \in (0, 1)\) was arbitrary, we have proved the desired result.

        \textbf{Case 2:} Suppose \(\frac{p}{4R} < s \leq \frac{p}{R} - \sqrt{\frac{p}{R}\log(eR)}\) and further suppose \(\sqrt{\frac{p}{R} \log(eR)} < \log(R)\). Consider the prior distribution \(\pi\) on \(\Theta(p, s, c\rho)\) defined as follows. A draw \(\mu \sim \pi\) is constructed by drawing \(K \sim \text{Uniform}([R])\) and setting the first \(s\) coordinates in \(B_K\) of \(\mu\) equal to \(\frac{c\rho}{\sqrt{s}}\). Note that \(||\mu||_0 = s\) and \(||\mu||^2 = c^2\rho^2\), so \(\pi\) is indeed supported on \(\Theta(p, s, c\rho)\). The calculation in Case 1 can be repeated (with the modification \(\eta := \frac{c \rho}{\sqrt{s}} \mathbf{1}_{B_K}\) and \(\widetilde{\eta} := \frac{c\rho}{\sqrt{s}}\mathbf{1}_{B_{\widetilde{K}}}\)) to obtain 
        \begin{align*}
            \chi^2\left(P_{\pi, \gamma, R} || P_{0, \gamma, R} \right) &\leq - \frac{1}{R} + \frac{e^{c^2\log(eR)}}{R} \exp\left(\frac{c^2\rho^2}{s(1-\gamma)} \left( s - \frac{R}{p} s^2 \right)  + \frac{c^2\rho^2}{s(1-\gamma+\gamma \frac{p}{R})} \frac{R}{p} s^2\right)
        \end{align*}
        Since \(\sqrt{\frac{p}{R}\log(eR)} < \log(R)\), we have
        \begin{align*}
            \frac{c^2\rho^2}{s(1-\gamma)} \left( s - \frac{R}{p} s^2 \right) &\leq \frac{2c^2\log(R)p}{p-Rs} \left(1 - \frac{R}{p}s \right) = 2c^2\log(R).
        \end{align*}
        Using also the fact that \(\rho^2 \leq \left(1-\gamma+\gamma \frac{p}{R}\right)\log(eR)\) and \(s < \frac{p}{R}\), we have 
        \begin{align*}
            \chi^2\left(P_{\pi, \gamma, R} || P_{0, \gamma, R} \right) &\leq -\frac{1}{R} + \frac{e^{c^2\log(eR)}}{R} \exp\left(2c^2\log(R) + c^2\log(eR)\right) \\
            &= -\frac{1}{R} + \exp\left(\left(4c^2 - 1\right) \log(R) + 2c^2\right) \\
            &\leq -\frac{1}{R} + \exp\left(\left(4c^2 - 1\right) \log(R) + 4c^2\right)
        \end{align*}
        Let \(L_{\eta}\) be the constant from Lemma \ref{lemma:NL_bound} at level \(\eta\) and let \(c_{\eta, 2}\) denote the largest constant such that \(0 < d < c_{\eta, 2}\) implies \(4c^2 \leq L_\eta\). Since \(c < c_\eta \leq c_{\eta, 2}\), we have that 
        \begin{equation*}
            \chi^2(P_{\pi, \gamma, R} || P_{0, \gamma, R}) \leq -\frac{1}{R} + \exp\left(\left(4c^2 - 1\right) \log(R) + 4c^2\right) \leq 4\eta^2. 
        \end{equation*}
        By Lemma \ref{lemma:general_lower_bound} it follows that 
        \begin{equation*}
            \mathcal{R}(c\rho) \geq 1 - \frac{1}{2}\sqrt{\chi^2(P_{\pi, \gamma, R} || P_{0, \gamma, R})} \geq 1-\eta.
        \end{equation*}
        Since \(0 < c < c_\eta\) was arbitrary and \(\eta \in (0, 1)\) was arbitrary, we have proved the desired result.

        \textbf{Case 3:} Suppose \(\frac{p}{R} - \sqrt{\frac{p}{R} \log(eR)} < s < \frac{p}{R}\) and \((\frac{p}{R} - s) \log\left(1 + \frac{Rp\log(eR)}{\left(p - Rs\right)^2}\right) \geq \log(R)\). Let \(\pi\) denote exactly the same prior as in Case 1. Note the supposition that \(s > \frac{p}{4R}\) in the statement of the lemma. The exact same analysis can be repeated to yield (\ref{eqn:chi_square_div_bound}). Observe that 
        \begin{align*}
            \frac{c^2\rho^2 \frac{p}{R}}{s^2}\left\langle \mathbf{1}_{T_1}, \Omega_1  \mathbf{1}_{\widetilde{T}_1}\right\rangle &= \frac{c^2 \rho^2 \frac{p}{R}}{s^2 (1-\gamma)} \left(|T_1 \cap \widetilde{T}_1| - \frac{\left(\frac{p}{R} - s\right)^2}{\frac{p}{R}} \right) + \frac{c^2 \rho^2}{1-\gamma+\gamma \frac{p}{R}} \cdot \frac{\left(\frac{p}{R} - s\right)^2}{s^2}.
        \end{align*}
        Since \(\rho^2 \leq (1-\gamma+\gamma \frac{p}{R})\log(eR)\) and \(\frac{p}{R} - s \leq 3s\) we have \(\frac{c^2\rho^2}{1-\gamma+\gamma \frac{p}{R}} \cdot \frac{\left(\frac{p}{R} - s\right)^2}{s^2} \leq 9c^2\log(eR)\). Additionally since \(\left(\frac{p}{R} - s\right)\log\left(1 + \frac{Rp \log(eR)}{(p-Rs)^2}\right) \geq \log(R)\) we have \(\frac{c^2 \rho^2 \frac{p}{R}}{s^2 (1-\gamma)} \leq 32c^2\log\left(1 + \frac{Rp\log(eR)}{(p-Rs)^2}\right)\). Therefore
        \begin{equation*}
            \frac{c^2\rho^2 \frac{p}{R}}{s^2}\left\langle \mathbf{1}_{T_1}, \Omega_1  \mathbf{1}_{\widetilde{T}_1}\right\rangle \leq 32c^2\log\left(1 +\frac{Rp \log(eR)}{(p-Rs)^2}\right)|T_1 \cap \widetilde{T}_1| + 9c^2\log(eR).
        \end{equation*}
        Using (\ref{eqn:chi_square_div_bound}), we have 
        \begin{align*}
            \chi^2(P_{\pi, \gamma, R} || P_{0, \gamma, R}) &\leq -\frac{1}{R} + \frac{e^{7c^2\log(eR)}}{R} E\left[ \exp\left(32c^2\log\left(1 +\frac{Rp\log(eR)}{(p-Rs)^2}\right)|T_1 \cap \widetilde{T}_1| + 9c^2\log(eR)\right) \right] \\
            &= -\frac{1}{R} + \exp\left(16c^2 + \left(16c^2 - 1\right) \log(R) \right) E\left[ \exp\left(32c^2\log\left(1 + \frac{Rp \log(eR)}{(p-Rs)^2}\right) |T_1 \cap \widetilde{T}_1| \right) \right].
        \end{align*}
        Note that \(|T_1 \cap \widetilde{T}_1|\) is distributed according to a hypergeometric distribution with probability mass function given by (\ref{eqn:T1_hyp}). By Lemma \ref{lemma:hypergeometric}, we have 
        \begin{align*}
            E\left[ \exp\left(32c^2\log\left(1 + \frac{Rp\log(eR)}{(p-Rs)^2}\right) |T_1 \cap \widetilde{T}_1| \right) \right] &\leq \left(1 - \frac{\frac{p}{R} - s}{\frac{p}{R}} + \frac{\frac{p}{R} - s}{\frac{p}{R}} \exp\left( 32c^2 \log\left(1 + \frac{Rp \log(eR)}{(p-Rs)^2}\right) \right) \right)^{\frac{p}{R} - s} \\
            &= \left(1 - \frac{\frac{p}{R} - s}{\frac{p}{R}} + \frac{\frac{p}{R} - s}{\frac{p}{R}} \left(1 + \frac{\frac{p}{R} \log(eR)}{\left(\frac{p}{R} - s\right)^2}\right)^{32c^2} \right)^{\frac{p}{R} - s} \\
            &\leq \left(1 + \frac{32c^2 \log(eR)}{\frac{p}{R} - s}\right)^{\frac{p}{R} - s} \\
            &\leq e^{32c^2 \log(eR)}
        \end{align*}
        where the penultimate inequality follows from the fact that \(32c^2 < 32c_\eta^2 \leq 1\) along with the inequality \((1+t)^\delta - 1 \leq \delta t\) for all \(0 < \delta < 1\) and \(t > 0\). Therefore we have 
        \begin{equation*}
            \chi^2(P_{\pi, \gamma, R} || P_{0, \gamma, R}) \leq -\frac{1}{R} + \exp\left(48c^2 + \left(48c^2 - 1\right)\log(R) \right). 
        \end{equation*}
        Let \(L_{\eta}\) be the constant from Lemma \ref{lemma:NL_bound} at level \(\eta\) and let \(c_{\eta, 3}\) denote the largest constant such that \(0 < d < c_{\eta, 3}\) implies \(48d^2 \leq L_\eta\). Since \(c < c_\eta \leq c_{\eta, 3}\), we have that 
        \begin{equation*}
            \chi^2(P_{\pi, \gamma, R} || P_{0, \gamma, R}) \leq -\frac{1}{R} + \exp\left(\left(48c^2 - 1\right) \log(R) + 48c^2\right) \leq 4\eta^2. 
        \end{equation*}
        By Lemma \ref{lemma:general_lower_bound} it follows that 
        \begin{equation*}
            \mathcal{R}(c\rho) \geq 1 - \frac{1}{2}\sqrt{\chi^2(P_{\pi, \gamma, R} || P_{0, \gamma, R})} \geq 1-\eta.
        \end{equation*}
        Since \(0 < c < c_\eta\) was arbitrary and \(\eta \in (0, 1)\) was arbitrary, we have proved the desired result.
        
        \textbf{Case 4:} Suppose \(\frac{p}{R} - \sqrt{\frac{p}{R} \log(eR)} < s < \frac{p}{R}\) and \((\frac{p}{R} - s) \log\left(1 + \frac{Rp \log(eR)}{\left(p - Rs\right)^2}\right) < \log(R)\). Note the supposition that \(s > \frac{p}{4R}\) in the statement of the lemma. The analysis of Case 2 can be exactly repeated. 
    \end{proof}

    Proposition \ref{prop:lbound_R_sparse} and Lemma \ref{lemma:lbound_pR_dense} are combined to give Propositions \ref{prop:lbound_pR_dense} and \ref{prop:lbound_pR_very_dense}.
    
    \begin{proof}[Proof of Proposition \ref{prop:lbound_pR_dense}]
        Fix \(\eta \in (0, 1)\). Let \(c_\eta := c_{\eta, 1} \wedge c_{\eta, 2}\) where \(c_{\eta, 1}\) and \(c_{\eta, 2}\) are the constants at level \(\eta\) from Proposition \ref{prop:lbound_R_sparse} and Lemma \ref{lemma:lbound_pR_dense} respectively. For all \(0 < c < c_\eta\), we have \(\mathcal{R}(c(\psi_1 \vee \upsilon)) \geq \mathcal{R}(c\psi_1) \vee \mathcal{R}(c\upsilon) \geq (1 - \eta) \vee (1-\eta) = 1-\eta\) since we have both \(c < c_{\eta, 1}\) and \(c < c_{\eta, 2}\). Since \(c < c_\eta\) was arbitrary and \(\eta \in (0, 1)\) was arbitrary, we have the desired result.
    \end{proof}

    \begin{proof}[Proof of Proposition \ref{prop:lbound_pR_very_dense}]
        The same proof for Proposition \ref{prop:lbound_pR_dense} applies here.  
    \end{proof}

    \subsubsection{\texorpdfstring{Regime \(\frac{p}{R} \leq s \leq p\)}{Regime p/R <= s <= p}}
    \begin{lemma}\label{lemma:lbound_R_avg_problem}
        Suppose \(\frac{p}{R} \leq s \leq p\) and \(\gamma \in [0, 1)\). Let 
        \begin{equation*}
            \rho^2 = 
            \begin{cases}
                \left(1-\gamma+\gamma \frac{p}{R}\right) \frac{Rs}{p} \log\left(1 + \frac{p^2}{Rs^2}\right) &\text{if } \frac{p}{R} \leq s < \frac{p}{\sqrt{R}}, \\
                \left(1 - \gamma + \gamma \frac{p}{R}\right) \sqrt{R} &\text{if } s \geq \frac{p}{\sqrt{R}}. 
            \end{cases}
        \end{equation*}
        If \(\eta\in (0, 1)\), then there exists a constant \(c_\eta\) depending only on \(\eta\) such that for all \(0 < c < c_\eta\) we have \(\mathcal{R}(c\rho) \geq 1-\eta\). 
    \end{lemma}
    \begin{proof}
        Fix \(\eta \in (0, 1)\). Set \(c_\eta := \frac{1}{\sqrt{2}} \wedge \sqrt{\frac{1}{2}\log\left(1 + 4\eta^2\right)} \wedge \sqrt{\frac{1}{2}\log(1 + \log(1 + 4\eta^2))}\). Let \(0 < c < c_\eta\). We deal with the two cases separately. 
        
        \textbf{Case 1:} Suppose \(\frac{p}{R} \leq s < \frac{p}{\sqrt{R}}\). Note that \(\rho^2 = \left(1 - \gamma + \gamma \frac{p}{R}\right) \frac{Rs}{p}\log\left(1 + \frac{p^2}{Rs^2}\right)\). Let \(m := \left\lfloor \frac{s}{p/R} \right\rfloor\) and note that \(1 \leq m < \sqrt{R}\). Let \(\pi\) be the prior on \(\Theta(p, s, c\rho)\) in which a draw \(\mu \sim \pi\) is given by the following construction. Draw \(K \subset [R]\) uniformly from the collection of all size \(m\) subsets of \([R]\) and set 
        \begin{equation*}
            \mu := \sum_{k \in K} \frac{c\rho}{\sqrt{m\frac{p}{R}}} \mathbf{1}_{B_k}. 
        \end{equation*}
        Observe that \(||\mu||_0 = m \cdot \frac{p}{R} \leq s\) and \(||\mu||^2 = c^2\rho^2\) almost surely. Thus, \(\pi\) is indeed supported on \(\Theta(p, s, c\rho)\). Let \(P_{\pi, \gamma, R} = \int P_{\theta, \gamma, R} \, \pi(d\theta)\) denote the Gaussian mixture induced by \(\pi\). 

        By Lemma \ref{lemma:Ingster_Suslina} and Lemma \ref{lemma:block_precision} we have 
        \begin{equation*}
            \chi^2\left( P_{\pi, \gamma, R} || P_{0, \gamma, R} \right) = E\left[ \exp\left( \left\langle \theta, \left[\sum_{k=1}^{R} \left(\frac{1}{1-\gamma} I_{B_k} - \frac{\gamma}{(1-\gamma)\left(1 - \gamma + \gamma \frac{p}{R}\right)}\mathbf{1}_{B_k}\mathbf{1}_{B_k}^\intercal \right)  \right] \widetilde{\theta} \right\rangle \right) \right] - 1
        \end{equation*}
        where \(\theta, \widetilde{\theta} \overset{iid}{\sim} \pi\). Write \(\theta = \sum_{k \in K} \frac{c\rho}{\sqrt{m \frac{p}{R}}} \mathbf{1}_{B_k}\) and \(\widetilde{\theta} = \sum_{k \in \widetilde{K}} \frac{c\rho}{\sqrt{m\frac{p}{R}}} \mathbf{1}_{B_k}\) where \(K, \widetilde{K}\) are iid uniformly drawn subsets of \([R]\) of size \(m\). Then, 
        \begin{align*}
            &\left\langle \theta, \left[\sum_{k=1}^{R} \left(\frac{1}{1-\gamma} I_{B_k} - \frac{\gamma}{(1-\gamma)\left(1 - \gamma + \gamma \frac{p}{R}\right)}\mathbf{1}_{B_k}\mathbf{1}_{B_k}^\intercal \right)\right] \widetilde{\theta} \right\rangle \\
            &= \left\langle \theta, \left[\sum_{k=1}^{R} \frac{1}{1-\gamma}\left(I_{B_k} - \frac{R}{p}\mathbf{1}_{B_k}\mathbf{1}_{B_k}^\intercal\right) + \frac{1}{1-\gamma+\gamma \frac{p}{R}} \cdot \frac{R}{p}\mathbf{1}_{B_k}\mathbf{1}_{B_k}^\intercal \right] \widetilde{\theta} \right\rangle \\
            &= \frac{c^2\rho^2}{m \frac{p}{R}} \left\langle \sum_{k \in K} \mathbf{1}_{B_k}, \left[\sum_{k=1}^{R} \frac{1}{1-\gamma}\left(I_{B_k} - \frac{R}{p}\mathbf{1}_{B_k}\mathbf{1}_{B_k}^\intercal\right) + \frac{1}{1-\gamma+\gamma \frac{p}{R}} \cdot \frac{R}{p}\mathbf{1}_{B_k}\mathbf{1}_{B_k}^\intercal  \right] \sum_{k \in \widetilde{K}} \mathbf{1}_{B_k} \right\rangle \\
            &= \frac{c^2\rho^2}{m \frac{p}{R}} \left\langle \sum_{k \in K} \mathbf{1}_{B_k}, \frac{1}{1-\gamma+\gamma\frac{p}{R}} \sum_{k \in \widetilde{K}} \mathbf{1}_{B_k}\right\rangle \\
            &= \frac{c^2\rho^2}{m \left(1-\gamma + \gamma \frac{p}{R}\right)} |K \cap \widetilde{K}| \\
            &= \frac{c^2 \frac{s}{p/R} \log\left(1 + \frac{p^2}{Rs^2}\right)}{m} |K \cap \widetilde{K}| \\
            &\leq 2c^2 \log\left(1 + \frac{p^2}{Rs^2}\right) |K \cap \widetilde{K}|
        \end{align*}
        where the final inequality follows from the fact that \(m = \left\lfloor \frac{s}{p/R}\right\rfloor\) and so \(\frac{\frac{s}{p/R}}{m} \leq 2\). Now, consider that \(|K \cap \widetilde{K}|\) is distributed according to a hypergeometric distribution with probability mass function 
        \begin{equation*}
            P\{|K \cap \widetilde{K}| = k \} = \frac{ \binom{m}{k} \binom{R-m}{m-k}}{ \binom{R}{m}}
        \end{equation*}
        for \(0 \leq k \leq m\). Consequently, we have by Lemma \ref{lemma:hypergeometric}
        \begin{align*}
            \chi^2\left( P_{\pi, \gamma, R} || P_{0, \gamma, R} \right) &= E\left[ \exp\left(2c^2 \log\left(1 + \frac{p^2}{Rs^2}\right) |K \cap \widetilde{K}| \right) \right] - 1 \\
            &= E\left[ \exp\left(2c^2 \log\left(1 + \frac{R}{\left(\frac{s}{p/R}\right)^2}\right) |K \cap \widetilde{K}| \right) \right] - 1 \\
            &\leq E\left[ \exp\left(2c^2 \log\left(1 + \frac{R}{m^2}\right) |K \cap \widetilde{K}| \right) \right] - 1 \\
            &\leq \left(1 - \frac{m}{R} + \frac{m}{R}\exp\left( 2c^2\log\left(1 + \frac{R}{m^2} \right) \right) \right)^m - 1 \\
            &= \left( 1 - \frac{m}{R} + \frac{m}{R} \left(1 + \frac{R}{m^2}\right)^{2c^2}\right)^m - 1 \\
            &\leq \left(1 + \frac{2c^2}{m} \right)^m - 1 \\
            &\leq e^{2c^2} - 1
        \end{align*}
        where in the penultimate inequality we have used that \(2c^2 < 2c_\eta^2 \leq 1\) and the inequality \((1+x)^\delta - 1 \leq \delta x\) for all \(0 < \delta < 1\) and \(x > 0\). Therefore, it follows that \(\chi^2\left( P_{\pi, \gamma, R} || P_{0, \gamma, R} \right) \leq e^{2c^2} - 1 \leq e^{2c_\eta^2} - 1 \leq 4\eta^2\). An application of Lemma \ref{lemma:general_lower_bound} yields \(\mathcal{R}(c\rho) \geq 1 - \frac{1}{2}\sqrt{\chi^2(P_{\pi, \gamma, R} || P_{0, \gamma, R})} \geq 1-\eta\). Since \(0 < c < c_\eta\) was arbitrary and \(\eta \in (0, 1)\) was arbitrary, we have proved the desired result for the case \(\frac{p}{R} \leq s < \frac{p}{\sqrt{R}}\). 

        \textbf{Case 2:} Suppose \(s \geq \frac{p}{\sqrt{R}}\). Note that \(\rho^2 = \left(1 - \gamma + \gamma \frac{p}{R}\right)\sqrt{R}\). Without loss of generality, assume \(\frac{p}{\sqrt{R}}\) is an integer. Repeating exactly the argument presented in Case 1 except now replacing every instance of \(s\) with \(\frac{p}{\sqrt{R}}\) yields 
        \begin{align*}
            \chi^2(P_{\pi, \gamma, R} || P_{0, \gamma, R}) &\leq \left(1 - \frac{m}{R} + \frac{m}{R} \exp\left(\frac{c^2\rho^2}{m\left(1 - \gamma + \gamma \frac{p}{R}\right)}\right)\right)^m - 1\\
            &\leq \left(1 - \frac{m}{R} + \frac{m}{R} \exp\left( \frac{c^2 \sqrt{R}}{\lfloor \sqrt{R}\rfloor} \right) \right)^m - 1 \\
            &\leq \left(1 - \frac{m}{R} + \frac{m}{R}\exp\left(2c^2\right)\right)^m - 1 \\
            &= \left(1 + \frac{m}{R}\left(e^{2c^2} - 1\right)\right)^m - 1 \\
            &= \left( 1 + \frac{\lfloor \sqrt{R} \rfloor}{R} \left(e^{2c^2} - 1\right)\right)^{\lfloor \sqrt{R}\rfloor} - 1 \\
            &\leq \left(1 + \frac{1}{\sqrt{R}} \left(e^{2c^2} - 1\right)\right)^{\sqrt{R}} - 1 \\
            &\leq \exp\left(e^{2c^2} - 1\right) - 1 \\
            &\leq \exp\left(e^{2c_\eta^2} - 1\right) - 1 \\
            &\leq 4\eta^2. 
        \end{align*}
        An application of Lemma \ref{lemma:general_lower_bound} yields \(\mathcal{R}(c\rho) \geq 1 - \frac{1}{2}\sqrt{\chi^2(P_{\pi, \gamma, R} || P_{0, \gamma, R})} \geq 1-\eta\). Since \(0 < c < c_\eta\) was arbitrary and \(\eta \in (0, 1)\) was arbitrary, we have proved the desired result for the case \(s \geq \frac{p}{\sqrt{R}}\).
    \end{proof}

    Proposition \ref{prop:lbound_R_sparse} and Lemma \ref{lemma:lbound_R_avg_problem} are combined to give Proposition \ref{prop:R_avg_lbound}.
    
    \begin{proof}[Proof of Proposition \ref{prop:R_avg_lbound}]
        The same proof for Proposition \ref{prop:lbound_pR_dense} applies here with the choice \(c_{\eta, 1}\) and \(c_{\eta, 2}\) given by the constants at level \(\eta\) from Proposition \ref{prop:lbound_R_sparse} and Lemma \ref{lemma:lbound_R_avg_problem}. 
    \end{proof}

    \subsection{Proofs of upper bounds in Section \ref{section:multiple_random_effects}}
    Recall the notation which was set up in Section \ref{section:R_prelim}.

    \subsubsection{Proof of Proposition \ref{prop:theta_upsilon}}
    \begin{proof}[Proof of Proposition \ref{prop:theta_upsilon}]
        Fix \(\theta \in \Theta(p, s, \varepsilon)\). Consider that \(||\theta||^2 = \sum_{k=1}^{R} ||\theta_{B_k}||^2\). The Pythagorean identity yields 
        \begin{equation*}
            \sum_{k=1}^{R} ||\theta_{B_k}||^2 = \sum_{k=1}^{R} ||\theta_{B_k} - \bar{\theta}_{B_k}\mathbf{1}_{B_k}||^2 + \sum_{k=1}^{R} ||\bar{\theta}_{B_k}\mathbf{1}_{B_k}||^2.
        \end{equation*}
        Since \(||\theta||^2 \geq \varepsilon^2\), it follows that either 
        \begin{equation}\label{eqn:condition_1}
            \sum_{k=1}^{R} ||\theta_{B_k} - \bar{\theta}_{B_k}\mathbf{1}_{B_k}||^2 \geq \frac{\varepsilon^2}{2}
        \end{equation}
        or 
        \begin{equation}\label{eqn:condition_2}
            \sum_{k=1}^{R} ||\bar{\theta}_{B_k}\mathbf{1}_{B_k}||^2 \geq \frac{\varepsilon^2}{2}
        \end{equation}
        holds. 

        Let us first examine condition (\ref{eqn:condition_1}). Notice that we can write
        \begin{equation*}
            \sum_{k=1}^{R} ||\theta_{B_k} - \bar{\theta}_{B_k}\mathbf{1}_{B_k}||^2 = \sum_{\substack{1\leq k \leq R: \\ |B_k \cap \supp(\theta)| \leq \frac{p}{4R}}} ||\theta_{B_k} - \bar{\theta}_{B_k}\mathbf{1}_{B_k}||^2 + \sum_{\substack{1\leq k \leq R: \\ |B_k \cap \supp(\theta)| > \frac{p}{4R}}} ||\theta_{B_k} - \bar{\theta}_{B_k}\mathbf{1}_{B_k}||^2.
        \end{equation*}
        If (\ref{eqn:condition_1}) holds, then it must be that either 
        \begin{equation}\label{eqn:condition_1.1}
            \sum_{\substack{1\leq k \leq R: \\ |B_k \cap \supp(\theta)| \leq \frac{p}{4R}}} ||\theta_{B_k} - \bar{\theta}_{B_k}\mathbf{1}_{B_k}||^2 \geq \frac{\varepsilon^2}{4}
        \end{equation}
        or 
        \begin{equation}\label{eqn:condition_1.2}
            \sum_{\substack{1\leq k \leq R: \\ |B_k \cap \supp(\theta)| > \frac{p}{4R}}} ||\theta_{B_k} - \bar{\theta}_{B_k}\mathbf{1}_{B_k}||^2 \geq \frac{\varepsilon^2}{4}
        \end{equation}
        holds. If (\ref{eqn:condition_1.1}) holds, then Lemma \ref{lemma:v_supp_approximation} gives
        \begin{align*}
            \sum_{\substack{1\leq k \leq R: \\ |B_k \cap \supp(\theta)| \leq \frac{p}{4R}}} ||\theta_{B_k} - \bar{\theta}_{B_k}\mathbf{1}_{B_k \cap \supp(\theta)}||^2 &\geq \sum_{\substack{1\leq k \leq R: \\ |B_k \cap \supp(\theta)| \leq \frac{p}{4R}}} ||\theta_{B_k}||^2 \cdot \frac{\frac{p}{R} - 2 \cdot \frac{p}{4R}}{\frac{p}{R}} \\
            &\geq \frac{1}{2} \cdot \sum_{\substack{1\leq k \leq R: \\ |B_k \cap \supp(\theta)| \leq \frac{p}{4R}}} ||\theta_{B_k} - \bar{\theta}_{B_k}\mathbf{1}_{B_k}||^2 \\
            &\geq \frac{\varepsilon^2}{8}.
        \end{align*}
        In other words, it has been shown that if (\ref{eqn:condition_1}) and (\ref{eqn:condition_1.1}) hold, then \(\theta \in \Upsilon_{\mathcal{I}}(p, s, \varepsilon)\). 

        Now let us suppose (\ref{eqn:condition_1}) and (\ref{eqn:condition_1.2}) hold. Consider that we can write 
        \begin{align*}
            \sum_{\substack{1\leq k \leq R: \\ |B_k \cap \supp(\theta)| > \frac{p}{4R}}} ||\theta_{B_k} - \bar{\theta}_{B_k}\mathbf{1}_{B_k}||^2 &= \sum_{\substack{1\leq k \leq R: \\ |B_k \cap \supp(\theta)| > \frac{p}{4R}}} ||\theta_{B_k} - \bar{\theta}_{B_k}\mathbf{1}_{B_k \cap \supp(\theta)}||^2 + \sum_{\substack{1\leq k \leq R: \\ |B_k \cap \supp(\theta)| > \frac{p}{4R}}} ||-\bar{\theta}_{B_k}\mathbf{1}_{B_k \cap \supp(\theta)^c}||^2 \\
            &= \sum_{\substack{1\leq k \leq R: \\ |B_k \cap \supp(\theta)| > \frac{p}{4R}}} ||\theta_{B_k} - \bar{\theta}_{B_k}\mathbf{1}_{B_k \cap \supp(\theta)}||^2 + \sum_{\substack{1\leq k \leq R: \\ |B_k \cap \supp(\theta)| > \frac{p}{4R}}} \bar{\theta}_{B_k}^2 |B_k \cap \supp(\theta)^c|.
        \end{align*}
        Since (\ref{eqn:condition_1.2}) holds, it follows that either 
        \begin{equation*}
            \sum_{\substack{1\leq k \leq R: \\ |B_k \cap \supp(\theta)| > \frac{p}{4R}}} ||\theta_{B_k} - \bar{\theta}_{B_k}\mathbf{1}_{B_k \cap \supp(\theta)}||^2 \geq \frac{\varepsilon^2}{8}
        \end{equation*}
        or 
        \begin{equation*}
            \sum_{\substack{1\leq k \leq R: \\ |B_k \cap \supp(\theta)| > \frac{p}{4R}}} \bar{\theta}_{B_k}^2 |B_k \cap \supp(\theta)^c| \geq \frac{\varepsilon^2}{8}.
        \end{equation*}
        If the first condition of these holds, then it clearly follows \(\theta \in \Upsilon_{\mathcal{I}}(p, s, \varepsilon)\). If the second condition holds, then we have
        \begin{align*}
            \sum_{\substack{1\leq k \leq R: \\ |B_k \cap \supp(\theta)| > \frac{p}{4R}}} ||\bar{\theta}_{B_k}\mathbf{1}_{B_k}||^2 &= \sum_{\substack{1\leq k \leq R: \\ |B_k \cap \supp(\theta)| > \frac{p}{4R}}} \bar{\theta}_{B_k}^2 \frac{p}{R} \\
            &\geq \sum_{\substack{1\leq k \leq R: \\ |B_k \cap \supp(\theta)| > \frac{p}{4R}}} \bar{\theta}_{B_k}^2 |B_k \cap \supp(\theta)^c| \\
            &\geq \frac{\varepsilon^2}{8}
        \end{align*}
        i.e. we have \(\theta \in \Upsilon_{\mathcal{II}}(p, s, \varepsilon)\). In other words, if (\ref{eqn:condition_1}) and (\ref{eqn:condition_1.2}) holds, then \(\theta \in \Upsilon_{\mathcal{I}}(p, s, \varepsilon) \cup \Upsilon_{\mathcal{II}}(p, s, \varepsilon)\). To sum up, we have showed that if (\ref{eqn:condition_1}) holds, then \(\theta \in \Upsilon_{\mathcal{I}}(p, s, \varepsilon) \cup \Upsilon_{\mathcal{II}}(p, s, \varepsilon)\). 

        We now examine condition (\ref{eqn:condition_2}). Observe that we can write 
        \begin{equation*}
            \sum_{k=1}^{R} ||\bar{\theta}_{B_k}\mathbf{1}_{B_k}||^2 = \sum_{\substack{1 \leq k \leq R: \\ |B_k \cap \supp(\theta)| \leq \frac{p}{4R}}} ||\bar{\theta}_{B_k}\mathbf{1}_{B_k}||^2 + \sum_{\substack{1 \leq k \leq R: \\ |B_k \cap \supp(\theta)| > \frac{p}{4R}}} ||\bar{\theta}_{B_k}\mathbf{1}_{B_k}||^2. 
        \end{equation*}
        Consequently, if (\ref{eqn:condition_2}) holds then it must be that either 
        \begin{equation}\label{eqn:condition_2.1}
            \sum_{\substack{1 \leq k \leq R: \\ |B_k \cap \supp(\theta)| \leq \frac{p}{4R}}} ||\bar{\theta}_{B_k}\mathbf{1}_{B_k}||^2 \geq \frac{\varepsilon^2}{4}
        \end{equation}
        or 
        \begin{equation}\label{eqn:condition_2.2}
            \sum_{\substack{1 \leq k \leq R: \\ |B_k \cap \supp(\theta)| > \frac{p}{4R}}} ||\bar{\theta}_{B_k}\mathbf{1}_{B_k}||^2 \geq \frac{\varepsilon^2}{4}.
        \end{equation}
        If (\ref{eqn:condition_2.1}) holds, then Lemma \ref{lemma:v_orthog_approximation} implies 
        \begin{align*}
            \sum_{\substack{1 \leq k \leq R: \\ |B_k \cap \supp(\theta)| \leq \frac{p}{4R}}} ||\theta_{B_k} - \bar{\theta}_{B_k} \mathbf{1}_{B_k \cap \supp(\theta)}||^2 &\geq \sum_{\substack{1 \leq k \leq R: \\ |B_k \cap \supp(\theta)| \leq \frac{p}{4R}}} ||\theta_{B_k}||^2 \cdot \frac{\frac{p}{R} - 2 \cdot \frac{p}{4R}}{\frac{p}{R}} \\
            &= \frac{1}{2} \sum_{\substack{1 \leq k \leq R: \\ |B_k \cap \supp(\theta)| \leq \frac{p}{4R}}} ||\theta_{B_k}||^2 \\
            &\geq \frac{1}{2} \sum_{\substack{1 \leq k \leq R: \\ |B_k \cap \supp(\theta)| \leq \frac{p}{4R}}} ||\bar{\theta}_{B_k} \mathbf{1}_{B_k}||^2 \\
            &\geq \frac{\varepsilon^2}{8}
        \end{align*}
        and so it immediately follows that \(\theta \in \Upsilon_{\mathcal{I}}(p, s, \varepsilon)\). On the other hand, if (\ref{eqn:condition_2.2}) holds then we clearly have \(\theta \in \Upsilon_{\mathcal{II}}(p,s,\varepsilon)\). Thus, we have proved that if (\ref{eqn:condition_2}) holds then \(\theta \in \Upsilon_{\mathcal{I}}(p, s, \varepsilon) \cup \Upsilon_{\mathcal{II}}(p, s, \varepsilon)\). 
    \end{proof}

    \subsubsection{\texorpdfstring{Regime \(1 \leq s \leq \frac{p}{4R}\)}{Regime 1 <= s <= p/4R}}

    We first establish the sensitivity of the test \(\varphi_{t, r}\) given by (\ref{test:orthogonal_tsybakov}) to the space \(\Upsilon_{\mathcal{I}}(p, s, \varepsilon)\) in the regime \(s < \sqrt{p}\).
    \begin{lemma}\label{lemma:test_upsilon1_sparse}
        Suppose \(1 \leq s < \sqrt{p}\) and \(\gamma \in [0, 1)\). Let \(\psi_1^2\) be given by (\ref{rate:problemI}). If \(\eta \in (0, 1)\), then there exists a constant \(C_\eta > 0\) depending only on \(\eta\) such that for all \(C > C_\eta\) the testing procedure \(\varphi_{t^*, r^*}\) given by (\ref{test:orthogonal_tsybakov}) with \(t^* = \sqrt{2 \log\left(1 + \frac{p}{s^2}\right)}\) and \(r^* = \frac{C^2}{64}s\log\left(1 + \frac{p}{s^2}\right)\) satisfies 
        \begin{equation*}
            P_{0, \gamma, R}\{\varphi_{t^*, r^*} = 1\} + \sup_{\theta \in \Upsilon_{\mathcal{I}}(p, s, C\psi_1)} P_{\theta, \gamma, R}\{\varphi_{t^*, r^*} = 0\} \leq \eta
        \end{equation*}
        where \(\Upsilon_{\mathcal{I}}(p, s, C\psi_1)\) is given by (\ref{space:Upsilon1}).
    \end{lemma}
    \begin{proof}
        Fix \(\eta \in (0, 1)\) and set \(C_\eta := \sqrt{8}C_\eta^*\) where \(C_\eta^*\) is given in Proposition \ref{prop:supp_tsybakov}. Let \(C > C_\eta\). Note that under the data-generating process \(P_{\theta, \gamma, R}\) we have 
        \begin{equation*}
            \widetilde{X}_{B_k} \sim N\left( \frac{\theta_{B_k} - \bar{\theta}_{B_k}\mathbf{1}_{B_k}}{\sqrt{1-\gamma}}, I_{\frac{p}{R}} \right)
        \end{equation*}
        and \(\{\widetilde{X}_{B_k}\}_{k=1}^{R}\) are independent. Let \(Y \in \R^p\) denote the vector by concatenating \(Y := (\widetilde{X}_{B_1},...,\widetilde{X}_{B_R})\). Let \(\mu \in \R^p\) denote the expectation of \(Y\). Note that \(Y \sim N(\mu, I_p)\) and \(\mu\) depends on \(\theta, \gamma,\) and \(R\). Further note 
        \begin{equation*}
            \sum_{k=1}^{R} \sum_{j=1}^{p/R} \left( (\widetilde{X}_{B_k})_j^2 - \alpha_{t^*} \right) \mathbf{1}_{\{|(\widetilde{X}_{B_k})_j| \geq t^*\}} = \sum_{i=1}^{p} (Y_i^2 - \alpha_{t^*})\mathbf{1}_{\{|Y_i| \geq t^*\}}.    
        \end{equation*}
        Now, consider that for \(\theta \in \Upsilon_{\mathcal{I}}(p, s, C\psi_1)\) we have \(|\supp(\theta)| \leq s\) and 
        \begin{equation*}
            ||\mu_{\supp(\theta)}||^2 = \sum_{k=1}^{R} \frac{||\theta_{B_k} - \bar{\theta}_{B_k}\mathbf{1}_{B_k \cap \supp(\theta)}||^2}{1-\gamma} \geq \frac{C^2}{8} s \log\left(1 + \frac{p}{s^2}\right). 
        \end{equation*}
        In particular, we have shown \(\theta \in \Upsilon_{\mathcal{I}}(p, s, C\psi_1)\) implies \(\mu \in \mathscr{M}\left(p, s, \frac{C}{\sqrt{8}}\sqrt{s \log\left(1 + \frac{p}{s^2}\right)}\right)\). Here, the latter parameter space is defined in the statement of Proposition \ref{prop:supp_tsybakov}. Consequently, since \(s < \sqrt{p}\), \(\frac{C}{\sqrt{8}} > \frac{C_\eta}{\sqrt{8}} \geq C_\eta^*\), and \(r^* = \frac{C^2/8}{8} s \log\left(1 + \frac{p}{s^2}\right)\), an application of Proposition \ref{prop:supp_tsybakov} yields
        \begin{align*}
            P_{0, \gamma, R} \left\{\varphi_{t^*, r^*} = 1 \right\} + \sup_{\theta \in \Upsilon_{\mathcal{I}}(p, s, C\psi_1)} P_{\theta, \gamma, R}\left\{ \varphi_{t^*, r^*} = 0 \right\} \leq \eta. 
        \end{align*}
        Since \(C > C_\eta\) was arbitrary and \(\eta \in (0, 1)\) was arbitrary, the desired result has been proved.
    \end{proof}

    We now establish the sensitivity of the test \(\varphi_{r}^{\chi^2}\) given by (\ref{test:block_orthogonal_chisquare}) to the space \(\Upsilon_{\mathcal{I}}(p , s, \varepsilon)\) when \(s \geq \sqrt{p}\). 
    \begin{lemma}\label{lemma:test_upsilon1_dense}
        Suppose \(s \geq \sqrt{p}\) and \(\gamma \in [0, 1)\). Let \(\psi_1^2\) be given by (\ref{rate:problemI}). If \(\eta \in (0, 1)\), then there exists a constant \(C_\eta > 0\) depending only on \(\eta\) such that for all \(C > C_\eta\) the testing procedure \(\varphi_{\frac{C^2}{16}}^{\chi^2}\) given by (\ref{test:block_orthogonal_chisquare}) satisfies 
        \begin{equation*}
            P_{0, \gamma, R}\left\{\varphi_{\frac{C^2}{16}}^{\chi^2} = 1 \right\} + \sup_{\theta \in \Upsilon_{\mathcal{I}}(p, s, C\psi_1)} P_{\theta, \gamma, R}\left\{\varphi_{\frac{C^2}{16}}^{\chi^2} = 0\right\} \leq \eta. 
        \end{equation*}
    \end{lemma}
    \begin{proof}
        Fix \(\eta \in (0, 1)\) and set \(C_\eta := \left( \frac{512 \cdot 2}{\eta} \right)^{1/4} \vee \left(\frac{512 \cdot 4}{\eta}\right)^{1/4} \vee \left( \frac{256 \cdot 4}{\eta}\right)^{1/2}\). Let \(C > C_\eta\). Note that under the data-generating process \(P_{\theta, \gamma, R}\) we have 
        \begin{equation*}
            \widetilde{X}_{B_k} \sim N\left( \frac{\theta_{B_k} - \bar{\theta}_{B_k}\mathbf{1}_{B_k}}{\sqrt{1-\gamma}}, I_{\frac{p}{R}} \right)
        \end{equation*}
        and \(\{\widetilde{X}_{B_k}\}_{k=1}^{R}\) are independent. Let \(Y \in \R^p\) denote the vector by concatenating \(Y := (\widetilde{X}_{B_1},...,\widetilde{X}_{B_R})\). Let \(\mu \in \R^p\) denote the expectation of \(Y\). Note that \(Y \sim N(\mu, I_p)\) and \(\mu\) depends on \(\theta, \gamma,\) and \(R\). Further note 
        \begin{equation*}
            ||Y||^2 = \sum_{k=1}^{R} \sum_{j=1}^{p/R} (\widetilde{X}_{B_k})_j^2.    
        \end{equation*}
        Likewise, consider that \(||\mu||^2 = \sum_{k=1}^{R} \frac{||\theta_{B_k} - \bar{\theta}_{B_k}\mathbf{1}_{B_k}||^2}{1-\gamma}\). Note that \(||Y||^2 \sim \chi^2_{p}(||\mu||^2)\). Therefore \(E_{\theta, \gamma, R}(||Y||^2) = p + ||\mu||^2\) and \(\Var_{\theta, \gamma, R} (||Y||^2) = 2p + 4||\mu||^2\). Therefore, the type I error is bounded as Chebyshev's inequality yields
        \begin{align*}
            P_{0, \gamma, R}\left\{ \varphi_{\frac{C^2}{16}}^{\chi^2} = 1 \right\} = P_{0, \gamma, R} \left\{ ||Y||^2 > p + \frac{C^2}{16}\sqrt{p} \right\} \leq \frac{2p}{\frac{C^4}{16^2} p} \leq \frac{512}{C_\eta^4} \leq \frac{\eta}{2}.
        \end{align*}
        Turning our attention to the type II error, consider that for \(\theta \in \Upsilon_{\mathcal{I}}(p, s, C\psi_1)\), we have \(||\mu||^2 \geq \frac{C^2}{8} \sqrt{p}\). Therefore, 
        \begin{align*}
            \sup_{\theta \in \Upsilon_{\mathcal{I}}(p, s, C\psi_1)} P_{\theta, \gamma, R}\left\{ \varphi_{\frac{C^2}{16}}^{\chi^2} = 0 \right\} &= \sup_{\theta \in \Upsilon_{\mathcal{I}}(p, s, C\psi_1)} P_{\theta, \gamma, R} \left\{ ||Y||^2 \leq p + \frac{C^2}{16} \sqrt{p} \right\} \\
            &= \sup_{\theta \in \Upsilon_{\mathcal{I}}(p, s, C\psi_1)} P_{\theta, \gamma, R} \left\{ ||\mu||^2 - \frac{C^2}{16}\sqrt{p} \leq p + ||\mu||^2 - ||Y||^2 \right\} \\
            &\leq \sup_{\theta \in \Upsilon_{\mathcal{I}}(p, s, C\psi_1)} \frac{\Var_{\theta, \gamma, R}(||Y||^2)}{\left( E_{\theta, \gamma, R}(||Y||^2) - \frac{C^2}{16}\sqrt{p} \right)^2} \\
            &\leq \frac{2p}{\frac{C^4}{16^2} p} + \sup_{\theta \in \Upsilon_{\mathcal{I}}(p, s, C\psi_1)} \frac{4||\mu||^2}{\frac{1}{4} ||\mu||^4} \\
            &\leq \frac{512}{C^4} + \frac{16^2}{C^2\sqrt{p}} \\
            &\leq \frac{512}{C_\eta^4} + \frac{256}{C_\eta^2} \\
            &\leq \frac{\eta}{2}.
        \end{align*}
        Hence, we have that the sum of the type I and type II errors are bounded as 
        \begin{equation*}
            P_{0, \gamma, R}\left\{ \varphi_{\frac{C^2}{16}}^{\chi^2} = 1 \right\} + \sup_{\theta \in \Upsilon_{\mathcal{I}}(p, s, C\psi_1)} P_{\theta, \gamma, R}\left\{ \varphi_{\frac{C^2}{16}}^{\chi^2} = 0 \right\} \leq \eta.
        \end{equation*}
        Since \(C > C_\eta\) was arbitrary and \(\eta \in (0, 1)\) was arbitrary, the desired result has been proved.
    \end{proof}

    Lemmas \ref{lemma:test_upsilon1_sparse} and \ref{lemma:test_upsilon1_dense} are combined to prove Proposition \ref{prop:ubound_R_sparse}.
    \begin{proof}[Proof of Proposition \ref{prop:ubound_R_sparse}]
        Fix \(\eta \in (0, 1)\) and set \(C_\eta := C_{\eta, 1} \vee C_{\eta, 2}\) where \(C_{\eta, 1}, C_{\eta, 2}\) are the constants at level \(\eta\) from Lemmas \ref{lemma:test_upsilon1_sparse} and \ref{lemma:test_upsilon1_dense} respectively. Note that since \(s \leq \frac{p}{4R}\), we immediately have \(\Theta(p, s, \varepsilon) \subset \Upsilon_{\mathcal{I}}(p, s, \varepsilon)\) for all \(\varepsilon > 0\). Consequently,
        \begin{align*}
            P_{0, \gamma, R}\left\{ \varphi_1 = 1 \right\} + \sup_{\theta \in \Theta(p, s, C\psi_1)} P_{\theta, \gamma, R}\left\{ \varphi_1 = 0 \right\}
            \leq  P_{0, \gamma, R}\left\{ \varphi_1 = 1 \right\} + \sup_{\theta \in \Upsilon_{\mathcal{I}}(p, s, C\psi_1)} P_{\theta, \gamma, R}\left\{ \varphi_1 = 0 \right\} \leq \eta
        \end{align*}
        where the second inequality follows from an application of Lemmas \ref{lemma:test_upsilon1_sparse} and \ref{lemma:test_upsilon1_dense} since we have \(C > C_{\eta, 1}\) and \(C > C_{\eta, 2}\). As \(C > C_\eta\) was arbitrary and \(\eta \in (0, 1)\) was arbitrary, the desired result has been proved. 
    \end{proof}

    \subsubsection{\texorpdfstring{Regime \(\frac{p}{4R} < s \leq \frac{p}{R} - \sqrt{\frac{p}{R}\log(eR)}\)}{Regime p/4R < s <= p/R - sqrt(p log(eR)/R) }}

    \begin{lemma}\label{lemma:linear_scan}
        Suppose \(1 \leq s \leq \frac{p}{R}\) and \(\gamma \in [0, 1)\). Set \(\rho^2 = \left(1 - \gamma + \gamma \frac{p}{R} \right)\log(eR)\). If \(\eta \in (0, 1)\), then there exists a constant \(C_\eta > 0\) depending only on \(\eta\) such that for all \(C > C_\eta\), the testing procedure \(\varphi_{\frac{C^2}{64}\log(eR)}^{\mathbf{1}-\text{scan}}\) given by (\ref{test:linear_scan}) satisfies 
        \begin{equation*}
            P_{0, \gamma, R}\left\{ \varphi_{\frac{C^2}{64} \log(eR)}^{\mathbf{1}-\text{scan}} = 1 \right\} + \sup_{\theta \in \Upsilon_{\mathcal{II}}(p, s, C\rho)} P_{\theta, \gamma, R}\left\{\varphi_{\frac{C^2}{64}\log(eR)}^{\mathbf{1}-\text{scan}} = 0\right\} \leq \eta. 
        \end{equation*}
    \end{lemma}
    \begin{proof}
        Fix \(\eta \in (0, 1)\) and set \(C_\eta := \left(128 \cdot 4 \log\left(\frac{2}{\eta}\right) \right)^{1/2} \vee \left(128 \cdot 4 \right)^{1/2} \vee \left(\frac{64^2 \cdot 2 \cdot 4}{\eta} \right)^{1/4} \vee \left( \frac{32 \cdot 16 \cdot 4}{\eta}\right)^{1/2}\). Let \(C > C_\eta\). For ease of notation, set \(\sigma^2 := 1 - \gamma + \gamma \frac{p}{R}\). Notice that under the data-generating process \(P_{\theta, \gamma, R}\), we have for \(k \in [R]\) that
        \begin{equation*}
            \left\langle \sqrt{\frac{R}{p}} \mathbf{1}_{B_k}, X \right\rangle^2 \sim \sigma^2 \chi^2_1\left(\frac{||\bar{\theta}_{B_k}\mathbf{1}_{B_k}||^2}{\sigma^2}\right).   
        \end{equation*}
        Thus \(E_{\theta, \gamma, R}\left(\left\langle \sqrt{\frac{R}{p}} \mathbf{1}_{B_k}, X \right\rangle^2\right) = \sigma^2 + ||\bar{\theta}_{B_k}\mathbf{1}_{B_k}||^2\) and \(\Var_{\theta, \gamma, R}\left(\left\langle \sqrt{\frac{R}{p}} \mathbf{1}_{B_k}, X \right\rangle^2\right) = \sigma^2\left(2\sigma^2 + 4||\bar{\theta}_{B_k}\mathbf{1}_{B_k}||^2 \right)\). We now bound the type I error. Consider that 
        \begin{equation*}
            \sigma^2\left(2\sqrt{\log(eR)- \log\left(\frac{\eta}{2}\right)} + 2\left(\log(eR) - \log\left(\frac{\eta}{2}\right) \right)\right) \leq \frac{C^2}{64} \sigma^2 \log(eR)
        \end{equation*}
        since \(\log(eR) \geq 1\) and \(\frac{C^2}{128} > \frac{C_\eta^2}{128} \geq 4 \log\left(\frac{2}{\eta}\right) \vee 4\). Consequently, 
        \begin{align*}
            P_{0, \gamma, R}\left\{\varphi_{\frac{C^2}{64} \log(eR)}^{\mathbf{1}-\text{scan}} = 1  \right\} &\leq \sum_{k = 1}^{R} P_{\theta, \gamma, R}\left\{ \left\langle \sqrt{\frac{R}{p}} \mathbf{1}_{B_k}, X \right\rangle^2 > \sigma^2\left( 1 + \frac{C^2}{64}\log(eR)\right) \right\} \\
            &\leq \sum_{k = 1}^{R} P_{\theta, \gamma, R}\left\{ \left\langle \sqrt{\frac{R}{p}} \mathbf{1}_{B_k}, X \right\rangle^2 > \sigma^2\left( 1 + 2\sqrt{\log(eR)- \log\left(\frac{\eta}{2}\right)} + 2\left(\log(eR) - \log\left(\frac{\eta}{2}\right) \right)\right) \right\} \\
            &\leq R \exp\left(-\log(eR) + \log\left(\frac{\eta}{2}\right)\right)  \\
            &\leq \frac{\eta}{2}
        \end{align*}
        where we have used Lemma \ref{lemma:laurent_massart} to obtain the penultimate inequality. We now turn to bounding the type II error. For \(\theta \in \Upsilon_{\mathcal{II}}(p, s, C\rho)\), let \(\mathcal{K}(\theta) := \left\{k \subset [R] : |B_k \cap \supp(\theta)| > \frac{p}{4R}\right\}\). Note that \(\mathcal{K}(\theta) \neq \emptyset\) by definition of \(\Upsilon_{\mathcal{II}}\), and consider that \(|\mathcal{K}(\theta)| \leq \frac{s}{\frac{p}{4R}} \leq \frac{4Rs}{p} \leq 4\) since \(s \leq \frac{p}{R}\). Consequently, by triangle inequality there exists \(1 \leq k^*(\theta) \leq R\) such that \(||\bar{\theta}_{B_{k^*(\theta)}}\mathbf{1}_{B_{k^*(\theta)}}||^2 \geq \frac{C^2\rho^2}{32}\). Then the type II error is bounded as 
        \begin{align*}
            &\sup_{\theta \in \Upsilon_{\mathcal{II}}(p, s, C\rho)} P_{\theta, \gamma, R}\left\{\varphi_{\frac{C^2}{64}\log(eR)}^{\mathbf{1}-\text{scan}} = 0 \right\} \\
            &\leq \sup_{\theta \in \Upsilon_{\mathcal{II}}(p, s, C\rho)} P_{\theta, \gamma, R}\left\{\left\langle \sqrt{\frac{R}{p}} \mathbf{1}_{B_{k^*(\theta)}}, X \right\rangle^2 \leq \sigma^2\left( 1 + \frac{C^2}{64}\log(eR)\right)\right\} \\
            &= \sup_{\theta \in \Upsilon_{\mathcal{II}}(p, s, C\rho)} P_{\theta, \gamma, R}\left\{||\bar{\theta}_{B_{k^*(\theta)}}\mathbf{1}_{B_{k^*(\theta)}}||^2 - \frac{C^2}{64}\sigma^2 \log(eR) \leq \sigma^2 + ||\bar{\theta}_{B_{k^*(\theta)}}\mathbf{1}_{B_{k^*(\theta)}}||^2 - \left\langle \sqrt{\frac{R}{p}} \mathbf{1}_{B_{k^*(\theta)}}, X \right\rangle^2\right\} \\
            &\leq \sup_{\theta \in \Upsilon_{\mathcal{II}}(p, s, C\rho)} \frac{\Var_{\theta, \gamma, R}\left( \left\langle \sqrt{\frac{R}{p}} \mathbf{1}_{B_{k^*(\theta)}}, X \right\rangle^2 \right)}{\left(||\bar{\theta}_{B_{k^*(\theta)}}\mathbf{1}_{B_{k^*(\theta)}}||^2 - \frac{C^2}{64}\sigma^2\log(eR) \right)^2} \\
            &\leq \frac{2\sigma^4}{\frac{C^4}{64^2} \sigma^4 \log^2(eR)} + \sup_{\theta \in \Upsilon_{\mathcal{II}}(p, s, C\rho)} \frac{4\sigma^2 ||\bar{\theta}_{B_{k^*(\theta)}}\mathbf{1}_{B_{k^*(\theta)}}||^2}{\frac{1}{4}||\bar{\theta}_{B_{k^*(\theta)}}\mathbf{1}_{B_{k^*(\theta)}}||^4} \\
            &\leq \frac{64^2 \cdot 2}{C^4 \log^2(eR)} + \sup_{\theta \in \Upsilon_{\mathcal{II}}(p, s, C\rho)} \frac{16\sigma^2}{||\bar{\theta}_{B_{k^*(\theta)}}\mathbf{1}_{B_{k^*(\theta)}}||^2} \\
            &\leq \frac{64^2 \cdot 2}{C^4 \log^2(eR)} + \frac{32 \cdot 16}{C^2\log(eR)} \\
            &\leq \frac{64^2 \cdot 2}{C_\eta^4} + \frac{32 \cdot 16}{C_\eta^2} \\
            &\leq \frac{\eta}{2}. 
        \end{align*}
        Thus, the sum of the type I and type II errors is bounded as 
        \begin{equation*}
            P_{0, \gamma, R}\left\{ \varphi_{\frac{C^2}{64} \log(eR)}^{\mathbf{1}-\text{scan}} = 1 \right\} + \sup_{\theta \in \Upsilon_{\mathcal{II}}(p, s, C\rho)} P_{\theta, \gamma, R}\left\{\varphi_{\frac{C^2}{64}\log(eR)}^{\mathbf{1}-\text{scan}} = 0 \right\} \leq \eta.
        \end{equation*}
        Since \(C > C_\eta\) was arbitrary and \(\eta \in (0, 1)\) was arbitrary, the desired result has been proved. 
    \end{proof}

    \begin{lemma}\label{lemma:chisquare_scan}
        Suppose \(\frac{p}{4R} < s \leq \frac{p}{R} - \sqrt{\frac{p}{R}\log(eR)}\) and \(\gamma \in [0, 1)\). Set 
        \begin{equation*}
            \rho^2 := \frac{(1-\gamma)p}{p-Rs}\left(\sqrt{\frac{p}{R} \log(eR)} + \log(R) \right).    
        \end{equation*}
        If \(\eta \in (0, 1)\), then there exists a constant \(C_\eta > 0\) depending only on \(\eta\) such that for all \(C > C_\eta\) the testing procedure \(\varphi_{\frac{C^2}{128}\log(eR)}^{\chi^2-\text{scan}}\) given by (\ref{test:pR_dense_chisquare}) satisfies 
        \begin{equation*}
            P_{0, \gamma, R}\left\{\varphi_{\frac{C^2}{128} \log(eR)}^{\chi^2-\text{scan}} = 1 \right\} + \sup_{\theta \in \Upsilon_{\mathcal{II}}(p, s, C\rho)} P_{\theta, \gamma, R}\left\{\varphi_{\frac{C^2}{128}\log(eR)}^{\chi^2-\text{scan}} = 0 \right\} \leq \eta.
        \end{equation*}
    \end{lemma}
    \begin{proof}
        Fix \(\eta \in (0,1)\) and set 
        \begin{equation*}
            C_\eta := \sqrt{128 \left(1 + \log\left(\frac{2}{\eta}\right)\right)} \vee \left( \frac{64^2 \cdot 2 \cdot 4}{\eta} \right)^{1/4} \vee \left(\frac{32 \cdot 16 \cdot 4}{\eta} \right)^{1/2}.
        \end{equation*}
        Let \(C > C_\eta\). We first bound the type I error. An application of union bound and Lemma \ref{lemma:laurent_massart} yields 
        \begin{align*}
            P_{0, \gamma, R}\left\{\varphi_{\frac{C^2}{128}\log(eR)}^{\chi^2-\text{scan}} = 1 \right\} &\leq \sum_{k=1}^{R} P_{0, \gamma, R}\left\{||\widetilde{X}_{B_k}||^2 > \frac{p}{R} + \frac{2C}{\sqrt{128}}\sqrt{\frac{p}{R}\log(eR)} + \frac{C^2}{64}\log(eR) \right\} \\
            &= R \cdot P\left\{ \chi^2_{\frac{p}{R}} > \frac{p}{R} + \frac{2C}{\sqrt{128}}\sqrt{\frac{p}{R}\log(eR)} + \frac{C^2}{64}\log(eR)\right\}\\
            &\leq Re^{-\frac{C^2}{128}\log(eR)} \\
            &\leq e^{\left(1-\frac{C^2}{128}\right)\log(eR)} \\
            &\leq e^{1 - \frac{C^2}{128}} \\
            &\leq \frac{\eta}{2}
        \end{align*}
        where the penultimate inequality uses the fact that \(\frac{C^2}{128} > \frac{C_\eta^2}{128} > 1\) and the final inequality uses \(C^2 > C_\eta^2 \geq 128\left(\log\left(\frac{2}{\eta}\right) + 1\right)\). The type I error is thus bounded. We now turn our attention to the type II error. For \(\theta \in \Upsilon_{\mathcal{II}}(p, s, C\rho)\), let \(\mathcal{K}(\theta) := \left\{k \subset [R] : |B_k \cap \supp(\theta)| > \frac{p}{4R}\right\}\). Note that \(\mathcal{K}(\theta) \neq \emptyset\) by definition of \(\Upsilon_{\mathcal{II}}\), and consider that \(|\mathcal{K}(\theta)| \leq \frac{s}{\frac{p}{4R}} \leq \frac{4Rs}{p} \leq 4\) since \(s < \frac{p}{R}\). Consequently, by triangle inequality there exists \(1 \leq k^*(\theta) \leq R\) such that \(||\bar{\theta}_{B_{k^*(\theta)}}\mathbf{1}_{B_{k^*(\theta)}}||^2 \geq \frac{C^2\rho^2}{32}\). Consider then that the type II error is bounded as 
        \begin{align*}
            &\sup_{\theta \in \Upsilon_{\mathcal{II}}(p, s, C\rho)} P_{\theta, \gamma, R}\left\{ \varphi_{\frac{C^2}{128}\log(eR)}^{\chi^2-\text{scan}} = 0 \right\} \\
            &\leq \sup_{\theta \in \Upsilon_{\mathcal{II}}(p, s, C\rho)} P_{\theta, \gamma, R}\left\{ ||\widetilde{X}_{B_{k^*(\theta)}}||^2 \leq \frac{p}{R} + \frac{2C}{\sqrt{128}}\sqrt{\frac{p}{R} \log(eR)} + \frac{C^2}{64} \log(eR) \right\}.
        \end{align*}
        Consider that under the data-generating process \(P_{\theta, \gamma, R}\), we have \(||\widetilde{X}_{B_{k^*(\theta)}}||^2 \sim \chi^2_{p/R}\left(\frac{||\theta_{B_{k^*(\theta)}} - \bar{\theta}_{B_{k^*(\theta)}} \mathbf{1}_{B_{k^*(\theta)}}||^2}{1-\gamma} \right)\). Consequently, 
        \begin{align*}
            E_{\theta, \gamma, R}\left(||\widetilde{X}_{B_{k^*(\theta)}}||^2\right) &= \frac{p}{R} + \frac{||\theta_{B_{k^*(\theta)}} - \bar{\theta}_{B_{k^*(\theta)}} \mathbf{1}_{B_{k^*(\theta)}}||^2}{1-\gamma}, \\
            \Var_{\theta, \gamma, R}\left(||\widetilde{X}_{B_{k^*(\theta)}}||^2\right) &= 2\frac{p}{R} + 4\frac{||\theta_{B_{k^*(\theta)}} - \bar{\theta}_{B_{k^*(\theta)}} \mathbf{1}_{B_{k^*(\theta)}}||^2}{1-\gamma}. 
        \end{align*}
        Since \(||\bar{\theta}_{B_{k^*(\theta)}}\mathbf{1}_{B_{k^*(\theta)}}||^2 \geq \frac{C^2\rho^2}{32}\) and \(||\theta_{B_{k^*(\theta)}}||_0 \leq s\), we have by Corollary \ref{corollary:orthog_approximation}
        \begin{align*}
            \frac{||\theta_{B_{k^*(\theta)}} - \bar{\theta}_{B_{k^*(\theta)}} \mathbf{1}_{B_{k^*}(\theta)}||^2}{1-\gamma} \geq \frac{||\bar{\theta}_{B_{k^*(\theta)}}\mathbf{1}_{B_{k^*(\theta)}}||^2}{1-\gamma} \cdot \frac{\frac{p}{R} - s}{\frac{p}{R}} \geq \frac{C^2 \rho^2}{32(1-\gamma)} \cdot \frac{\frac{p}{R} - s}{\frac{p}{R}}.
        \end{align*}
        Plugging \(\rho^2 = \frac{(1-\gamma)\frac{p}{R}}{\frac{p}{R}-s}\left( \sqrt{\frac{p}{R}\log(eR)} + \log(R)\right)\) into the bound yields
        \begin{equation}\label{eqn:chisquare_scan_expectation}
            \frac{||\theta_{B_{k^*(\theta)}} - \bar{\theta}_{B_{k^*(\theta)}} \mathbf{1}_{B_{k^*(\theta)}}||^2}{1-\gamma} \geq \frac{C^2}{32} \left(\sqrt{\frac{p}{R}\log(eR)} + \log(R) \right). 
        \end{equation}

        Therefore, using that \(C > C_\eta \geq \sqrt{128}\) we have 
        \begin{align*}
            &\sup_{\theta \in \Upsilon_{\mathcal{II}}(p, s, C\rho)} P_{\theta, \gamma, R}\left\{ ||\widetilde{X}_{B_{k^*(\theta)}}||^2 \leq \frac{p}{R} + \frac{2C}{\sqrt{128}}\sqrt{\frac{p}{R} \log(eR)} + \frac{C^2}{64}\log(R) \right\} \\
            &\leq \sup_{\theta \in \Upsilon_{\mathcal{II}}(p, s, C\rho)} P_{\theta, \gamma, R}\left\{ ||\widetilde{X}_{B_{k^*(\theta)}}||^2 \leq \frac{p}{R} + \frac{C^2}{64}\left(\sqrt{\frac{p}{R} \log(eR)} + \log(R)\right) \right\} \\
            &= \sup_{\theta \in \Upsilon_{\mathcal{II}}(p, s, C\rho)} P_{\theta, \gamma, R}\left\{ \frac{||\theta_{B_{k^*(\theta)}} - \bar{\theta}_{B_{k^*(\theta)}} \mathbf{1}_{B_{k^*(\theta)}}||^2}{1-\gamma} - \frac{C^2}{64}\left(\sqrt{\frac{p}{R} \log(eR)} + \log(R)\right) \right. \\
            &\;\;\;\;\;\; \left. \leq \frac{p}{R} + \frac{||\theta_{B_{k^*(\theta)}} - \bar{\theta}_{B_{k^*(\theta)}} \mathbf{1}_{B_{k^*(\theta)}}||^2}{1-\gamma} - ||\widetilde{X}_{B_{k^*(\theta)}}||^2 \right\} \\
            &\leq \sup_{\theta \in \Upsilon_{\mathcal{II}}(p, s, C\rho)} \frac{\Var_{\theta, \gamma, R}\left(||\widetilde{X}_{B_{k^*(\theta)}}||^2\right)}{\left(\frac{||\theta_{B_{k^*(\theta)}} - \bar{\theta}_{B_{k^*(\theta)}} \mathbf{1}_{B_{k^*(\theta)}}||^2}{1-\gamma} - \frac{C^2}{64}\left(\sqrt{\frac{p}{R} \log(eR)} + \log(R)\right) \right)^2}
        \end{align*}
        Now applying the bound (\ref{eqn:chisquare_scan_expectation}) enables us to continue the calculation
        \begin{align*}
            &\sup_{\theta \in \Upsilon_{\mathcal{II}}(p, s, C\rho)} \frac{\Var_{\theta, \gamma, R}\left(||\widetilde{X}_{B_{k^*(\theta)}}||^2\right)}{\left(\frac{||\theta_{B_{k^*(\theta)}} - \bar{\theta}_{B_{k^*(\theta)}} \mathbf{1}_{B_{k^*(\theta)}}||^2}{1-\gamma} - \frac{C^2}{64}\left(\sqrt{\frac{p}{R} \log(eR)} + \log(R)\right) \right)^2} \\
            &\leq \frac{2\frac{p}{R}}{\frac{C^4}{64^2} \left(\sqrt{\frac{p}{R} \log(eR)} + \log(R)\right)^2} + \sup_{\theta \in \Upsilon_{\mathcal{II}}(p, s, C\rho)} \frac{4\frac{||\theta_{B_{k^*(\theta)}} - \bar{\theta}_{B_{k^*(\theta)}} \mathbf{1}_{B_{k^*(\theta)}}||^2}{1-\gamma}}{\frac{1}{4}\left(\frac{||\theta_{B_{k^*(\theta)}} - \bar{\theta}_{B_{k^*(\theta)}} \mathbf{1}_{B_{k^*(\theta)}}||^2}{1-\gamma}\right)^2} \\
            &\leq \frac{64^2 \cdot 2}{C_\eta^4} + \sup_{\theta \in \Upsilon_{\mathcal{II}}(p, s, C\rho)} \frac{16}{\frac{||\theta_{B_{k^*(\theta)}} - \bar{\theta}_{B_{k^*(\theta)}} \mathbf{1}_{B_{k^*(\theta)}}||^2}{1-\gamma}} \\
            &\leq \frac{64^2 \cdot 2}{C_\eta^4} + \frac{32 \cdot 16}{C^2 \left(\sqrt{\frac{p}{R} \log(eR)} + \log(R)\right)} \\
            &\leq \frac{64^2 \cdot 2}{C_\eta^4} + \frac{32 \cdot 16}{C_\eta^2} \\
            &\leq \frac{\eta}{2}
        \end{align*}
        where we have applied (\ref{eqn:chisquare_scan_expectation}) again to obtain the third-to-last inequality. Hence, we have shown that the sum of the type I and type II errors is bounded as 
        \begin{equation*}
            P_{0, \gamma, R}\left\{\varphi_{\frac{C^2}{128}\log(eR)}^{\chi^2-\text{scan}} = 1 \right\} + \sup_{\theta \in \Upsilon_{\mathcal{II}}(p, s, C\rho)} P_{\theta, \gamma, R}\left\{ \varphi_{\frac{C^2}{128}\log(eR)}^{\chi^2-\text{scan}} = 0 \right\} \leq \eta. 
        \end{equation*}
        Since \(C > C_\eta\) was arbitrary and \(\eta \in (0, 1)\) was arbitrary, the desired result has been proved. 
    \end{proof}
    
    \begin{proof}[Proof of Proposition \ref{prop:ubound_pR_dense}]
        Fix \(\eta \in (0, 1)\) and set \(C_{\eta} := C_{\eta, 1} \vee C_{\eta, 2} \vee C_{\eta, 3} \vee C_{\eta, 4}\) where \(C_{\eta, 1}, C_{\eta, 2}, C_{\eta, 3},\) and \(C_{\eta, 4}\) are the constants at level \(\eta/2\) from Lemmas \ref{lemma:test_upsilon1_sparse}, \ref{lemma:test_upsilon1_dense}, \ref{lemma:chisquare_scan}, and \ref{lemma:linear_scan} respectively. By Proposition \ref{prop:theta_upsilon}, we have 
        \begin{equation*}
            \Theta(p, s, C(\psi_1 \vee \upsilon)) \subset \Upsilon_{\mathcal{I}}(p, s, C(\psi_1 \vee \upsilon)) \cup \Upsilon_{\mathcal{II}}(p, s, C(\psi_1 \vee \upsilon)) \subset \Upsilon_{\mathcal{I}}(p, s, C\psi_1) \cup \Upsilon_{\mathcal{II}}(p, s, C\upsilon). 
        \end{equation*}
        Therefore, 
        \begin{align*}
            &P_{0, \gamma, R}\{\varphi^* = 1\} + \sup_{\theta \in \Theta(p, s, C(\psi_1 \vee \upsilon))} P_{\theta, \gamma, R}\{\varphi^* = 0\} \\
            &\leq P_{0, \gamma, R}\{\varphi^* = 1\} + \sup_{\theta \in \Upsilon_{\mathcal{I}}(p, s, C\psi_1) \cup \Upsilon_{\mathcal{II}}(p, s, C\upsilon)} P_{\theta, \gamma, R}\{\varphi_1 \vee \varphi_2 = 0\} \\
            &= P_{0, \gamma, R}\{\varphi^* = 1\} + \left[\sup_{\theta \in \Upsilon_{\mathcal{I}}(p, s, C\psi_1)} P_{\theta, \gamma, R}\{\varphi_1 \vee \varphi_2 = 0\}\right] + \left[\sup_{\theta \in \Upsilon_{\mathcal{II}}(p, s, C\upsilon)} P_{\theta, \gamma, R}\{\varphi_1 \vee \varphi_2 = 0\}\right] \\
            &\leq \left[P_{0, \gamma, R}\{\varphi_1 = 1\} +  \sup_{\theta \in \Upsilon_{\mathcal{I}}(p, s, C\psi_1)} P_{\theta, \gamma, R}\{\varphi_1 = 0\} \right] + \left[P_{0, \gamma, R}\{\varphi_2 = 1\} + \sup_{\theta \in \Upsilon_{\mathcal{II}}(p, s, C\upsilon)} P_{\theta, \gamma, R}\{\varphi_2 = 0\}\right] \\
            &\leq \eta.
        \end{align*}
        where the final inequality follows from \(C > C_\eta\) combined with Lemmas \ref{lemma:test_upsilon1_sparse}, \ref{lemma:test_upsilon1_dense}, \ref{lemma:chisquare_scan}, and \ref{lemma:linear_scan}. Since \(C > C_\eta\) was arbitrary and \(\eta \in (0, 1)\) was arbitrary, the desired result has been proved. 
    \end{proof}

    \subsubsection{\texorpdfstring{Regime \(\frac{p}{R} - \sqrt{\frac{p}{R}\log(eR)} < s < \frac{p}{R}\)}{Regime p/R - sqrt(p log(eR)/R) < s < p/R}}

    We begin by establishing the sensitivity of \(\varphi_{t, r}^{\text{scan}}\) given by (\ref{test:pR_very_dense_tsybakov}) to \(\Upsilon_{\mathcal{II}}(p, s, \varepsilon)\).

    \begin{lemma}\label{lemma:tsybakov_scan}
        Suppose \(\frac{p}{R} - \sqrt{\frac{p}{R}\log(eR)} < s < \frac{p}{R}\) and \(\gamma \in [0, 1)\). Set
        \begin{equation*}
            \rho^2 = \frac{(1-\gamma)p}{p-Rs} \left(\left(\frac{p}{R} - s\right) \log\left(1 + \frac{Rp \log(eR)}{(p-Rs)^2}\right) + \log(R)\right).
        \end{equation*}
        If \(\eta \in (0, 1)\), then there exists a constant \(C_\eta > 0\) depending only on \(\eta\) such that for all \(C > C_\eta\), the testing procedure \(\varphi_{\widetilde{t}, \widetilde{r}}^{\text{scan}}\) given by (\ref{test:pR_very_dense_tsybakov}) with \(\widetilde{t} = \sqrt{2 \log\left( 1 + \frac{Rp \log(eR)}{(p-Rs)^2}\right)}\) and \(\widetilde{r} = \frac{C^2}{64} \left(\left(\frac{p}{R}-s\right) \log\left(1 + \frac{Rp \log(eR)}{(p-Rs)^2}\right) + \log(R)\right)\) satisfies 
        \begin{equation*}
            P_{0, \gamma, R} \left\{ \varphi_{\widetilde{t}, \widetilde{r}}^{\text{scan}} = 1 \right\} + \sup_{\theta \in \Upsilon_{\mathcal{II}}(p, s, C\rho)} P_{0, \gamma, R} \left\{ \varphi_{\widetilde{t}, \widetilde{r}}^{\text{scan}} = 0 \right\} \leq \eta.
        \end{equation*}
    \end{lemma}
    \begin{proof}
        Fix \(\eta \in (0, 1)\) and set 
        \begin{equation*}
            C_\eta := 24 \vee \sqrt{64 \cdot 9} \vee C^*_{\eta}
        \end{equation*}
        where \(C^*_\eta\) is a constant depending only on \(\eta\) to be defined later. Let \(C > C_\eta\). We first bound the type I error. Consider that 
        \begin{align*}
            9 \cdot \left(\sqrt{\frac{p}{R} e^{-(t^*)^2/2} \log(eR) \cdot \frac{C^2}{64 \cdot 9}} + \frac{C^2}{64 \cdot 9} \log(eR) \right) &= \frac{3C}{8} \sqrt{\frac{p}{R} \frac{\left(\frac{p}{R}-s\right)^2}{\frac{p}{R}\log(eR) + \left(\frac{p}{R}-s\right)^2} \cdot \log(eR)} + \frac{C^2}{64} \log(eR) \\
            &\leq \frac{3C}{8} \left(\frac{p}{R} - s\right) + \frac{C^2}{64} \log(eR) \\
            &\leq \widetilde{r}
        \end{align*}
        where we have used that \(\frac{3C}{8} \leq \frac{C^2}{64}\) since \(C > C_\eta \geq 24\). The type I error is bounded via an application of Lemma \ref{lemma:tsybakov_typeI_error}
        \begin{align*}
            P_{0, \gamma, R}\left\{\varphi_{\widetilde{t}, \widetilde{r}} = 1 \right\} &\leq \sum_{k=1}^{R} P_{0, \gamma, R}\left\{Y_t^{(k)} > \widetilde{r} \right\} \\
            &\leq \sum_{k=1}^{R} P_{0, \gamma, R}\left\{Y_t^{(k)} > 9 \cdot \left(\sqrt{\frac{p}{R} e^{-(t^*)^2/2} \log(eR) \cdot \frac{C^2}{64 \cdot 9}} + \frac{C^2}{64 \cdot 9} \log(eR) \right) \right\} \\
            &\leq R \exp\left(-\frac{C^2}{64 \cdot 9}\log(eR)\right) \\
            &\leq \exp\left(\log(R)\left(1 - \frac{C^2}{64 \cdot 9}\right) - \frac{C^2}{64 \cdot 9}\right) \\
            &\leq \exp\left(-\frac{C_\eta^2}{64 \cdot 9}\right) \\
            &\leq \frac{\eta}{2}
        \end{align*}
        where we have used \(\frac{C^2}{64 \cdot 9} > \frac{C_\eta^2}{64 \cdot 9} \geq 1\). 
        
        We now turn our attention to bounding the type II error. For \(\theta \in \Upsilon_{\mathcal{II}}(p, s, C\rho)\), let \(\mathcal{K}(\theta) := \left\{k \subset [R] : |B_k \cap \supp(\theta)| > \frac{p}{4R}\right\}\). Note that \(\mathcal{K}(\theta) \neq \emptyset\) by definition of \(\Upsilon_{\mathcal{II}}\), and consider that \(|\mathcal{K}(\theta)| \leq \frac{s}{\frac{p}{4R}} \leq \frac{4Rs}{p} \leq 4\) since \(s < \frac{p}{R}\). Consequently, by triangle inequality there exists \(1 \leq k^*(\theta) \leq R\) such that \(||\bar{\theta}_{B_{k^*(\theta)}}\mathbf{1}_{B_{k^*(\theta)}}||^2 \geq \frac{C^2\rho^2}{32}\). Since \(||\theta||_0 \leq s\) implies \(||\theta_{B_{k^*(\theta)}}||_0 \leq s\), Lemma \ref{lemma:v_orthog_approximation} yields 
        \begin{equation}\label{eqn:pR_very_dense_orthog_bound}
            \frac{||\theta_{B_{k^*(\theta)}} - \bar{\theta}_{B_{k^*(\theta)}}\mathbf{1}_{B_{k^*(\theta)}}||^2}{1-\gamma} \geq \frac{C^2}{32} \cdot \left(\left(\frac{p}{R}-s\right) \log\left(1 + \frac{Rp \log(eR)}{\left(p-Rs\right)^2}\right) + \log(R)\right) = 2\widetilde{r}.
        \end{equation}
        With (\ref{eqn:pR_very_dense_orthog_bound}) in hand, the type II error analysis in the proof of Proposition \ref{prop:supp_tsybakov} can be repeated with appropriate and slight modifications to obtain the type II error bound \(\sup_{\theta \in \Upsilon_{\mathcal{II}}(p, s, C\rho)} P_{\theta, \gamma, R}\left\{\varphi_{\widetilde{t}, \widetilde{r}} = 0 \right\} \leq \frac{\eta}{2}\). The constant \(C_\eta^*\) is the appropriate constant depending only on \(\eta\) such that \(C > C_\eta\) implies this type II error bound. Thus, the sum of the type I and type II errors is bounded as 
        \begin{equation*}
            P_{0, \gamma, R}\left\{\varphi_{\widetilde{t}, \widetilde{r}} = 1 \right\} + \sup_{\theta \in \Upsilon_{\mathcal{II}}(p, s, C\rho)} P_{\theta, \gamma, R}\left\{\varphi_{\widetilde{t}, \widetilde{r}} = 0 \right\} \leq \eta.
        \end{equation*}
        Since \(C > C_\eta\) was arbitrary and \(\eta \in (0, 1)\) was arbitrary, the desired result has been proved.
    \end{proof}

    \begin{proof}[Proof of Proposition \ref{prop:ubound_pR_very_dense}]
        Fix \(\eta \in (0, 1)\) and set \(C_{\eta} := C_{\eta, 1} \vee C_{\eta, 2} \vee C_{\eta, 3} \vee C_{\eta, 4}\) where \(C_{\eta, 1}, C_{\eta, 2}, C_{\eta, 3},\) and \(C_{\eta, 4}\) are the constants at level \(\eta/2\) from Lemmas \ref{lemma:test_upsilon1_sparse}, \ref{lemma:test_upsilon1_dense}, \ref{lemma:tsybakov_scan}, and \ref{lemma:linear_scan} respectively. Repeating the argument in the proof of Proposition \ref{prop:ubound_pR_dense} yields the desired result.
    \end{proof} 

    \subsubsection{\texorpdfstring{Regime \(\frac{p}{R} \leq s \leq p\)}{Regime p/R <= s <= p}} 
    
    We first establish the sensitivity of \(\bar{\varphi}_{t, r}\) given by (\ref{test:average_tsybakov}) to the space \(\Upsilon_{\mathcal{II}}(p, s, \varepsilon)\) in the regime \(\frac{p}{R} \leq s < \frac{p}{\sqrt{R}}\).

    \begin{lemma}\label{lemma:test_upsilon2_avg_sparse}
        Suppose \(\frac{p}{R} \leq s < \frac{p}{\sqrt{R}}\) and \(\gamma \in [0, 1]\). Let \(\rho^2 = \left(1-\gamma+\gamma \frac{p}{R}\right)\frac{4Rs}{p}\log\left(1 + \frac{p^2}{16Rs^2}\right)\). If \(\eta \in (0, 1)\), then there exists a constant \(C_\eta > 0\) depending only on \(\eta\) such that for all \(C > C_\eta\), the testing procedure \(\bar{\varphi}_{\bar{t}, \bar{r}}\) given by (\ref{test:average_tsybakov}) with \(\bar{t} = \sqrt{2\log\left(1 + \frac{p^2}{Rs^2}\right)}\) and \(\bar{r} = \frac{C^2}{64} \frac{4Rs}{p} \log\left(1 + \frac{p^2}{16Rs^2}\right)\) satisfies 
        \begin{equation*}
            P_{0, \gamma, R}\{\bar{\varphi}_{\bar{t}, \bar{r}} = 1\} + \sup_{\theta \in \Upsilon_{\mathcal{II}}(p, s, C\rho)}  P_{\theta, \gamma, R}\{\bar{\varphi}_{\bar{t}, \bar{r}} = 0\} \leq \eta
        \end{equation*}
        where \(\Upsilon_{\mathcal{II}}(p, s, C\rho)\) is given by (\ref{space:Upsilon2}).
    \end{lemma}
    \begin{proof}
        Fix \(\eta \in (0, 1)\) and set \(C_\eta := \sqrt{8}C_{\eta}^*\) where \(C_{\eta}^*\) is given in Proposition \ref{prop:supp_tsybakov}. Let \(C > C_\eta\). For ease of notation, set \(\sigma^2 := 1-\gamma+\gamma \frac{p}{R}\).  Notice that under the data-generating process \(P_{\theta, \gamma, R}\), we have 
        \begin{equation*}
            \sqrt{\frac{p}{R}} \bar{X}_{B_k} \sim N\left( \sqrt{\frac{p}{R}}\bar{\theta}_{B_k}, \sigma^2\right)
        \end{equation*}
        and \(\{\bar{X}_{B_k}\}_{k=1}^{R}\) are independent. For \(\theta \in \Upsilon_{\mathcal{II}}(p, s, C\rho)\), set \(\mathcal{K}(\theta) := \left\{k \in [R] : |B_k \cap \supp(\theta)| > \frac{p}{4R}\right\}\). Note that 
        \begin{equation}\label{eqn:effective_sparsity}
            |\mathcal{K}(\theta)| \leq \frac{s}{\frac{p}{4R}} = \frac{4Rs}{p}.
        \end{equation} 
        Let \(Y \in \R^{R}\) denote the vector obtained by concatenation \(Y = \left(\sigma^{-1}\sqrt{\frac{p}{R}}\bar{X}_{B_1},...,\sigma^{-1}\sqrt{\frac{p}{R}}\bar{X}_{B_R} \right)\). Let \(\mu\) denote the mean of \(Y\) and note \(Y \sim N(\mu, I_R)\). Note that (\ref{eqn:effective_sparsity}) and \(s < \frac{p}{\sqrt{R}}\) gives \(|\mathcal{K}(\theta)| \leq \frac{4Rs}{p} < 4\sqrt{R}\). Note also that since \(\theta \in \Upsilon_{\mathcal{II}}(p, s, C\rho)\), we have by definition of \(\Upsilon_{\mathcal{II}}(p, s, C\rho)\)
        \begin{align*}
            \left|\left|\mu_{\mathcal{K}(\theta)}\right|\right|^2 &= \sum_{k \in \mathcal{K}(\theta)} \frac{||\bar{\theta}_{B_k}\mathbf{1}_{B_k}||^2}{\sigma^2} \\
            &\geq \frac{C^2\rho^2}{8 \sigma^2} \\
            &= \frac{C^2}{8} \frac{4Rs}{p}\log\left(1 + \frac{p^2}{16Rs^2}\right) \\
            &= \frac{C^2}{8} \frac{4Rs}{p} \log\left(1 + \frac{R}{\left(\frac{4Rs}{p}\right)^2}\right).
        \end{align*}
        In other words, we have shown that \(\theta \in \Upsilon_{\mathcal{II}}(p, s, C\rho)\) implies \(\mu \in \mathscr{M}\left(R, \frac{4Rs}{p}, \frac{C}{\sqrt{8}} \sqrt{\frac{4Rs}{p}\log\left(1 + \frac{R}{\left(\frac{4Rs}{p}\right)^2} \right)}\right)\) where the latter parameter space is given by (\ref{space:supp_space}). Since \(\frac{4Rs}{p} < 4\sqrt{R}\), \(\frac{C}{\sqrt{8}} > \frac{C_\eta}{\sqrt{8}} \geq C_{\eta}^*\), and \(\bar{r} = \frac{C^2/8}{8} \frac{4Rs}{p} \log\left(1 + \frac{p^2}{16Rs^2}\right)\), we can apply Proposition \ref{prop:supp_tsybakov} to obtain 
        \begin{equation*}
            P_{0, \gamma, R}\{ \bar{\varphi}_{\bar{t}, \bar{r}} = 1\} + \sup_{\theta \in \Upsilon_{\mathcal{II}}(p, s, C\rho)} P_{\theta, \gamma, R}\left\{ \bar{\varphi}_{\bar{t}, \bar{r}} = 0\right\} \leq \eta. 
        \end{equation*}
        Since \(C > C_\eta\) was arbitrary and \(\eta \in (0, 1)\) was arbitrary, the desired result has been proved.
    \end{proof}

    We now establish the sensitivity of \(\bar{\varphi}_{r}^{\chi^2}\) to \(\Upsilon_{\mathcal{II}}(p, s, \varepsilon)\). 
    \begin{lemma}\label{lemma:test_upsilon2_avg_dense}
        Suppose \(1 \leq s \leq p\) and \(\gamma \in [0, 1]\). Let \(\rho^2 := \left(1 - \gamma +\gamma \frac{p}{R}\right) \sqrt{R}\). If \(\eta \in (0, 1)\), then there exists a constant \(C_\eta > 0\) depending only on \(\eta\) such that for all \(C > C_\eta\) the testing procedure \(\bar{\varphi}_{C^2/16}^{\chi^2}\) given by (\ref{test:average_chisquare}) satisfies 
        \begin{equation*}
            P_{0, \gamma, R}\left\{ \bar{\varphi}_{C^2/16}^{\chi^2} = 1 \right\} + \sup_{\theta \in \Upsilon_{\mathcal{II}}(p, s, C\rho)} P_{\theta, \gamma, R}\left\{\bar{\varphi}_{C^2/16}^{\chi^2} = 0\right\} \leq \eta. 
        \end{equation*}
    \end{lemma}
    \begin{proof}
        Fix \(\eta \in (0, 1)\) and set \(C_\eta := \left(\frac{512 \cdot 4}{\eta}\right)^{1/4} \vee \left( \frac{128 \cdot 4}{\eta} \right)^{1/2}\). Let \(C > C_\eta\). For ease of notation, set \(\sigma^2 = 1-\gamma+\gamma \frac{p}{R}\). Under the data-generating process \(P_{\theta, \gamma, R}\), we have \(\bar{X}_{B_k}\mathbf{1}_{B_k} \sim N\left(\bar{\theta}_{B_k}\mathbf{1}_{B_k}, \sigma^2 \frac{R}{p} \mathbf{1}_{B_k}\mathbf{1}_{B_k}^\intercal \right)\). Since \(\{\bar{X}_{B_k}\}_{k=1}^{R}\) are independent, it follows that 
        \begin{equation*}
            \sum_{k=1}^{R} ||\bar{X}_{B_k}\mathbf{1}_{B_k}||^2 \sim \sigma^2 \chi^2_{R}\left(\frac{\sum_{k=1}^{R} ||\bar{\theta}_{B_k}\mathbf{1}_{B_k}||^2}{\sigma^2} \right).
        \end{equation*}
        Thus \(E_{\theta, \gamma, R}\left( \sum_{k=1}^{R} ||\bar{X}_{B_k}\mathbf{1}_{B_k}||^2 \right) = R\sigma^2 + \sum_{k=1}^{R} ||\bar{\theta}_{B_k}\mathbf{1}_{B_k}||^2\) and \(\Var_{\theta, \gamma, R}\left( \sum_{k=1}^{R} ||\bar{X}_{B_k}\mathbf{1}_{B_k}||^2 \right) = 2R\sigma^4 + 4\sigma^2 \sum_{k=1}^{R} ||\bar{\theta}_{B_k}\mathbf{1}_{B_k}||^2\). The type I error is thus bounded as 
    \begin{align*}
        P_{0, \gamma, R}\left\{\bar{\varphi}_{C^2/16}^{\chi^2} = 1 \right\} &= P_{0, \gamma, R}\left\{\sum_{k=1}^{R} ||\bar{X}_{B_k}\mathbf{1}_{B_k}||^2 > \sigma^2\left(R + \frac{C^2}{16}\sqrt{R}\right) \right\} \\
        &\leq \frac{\Var_{0, \gamma, R}\left( \sum_{k=1}^{R} ||\bar{X}_{B_k}\mathbf{1}_{B_k}||^2 \right)}{\frac{C^4}{16^2} R\sigma^4 } \\
        &\leq \frac{512}{C^4} \\
        &\leq \frac{512}{C_\eta^4} \\
        &\leq \frac{\eta}{2}. 
    \end{align*}
    Turning our attention to the type II error, consider 
    \begin{align*}
        &\sup_{\theta \in \Upsilon_{\mathcal{II}}(p, s, C\rho)} P_{\theta, \gamma, R} \left\{ \bar{\varphi}_{C^2/16}^{\chi^2} = 0 \right\} \\
        &= \sup_{\theta \in \Upsilon_{\mathcal{II}}(p, s, C\rho)} P_{\theta, \gamma, R} \left\{ \sum_{k=1}^{R} ||\bar{X}_{B_k}\mathbf{1}_{B_k}||^2 \leq \sigma^2\left(R + \frac{C^2}{16}\sqrt{R}\right) \right\} \\
        &= \sup_{\theta \in \Upsilon_{\mathcal{II}}(p, s, C\rho)} P_{\theta, \gamma, R} \left\{\sum_{k=1}^{R} ||\bar{\theta}_{B_k}\mathbf{1}_{B_k}||^2 - \frac{C^2}{16}\sigma^2 \sqrt{R} \leq R\sigma^2 + \sum_{k=1}^{R} ||\bar{\theta}_{B_k}\mathbf{1}_{B_k}||^2 - \sum_{k=1}^{R} ||\bar{X}_{B_k}\mathbf{1}_{B_k}||^2 \right\} \\
        &\leq \sup_{\theta \in \Upsilon_{\mathcal{II}}(p, s, C\rho)} \frac{\Var_{\theta, \gamma, R}\left(\sum_{k=1}^{R} ||\bar{X}_{B_k}\mathbf{1}_{B_k}||^2 \right)}{\left(\sum_{k=1}^{R} ||\bar{\theta}_{B_k}\mathbf{1}_{B_k}||^2 - \frac{C^2}{16}\sigma^2 \sqrt{R} \right)^2} \\
        &\leq \frac{2R\sigma^4}{\frac{C^4}{16^2}\rho^2} + \sup_{\theta \in \Upsilon_{\mathcal{II}}(p, s, C\rho)} \frac{4\sigma^2 \sum_{k=1}^{R}||\bar{\theta}_{B_k}\mathbf{1}_{B_k}||^2}{\frac{1}{4}\left(\sum_{k=1}^{R}||\bar{\theta}_{B_k}\mathbf{1}_{B_k}||^2\right)^2} \\
        &= \frac{512}{C^4} + \sup_{\theta \in \Upsilon_{\mathcal{II}}(p, s, C\rho)} \frac{16\sigma^2}{\sum_{k=1}^{R} ||\bar{\theta}_{B_k}\mathbf{1}_{B_k}||^2} \\
        &\leq \frac{512}{C^4} + \frac{16 \cdot 8}{C^2\sqrt{R}} \\
        &\leq \frac{512}{C_\eta^4} + \frac{128}{C_\eta^2} \\
        &\leq \frac{\eta}{2}.
    \end{align*}
    Hence, we have shown that the sum of type I and type II errors is bounded as 
    \begin{equation*}
        P_{0, \gamma, R}\left\{\bar{\varphi}_{C^2/16}^{\chi^2} = 1 \right\} + \sup_{\theta \in \Upsilon_{\mathcal{II}}(p, s, C\rho)} P_{\theta, \gamma, R} \left\{ \bar{\varphi}_{C^2/16}^{\chi^2} = 0 \right\} \leq \eta. 
    \end{equation*}
    Since \(C > C_\eta\) was arbitrary and \(\eta \in (0, 1)\) was arbitrary, the desired result has been proved.
\end{proof}

\begin{proof}[Proof of Proposition \ref{prop:R_avg_ubound}]
    Fix \(\eta \in (0, 1)\) and set \(C_{\eta} = C_{\eta, 1} \vee C_{\eta, 2} \vee C_{\eta, 3} \vee C_{\eta, 4}\) where \(C_{\eta, 1}, C_{\eta, 2}, C_{\eta, 3}, C_{\eta, 4}\) are the constants at level \(\eta/2\) from Lemmas \ref{lemma:test_upsilon1_dense}, \ref{lemma:test_upsilon1_sparse}, \ref{lemma:test_upsilon2_avg_sparse} and \ref{lemma:test_upsilon2_avg_dense} respectively. Repeating the argument in the proof of Proposition \ref{prop:ubound_pR_dense} yields the desired result.
\end{proof}

    \subsection{Proof of result in Section \ref{section:R_perfect_correlation}}
    
    \begin{lemma}\label{lemma:nontrivial_projection}
        For \(1 \leq s \leq p\) define the space 
        \begin{equation}\label{defn:theta_dag}
            \Theta^\dag(p, s, R) := \left\{ \theta \in \R^p\setminus\{0\} : ||\theta||_0 \leq s \text{ and } \theta_{B_k} - \bar{\theta}_{B_k}\mathbf{1}_{B_k} \neq 0 \text{ for some } k \in [R] \right\}.
        \end{equation}
        If \(1 \leq s < \frac{p}{R}\) and \(\varepsilon > 0\), then \(\Theta(p, s, \varepsilon) \subset \Theta^\dag(p, s, R)\). 
    \end{lemma}
    \begin{proof}
        Suppose \(\theta \in \Theta(p, s, \varepsilon)\). Since \(||\theta|| > \varepsilon\) and \(\{B_k\}_{k=1}^{R}\) are a partition of \([p]\), it follows that \(||\theta_{B_k}|| > 0\) for some \(k \in [R]\). Since \(s < \frac{p}{R}\), it immediately follows \(\theta_{B_k} \not \in \spn\left\{\mathbf{1}_{B_k}\right\}\). Thus we must have \(\theta_{B_k} - \bar{\theta}_{B_k}\mathbf{1}_{B_k} \neq 0\), and so \(\theta \in \Theta^\dag(p, s, R)\). 
    \end{proof}

    \begin{proposition}\label{prop:R_perfect_correlation_sparse}
        If \(1 \leq s < \frac{p}{R}\) and \(\gamma = 1\), then the testing procedure
        \begin{equation}\label{test:R_noiseless}
            \varphi^\dag := \mathbf{1}_{\left\{\sum_{k=1}^{R} ||X_{B_k} - \bar{X}_{B_k}\mathbf{1}_{B_k}||^2 > 0 \right\}}
        \end{equation}
        satisfies \(P_{0, 1, R}\left\{ \varphi^\dag = 1 \right\} + \sup_{\theta \in \Theta^\dag(p, s, R)} P_{\theta, 1, R}\left\{\varphi^\dag = 0\right\} = 0\).
    \end{proposition}
    \begin{proof}
        Let us examine the data-generating process \(P_{\theta, 1, R}\) for \(\theta \in \R^p\). Since \(\gamma = 1\), consider from (\ref{model:multiple_random_effects_additive}) we can write for \(k \in [R]\)
        \begin{equation*}
            X_{B_k} = \theta_{B_k} + W_k\mathbf{1}_{B_k}
        \end{equation*}
        where \(W_1,...,W_R \overset{iid}{\sim} N(0, 1)\). Therefore, we have \(X_{B_{k}} - \bar{X}_{B_{k}} \mathbf{1}_{B_{k}} = 0\) for all \(k \in [R]\) almost surely. Hence, the type I error satisfies \(P_{0, 1, R}\left\{ \varphi^\dag = 1 \right\} = 0\). Now, focusing on \(\theta \in \Theta^\dag(p, s, \varepsilon)\), by definition there exists some \(k(\theta) \in [R]\) such that 
        \begin{equation*}
            \theta_{B_{k(\theta)}} - \bar{\theta}_{B_{k(\theta)}}\mathbf{1}_{B_{k(\theta)}} \neq 0.
        \end{equation*}
        Therefore, \(\sum_{k=1}^{R} ||X_{B_k} - \bar{X}_{B_k}\mathbf{1}_{B_k}||^2 > 0\) almost surely under the data-generating process \(P_{\theta, 1, R}\) when \(\theta \in \Theta^\dag(p, s, R)\). Consequently, the type II error satisfies \(\sup_{\theta \in \Theta^\dag(p, s, R)} P_{\theta, 1, R}\{\varphi^\dag = 0\} = 0\). Thus the desired result has been proved.
    \end{proof}

    \begin{proposition}\label{prop:R_perfect_correlation_ubound}
        Suppose \(1 \leq s \leq p\) and \(\gamma = 1\). Set
        \begin{equation*}
            \rho^2 :=
            \begin{cases}
                0 &\text{if } s < \frac{p}{R}, \\
                4s\log\left(1 + \frac{p^2}{16Rs^2}\right) &\text{if } \frac{p}{R} \leq s < \frac{p}{\sqrt{R}}, \\
                \frac{p}{\sqrt{R}} &\text{if } \frac{p}{\sqrt{R}} \leq s \leq p. 
            \end{cases}
        \end{equation*}
        If \(\eta \in (0, 1)\), then there exists a constant \(C_\eta > 0\) depending only on \(\eta\) such that for all \(C > C_\eta\) the testing procedure \(\varphi^* = \varphi^\dag \vee \varphi_2\) satisfies 
        \begin{equation*}
            P_{0, 1, R}\left\{ \varphi^* = 1\right\} + \sup_{\theta \in \Theta(p, s, C\rho)} P_{\theta, 1, R}\left\{ \varphi^* = 0 \right\} \leq \eta.
        \end{equation*}
        Here \(\varphi^\dag\) and \(\varphi_2\) are given by (\ref{test:R_noiseless}) and (\ref{test:R_avg}) respectively.
    \end{proposition}
    \begin{proof}
        Fix \(\eta \in (0, 1)\) and set \(C_{\eta} = C_{\eta, 1} \vee C_{\eta, 2}\) where \(C_{\eta, 1}\) and \(C_{\eta, 2}\) are the constants at level \(\eta/2\) from Lemmas \ref{lemma:test_upsilon2_avg_sparse} and \ref{lemma:test_upsilon2_avg_dense}. Let \(C > C_\eta\). By Proposition \ref{prop:theta_upsilon} we have \(\Theta(p, s, C\rho) \subset \Upsilon_{\mathcal{I}}(p, s, C\rho) \cup \Upsilon_{\mathcal{II}}(p, s, C\rho)\). Clearly we also have \(\Upsilon_{\mathcal{I}}(p, s, C\rho) \subset \Theta^\dag(p, s, R)\), and so it follows that \(\Theta(p, s, C\rho) \subset \Theta^\dag(p, s, R) \cup \Upsilon_{\mathcal{II}}(p, s, C\rho)\). Therefore, 
        \begin{align*}
            &P_{0, 1, R}\{\varphi^* = 1\} + \sup_{\theta \in \Theta(p, s, C\rho)} P_{\theta, 1, R}\{\varphi^* = 0\} \\
            &\leq P_{0, 1, R}\{\varphi^* = 1\} + \sup_{\theta \in \Theta^\dag(p, s, R) \cup \Upsilon_{\mathcal{II}}(p, s, C\rho)} P_{\theta, 1, R}\{\varphi^\dag \vee \varphi_2 = 0\} \\
            &= P_{0, 1, R}\{\varphi^* = 1\} + \left[\sup_{\theta \in \Theta^\dag(p, s, R)} P_{\theta, 1, R}\{\varphi^\dag \vee \varphi_2 = 0\}\right] + \left[\sup_{\theta \in \Upsilon_{\mathcal{II}}(p, s, C\rho)} P_{\theta, 1, R}\{\varphi^\dag \vee \varphi_2 = 0\}\right] \\
            &\leq \left[P_{0, 1, R}\{\varphi^\dag = 1\} +  \sup_{\theta \in \Theta^\dag(p, s, R)} P_{\theta, 1, R}\{\varphi^\dag = 0\} \right] + \left[P_{0, 1, R}\{\varphi_2 = 1\} + \sup_{\theta \in \Upsilon_{\mathcal{II}}(p, s, C\rho)} P_{\theta, 1, R}\{\varphi_2 = 0\}\right] \\
            &\leq \eta
        \end{align*}
        where the last inequality follows from Proposition \ref{prop:R_perfect_correlation_sparse} as well as Lemmas \ref{lemma:test_upsilon2_avg_sparse} and \ref{lemma:test_upsilon2_avg_dense}. 
    \end{proof}

    We now proceed to proving the lower bound. In order to do so, some preliminary lemmas are needed. 

    \begin{lemma}\label{lemma:TV_reduction}
        For \(\nu \in \R^R\), denote \(Q_{\nu} = N(\nu, I_R)\). If \(\mathfrak{p}\) is a distribution on \(\R^R\), let \(Q_{\mathfrak{p}} = \int Q_{\nu} \, \mathfrak{p}(d\nu)\) denote the induced Gaussian mixture. Suppose \(\pi\) is a distribution on \(\R^p\) in which a draw \(\mu \sim \pi\) satisfies \(\mu_{B_k} \in \spn\left\{\mathbf{1}_{B_k}\right\}\) for all \(1 \leq k \leq R\) almost surely. Then 
        \begin{equation*}
            d_{TV}(P_{0, 1, R}, P_{\pi, 1, R}) = d_{TV}(Q_{0}, Q_{\widetilde{\pi}})
        \end{equation*}
        where \(\widetilde{\pi}\) is the distribution on \(\R^R\) in which a draw \(\nu \sim \widetilde{\pi}\) is obtained by drawing \(\mu \sim \pi\) and setting \(\nu_k := \bar{\mu}_{B_k}\) for all \(1 \leq k \leq R\). 
    \end{lemma}
    \begin{proof}
        Let \(\Gamma := \left\{ \mu \in \R^p : \mu_{B_k} \in \spn\left\{ \mathbf{1}_{B_k} \right\} \text{ for all } 1 \leq k \leq R\right\}\). Note that \(\pi\) is supported on \(\Gamma\). Consider the transformation \(T : \Gamma \to \R^R\) with \(T(\mu) = (\bar{\mu}_{B_1},...,\bar{\mu}_{B_R})\). We claim \(T\) is an injective map. Suppose \(T(\mu) = T(\mu')\) for \(\mu, \mu' \in \Gamma\). Since \(\mu \in \Gamma\), it immediately follows that \(\mu_{B_k} = \left(I_{B_k} - \frac{R}{p}\mathbf{1}_{B_k}\mathbf{1}_{B_k}^\intercal \right)\mu + \frac{R}{p}\mathbf{1}_{B_k}\mathbf{1}_{B_k}^\intercal \mu = \bar{\mu}_{B_k}\mathbf{1}_{B_k}\). Likewise, \(\mu'_{B_k} = \bar{\mu'}_{B_k}\mathbf{1}_{B_k}\). Now since \(T(\mu) = T(\mu')\), we have \(\bar{\mu}_{B_k} = \bar{\mu'}_{B_k}\) for all \(1 \leq k \leq R\). Consequently, we have \(\mu_{B_k} = \mu'_{B_k}\) for all \(k\), which immediately gives \(\mu = \mu'\). Thus \(T\) is injective. Since \(T\) is injective and \(Q_0, Q_{\widetilde{\pi}}\) are the pushforward measures of \(P_{0, 1, R}, P_{\pi, 1, R}\) through \(T\) respectively, we immediately have the desired result.
    \end{proof}

    \begin{proposition}\label{prop:R_perfect_correlation_lbound}
        Suppose \(1 \leq s \leq p\) and \(\gamma = 1\). Set
        \begin{equation*}
            \rho^2 :=
            \begin{cases}
                0 &\text{if } s < \frac{p}{R}, \\
                s\log\left(1 + \frac{p^2}{Rs^2}\right) &\text{if } \frac{p}{R} \leq s < \frac{p}{\sqrt{R}}, \\
                \frac{p}{\sqrt{R}} &\text{if } \frac{p}{\sqrt{R}} \leq s \leq p. 
            \end{cases}
        \end{equation*}
        If \(\eta \in (0,1)\), then there exists a constant \(c_\eta > 0\) depending only on \(\eta\) such that for all \(0 < c < c_\eta\) we have \(\mathcal{R}(c\rho) \geq 1-\eta\). 
    \end{proposition}
    \begin{proof}
        Fix \(\eta \in (0,1)\) and set \(c_\eta := \frac{1}{\sqrt{2}} \wedge \sqrt{\frac{1}{2}\log(1 + 4\eta^2)} \wedge \sqrt{\frac{1}{2}\log(1 + \log(1 + 4\eta^2))}\). Let \(0 < c < c_\eta\). Note that the case \(s < \frac{p}{R}\) is trivial and so we need only focus on the remaining two cases. 

        \textbf{Case 1:} Suppose \(\frac{p}{R} \leq s < \frac{p}{\sqrt{R}}\). Note \(\rho^2 = s\log\left(1 + \frac{p^2}{Rs^2}\right)\). Let \(m := \left\lfloor \frac{s}{p/R}\right\rfloor\) and note that \(1 \leq m < \sqrt{R}\). Let \(\pi\) be the prior on \(\Theta(p, s, c\rho)\) in which a draw \(\mu \sim \pi\) is given by the following construction. Draw \(K \subset [R]\) uniformly from the collection of all size \(m\) subsets of \([R]\) and set 
        \begin{equation*}
            \mu := \sum_{k \in K} \frac{c\rho}{\sqrt{m \frac{p}{R}}} \mathbf{1}_{B_k}.
        \end{equation*}
        Observe that \(||\mu||_0 = m \frac{p}{R} \leq s\) and \(||\mu||^2 = c^2\rho^2\) almost surely. Thus \(\pi\) is indeed supported on \(\Theta(p, s, c\rho)\). Consequently,
        \begin{align*}
            \mathcal{R}(c\rho) \geq \inf_{\varphi}\left\{ P_{0, 1, R}\{\varphi = 1\} + P_{\pi, 1, R}\{\varphi = 0\}\right\} = 1-d_{TV}(P_{0, 1, R}, P_{\pi, 1, R}).
        \end{align*}
        Note further that \(\mu \sim \pi\) satisfies \(\mu_{B_k} \in \spn\left\{\mathbf{1}_{B_k}\right\}\) for all \(1 \leq k \leq R\) almost surely. Consequently, Lemma \ref{lemma:TV_reduction} yields \(d_{TV}(P_{0, 1, R}, P_{\pi, 1, R}) = d_{TV}(Q_0, Q_{\widetilde{\pi}})\). Notice that a draw \(\nu \sim \widetilde{\pi}\) is generated by drawing a subset \(S \subset [R]\) of size \(m\) uniformly at random and setting \(\nu_i = \frac{c\rho}{\sqrt{m\frac{p}{R}}}\) if \(i \in S\) and to zero otherwise. Letting \(\nu = \frac{c\rho}{\sqrt{m\frac{p}{R}}}\mathbf{1}_S\) and \(\nu' = \frac{c\rho}{\sqrt{m\frac{p}{R}}}\mathbf{1}_{S'}\) for \(S, S'\) iid uniformly at random size \(m\) subsets \([R]\), Lemma \ref{lemma:Ingster_Suslina} and Lemma \ref{lemma:hypergeometric} give 
        \begin{align*}
            \chi^2\left(Q_{\widetilde{\pi}}|| Q_0\right) &= E\left[ \exp\left(\langle \nu, I_R \nu'\rangle\right) \right] - 1\\
            &= E\left[ \exp\left(\frac{c^2\rho^2}{m\frac{p}{R}} |S \cap S'|\right) \right] - 1 \\
            &\leq E\left[ \exp\left(2c^2 \log\left(1 + \frac{p^2}{Rs^2}\right) |S \cap S'|\right) \right] - 1 \\
            &\leq \left(1 - \frac{m}{R} + \frac{m}{R}\exp\left(2c^2\log\left(1 + \frac{p^2}{Rs^2}\right)\right)\right)^m - 1\\
            &\leq \left(1 - \frac{m}{R} + \frac{m}{R}\exp\left(2c^2\log\left(1 + \frac{R}{m^2}\right)\right)\right)^m - 1 \\
            &= \left(1 - \frac{m}{R} + \frac{m}{R} \left(1 + \frac{R}{m^2}\right)^{2c^2}\right)^m - 1 \\
            &\leq \left(1 + \frac{2c^2}{m}\right)^m - 1 \\
            &\leq e^{2c^2} - 1
        \end{align*}
        where in the penultimate inequality we have used \(2c^2 < 2c_\eta^2 \leq 1\) and \((1+x)^\delta - 1 \leq \delta x\) for all \(0 < \delta < 1\) and \(x > 0\). Thus, \(\chi^2\left(Q_{\widetilde{\pi}} || Q_0\right) \leq e^{2c^2} - 1 \leq 4\eta^2\). Therefore, \(\mathcal{R}(c\rho) \geq 1 - d_{TV}(Q_{0}, Q_{\widetilde{\pi}}) \geq 1-\frac{1}{2}\sqrt{\chi^2(Q_{\widetilde{\pi}}||Q_0)} \geq 1-\eta\). Since \(0 < c <c_\eta\) was arbitrary and \(\eta \in (0, 1)\) was arbitrary, we have proved the desired result for the case \(\frac{p}{R} \leq s < \frac{p}{\sqrt{R}}\). 

        \textbf{Case 2:} Suppose \(s \geq \frac{p}{\sqrt{R}}\). Note \(\rho^2 = \frac{p}{\sqrt{R}}\). Without loss of generality, assume \(\frac{p}{\sqrt{R}}\) is an integer. Repeating exactly the argument presented in Case 1 except now replacing every instance of \(s\) with \(\frac{p}{\sqrt{R}}\) yields 
        \begin{equation*}
            \chi^2(Q_{\widetilde{\pi}} || Q_{0}) \leq \left( 1 - \frac{m}{R} + \frac{m}{R}\exp\left(\frac{c^2\rho^2}{m\frac{p}{R}}\right)\right)^m - 1 \leq \left(1 - \frac{m}{R} + \frac{m}{R} \exp\left(\frac{c^2\sqrt{R}}{\lfloor \sqrt{R}\rfloor}\right) \right)^m - 1.
        \end{equation*}
        Continuing the calculation exactly as in the analysis of Case 2 in the proof of Lemma \ref{lemma:lbound_R_avg_problem} yields the desired result.
    \end{proof}

    \begin{proof}[Proof of Theorem \ref{thm:R_perfect_correlation}]
        Theorem \ref{thm:R_perfect_correlation} follows immediately from Propositions \ref{prop:R_perfect_correlation_ubound} and \ref{prop:R_perfect_correlation_lbound}.
    \end{proof}

    \subsection{Proofs of results in Section \ref{section:sparsity_adaptive}}

    \begin{proof}[Proof of Theorem \ref{thm:sparsity_adaptive}]
        Fix \(\eta \in (0, 1)\) and set
        \begin{equation*}
            C_\eta := \frac{16}{\eta} \sqrt{64 \cdot 9 \cdot 2\pi} \vee \sqrt{64 \cdot 9 \cdot \log\left(\frac{64}{\eta} \right)}\vee \sqrt{64 \cdot 9 \left(\frac{\log(9)}{3} + \log\left(\frac{32\pi^2}{6\eta}\right) \right)}\vee \sqrt{64 \cdot 9} \vee C_{\frac{\eta}{8}}^*
        \end{equation*}
        where \(C_{\frac{\eta}{8}}^*\) is the constant at level \(\frac{\eta}{8}\) from Proposition \ref{prop:full_test}. Let \(C > C_\eta\). We first bound the type I error. By an application of union bound,
        \begin{align*}
            &P_{0, \gamma}\left\{ \varphi_{\text{adaptive}} = 1\right\} \\
            &\leq P_{0, \gamma} \left\{\max_{1 \leq s < \sqrt{p}} \varphi_{t(s), r(s)} = 1 \right\} + P_{0, \gamma} \left\{ \max_{p-\sqrt{p} < s < p} \varphi_{\widetilde{t}(s), \widetilde{r}(s)} = 1 \right\} + P_{0, \gamma} \left\{ \varphi_{\frac{C^2}{2}}^{\chi^2} = 1 \right\} + P_{0, \gamma} \left\{ \varphi_{\frac{C^2}{2}}^{\mathbf{1}_p} = 1 \right\} \\
            &\leq P_{0, \gamma} \left\{\max_{1 \leq s < \sqrt{p}} \varphi_{t(s), r(s)} = 1 \right\} + P_{0, \gamma} \left\{ \max_{p-\sqrt{p} < s < p} \varphi_{\widetilde{t}(s), \widetilde{r}(s)} = 1 \right\} + \frac{\eta}{4}
        \end{align*}
        since \(C > C_\eta \geq C_{\frac{\eta}{8}}^*\). We separately examine the first two terms on the right hand side in the above display. Firstly for \(1 \leq s < \sqrt{p}\), letting \(u(s) := \frac{C^2}{64 \cdot 9} \left[ \log^2\left(1 + \frac{p}{s^2}\right) \wedge s \log\left(1 + \frac{p}{s^2}\right) \right]\) we have that
        \begin{align*}
            9\sqrt{pe^{-t(s)^2/2} u(s)} + 9u(s) &\leq 9\sqrt{p \frac{s^2}{s^2+p} u(s)} + 9u(s) \\
            &\leq 9 s\sqrt{u(s)} + \frac{C^2}{64} s\log\left(1 + \frac{p}{s^2}\right) \\
            &\leq \frac{C^2}{64}s \sqrt{\log^2\left(1 + \frac{p}{s^2}\right)} + \frac{C^2}{64} s\log\left(1 + \frac{p}{s^2}\right) \\
            &= r(s)
        \end{align*}
        where we have used \(\frac{C^2}{64 \cdot 9} > \frac{C_\eta^2}{64 \cdot 9} \geq 1\). Applying union bound and Lemma \ref{lemma:tsybakov_typeI_error} yields
        \begin{align*}
            P_{0, \gamma}\left\{\max_{1 \leq s < \sqrt{p}} \varphi_{t(s), r(s)} = 1 \right\} &\leq \sum_{1 \leq s < \sqrt{p}} P_{0, \gamma}\{Y_{t(s)} > r(s)\} \\
            &\leq \sum_{1 \leq s < \sqrt{p}} P_{0, \gamma} \left\{ Y_{t(s)} > 9 \cdot \sqrt{pe^{-t(s)^2/2}} + 9u(s) \right\} \\
            &\leq \sum_{1 \leq s < \sqrt{p}} e^{-u(s)} \\
            &\leq \sum_{1 \leq s < \sqrt{p}} \exp\left(- \frac{C^2}{64 \cdot 9} \log^2\left(1 + \frac{p}{s^2}\right) \right) + \sum_{1 \leq s < \sqrt{p}} \exp\left(-\frac{C^2}{64 \cdot 9} s \log\left(1 + \frac{p}{s^2}\right) \right). 
        \end{align*}
        Examining the second sum, observe that since \(\frac{C^2}{64 \cdot 9} \geq \frac{\log(9)}{3} + \log\left(\frac{32 \pi^2}{6\eta}\right)\) it follows that 
        \begin{equation*}
            \frac{C^2}{64 \cdot 9} \geq \frac{\log(s^2)}{s} + \log\left(\frac{32\pi^2}{6\eta}\right)
        \end{equation*}
        for \(s \geq 3\). Since we also have \(\frac{C^2}{64 \cdot 9} \geq \log\left(\frac{64}{\eta}\right)\), it follows that 
        \begin{align*}
            \sum_{1 \leq s < \sqrt{p}} \exp\left(-\frac{C^2}{64 \cdot 9} s \log\left(1 + \frac{p}{s^2}\right) \right) &\leq \frac{\eta}{64} + \frac{\eta}{64} + \sum_{3 \leq s < \sqrt{p}} \exp\left(-\log(s^2) - \log\left(\frac{32\pi^2}{6\eta} \right)\right) \\
            &\leq \frac{\eta}{32} + \frac{6\eta}{32\pi^2} \sum_{s=1}^{\infty} \frac{1}{s^2} \\
            &\leq \frac{\eta}{32} + \frac{6\eta}{32 \pi^2} \cdot \frac{\pi^2}{6} \\
            &\leq \frac{\eta}{16}. 
        \end{align*}
        Now turning our attention to the first sum, observe that 
        \begin{align*}
            \sum_{1 \leq s < \sqrt{p}}\exp\left(- \frac{C^2}{64 \cdot 9} \log^2\left(1 + \frac{p}{s^2}\right) \right) &\leq \int_{-\infty}^{\infty} \exp\left(-\frac{C^2}{64 \cdot 9} x^2\right) \, dx \leq \sqrt{2\pi \frac{64 \cdot 9}{2C^2}} \leq \frac{\eta}{16}
        \end{align*}
        where we have used \(C > C_\eta \geq \frac{16}{\eta} \sqrt{64 \cdot 9 \cdot 2\pi}\). Therefore, 
        \begin{equation}\label{eqn:adaptive_sparse_typeI}
            P_{0, \gamma}\left\{\max_{1 \leq s < \sqrt{p}} \varphi_{t(s), r(s)} = 1 \right\} \leq \frac{\eta}{8}. 
        \end{equation}

        We now turn our attention to bounding \(P_{0, \gamma} \left\{ \max_{p-\sqrt{p} < s < p} \varphi_{\widetilde{t}(s), \widetilde{r}(s)} = 1 \right\}\) and we pursue the same strategy. For \(p-\sqrt{p} < s < p\), let \(v(s) = \frac{C^2}{32 \cdot 9}\left[\log^2\left(1 + \frac{p}{(p-s)^2}\right) \wedge (p-s)\log\left(1 + \frac{p}{(p-s)^2}\right) \right]\). Repeating the above analysis with \(v(s)\) in place of \(u(s)\) as well as \(\widetilde{t}(s), \widetilde{r}(s)\) in place of \(t(s), r(s)\) respectively yields the bound 
        \begin{equation}\label{eqn:adaptive_very_dense_typeI}
            P_{0, \gamma}\left\{\max_{p-\sqrt{p} < s < p} \varphi_{\widetilde{t}(s), \widetilde{r}(s)} = 1 \right\} \leq \frac{\eta}{8}. 
        \end{equation}
        With the bounds (\ref{eqn:adaptive_sparse_typeI}) and (\ref{eqn:adaptive_very_dense_typeI}) in hand, it follows that \(P_{0, \gamma}\left\{\varphi_{\text{adaptive}} = 1\right\} \leq \frac{\eta}{2}\) and so the type I error is bounded. 

        We now work to bound the type II error. Note that 
        \begin{align*}
            \sup_{\theta \in \Theta(p, s^*, C\psi)} P_{\theta, \gamma}\left\{ \varphi_{\text{adaptive}} = 0\right\} &\leq 
            \begin{cases}
                \sup_{\theta \in \Theta(p, s^*, C\psi)} P_{\theta, \gamma} \left\{ \varphi_{t(s^*), r(s^*)} = 0 \right\} &\text{if } 1 \leq s^* < \sqrt{p}, \\
                \sup_{\theta \in \Theta(p, s^*, C\psi)} P_{\theta, \gamma} \left\{ \varphi_{\frac{C^2}{2}}^{\chi^2} \vee \varphi_{\frac{C^2}{2}}^{\mathbf{1}_p} = 0 \right\} &\text{if } \sqrt{p} \leq s^* \leq p-\sqrt{p}, \\
                \sup_{\theta \in \Theta(p, s^*, C\psi)} P_{\theta, \gamma} \left\{ \varphi_{\frac{C^2}{2}}^{\chi^2} \vee \varphi_{\widetilde{t}(s^*), \widetilde{r}(s^*)} \vee \varphi_{\frac{C^2}{2}}^{\mathbf{1}_p} = 0 \right\} &\text{if } p-\sqrt{p} < s^* < p, \\
                \sup_{\theta \in \Theta(p, s^*, C\psi)} P_{\theta, \gamma} \left\{ \varphi_{\frac{C^2}{2}}^{\chi^2} \vee \varphi_{\frac{C^2}{2}}^{\mathbf{1}_p} = 0 \right\} &\text{if } s^* = p.
            \end{cases}
        \end{align*}
        Since \(C > C_{\eta/8}^*\), Proposition \ref{prop:full_test} together with the above display yields \(\sup_{\theta \in \Theta(p, s^*, C\psi)} P_{\theta, \gamma}\left\{ \varphi_{\text{adaptive}} = 0\right\} \leq \frac{\eta}{8} \leq \frac{\eta}{2}\). Thus, we have shown that the sum of the type I and type II errors is bounded as 
        \begin{equation*}
            P_{0, \gamma}\left\{ \varphi_{\text{adaptive}} = 1\right\} + \sup_{\theta \in \Theta(p, s^*, C\psi)} P_{\theta, \gamma} \left\{ \varphi_{\text{adaptive}} = 0 \right\} \leq \eta. 
        \end{equation*}
        Since \(C > C_\eta\) was arbitrary and \(\eta \in (0, 1)\) was arbitrary, the desired result has been proved. 
    \end{proof}

    \subsection{Proofs of results in Section \ref{section:rank_one_correlation}}
    
    In this section, we present the proofs for the results presented in Section \ref{section:rank_one_correlation}. 
    \begin{proof}[Proof of Proposition \ref{prop:known_v_perfect_correlation}]
        We consider the two cases \(s < ||v||_0\) and \(s \geq ||v||_0\) separately. 

        \textbf{Case 1:} Suppose \( 1 \leq s < ||v||_0\). It trivially holds \(\varepsilon^*(p, s, 1, v)^2 \geq 0\) and so the lower bound is proved. Define the test \(\varphi^* = \mathbbm{1}_{\{X - p^{-1} \langle v, X \rangle v \neq 0\}}\). Observe that under the data-generating process \(P_{\theta, 1, v}\), we have \(X = \theta + W v\) where \(W \sim N(0, 1)\). Consequently \(X - p^{-1}\langle v, X \rangle v = \theta - p^{-1} \langle v, \theta \rangle v\) almost surely. Since \(||\theta|| \geq \varepsilon\) and \(||\theta||_0 \leq s < ||v||_0\), it follows immediately that \(\theta \not \in \spn\{v\} \setminus \{0\}\). Consequently, \(\theta \neq 0\) if and only if \(\theta - p^{-1}\langle v, \theta \rangle v \neq 0\). Examining the type I error, observe that \(P_{0, 1, v}\{\varphi^* = 1\} = P_{0, 1, v}\{X - p^{-1}\langle v, X \rangle v \neq 0\} = 0\) since \(X - p^{-1}\langle v, X \rangle v = 0\) almost surely under \(P_{0, 1, v}\). Examining the type II error, observe that for any \(\varepsilon > 0\) we have \(\sup_{\theta \in \Theta(p, s, \varepsilon)} P_{\theta, 1, v} \{\varphi^* = 0\} = \sup_{\theta \in \Theta(p, s, \varepsilon)} P_{\theta, 1, v}\{ X - p^{-1}\langle v, X \rangle v = 0\} = 0\) since \(X-p^{-1}\langle v, X\rangle v = \theta - p^{-1}\langle v, \theta \rangle v\) almost surely under \(P_{\theta, 1, v}\) and \(\theta - p^{-1}\langle v, \theta \rangle v \neq 0\) for all \(\theta \in \Theta(p, s, \varepsilon)\). Therefore, \(\mathcal{R}(\varepsilon) = 0\). Since \(\varepsilon > 0\) was arbitrary, we have proved the upper bound. 
        
        \textbf{Case 2:} Suppose \(||v||_0 \leq s \leq p\). We first prove the upper bound. Let \(\eta \in (0, 1)\) and set \(C_\eta\) to be any value satisfying \(\frac{16}{C_\eta^4} + \frac{16}{C_\eta^2} \leq \eta\). Such a \(C_\eta\) clearly exists by taking \(C_\eta\) sufficiently large depending only on \(\eta\). Let \(C > C_\eta\). Define the test \(\varphi^* = \mathbbm{1}_{\left\{ ||X||^2 > p + \frac{C^2}{2} p \right\}}\). Note that 
        \begin{equation*}
            ||X||^2 = ||X - p^{-1}\langle v, X\rangle v||^2 + ||p^{-1}\langle v, X\rangle v||^2 \sim ||\theta - p^{-1}\langle v, \theta \rangle v||^2 + p\chi^2_1\left(\frac{||p^{-1}\langle v, \theta\rangle v||^2}{p}\right).
        \end{equation*}
        Consequently \(E_{\theta, 1, v}(||X||^2) = ||\theta||^2 + p\) and \(\Var_{\theta, 1, v}(||X||^2) = 2p^2 + 4p ||p^{-1} \langle v, \theta \rangle v||^2\). Examining the type I error, consider by Chebyshev's inequality 
        \begin{align*}
            P_{0, 1, v}\{\varphi^* = 1\} = P_{0, 1, v} \left\{ ||X||^2 > p + \frac{C^2}{2} p\right\} \leq \frac{2p^2}{\frac{C^4}{4} p^2} \leq \frac{8}{C^4} \leq \frac{8}{C_\eta^4}. 
        \end{align*}
        Examining the type II error, observe that 
        \begin{align*}
            \sup_{\theta \in \Theta(p, s, C\sqrt{p})} P_{\theta, 1, v}\{\varphi^* = 0\} &= \sup_{\theta \in \Theta(p, s, C\sqrt{p})} P_{\theta, 1, v}\left\{ ||X||^2 \leq p + \frac{C^2}{2}p \right\} \\
            &= \sup_{\theta \in \Theta(p, s, C\sqrt{p})} P_{\theta, 1, v}\left\{ ||\theta||^2 - \frac{C^2}{2}p \leq p + ||\theta||^2 - ||X||^2 \right\} \\
            &\leq \sup_{\theta \in \Theta(p, s, C\sqrt{p})} \frac{\Var_{\theta, 1, v}\left(||X||^2 \right)}{\left( ||\theta||^2 - \frac{C^2}{2}p \right)^2} \\
            &\leq \frac{2p^2}{\frac{C^4}{4}p^2} + \sup_{\theta \in \Theta(p, s, C\sqrt{p})} \frac{4p ||p^{-1}\langle v, \theta\rangle v||^2}{\frac{1}{4}||\theta||^4} \\
            &\leq \frac{8}{C^4} + \sup_{\theta \in \Theta(p, s, C\sqrt{p})} \frac{16p}{||\theta||^2} \\
            &\leq \frac{8}{C^4} + \frac{16}{C^2} \\
            &\leq \frac{8}{C_\eta^4} + \frac{16}{C_\eta^2}.
        \end{align*}
        Therefore \(P_{0, 1, v}\{\varphi^* = 1\} + \sup_{\theta \in \Theta(p, s, C\sqrt{p})} P_{\theta, 1, v}\{\varphi^* = 0\} \leq \frac{16}{C_\eta^4} + \frac{16}{C_\eta^2} \leq \eta\). Since \(C > C_\eta\) was arbitrary and \(\eta \in (0, 1)\) was arbitrary, we have proved \(\varepsilon^*(p, s, 1, v)^2 \lesssim p\) for \(s \geq ||v||_0\). 

        We now prove the lower bound. Let \(\eta \in (0, 1)\) and set \(c_\eta := \sqrt{\log\left(1 + 4\eta^2\right)}\). Let \(0 < c < c_\eta\). Lemma \ref{lemma:general_lower_bound} will be used to prove the lower bound. Let \(\pi\) be the prior on \(\Theta(p, s, c\sqrt{p})\) which is a point mass at \(cv\). Note that \(||cv|| = c\sqrt{p}\) and \(||cv||_0 = ||v||_0 \leq s\). Thus, \(\pi\) is indeed supported on \(\Theta(p, s, c\sqrt{p})\). A direct calculation shows 
        \begin{align*}
            \chi^2(P_{cv, 1, v} || P_{0, 1, v}) = \exp\left( \left\langle cv, \frac{1}{p^2}vv^\intercal \cdot cv \right\rangle \right) - 1 = e^{c^2} - 1 \leq e^{c_\eta^2} - 1 \leq 4\eta^2.
        \end{align*}
        Therefore, \(1 - \frac{1}{2}\sqrt{\chi^2(P_{cv, 1, v}||P_{0, 1, v})} \geq 1-\eta\). Lemma \ref{lemma:general_lower_bound} thus implies \(\mathcal{R}(c\sqrt{p}, v) \geq 1-\eta\). Since \(0 < c < c_\eta\) was arbitrary and \(\eta \in (0, 1)\) was arbitrary, it follows that \(\varepsilon^*(p, s, 1, v)^2 \gtrsim p\) for \(s \geq ||v||_0\). 
    \end{proof}

    The proof of Theorem \ref{thm:known_v_rate} consists of finding matching upper and lower bounds for the minimax separation rate. Focusing on the upper bound, the core strategy is very similar to the strategy adopted in Section \ref{section:ProblemI}. Specifically, we transform the data to decorrelate the observations. Recall that \(X \sim N(\theta, (1-\gamma)I_p + \gamma vv^\intercal)\). Drawing \(\xi \sim N(0, 1)\) independently from \(X\), consider the transformation 
    \begin{equation*}
        \widetilde{X} := \frac{1}{\sqrt{1-\gamma}}\left(I_p - \frac{1}{p} vv^\intercal\right) X + \frac{\xi}{\sqrt{p}} v
    \end{equation*}
    when \(\gamma \in [0, 1)\). It is readily seen that 
    \begin{equation*}
        \widetilde{X} \sim N\left( \frac{\theta - p^{-1}\langle v, \theta \rangle v}{\sqrt{1-\gamma}}, I_p \right). 
    \end{equation*}
    It turns out the statistic formulated by Collier et al. \cite{collierMinimaxEstimationLinear2017} can still be used for signal detection in the sparse regime. As in Section \ref{section:ProblemI}, define for \(t \geq 0\) the statistic
    \begin{equation*}
        Y_t := \sum_{i=1}^{p}(\widetilde{X}_i^2 - \alpha_{t})\mathbbm{1}_{\{|\widetilde{X}_i| \geq t\}}
    \end{equation*}
    where \(\alpha_t\) is given by (\ref{eqn:alpha_t}). Define the test 
    \begin{equation}\label{test:known_v_tsybakov}
        \varphi_{t, r} := \mathbbm{1}_{\{Y_t > r\}}
    \end{equation}
    where \(r \in \R\) is set according to the desired testing risk. A \(\chi^2\)-type test will be used in the dense regime. For \(r \in \R\) used to set the desired testing risk, define the test 
    \begin{equation}\label{test:known_v_chisquare}
        \varphi_{r}^{\chi^2} = \mathbbm{1}_{\{||\widetilde{X}||^2 > p + r\sqrt{p}\}}.
    \end{equation}

    These two tests are combined to obtain a testing procedure with the following guarantee. Recall the definition of \(\omega(v)\) given by (\ref{eqn:omega_v}). In preparation for the statement of the result, set 
    \begin{equation}\label{rate:known_v}
        \psi^2 := 
        \begin{cases}
            (1-\gamma) s \log \left(1 + \frac{p}{s^2} \right) &\text{if } s \leq \sqrt{p} \wedge \omega(v), \\
            (1-\gamma) \sqrt{p} &\text{if } \sqrt{p} < s \leq \omega(v).
        \end{cases}
    \end{equation}

    \begin{proposition}\label{prop:known_v_upperbound}
        Let \(1 \leq s \leq p\), \(\gamma \in [0, 1)\), and \(t^* = \sqrt{2 \log\left(1 + \frac{p}{s^2}\right)}\). If \(\eta \in (0, 1)\), then there exists a constant \(C_\eta > 0\) depending only on \(\eta\) such that for all \(C > C_\eta\) the testing procedure 
        \begin{equation}\label{test:known_v}
            \varphi^* = 
            \begin{cases}
                \varphi_{t^*, r^*} &\text{if } s \leq \sqrt{p} \wedge \omega(v), \\
                \varphi_{C^2/2}^{\chi^2} &\text{if } \sqrt{p} < s \leq \omega(v)
            \end{cases}
        \end{equation}
        with \(t^* = \sqrt{2 \log\left(1 + \frac{p}{s^2}\right)}\) and \(r^* = \frac{C^2}{16} s \log\left(1 + \frac{p}{s^2}\right)\) satisfies 
        \begin{equation*}
            P_{0, \gamma, v}\left\{\varphi^* = 1\right\} + \sup_{\theta \in \Theta(p, s, C\psi)} P_{\theta, \gamma, v}\left\{\varphi^* = 0 \right\} \leq \eta.
        \end{equation*}
        Here, \(\psi^2\) is given by (\ref{rate:known_v}).
    \end{proposition}
    
    To prove Proposition \ref{prop:known_v_upperbound}, two lemmas are needed. 
    
    \begin{lemma}\label{lemma:known_v_chisquare}
        Suppose \(1 \leq s \leq \omega(v)\) and \(\gamma \in [0, 1)\). If \(\eta \in (0, 1)\), then there exists a constant \(C_\eta > 0\) depending only on \(\eta\) such that for all \(C > C_\eta\), the test \(\varphi_{C^2/2}^{\chi^2}\) given by (\ref{test:known_v_chisquare}) satisfies 
        \begin{equation*}
            P_{0, \gamma, v}\left\{ \varphi_{C^2/2}^{\chi^2} = 1\right\} + \sup_{\theta \in \Theta\left(p, s, C\sqrt{(1-\gamma)\sqrt{p}}\right)} P_{\theta, \gamma, v}\left\{ \varphi_{C^2/2}^{\chi^2} = 0\right\} \leq \eta. 
        \end{equation*}
        Here, \(\psi^2\) is given by (\ref{rate:known_v}).
    \end{lemma}
    \begin{proof}
        Fix \(\eta \in (0, 1)\) and let \(C_\eta > 0\) be any value satisfying \(\frac{40}{C_\eta^4} + \frac{48}{C_\eta^2} \leq \eta\). Note that such a value of \(C_\eta\) clearly exists by taking \(C_\eta\) sufficiently large. Let \(C > C_\eta\). For ease of notation, let \(M(v, s) := \max_{S \subset [p], |S| \leq s} ||v_S||^2\). Under the data-generating process \(P_{\theta, \gamma, v}\), recall that \(\widetilde{X} \sim N\left(\frac{\theta - p^{-1}\langle v, \theta\rangle v}{\sqrt{1-\gamma}}, I_p\right)\), where \(\widetilde{X}\) is the transformed data used in the definition of the test \(\varphi_{C^2/2}^{\chi^2}\). Therefore, \(||\widetilde{X}||^2 \sim \chi^2_p\left(\frac{||\theta - p^{-1}\langle v, \theta \rangle v||^2}{1-\gamma} \right)\). Consequently \(E_{\theta, \gamma, v}(||\widetilde{X}||^2) = p + \frac{||\theta - p^{-1}\langle v, \theta \rangle v||^2}{1-\gamma}\) and \(\Var_{\theta, \gamma, v}(||\widetilde{X}||^2) = 2p + 4 \frac{||\theta - p^{-1}\langle v, \theta\rangle v||^2}{1-\gamma}\). A direct calculation using Chebyshev's inequality bounds the type I error
        \begin{align*}
            P_{0, \gamma, v}\left\{ \varphi_{C^2/2}^{\chi^2} = 1 \right\} = P_{0, \gamma, v}\left\{ ||\widetilde{X}||^2 > p + \frac{C^2}{2}\sqrt{p} \right\} \leq \frac{\Var_{0, \gamma, v}\left( ||\widetilde{X}||^2 \right)}{\frac{C^4}{4} p} = \frac{2p}{\frac{C^4}{4} p} = \frac{8}{C^4} \leq \frac{8}{C_\eta^4}. 
        \end{align*}
        We now turn our attention to the type II error. Before explicitly obtaining a bound, first note that \(s \leq \omega(v)\) implies \(M(v, s) \leq \frac{p}{4}\). Therefore, \(p-M(v, s)\geq \frac{3}{4}p\), and so we have by Lemma \ref{lemma:v_orthog_approximation}
        \begin{equation}\label{eqn:signal_projection_lbound}
            \frac{||\theta - p^{-1}\langle v, \theta \rangle v||^2}{1-\gamma} \geq \frac{3C^2}{4} \sqrt{p}
        \end{equation}
        for all \(\theta \in \Theta\left(p, s, C\sqrt{(1-\gamma)\sqrt{p}}\right)\). Examining the type II error, consider
        \begin{align*}
            &\sup_{\theta \in \Theta\left(p, s, C\sqrt{(1-\gamma)\sqrt{p}}\right)} P_{\theta, \gamma, v}\left\{ \varphi_{C^2/2}^{\chi^2} = 0 \right\} \\
            &= \sup_{\theta \in \Theta\left(p, s, C\sqrt{(1-\gamma)\sqrt{p}}\right)} P_{\theta, \gamma, v}\left\{ ||\widetilde{X}||^2 \leq p + \frac{C^2}{2}\sqrt{p} \right\} \\
            &= \sup_{\theta \in \Theta\left(p, s, C\sqrt{(1-\gamma)\sqrt{p}}\right)} P_{\theta, \gamma, v}\left\{ \frac{||\theta - p^{-1}\langle v, \theta\rangle v||^2}{1-\gamma} - \frac{C^2}{2} \sqrt{p} \leq p + \frac{||\theta - p^{-1}\langle v, \theta \rangle v||^2}{1-\gamma} - ||\widetilde{X}||^2 \right\} \\
            &\leq \sup_{\theta \in \Theta\left(p, s, C\sqrt{(1-\gamma)\sqrt{p}}\right)} \frac{\Var_{\theta, \gamma, v}(||\widetilde{X}||^2)}{\left(\frac{||\theta - p^{-1}\langle v, \theta\rangle v||^2}{1-\gamma} - \frac{C^2}{2}\sqrt{p} \right)^2} \\
            &\leq \frac{2p}{\frac{C^4}{16}p} + \sup_{\theta \in \Theta\left(p, s, C\sqrt{(1-\gamma)\sqrt{p}}\right)} \frac{4 \frac{||\theta - p^{-1}\langle v, \theta \rangle v||^2}{1-\gamma}}{\frac{1}{9} \frac{||\theta - p^{-1}\langle v, \theta \rangle v||^4}{(1-\gamma)^2}} \\
            &= \frac{32}{C^2} + \sup_{\theta \in \Theta\left(p, s, C\sqrt{(1-\gamma)\sqrt{p}}\right)} \frac{36 (1-\gamma)}{||\theta - p^{-1}\langle v, \theta \rangle v||^2}
        \end{align*}
        where the penultimate inequality follows by (\ref{eqn:signal_projection_lbound}). Continuing the calculation with another application of (\ref{eqn:signal_projection_lbound}) yields 
        \begin{align*}
            \frac{32}{C^4} + \sup_{\theta \in \Theta\left(p, s, C\sqrt{(1-\gamma)\sqrt{p}}\right)} \frac{36 (1-\gamma)}{||\theta - p^{-1}\langle v, \theta \rangle v||^2} \leq \frac{32}{C^4} + \frac{48}{C^2\sqrt{p}} \leq \frac{32}{C_\eta^4} + \frac{48}{C_\eta^2}. 
        \end{align*}
        Combining the bounds for the type I and type II errors shows 
        \begin{align*}
            P_{0, \gamma, v}\left\{ \varphi_{C^2/2}^{\chi^2} = 1 \right\} + \sup_{\theta \in \Theta\left(p, s, C\sqrt{(1-\gamma)\sqrt{p}}\right)} P_{\theta, \gamma, v}\left\{ \varphi_{C^2/2}^{\chi^2} = 0 \right\} \leq \frac{40}{C_\eta^4} + \frac{48}{C_\eta^2} \leq \eta. 
        \end{align*}
        Since \(C > C_\eta\) was arbitrary and \(\eta \in (0, 1)\) was arbitrary, we have shown the desired result. 
    \end{proof}

    \begin{lemma}\label{lemma:known_v_tsybakov}
        Suppose \(1 \leq s \leq \sqrt{p} \wedge \omega(v)\) and \(\gamma \in [0, 1)\). If \(\eta \in (0, 1)\), then there exists a constant \(C_\eta > 0\) depending only on \(\eta\) such that for all \(C > C_\eta\), the test \(\varphi_{t^*, r^*}\) given by (\ref{test:known_v_tsybakov}) with \(t^* = \sqrt{2 \log\left(1 + \frac{p}{s^2}\right)}\) and \(r^* = \frac{C^2}{16} s\log\left(1 + \frac{p}{s^2}\right)\) satisfies 
        \begin{equation*}
            P_{0, \gamma, v}\left\{ \varphi_{t^*, r^*} = 1 \right\} + \sup_{\theta \in \Theta\left(p, s, C\sqrt{(1-\gamma)s\log\left(1 + \frac{p}{s^2}\right)}\right)} P_{\theta, \gamma, v}\left\{ \varphi_{t^*, r^*} = 0 \right\} \leq \eta. 
        \end{equation*}
    \end{lemma}
    \begin{proof}
        Fix \(\eta \in (0, 1)\) and set \(C_\eta := \sqrt{2}C_\eta^*\) where \(C_\eta^*\) is given by Proposition \ref{prop:supp_tsybakov}. For ease of notation, set \(\kappa := (1-\gamma) s\log\left(1 + \frac{p}{s^2}\right)\). Also, let \(M(v, s) := \max_{S \subset [p], |S| \leq s} ||v_S||^2\). Consider that since \(s < \omega(v)\), we have \(M(v, s) \leq \frac{p}{4}\). Therefore, for all \(\theta \in \Theta\left(p, s, C\sqrt{\kappa}\right)\) we have by Lemma \ref{lemma:v_supp_approximation}
        \begin{equation*}
            \frac{||\theta - p^{-1}\langle v, \theta \rangle v_{\supp(\theta)}||^2}{1-\gamma} \geq \frac{||\theta||^2}{1-\gamma} \cdot \frac{p - 2 M(v, s)}{p} \geq \frac{C^2\kappa}{2(1-\gamma)} = \frac{C^2}{2} s \log\left(1 + \frac{p}{s^2}\right).  
        \end{equation*}
        In other words, we have shown that \(\theta \in \Theta\left(p, s, C\sqrt{\kappa}\right)\) implies \(\frac{\theta - p^{-1}\langle v, \theta \rangle v}{\sqrt{1-\gamma}} \in \mathscr{M}\left(p, s, \frac{C^2}{2}s\log\left(1 + \frac{p}{s^2}\right)\right)\) where the latter parameter space is given by (\ref{space:supp_space}). Recall that, under the data-generating process \(P_{\theta, \gamma, v}\), we have \(\widetilde{X} \sim N\left(\frac{\theta - p^{-1}\langle v, \theta \rangle v}{\sqrt{1-\gamma}}, I_p \right)\) where \(\widetilde{X}\) is the transformed data used in the definition of the test \(\varphi_{t^*, r^*}\). Since \(s \leq \sqrt{p}\), \(\frac{C}{\sqrt{2}} > \frac{C_\eta}{\sqrt{2}} \geq C_\eta^*\), and \(r^* = \frac{C^2/2}{8}s \log\left(1 + \frac{p}{s^2}\right)\), it follows from Proposition \ref{prop:supp_tsybakov}
        \begin{equation*}
            P_{0, \gamma, v}\left\{ \varphi_{t^*, r^*} = 1 \right\} + \sup_{\theta \in \Theta\left(p, s, C\sqrt{\kappa}\right)} P_{\theta, \gamma, v}\left\{ \varphi_{t^*, r^*} = 0 \right\} \leq \eta. 
        \end{equation*}
        Since \(C > C_\eta\) was arbitrary and \(\eta \in (0, 1)\) was arbitrary, we have proved the desired result.
    \end{proof}

    \begin{proof}[Proof of Proposition \ref{prop:known_v_upperbound}]
        Fix \(\eta \in (0, 1)\) and let \(C_{\eta} := C_{\eta, 1} \vee C_{\eta, 2}\) where \(C_{\eta, 1}\) and \(C_{\eta, 2}\) are the constants depending on \(\eta\) from the proofs of Lemmas \ref{lemma:known_v_tsybakov} and \ref{lemma:known_v_chisquare} respectively. Let \(C > C_\eta\). We now consider the two sparsity regimes separately. 

        \textbf{Case 1:} Suppose \(1 \leq s \leq \sqrt{p} \wedge \omega(v)\). Then \(\varphi^* = \varphi_{t^*, r^*}\) and \(\psi^2 = (1-\gamma) s\log\left(1 + \frac{p}{s^2}\right)\). Since \(C > C_{\eta, 1}\), it follows from Lemma \ref{lemma:known_v_tsybakov} that \(P_{0, \gamma, v}\left\{ \varphi^* = 1 \right\} + \sup_{\theta \in \Theta(p, s, C\psi)} P_{\theta, \gamma, v}\left\{ \varphi^* = 0 \right\} \leq \eta\).

        \textbf{Case 2:} Suppose \(\sqrt{p} < s \leq \omega(v)\). Then \(\varphi^* = \varphi_{C^2/2}^{\chi^2}\) and \(\psi^2 = (1-\gamma)\sqrt{p}\). Since \(C > C_{\eta, 2}\), it follows from Lemma \ref{lemma:known_v_chisquare} that \(P_{0, \gamma, v}\{\varphi^* = 1\} + \sup_{\theta \in \Theta(p, s, C\psi)} P_{\theta, \gamma, v}\{ \varphi^* = 0 \} \leq \eta\).

        Having dealt with the two cases, consider that since \(C > C_\eta\) was arbitrary and \(\eta \in (0, 1)\) was arbitrary, the desired result has been proved.
    \end{proof}

    The testing procedure (\ref{test:known_v}) is rate optimal as the following proposition establishes a matching lower bound.
    \begin{proposition}\label{prop:known_v_lowerbound}
        Suppose \(1 \leq s \leq \omega(v)\) and \(\gamma \in [0, 1)\). If \(\eta \in (0, 1)\), then there exists a constant \(c_\eta > 0\) depending only on \(\eta\) such that \(\mathcal{R}(c\psi) \geq 1-\eta\) for all \(0 < c < c_\eta\). Here, \(\psi^2\) is given by (\ref{rate:known_v}).
    \end{proposition}
    \begin{proof}
        Fix \(\eta \in (0, 1)\) and set \(c_\eta := 1 \wedge \sqrt{\log(1 + 4\eta^2)} \wedge \sqrt{\log\left(1 + \log\left(1 + 4\eta^2\right)\right)}\). Let \(0 < c < c_\eta\). We deal with the two sparsity regimes separately. 

        \textbf{Case 1:} Suppose \(1 \leq s \leq \sqrt{p} \wedge \omega(v)\). Then \(\psi^2 = (1-\gamma)s \log\left(1 + \frac{p}{s^2}\right)\). Let \(\pi\) be the prior on \(\Theta(p, s, c\psi)\) such that a draw \(\mu \sim \pi\) is obtained by drawing a subset \(S \subset [p]\) of size \(s\) uniformly at random from the set of all subsets of \([p]\) of size \(s\) and setting 
        \begin{equation*}
            \mu_i :=
            \begin{cases}
                \sgn(v_i) \cdot \frac{c\psi}{\sqrt{s}} &\text{if } i \in S, \\
                0 &\text{if } i \in S^c.
            \end{cases}
        \end{equation*}
        Note that \(||\mu||_0 = s\) almost surely since \(|S| = s\) and \(||\mu||^2 = c^2\psi^2\). Therefore, \(\pi\) is indeed supported on \(\Theta(p, s, c\psi)\). Let \(P_{\pi, \gamma, v} = \int P_{\theta, \gamma, v} \pi(d\theta)\) denote the Gaussian mixture induced by \(\pi\). By Lemma \ref{lemma:Ingster_Suslina} and Lemma \ref{lemma:v_inverse}, 
        \begin{align*}
            \chi^2(P_{\pi, \gamma, v} || P_{0, \gamma, v}) = E\left[ \exp\left( \left\langle \theta, \left[ \frac{1}{1-\gamma} I_p - \frac{\gamma}{(1-\gamma)^2 + (1-\gamma)\gamma p} vv^\intercal \right] \widetilde{\theta}\right\rangle \right)  \right] - 1
        \end{align*}
        where \(\theta, \widetilde{\theta} \overset{iid}{\sim} \pi\). Letting \(S\) and \(\widetilde{S}\) denote \(\supp(\theta)\) and \(\supp(\widetilde{\theta})\), observe that 
        \begin{align*}
            \left\langle \theta, \left[ \frac{1}{1-\gamma} I_p - \frac{\gamma}{(1-\gamma)^2 + (1-\gamma)\gamma p} vv^\intercal \right] \widetilde{\theta} \right\rangle &= \frac{c^2\psi^2}{s} \left[\frac{1}{1-\gamma} |S \cap \widetilde{S}| - \frac{\gamma}{(1-\gamma)^2 + \gamma(1-\gamma)p} \cdot \frac{||v_S||_1 \cdot ||v_{\widetilde{S}}||_1}{p} \right] \\
            &\leq \frac{c^2\psi^2}{s(1-\gamma)} |S \cap \widetilde{S}|.
        \end{align*}
        We now use the fact that \(|S \cap \widetilde{S}|\) is distributed according to the hypergeometric distribution with probability mass function given in Lemma \ref{lemma:hypergeometric}. Following a calculation similar to the one in the Case 1 analysis of the proof of Proposition \ref{prop:problemI_lowerbound}, it follows that \(\chi^2(P_{\pi, \gamma, v}||P_{0, \gamma, v}) \leq e^{c^2}-1 \leq e^{c_\eta^2}-1 \leq 4\eta^2\). Then, Lemma \ref{lemma:general_lower_bound} implies \(\mathcal{R}(c\psi, v) \geq 1- \frac{1}{2}\sqrt{\chi^2(P_{\pi, \gamma, v} || P_{0, \gamma, v})} \geq 1-\eta\). Since \(0 < c < c_\eta\) was arbitrary and \(\eta \in (0, 1)\) was arbitrary, we have proved the desired result for the case \(1 \leq s \leq \sqrt{p} \wedge \omega(v)\).
        
        \textbf{Case 2:} Suppose \(\sqrt{p} < s \leq \omega(v)\). Note \(\psi^2 = (1-\gamma)\sqrt{p}\). Without loss of generality, assume \(\sqrt{p}\) is an integer. Repeating exactly the argument presented in Case 1 except now replacing every instance of \(s\) with \(\sqrt{p}\) yields \(\chi^2(P_{\pi, \gamma, v} || P_{0, \gamma, v}) \leq 4\eta^2\). This bound is obtained via a calculation similar to the one in the Case 2 analysis of the proof of Proposition \ref{prop:problemI_lowerbound}. Lemma \ref{lemma:general_lower_bound} then implies \(\mathcal{R}(c\psi, v) \geq 1-\eta\). Since \(0 < c < c_\eta\) was arbitrary and \(\eta \in (0, 1)\) was arbitrary, we have proved the desired result for the case \(\sqrt{p} < s \leq \omega(v)\).
    \end{proof}

    A combination of Propositions \ref{prop:known_v_perfect_correlation}, \ref{prop:known_v_upperbound}, and \ref{prop:known_v_lowerbound} yields Theorem \ref{thm:known_v_rate}.

    \subsection{Technical Lemmas}
    \begin{lemma}[Lemma 1 of \cite{laurentAdaptiveEstimationQuadratic2000}]\label{lemma:laurent_massart}
        Let \(Y_1,...,Y_p \overset{iid}{\sim} N(0, 1)\) and let \(a_1,...,a_p \in \R^p\) be nonnegative. Then, for any \(x > 0\), we have 
        \begin{equation*}
            P\left\{\sum_{j=1}^{p} a_j Y_j^2 \geq \sum_{j=1}^{p} a_j + 2\sqrt{x \sum_{j=1}^{p}a_j^2} + 2x\max_{1\leq j\leq p} a_j \right\} \leq e^{-x},
        \end{equation*}
        and 
        \begin{equation*}
            P\left\{ \sum_{j=1}^{p} a_j Y_j^2 \leq \sum_{j=1}^{p} a_j - 2\sqrt{x\sum_{j=1}^{p}a_j^2} \right\} \leq e^{-x}.
        \end{equation*}
    \end{lemma}

    The following lemmas are used extensively in our arguments proving the main results of the paper. In preparation for the statements, recall the notation \(\alpha_t\) given by (\ref{eqn:alpha_t}).

    \begin{lemma}[Lemma 18 of \cite{liuMinimaxRatesSparse2021}]\label{lemma:tsybakov_typeI_error}
        Let \(Z_1,...,Z_p \overset{iid}{\sim}N(0, 1)\). Then there exists a universal constant \(C^* > 0\) such that for any \(t > 0\) and \(x > 0\), we have 
        \begin{equation*}
            P\left\{ \sum_{j=1}^{p}(Z_j^2 - \alpha_t) \mathbbm{1}_{\{|Z_j| \geq t\}} \geq C^* \left( \sqrt{pe^{-t^2/2}x} + x \right) \right\} \leq e^{-x}.
        \end{equation*}
        In fact, we may take \(C^* = 9\).
    \end{lemma}

    \begin{lemma}[Lemma 19 of \cite{liuMinimaxRatesSparse2021}]\label{lemma:tsybakov_expectation}
        Suppose \(Y \sim N(\theta, 1)\) for some \(\theta \in \R\). Then there exists a universal constant \(C > 0\) such that for every \(t \geq 1\), 
        \begin{equation*}
            E\left((Y^2 - \alpha_t)\mathbbm{1}_{\{|Y| \geq t\}}\right) 
            \begin{cases}
                = 0 &\text{if } \theta = 0, \\
                \in [0, C^2t^2 + 1] &\text{if } |\theta| < Ct, \\
                \geq \frac{\theta^2}{2} &\text{if } |\theta| \geq Ct.
            \end{cases}
        \end{equation*}
        In fact, we may take \(C = 8\). Moreover, for any \(\delta > 0\), there exist constants \(c_1^*, C_1^* > 0\), such that as long as \(|\theta| \geq (1+\delta) t\) and \(t > C_1^*\), we have \(E\left( (Y^2 - \alpha_t)\mathbbm{1}_{\{|Y| \geq t\}} \right) \geq c_1^*\theta^2\).
    \end{lemma}

    \begin{lemma}[Lemma 20 of \cite{liuMinimaxRatesSparse2021}]\label{lemma:tsybakov_variance}
        Suppose \(Y \sim N(\theta, 1)\) for some \(\theta \in \R\). Then there exists a universal constant \(C_1 \geq 1\) such that 
        \begin{equation*}
            \Var\left((Y^2 - \alpha_t)\mathbbm{1}_{\{|Y| \geq t\}}\right) \leq 
            \begin{cases}
                C_1 t^3 e^{-t^2/2} &\text{if } \theta = 0, \\
                C_1t^4 &\text{if } 0 < |\theta| < 2t, \\
                C_1\theta^2 &\text{if } |\theta| \geq 2t,
            \end{cases}
        \end{equation*}
        as long as \(t \geq 1\). 
    \end{lemma}

    \begin{lemma}[\cite{collierMinimaxEstimationLinear2017}]\label{lemma:hypergeometric}
        If \(Y\) is distributed according to the hypergeometric distribution with probability mass function \(P\left\{Y = k\right\} = \frac{\binom{s}{k} \binom{p-s}{s-k}}{ \binom{p}{s}}\) for \(0 \leq k \leq s\), then \(E[Y] = \frac{s^2}{p}\) and \(E\left[\exp\left(\lambda^2 Y\right) \right] \leq \left(1 - \frac{s}{p} + \frac{s}{p} e^{\lambda^2} \right)^s\) for \(\lambda \in \R\). 
    \end{lemma}
    
    \begin{lemma}[\cite{tsybakovLowerBoundsMinimax2009}]\label{lemma:chisquare_tv_bound}
        If \(P, Q\) are probability measures on a measurable space \((\mathcal{X}, \mathcal{A})\) with \(P \ll Q\), then \(d_{TV}(P, Q) \leq \frac{1}{2}\sqrt{\chi^2(P||Q)}\).
    \end{lemma}
   
    \begin{lemma}[\cite{tsybakovLowerBoundsMinimax2009}]\label{lemma:general_lower_bound}
        Suppose \(\Sigma \in \R^{p \times p}\) is a positive semi-definite matrix and \(\Theta \subset \R^p\) is a parameter space. Let \(P_{\theta}\) denote the distribution \(N(\theta, \Sigma)\). If \(\pi\) is a probability distribution supported on \(\Theta\), then 
        \begin{equation*}
            \inf_{\varphi}\left\{ P_0\left\{ \varphi = 1\right\} + \sup_{\theta \in \Theta} P_{\theta}\left\{\varphi = 0\right\} \right\} \geq 1 - \frac{1}{2} \sqrt{\chi^2(P_{\pi} || P_0)}
        \end{equation*}
        where \(P_{\pi} = \int_{\theta} P_{\theta} \, \pi(d\theta)\) and \(\chi^2(\cdot || \cdot)\) denotes the \(\chi^2\)-divergence.
    \end{lemma}
    
    \begin{lemma}[\cite{ingsterNonparametricGoodnessoffitTesting2003}]\label{lemma:Ingster_Suslina}
        Suppose \(\Sigma \in \R^{p \times p}\) is a positive definite matrix and \(\Theta \subset \R^p\) is a parameter space. Let \(P_{\theta}\) denote the distribution \(N(\theta, \Sigma)\). If \(\pi\) is a probability distribution supported on \(\Theta\), then 
        \begin{equation*}
            \chi^2(P_{\pi} || P_0) \leq E\left[\exp\left( \left\langle \theta, \Sigma^{-1} \widetilde{\theta} \right\rangle \right) \right] - 1
        \end{equation*}
        where \(\theta, \widetilde{\theta} \overset{iid}{\sim} \pi\). Here, \(P_{\pi} = \int_{\theta} P_\theta \pi(d\theta)\) and \(\chi^2(\cdot || \cdot)\) denotes the \(\chi^2\)-divergence.
    \end{lemma}
    \begin{proof}
        See the proof of Lemma 23 in \cite{liuMinimaxRatesSparse2021}.
    \end{proof}

    \printbibliography
\end{document}